\documentclass[10.6pt]{amsart}

\usepackage{slashed}
\usepackage{graphicx}
\usepackage{amsfonts}
\usepackage{amssymb}
\usepackage{amsmath}
\usepackage{amsxtra}
\usepackage{latexsym}

\newtheorem{theorem}{Theorem}[section]
\newtheorem{lemma}[theorem]{Lemma}
\newtheorem{corollary}[theorem]{Corollary}
\newtheorem{proposition}[theorem]{Proposition}
\newtheorem{definition}[theorem]{Definition}

\newtheorem{assumption}[theorem]{Assumption}

\theoremstyle{remark}

\numberwithin{equation}{section}

\def\T{\mathcal{T}}
\def\slP{\slashed{P}}

\def\Rp{\R_{0,p}}

\def\slp{\slashed{\pi}}
\def\Et{\mbox {Ext}}
\def\smu{\slashed{\mu}}
\def\cpi{\stackrel{\circ}\pi}

\def\hg{{\hat g}}

\def\hR{\hat{R}}
\def\tt{{t'}}

\def\bb{{\mathbf{b}}}

\def\er{\mbox{err}}
\def\Er{\mbox{Er}}

\def\sn{{\slashed{\nabla}}}

\def\bkpi{{{}^{(K)}\bar\pi}}

\def\ab{{\underline{\a}}}
\def\eh{\hat{\eta}}

\def\zb{{\underline{\zeta}}}
\def\bpi{{\bar{\pi}}}

\def\J{{\mathcal{J}}}

\def\Ab{{\underline{A}}}
\def\bM{\textbf{M}}
\def\M{{\mathcal{M}}}
\def\bT{{\textbf{T}}}

\def\bR{{\textbf{R}}}

\def\bd{{\textbf{D}}}
\def\ti{\tilde}

\def\bg{\mathbf{g}}

\def\hk{{\hat{k}}}
\def\I{{\mathcal I}}

\def\beaa{\begin{eqnarray*}}
\def\eeaa{\end{eqnarray*}}

\def\ba{\begin{array}}
\def\ea{\end{array}}
\def\d{\delta}
\def\be#1{\begin{equation} \label{#1}}
\def \eeq{\end{equation}}

\newcommand{\nn}{\nonumber}

\def\l{\langle}
\def\r{\rangle}

\def\nn{\nonumber}
\def\S{{\mathcal S}}

\def\S2{{\Bbb S}^2}

\def\A{\mathcal {A}}

\def\E{{\mathcal E}}

\def\U{{\mathcal U}}

\def\ze{{\zeta}}

\def\ub{\underline{u}}
\def\udb{\underline{\b}}
\def\Lb{\underline{L}}
\def\Lie{{\mathcal L}}
\def\tr{\mbox{tr}}

\def\D{{\mathcal D}}
\def\H{{\mathcal H}}

\def\N{{\mathcal N}}

\def\La{{\Lambda}}

\def\B{{\mathcal B}}
\def\F{{\mathcal F}}
\def\P{{\mathcal P}}
\def\R{{\mathcal R}}
\def\c{\cdot}
\def\hot{\widehat{\otimes}}
\def\sig{\sigma}
\def\s{\sigma}
\def\a{\alpha}
\def\b{\beta}
\def\e{\eta}

\def\ep{\epsilon}
\def\l{\langle}
\def\r{\rangle}
\def\ga{\gamma}
\def\Ga{\Gamma}
\def\la{\lambda}

\def\p{\partial}
\def\P{{\mathcal P}}

\def\nab{\nabla}
\def\F{{\mathcal{F}}}

\def\C{{\mathcal C}}

\def\Lb{{\underline{L}}}

\def\div{\mbox{\,div\,}}
\def\curl{\mbox{\,curl\,}}

\def\tr{\mbox{tr}}
\def\Tr{\mbox{Tr}}

\def\itt{{\mbox{Int}}}

\def\f14{\frac{1}{4}}
\def\f12{{\frac{1}{2}}}

\def\t1a{t^{-\frac{1}{a}}}
\def\kp{{\kappa}}

\def\sl{\slashed}
\def\sD{\slashed{\Delta}}

\def\er{\mbox{err}}

\def\sn{{\slashed{\nabla}}}

\def\ab{{\underline{\a}}}
\def\eh{\hat{\eta}}

\def\zb{{\underline{\zeta}}}
\def\bpi{{\bar{\pi}}}

\def\J{{\mathcal{J}}}

\def\Ab{{\underline{A}}}
\def\bM{\textbf{M}}
\def\M{{\mathcal{M}}}
\def\bT{{\textbf{T}}}

\def\bR{{\textbf{R}}}

\def\bd{{\textbf{D}}}
\def\ti{\tilde}

\def\hk{{\hat{k}}}
\def\I{{\mathcal I}}

\def\beaa{\begin{eqnarray*}}
\def\eeaa{\end{eqnarray*}}

\def\ba{\begin{array}}
\def\ea{\end{array}}
\def\be#1{\begin{equation} \label{#1}}
\def \eeq{\end{equation}}
\def\nn{\nonumber}

\def\l{\langle}
\def\r{\rangle}

\def\nn{\nonumber}
\def\S{{\mathcal S}}

\def\S2{{\Bbb S}^2}

\def\A{\mathcal {A}}

\def\E{{\mathcal E}}

\def\U{{\mathcal U}}

\def\ze{{\zeta}}

\def\ub{\underline{u}}
\def\udb{\underline{\b}}
\def\Lb{\underline{L}}
\def\Lie{{\mathcal L}}
\def\tr{\mbox{tr}}

\def\D{{\mathcal D}}
\def\H{{\mathcal H}}

\def\N{{\mathcal N}}

\def\La{{\Lambda}}

\def\B{{\mathcal B}}
\def\F{{\mathcal F}}
\def\P{{\mathcal P}}
\def\R{{\mathcal R}}
\def\c{\cdot}
\def\hot{\widehat{\otimes}}
\def\sig{\sigma}
\def\s{\sigma}
\def\a{\alpha}
\def\b{\beta}
\def\e{\eta}

\def\ep{\epsilon}
\def\l{\langle}
\def\r{\rangle}
\def\ga{\gamma}
\def\Ga{\Gamma}
\def\la{\lambda}

\def\p{\partial}
\def\P{{\mathcal P}}

\def\nab{\nabla}
\def\F{{\mathcal{F}}}

\def\Lb{{\underline{L}}}

\def\div{\mbox{\,div\,}}
\def\curl{\mbox{\,curl\,}}

\def\tr{\mbox{tr}}
\def\Tr{\mbox{Tr}}

\def\itt{{\mbox{Int}}}

\def\f14{\frac{1}{4}}
\def\f12{{\frac{1}{2}}}

\def\t1a{t^{-\frac{1}{a}}}
\def\kp{{\kappa}}

\def\sl{\slashed}

\def\sD{\slashed{\Delta}}

\def\ckk{\check}

\def\RR{{\mathcal R}_0}

\def\Ric{\mbox{Ric}}
\def\hn{\widehat \nab}
\def\hdt{\hat \Delta}
\def\dum{\mbox{ }}
\newcommand{\bea}{\begin{eqnarray}}
\newcommand{\eea}{\end{eqnarray}}

\def\nn{\nonumber}
\def\be{{(e)}}
\def\bi{{(i)}}
\def\np{\N_{1,p}}
\def\nps{\N_{1,p}}

\def\G{\mathcal{G}}
\def\sz{\sn^{(0)}}
\def\gaz{\gamma^{(0)}}
\def\gas{\ga_{{{\Bbb
S}^2}}}
\def\dvz{\div^{(0)}}
\def\Gz{\Gamma^{(0)}}
\def\Pz#1{P_{#1}^{(0)}}
\def\tPz#1{\ti P_{#1}^{(0)}}
\def\schih{\sl{\chih}}
\def\ei{\E^{(1)}}

\newcommand{\chih}{\hat{\chi}}
\newcommand{\chib}{\underline{\chi}}

\newcommand{\chibh}{\underline{\hat{\chi}}\,}
\newcommand{\les}{\lesssim}

\newcommand\ovl{\overline}

\def\gac{\stackrel{\circ}\ga}

\begin{document}

\title[]
{Rough solutions of Einstein vacuum equations in CMCSH gauge}

\author{Qian Wang}
\address{Max Planck Institute for Gravitational Physics\\ 
Albert Einstein Institute\\
Am M\"{u}lenberg 1\\
D-14476 Golm\\
Germany}
\email{qwang@aei.mpg.de}
\curraddr{} \email{}

\date{}





\begin{abstract}
In this paper, we consider very rough solutions to Cauchy problem for
the Einstein vacuum equations in CMC spacial harmonic gauge, and obtain the local well-posedness result in $H^s, s>2$. The novelty of our approach lies in that,  without  resorting to the standard paradifferential regularization over the rough, Einstein metric $\bg$,  we manage to implement the commuting vector field approach to prove Strichartz estimate for  geometric wave equation  $\Box_\bg \phi=0$ directly.
\end{abstract}

\maketitle

\section{\bf  Introduction} \label{intr}
\setcounter{equation}{0}

In mathematical relativity, a fundamental question is to find a four dimensional
Lorentz metric $\bg$ that satisfies the vacuum Einstein equation
\begin{equation}\label{10.15.1}
{\bf Ric}(\bg)=0.
\end{equation}
Since the equation is diffeomorphic invariant, certain gauge should be fixed before
solving it.  There exist extensive works on (\ref{10.15.1}) under the wave condition
gauge or the constant mean curvature gauge.

In \cite{AnMon} Andersson and Moncrief consider the vacuum Einstein equation (\ref{10.15.1})
under the so-called constant mean curvature and spatial harmonic coordinate (CMCSH) gauge
condition. To set up the framework, let $\Sigma$ be a $3$-dimensional compact, connected
and orientable smooth manifold, and let ${\mathcal M}:={\mathbb R}\times \Sigma$. Let
$t: {\mathcal M}\to {\mathbb R}$ be the projection on the first component and let $\Sigma_t:=
\{t\}\times \Sigma$ be the level sets of $t$.  One may construct solutions of
(\ref{10.15.1}) by considering Lorentz metrics $\bg$ of form
\begin{equation*}
\bg=-n^2 dt\otimes dt +g_{ij} (dx^i +Y^i dt) \otimes (dx^j+Y^j dt)
\end{equation*}
with suitable determination of the scalar function $n$, the vector field $Y:=Y^j \p_j$ and the
Riemannian metric $g:=g_{ij} dx^i\otimes dx^j$ on $\Sigma$. In order for $\p_t$ to be
time-like, it is necessary to have $n^2-g_{ij} Y^i Y^j>0$. Let $\bT$ be the time-like normal to
$\Sigma_t$, then
\begin{equation*}
\p_t =n \bT + Y.
\end{equation*}
We call $n$ the lapse function and $Y$ the shift vector field.

Let $\hg$ be a fixed smooth Riemannian metric on $\Sigma$ with Levi-Civita connection $\hn$ and Christoffel
symbol $\hat{\Gamma}_{ij}^k$. Let $\Ga_{ij}^k$ denote the Christoffel symbol with
respect to $g$, we may introduce the vector field $U=U^l \p_l$ with
\begin{equation*}
U^l:=g^{ij} (\Ga_{ij}^l -\hat{\Ga}_{ij}^l).
\end{equation*}
Let $k$ be the second fundamental form of $\Sigma_t$ in ${\mathcal M}$,
i.e. $k_{ij}=-\frac{1}{2} \Lie_\bT \bg_{ij}$. The solution of (\ref{10.15.1}) constructed in
\cite{AnMon} is to find the pair $(g, k)$ such that they satisfy the CMCSH condition
\begin{equation}\label{CMCSH}
\Tr k:=g^{ij} k_{ij} =t\quad \mbox{and} \quad U^j=0
\end{equation}
and the vacuum Einstein evolution equations
\begin{align}
\p_t g_{ij}&=-2n k_{ij}+\Lie_Y g_{ij}\label{eqn8}\\
\p_t k_{ij}&=-\nab_i \nab_j n+n(R_{ij}+\Tr k k_{ij} -2 k_{im}
k_j^m)+\Lie_Y k_{ij}\label{eqn9}
\end{align}
with the constraint equations
\begin{equation}\label{constraint}
R-|k|^2+(\Tr k)^2=0 \quad \mbox{and} \quad
\nab_i \Tr k-\nab^j k_{ij}=0.
\end{equation}
It has been shown in \cite{AnMon} that for initial data $(g^0, k^0)\in H^s\times H^{s-1}$
with $s>5/2$ satisfying the constraint equation (\ref{constraint}) with $t_0:=\Tr k^0<0$,
the Cauchy problem for the system (\ref{CMCSH})--(\ref{constraint}) is locally well-posed.
In particular, there is a time $T_*>0$ depending on $\|g^0\|_{H^s}$ and $\|k^0\|_{H^{s-1}}$ such that
the Cauchy problem has a unique solution defined on $[t_0-T_*, t_0+T_*]\times \Sigma$. We should mention that,
for the solution constructed in this way, the lapse function $n$ and the shift vector field $Y$ satisfy
the elliptic equations
\begin{align}
-\Delta n+|k|^2 n&=1 \label{n1}
\end{align}
and
\begin{align}
\Delta Y^i+R^i_j Y^j&=\left(-2n k^{jl}+2 \nab^j Y^l\right) U^i_{jl}
+2 \nab^j n k_{j}^i-\nab^i n k_j^j,\label{shi1}
\end{align}
where $U^i_{jl}$ is the tensor defined by
\begin{equation}\label{11.5.2}
U^i_{jl}:=\Ga^i_{jl}-\hat \Ga^i_{jl}.
\end{equation}

It is natural to ask under what minimal regularity on the initial data the CMCSH Cauchy problem
(\ref{CMCSH})--(\ref{constraint}) is locally well-posed. In this paper we prove the
following result which shows the well-posedness of the problem when the initial data is in
$H^s\times H^{s-1}$ with $s>2$.

\begin{theorem}[Main Theorem]\label{main1}
For any $s>2$, $t_0<0$ and $M_0>0$,  there exist positive constants $T_*$, $M_1$ and $M_2$
such that the following properties hold true:

\begin{enumerate}
\item[(i)] For any initial data set $(g^0, k^0)$ satisfying (\ref{constraint}) with $t_0:=\Tr k^0<0$
and $\|g^0\|_{ H^s(\Sigma_{t_0})}+\|k^0\|_{H^{s-1}(\Sigma_{t_0})}\le M_0$, there exists
a unique solution $(g, k)\in C(I_*, H^s\times H^{s-1})\times C^1(I_*, H^{s-1}\times H^{s-2})$
to the problem (\ref{CMCSH})--(\ref{constraint});

\item[(ii)] There holds
\begin{equation*}
\|\hn g, k\|_{L^2_{I_*} L_x^\infty}+\|\hn g, k\|_{L_{I_*}^\infty
H^{s-1}}\le M_1;
\end{equation*}

\item[(iii)] For $2<r\le s$, and for each $\tau \in I_*$
the linear equation
\begin{equation*}
\left\{\begin{array}{lll}
\Box_\mathbf{g} \psi=0, \quad&(t,x)\in I_*\times \Sigma\\
 \psi(\tau, \cdot) =\psi_0\in H^r(\Sigma), &\p_t \psi(\tau,
\cdot)=\psi_1\in H^{r-1}(\Sigma)
\end{array}\right.
\end{equation*}
admits a unique solution $\psi\in C(I_*, H^r)\times C^1(I_*, H^{r-1})$ satisfying the estimates
\begin{equation*} 
\|\psi\|_{L_t^\infty H^{r}}+\|\p_t \psi\|_{L_t^\infty
H^{r-1}}\le M_2\|\psi_0, \psi_1\|_{H^{r}\times H^{r-1}}
\end{equation*}
and
\begin{equation*}
\|\bd\psi\|_{L_t^2 L_x^\infty}\le M_2\|(\psi_0,
\psi_1)\|_{H^{r}\times H^{r-1}};
\end{equation*}
 \end{enumerate}
 where $I_*:=[t_0-T_*, t_0+T_*]$. 
 
\end{theorem}

We actually obtain a stronger result than Theorem \ref{main1}, which is contained in Theorem  \ref{mainformetric}.

\subsection{ Review and Motivation}

Since the pioneer work of Choquet-Bruhat \cite{Ch}, there has been extensive work
on the well-posedness of quasilinear wave equation
\begin{equation}\label{gew}
\left\{\begin{array}{lll}
\Box_{g(\phi)}\phi:=\p_t^2 \phi-g^{ij}(\phi)\p_i \p_j\phi =N(\phi, \p \phi), \\
\phi|_{t=0}=\phi_0,\quad \p_t \phi|_{t=0}=\phi_1
\end{array}\right.
\end{equation}
in ${\Bbb R}^{n+1}$, where the symmetric matrix $g^{ij}(\phi)$ is positive definite and smooth as a
function of $\phi$, and the function $N(\phi, \p\phi)$ is smooth in its arguments and is quadratic
in $\p \phi$. In view of the energy estimate
\begin{equation}\label{eqin1}
\|\p \phi(t)\|_{H^{s-1}}\les \|\p \phi(0)\|_{H^{s-1}}\c
\exp \left(\int_0^t  \|\p \phi(\tau)\|_{L_x^\infty} d\tau\right),
\end{equation}
the Sobolev embedding and a standard iteration argument, the classical result of Hughes-Kato-Marsden \cite{HKM} of
well-posedness in the Sobolev space $H^s$  follows for any $s>\frac{n}{2}+1$, where the estimate
on $\|\p \phi\|_{L_t^\infty L_x^\infty}$ is heavily relied. To improve the classical
result, it is crucial to get a good estimate on $\|\p \phi\|_{L_t^1 L_x^\infty}$. This is
naturally reduced to deriving the Strichartz estimate for the wave operator $\Box_{g(\phi)}$ which has
rough coefficients since $g^{ij}(\phi)$ depend on the solution $\phi$ and thus at most have as much regularity as $\phi$.
The first important breakthrough was achieved by Bahouri-Chemin \cite{BC2} and by Tataru \cite{T1}
 using parametrix constructions. They obtained the well-posedness of (\ref{gew}) in $H^s$
with $s>\frac{n}{2}+\f12+\frac{1}{4}$ by establishing a Strichartz estimate of the following form,  if $\p g\in L_t^2 L_x^\infty$,
\begin{equation*}
\|\p \phi\|_{L_I^2 L_x^\infty}\le c(\|\phi_0\|_{H^{\frac{n}{2}+\f12+\sigma}}+\|\phi_1\|_{H^{\frac{n}{2}-\f12 +\sigma}})
\end{equation*}
with a loss of $\sigma>\frac{1}{4}$, for solutions to linearized equations. This well-posedness result
was later improved to $s>\frac{n}{2}+\f12+\frac{1}{6}$ in \cite{T3}.

The next important progress was made by Klainerman in \cite{KCom} where a vector field
approach was developed to establish the Strichartz estimate. This approach
was further developed by Klainerman-Rodnianski in \cite{KRduke} where they successfully
improved the local well-posedness  of (\ref{gew}) in ${\mathbb R}^{3+1}$
to the Sobolev space $H^{s}$ with $s>2+\frac{2-\sqrt{3}}{2}$.  Due to the limited
regularity of the coefficients, the paradifferential localization procedure in
\cite{BC2,T2,T3} was adopted in \cite{KRduke} to consider the Strichartz estimate for
solutions of linearized wave equation $\Box_{g_{\le \la^a}}\psi=0$ for some $0<a\le 1$, where
$g_{\le\la^a}:=S_{\la^a}(g(S_{\la^a}(\phi)))$ is the truncation of $g(\phi)$ at the frequency level $\la^a$.
Here $S_\la:=\sum_{\mu\le \la} P_\mu$ and $P_\la$ is the Littlewood-Paley projector  with
frequency $\la=2^k $ defined for any function $f$ by
\begin{equation}\label{LP2012}
P_\la f(x)=f_\la(x)=\int e^{-i x\c \xi} \ze(\la^{-1}\xi) \hat f(\xi) d\xi
\end{equation}
with $\zeta$ being a smooth function supported in the shell $\{\xi: \f12\le |\xi|\le 2\}$
satisfying $\sum_{k\in {\Bbb Z}}\zeta(2^k \xi)=1$ for $\xi\ne 0$. We refer to \cite{Stein2,KRsurf} for detailed properties of Littlewood-Paley decompositions.
With the help of a $\T\T^*$ argument, such Strichartz estimate was reduced to the
dispersive estimate for solutions of $\Box_{g_{\le \la^a}} \psi=0$ with frequency localized
initial data. It was then further reduced  to deriving the boundedness of
Morewatz type energy for $\p \psi$ and its higher derivatives. To derive these energy estimates requires us to  control deformation tensor of Morawetz vector field, which involves the Ricci coefficients
relative to the Lorentzian metric $-d t^2 +(g_{\le \la^a})_{ij} dx^i\otimes dx^j$. Since  $\bf{Ric}$ appears crucially in the structure equations for Ricci coefficients, deep observations and techniques were developed to control $\bf{Ric}$ relative to the smoothed metric, such as taking advantage of the facts that  the coefficients $g$ themselves verify
equations of the form (\ref{gew}),  and the observation that $\bR_{44}$, the tangential component of $\bf{Ric}$ along null hypersurfaces, has better structure.
For Einstein vacuum equation under the wave coordinates gauge, the local well-posedness
were obtained in $H^s$ for any $s>2$ in \cite{KR1,KR2,KRd}.
 The core progress which enables the
improvement from $s>2+\frac{2-\sqrt{3}}{2}$ to $s>2$ was made in \cite{KRd} by showing that
the Ricci tensor relative to the frequency-truncated metric ${\bf h}:=\bg_{\le \la}$
does not deviate from $0$ to a harmful level; the decay rate of $\bf{Ric}(h)$  and its
derivatives were proven to be sufficiently strong in terms of $\la$. However, similar
estimates for $\bf{Ric}(h)$ can hardly be obtained for (\ref{gew}).  The sharp local well-posedness for type (\ref{gew}) in $H^s$ with $s>2$ was achieved by
Smith and Tataru in \cite{SmTT} based on the wave packet
parametrix construction, with an innovative application to represent the nontruncated metric $\bg$.
The particular
structure of $\bR_{44}$ observed in \cite{KRduke} also played important role
to control the geometry of null surface. The local well-posedness  with $s=2$ for Einstein vacuum equation was conjectured by Klainerman in \cite{uni}. Recently we learn that significant progress has been achieved for this so-called $L^2$ curvature conjecture \cite{psc}.

A reduction to consider  $\Box_{\bg_{\le \la^a} }\psi=0,$ with $ 0<a\le 1$  appeared in
all the above mentioned work. This regularization on metric is used to phase-localize
the solution, and in most of the works, to balance the differentiability on coefficients
required either by parametrix construction or by energy method.
 Such a regularization on metric,
nevertheless, poses major technical baggage, in particular, to carry out the vector field approach in  Einstein vacuum spacetime, since
$\bf{Ric}(\bg_{\le \la})$ no longer vanishes. The  analysis in \cite{KRd} on the defected Ricci tensor
and its derivatives is a very delicate procedure, which relies crucially on full force of $\p \bf{h}$, hence, on their non-smoothed counter part $\p \bf{g}$ as well. One particular issue tied to CMCSH gauge itself arises due to the lack of control on $\bd_\bT Y$, the time derivative
of the shift vector field. Although $\bd_\bT Y$  satisfies an elliptic equation, that equation is
not good enough to provide a valid control on $\bd_\bT Y$ even in terms of $L^2$-norm.
The loss of control over some components of $\p {\bf g}$ becomes a serious hurdle in
recovering the decay for ${\bf{Ric}}({\bf h})$ and its derivatives.
The potential issue on Ricci defect forces us to abandon the frequency truncation on metric.

 The  important aspect of our analysis is to implement the vector field approach
directly in the non-smoothed Einstein spacetime $(\M,\bg)$ to establish the Strichartz estimate with
an arbitrarily small loss for the linearized problem $\Box_\bg \psi=0$.  This confirms that,
 due to the better behavior of $\bf{Ric}$,  Einstein metric is in nature ``smooth"
enough to implement the vector field approach without
 the truncation on $\bg$ in Fourier space, and  leads to the $H^s$ well-posedness result with $s>2$
for Einstein equation.

Note that in \cite{KRduke}, \cite{KR1}-\cite{KRd},  deriving the bounded Morawetz type energy of derivatives of $\psi$, with $\Box_{\bg_{\le\la^a}}\psi=0$, is the main building block to obtain the dispersive estimate for $P \p_t \psi$ required by Strichartz estimate. This procedure relies on $H^\s, \s>\frac{1}{2}$ norm of curvature, which is impossible to be obtained relative to the rough and non-smoothed metric.  Our strategy is to derive the dispersive estimate merely using Morawetz type energy for $\psi$ itself. The analysis to control such energy  is mainly focused on the Ricci coefficients relative to Einstein metric. Although part of such analysis benefits from $\bf{Ric}=0$, the crucial estimates such as  strichartz type norm $\|\chih\|_{L_t^2 L_x^\infty}$ require the bound of $H^{\s}, \s>1/2$ for $\hn k$ and $\hn^2 g$ if only the  classic $L_x^\infty$ Calderson-Zygmund inequality is employed. We solve this problem by modifing Calderson Zygmund inequality followed by taking advantage of the extra differentialbity for $\hn g, k$ that can be obtained by Strichartz estimates.
The  difficulty coming from $\bd_\bT Y$ still penetrates in key steps in the vector fields approach, where all components of $\p \bg$ were typically involved. We exclude such term by modifying the standard treatments including modifying energy momentum tensor,  refining $\T\T^*$ argument and curvature decomposition into more invariant fashion.

 Our approach can be directly applied for reproducing $H^{2+\ep}$ result for Einstien equations in wave coordinates gauge. It actually works better under wave coordinates since  $\bd_\bT Y$ can be well controlled in this situation.  Steps which are involved with getting around this term in CMCSH gauge, such as energy estimate, estimate for flux, $\T\T^*$ argument and decompositions for curvature,  take simpler and more straightforward form in wave coordinate gauge. Thus our approach gives a vast simplification over the the methodology in \cite{KR1}-\cite{KRd}.

\subsection{Outline of the proof}

According to \cite{KR2, SmTT}, in order to complete the proof of Theorem \ref{main1} it
suffices to show that for any $s>2$ there exist two positive constants $C$ and $T$
depending on $\|g\|_{H^s(\Sigma_0)}$ and $\|k\|_{H^{s-1}(\Sigma_0)}$ such that
\begin{equation}\label{main10}
\|g\|_{L_t^\infty H^s(I\times \Sigma)} +\|k\|_{L_t^\infty H^{s-1}(I\times \Sigma)}\le C,
\end{equation}
where $I:=[t_0-T, t_0+T]$. We achieve this by a bootstrap argument.
That is, we first make the bootstrap assumption
\begin{equation}\label{ba1}\tag{BA1}
\int_{t_0-T}^{t_0+T} \|\hn g,\, k,\, \hn Y,\, \hn n \|_{L^\infty(\Sigma_t)} dt \le B_1,
\end{equation}
where, for any $\Sigma$-tangent tensor $F$, we will use $\|F\|_{L^\infty(\Sigma_t)}$ to denote
its $L^\infty$-norm with respect to the Riemannian metric $g$ on $\Sigma_t$.
We then show that (\ref{ba1}) and some auxiliary bootstrap assumptions imply (\ref{main10}). 
We prove these bootstrap assumptions can be improved for small but universal $T>0$.

We will only work on the time interval $[t_0, t_0+T]$ since the same procedure applies to the time
interval $[t_0-T, t_0]$ by simply reversing the time.
In view of (\ref{ba1}) and elliptic estimates, we derive in Section \ref{h2es} better
estimates for $\hn Y $ and $\hn n$. That is, we show that, for any $1< b<2$, there holds
\begin{equation*}
\|\hn Y, \hn n \|_{L^{b}_{[t_0,t_0+T]} L_x^\infty}\le C
\end{equation*}
which improves the estimates for $\hn n$ and $\hn Y$ in (\ref{ba1}) with
$T$ sufficiently small. In order to improve the estimates for $\hn g$ and $k$, we
establish the core estimates in Theorem \ref{main1}(ii) by showing that
\begin{equation*}
\|\hn g,\, k\|_{L^2_{[t_0, t_0+T]}L_x^\infty}\le C T^\delta,
\end{equation*}
for some $\delta>0$. Here we briefly describe the ideas behind the proof.

\subsubsection{Step 1.  Energy estimates and flux}

In Sections \ref{h2es} and \ref{sec3}, we derive (\ref{main10}) under bootstrap assumptions. 
We also derive for the scalar
solution of homogeneous geometric wave equation $\Box_\bg \phi=0$, the
energy estimate
\begin{equation*}
\|\p \phi\|_{H^{s-1}}\les \|\phi(0)\|_{H^s}+\|\p \phi(0)\|_{H^{s-1}}.
\end{equation*}
The typical argument relies on the estimate of $\|\p \bg\|_{L_t^1 L_x^\infty}$, including  the one for $\bd_\bT Y$ in this norm.
In view of (\ref{shi1}), $\bd_\bT Y$ satisfies the elliptic equation
\begin{align*}
\Delta \nab_{n\bT} Y^i+R_j^i \nab_{n\bT} Y^j &-2 U_{mp}^{\dum\dum i}
\nab^m \nab_{n\bT} Y^p=-n (\curl H)_j^i Y^j+\bg\c \nab \ti\pi\c \ti\pi,
\end{align*}
where $\ti \pi$ denotes components of $\p \bg$ excluding $\p_t Y$.
Due to the appearance of the term $R_j^i \nab_{n\bT} Y^j$, this elliptic equation is
not good enough to provide valid control for $\bd_\bT Y$ unless the space
metric $g$ has negative sectional curvature. In order to avoid the difficulty
coming from $\bd_\bT Y$, in Section \ref{sec3} we derive the energy estimate by
considering the first order hyperbolic system
\begin{equation}\label{lie3}
\left\{\begin{array}{lll}
\p_t u-\hn_Y u=n v+F_u\\
\p_t v-\hn_Y v=n \hdt u+F_v
\end{array}\right.
\end{equation}
for the pairs $(u,v)=(g, -2k), (k, E)$ and $(\phi, e_0\phi)$ with corresponding
remainder terms $(F_u, F_v)$, where, for any $\Sigma$ tangent tensor $F$,
\begin{equation}\label{11.13.1}
\hdt F:= g^{ij} \hn_i \hn_j F.
\end{equation}
Consistent with these energy estimates, we also obtain the $L^\infty_t H^{s-\f12}$
and $L^1_t H^s$ estimates of $\hn n, \hn Y, \bd_\bT n$
with the help of elliptic equations (\ref{n1}) and (\ref{shi1}).

However, to  derive the  flux estimate for $\hn g,$ $k, $ $P_\mu k,$ and $P_\mu \hn g$,
we still have to rely on the second order hyperbolic system
of $k$ and $\hn g$, both of which contain the time derivative of the shift vector
field. This issue is solved in Section \ref{flx} by introducing a
modified energy momentum tensor.

As the major technicality to carry out energy estimates  in fractional Sobolev space and dyadic flux,
a series of more delicate commutator estimates are established in  Appendix III (Section \ref{apiii})
on the Littlewood Paley projection and the rough metric, particularly to handle the
decreased differentiability of coefficients.

\subsubsection{ Step 2. Reduction to dyadic Strichartz estimates on frequency dependent time intervals}

By using the Littlewood-Paley decomposition, it is easy to reduce the proof of
Theorem \ref{main1} to establishing for sufficiently large $\la$ the estimates
\begin{equation}\label{11.3.1}
\|P_\la \hn g, P_\la k\|_{L_I^2 L_x^\infty}
\les \la^{-\d} |I|^{\f12-\frac{1}{q}}\|\hn g, k\|_{H^{s-1}(\Sigma_0)}
\end{equation}
and
\begin{equation}\label{plphi}
\| P_\la \p\phi\|_{L_I^2 L_x^\infty}
\les \la^{-\d} |I|^{\f12-\frac{1}{q}}\|\hn \phi, e_0 \phi\|_{H^{s-1}(\Sigma_0)}
\end{equation}
for any solution $\phi$ of the equation $\Box_\bg \phi=0$,
where $I=[t_0, t_0+T]$, $q>2$ is sufficiently close to $2$, and $\d>0$ is sufficiently close to $0$.

We reduce the proof of (\ref{11.3.1}) and (\ref{plphi}) to Strichartz estimates
on small time intervals. We pick a sufficiently small $\epsilon_0>0$ and partition $[t_0, t_0+T]$ into
disjoint union of subintervals $I_k:=[t_{k-1}, t_k]$ of total number $\les \la^{8\epsilon_0}$
with the properties that
\begin{equation}\label{11.8.1}
|I_k|\les \la^{-8\epsilon_0} T \quad \mbox{and}
\quad \|k, \hn g, \hn Y, \hn n\|_{L_{I_k}^2 L_x^\infty}\le \la^{-4\epsilon_0}.
\end{equation}
To explain our approach, we take the derivation of (\ref{plphi}) as an example.
We consider on each  $I_k$ the Strichartz norm for $P_\la \p \phi$. By commuting
$P_\la$ with $\Box_{\bg}$ we have
$$
\Box_\bg P_\la \phi=F_\la,
$$
where $F_\la =[\Box_\bg, P_\la]\phi$ which can be treated as phase-localized at level of $\la$
in certain sense although it is not frequency-localized.
We use $W(t,s)$ to denote the operator that sends $(f_0, f_1)$ to the solution of $\Box_{\bg} \psi=0$
satisfying the initial conditions $\psi(s)=f_0$ and $\p_t \psi(s)=f_1$ at the time $s$.
Using Duhamel principle followed by differentiation, we can represent $P_\la \p\phi$ as
\begin{equation}\label{plagen}
P_\la \p  \phi(t)= \p W(t, t_{k-1})P_\la  \phi[t_{k-1}]+\int_s^t\p W(t,s)(0,F_\la (s))ds,
\end{equation}
where we used the convention $\phi[t]:=(\phi(t), \p_t \phi(t))$.
Running a $\T \T^*$ argument leads to Strichartz estimate for one dyadic piece of $\p \psi$,
\begin{equation}\label{plagen2}
  \|P_\la \p \psi\|_{L^q_{I_k} L_x^\infty}\les\la^{\frac{3}{2}-\frac{1}{q}}\|\psi[0]\|_{H^1},
\end{equation}
where $q>2$ is sufficiently close to $2$.

Similar procedure was used in \cite{KR1} for $\Box_{\bg_{\le \la}} \phi=0$. Observe that the solution
of this homogeneous wave equation is frequency-localized at the level of $\la$ if the data is
localized in Fourier space at the dyadic shell $\{\xi: \frac{\la}{2}\le |\xi|\le 2\la\}$.
Therefore, the dyadic Strichartz estimates (\ref{plagen2}) can be applied directly to the
representation of $P_\la \p\phi$. Since we will work for the metric $\bg$ without frequency truncation,
the corresponding operator $W(t,s)$ does not preserve the frequency-localized feature of data.
The Strichartz estimate for $\p W(t, t_{k-1})P_\la  \phi[t_{k-1}]$ is no longer expected to be
obtained directly from (\ref{plagen2}).  We solve this problem in Section 4  by modifying (\ref{plagen}) with the help of  the reproducing property of the
Littlewood-Paley projections, i.e. $P_\la=\ti P_\la\ti P_\la$, as follows,
\begin{equation}\label{plagen1}
 P_\la \p \phi(t)=\ti P_\la \p W(t, t_{k-1})\ti P_\la \phi[t_{k-1}]
 + \int_s^t \ti P_\la \p W(t,s)(0,F_\la (s))ds.
\end{equation}
This makes it possible to apply (\ref{plagen2}). The effort then goes into piecing together the
result of dyadic Strichartz estimates over intervals $I_k$ with the help of (\ref{11.8.1}).
This trick would have successfully reduced the main estimates to dyadic strichartz estimate
for the solution of $\Box_\bg \phi=0$ on  one sub-interval $I_k$,  had the term of $\bd_\bT Y$
not appeared in $F_\la$. We then  refine (\ref{plagen1}) further by modifying the application of Duhamel principle.

\subsubsection{Step 3. Reduction to dispersive estimates and  boundedness theorem}

By rescaling coordinates as $(t, x)\rightarrow ((t-t_{k-1})/\la, x/\la)$,
we need only to consider (\ref{plagen2}) on $[0,t_*]\times \Sigma$ with $t_*\le \la^{1-8\ep_0}T$.
In view of a $\T\T^*$ argument, this essentially relies on the dispersive estimate
\begin{equation}\label{dispst1}
\|P\bd_\bT W(t,s) I[s]\|_{L_x^\infty}
\les \left((1+|t-s|)^{-\frac{2}{q}}+d(t)\right)\sum_{k=0}^m \|\hn^k I[s]\|_{L_x^1}
\end{equation}
with initial data  $I[s]=(\psi(s), \bd_\bT \psi(s))$ for all $0<s\le t_*$,
where $m$ is a positive integer, and $d(t)$  is a function satisfying
$\|d\|_{L^{\frac{q}{2}}}\les 1$ for  $q>2$ sufficiently close to $2$.

Let $\{\chi_J\}$  be a suitable partition of unity on $\Sigma$ supported on balls of
radius $1$ in rescaled coordinates. We localize the solution of $\Box_\bg \psi=0$
by writing $\psi(t,x)=\sum_J\psi_J(t,x)$, where $\psi_J(t,x)$ is the solution of
$\Box_{\bg} \psi_J=0$ with the initial data  $\psi_J[\tau_0]=\chi_J\c \psi[\tau_0]$.
We then reduce the derivation of (\ref{dispst1}) to proving that
\begin{equation}\label{dispst2}
\|P \bd_\bT \phi(t) \|_{L_x^\infty}\le \left(\frac{1}{{(1+|t-\tau_0|)}^{\frac{2}{q}}}+d(t)\right)
\sum_{k=0}^{m-2}\|\hn^{k}\phi[\tau_0]\|_{L^2},
\end{equation}
with $\phi$ the solution of $\Box_\bg \phi=0$ and with data supported within a unit ball at $\Sigma_{\tau_0}$. It then suffices to consider  (\ref{dispst2}) on $\J_0^+$, the causal
future of the support of $\chi_J$ from $t=\tau_0\approx 1$, where one can introduce optical
function $u$ whose level sets  are null cones $C_u$. Thus $\J^+_0$ can be foliated by
$S_{t,u}:=C_u\cap \Sigma_t$ and a null frame $\{L,\Lb,e_1, e_2\}$ can be naturally defined,
where $e_A, A=1,2$, are tangent to $S_{t, u}$. Using these vector fields and $\ub=2t-u$,
one can introduce the Morawetz vector field $K=\frac{1}{2} n(u^2 \Lb +\ub^2 L)$.
Consequently, for any function $f$, one can introduce the generalized energy
$$
\ti Q[f](t):=\int_\Sigma \bar Q(K,\bT)[f],
$$
where $\bar Q(K,\bT)[f]$ is defined by applying $X=K, Y=\bT, \Omega=4t$ to
\begin{equation}\label{gneng1}\mathbf{}
\bar Q(X,Y)[f] =Q(X,Y)[f]+\f12\Omega f Y(f)-\frac{1}{4} f^2 Y(\Omega)
\end{equation}
with  $Q_{\mu\nu}$ being the standard energy momentum tensor
$$
Q_{\mu\nu}:=Q[f]_{\mu\nu}=\p_\mu f \p_\nu f-\f12 \bg_{\mu\nu}(\bg^{\a\b}\p_\a f \p_\b f ).
$$

The typical energy method gives
\begin{equation}\label{eniden}
\ti Q[f](t)-\ti Q[f](\tau_0)= -\f12\int_{\J_0^+}{}^{(K)}\bar\pi_{\a\b}Q[f]_{\a\b}
+\int_{\Sigma \times I} \Box_{\bg} f\cdot K f+l.o.t,
\end{equation}
where, for any vector field $X$,  the deformation tensor ${}^{(X)}\pi_{\a\b}:=\Lie_X \bg_{\a\b}$ and ${}^{(K)}\bar\pi_{\a\b}:={}^{(K)}\pi_{\a\b}-4t\bg_{\a\b}$.
 By applying (\ref{eniden}) to $f=\bd_\bT \phi$, we consider to bound generalized energy $\ti Q[\bd_\bT \phi]$ in
terms of their initial values at $t=\tau_0\approx 1$.
Due to one bad term contained in  $\Box_\bg \bd_\bT \phi=[\Box_\bg, \bd_\bT]\phi$,
  the estimate of $\ti Q[\bd_\bT \phi]$ has to be coupled  with $ \ti Q[\bd_Z\phi]$ with $Z$ either $\Lb$ or $ e_A$, for which we need to control
$\|\bd{}^{(Z)}\pi\|_{L_t^1 L_x^\infty}$ and $\int_0^{t_*} \sup_{u}\|\bd{}^{(Z)}\pi\|_{L^2(S_{t,u})} dt$. Since  $\bd{}^{(Z)}\pi$ contains curvature terms, such estimates relative to non-smoothed metric can only be obtained under the assumption of $H^{\frac{5}{2}+\ep}$ on data. Similar regularity issue occurs for the estimates required for $\bd{}^{(\bT)} \pi$ due to the integration by part argument  employed to handle that bad term.
Therefore, we no longer expect to obtain the boundedness of the conformal energy for any derivative
of $\phi$, including the one for $\bd_\bT \phi$.

Our strategy is  to  control
$\|P\bd_\bT \phi(t)\|_{L_x^\infty}$ merely in terms of $\ti Q[\phi](t)$, with certain loss of decay rate and with error incorporated into $d(t)$ in (\ref{dispst2}). With $\varpi$  a cut-off function whose
support is in a so-called exterior region, our treatment concerning  the harder part, $P(\varpi\bd_\bT  \phi)$, starts with writing it as
$P(\varpi\bd_\bT  \phi)=P( \varpi L\phi)- P(\varpi N \phi)$ with $N$ the unit outward normal
vector fields on $S_{t,u}\subset \Sigma$. The first term is controlled by Bernstein inequality
and $\ti Q[\phi](t)$. The second term is treated in view of
\begin{equation}\label{trt1}
P(\varpi N\phi)=\varpi N^l \p_l P\phi+[P, \varpi N^l]\p_l \phi.
\end{equation}
The first term of (\ref{trt1})  is then related to $\ti Q[\phi]$ with the help of  Sobolev embedding and commutator estimates.
By using the machinery developed in Section \ref{apiii}, the treatment on the commutators involved in both terms in (\ref{trt1}) is reduced to  estimating $\|\p (\varpi N)\|_{L_x^\infty}$.
Note that $\p N $ can be expressed as $\bg\c (\chi, \zeta, \hn g, k)$,
  thus we need to establish estimates on $L_t^{\frac{q}{2}} L_x^\infty, \,q>2$ of Ricci coefficients $\chih, \zeta$ and $L^\infty$ estimate on $\tr\chi$.
  The components of
${}^{(K)}\bar\pi_{\a\b}$ in (\ref{eniden}) involve ${}^{(\bT)}\pi$, $\chi$,  $\zeta$ and other Ricci coefficients as well. By assuming suitable control on Ricci coefficients, the proof of boundedness theorem is given in Section \ref{sec5}. We accomplish this step by showing that (\ref{dispst2}) holds true with  $m=3$.

\subsubsection{Step 4. $L^{2+}$ type flux and Ricci coefficients}

The control of Ricci coefficients consistent with $H^2$ Einstein metrics
has been studied in \cite{KR1, Wang09, Wang1},  where a set of estimates concerning $\tr\chi, \chih, \zeta, \zb$ was achieved   in terms of  curvature flux, combined with flux of $k$ if null hypersurface is foliated by $S_t$, level sets of $t$.  Bearing the flavor of these works, in the situation when $H^{2+\ep}$ estimates
for $\bg$ can be established, we first manage to gain from the extra differentiability
of metric a slightly stronger flux type control for $\hn g, k$. We then obtain a
 stronger set of estimates on Ricci coefficients
in terms of the $L^{2+}$ type flux, which contains the
$L_u^\infty L_t^2 L_\omega^\infty$-norms for $\chih, \zeta$
and corresponding estimates for $\mu, \tr\chi-\frac{2}{n(t-u)}$. This enables us to carry out
delicate analysis such as the standard $L^p, 1<p\le\infty$
type Calderon Zygmund inequality on null hypersurfaces under rough metric.
In this procedure, thanks to working directly in vacuum spacetime, we no longer
encounter the technical baggage in  \cite{KR1}--\cite{KRd} posed by defected $\bf{Ric}(h)$.
Nevertheless, this set of estimates is  far from
sufficient to control $\|\cdot\|_{L_t^1 L_u^\infty L_\omega^\infty}$ norm of deformation
tensor ${}^{(K)} \pi$, which relies on  estimates for $\chih, \zeta$ in this norm.
The idea of  Klainerman-Rodnianski \cite{KR2} is to rely on the Hodge system such as the equation of (\ref{stc1})
\begin{equation*}
\div \chih=\frac{1}{2}\sn \tr \chi -\beta+\cdots.
\end{equation*}
The key points of deriving a strichartz type norm for $\chih$ lie in:

(1) To obtain a decomposition of the form $\b=\sn \ti\pi+\cdots$ so that
\begin{equation*}
\div \chih=\sn (\tr\chi-\frac{2}{n(t-u)})+\sn \ti \pi+\cdots
\end{equation*}
with $\ti\pi$ being certain components of $\p \bg$.  The new difficulty arising in our
situation is to exclude the time derivative of shift in
the decomposition for the null curvature component of $\b$. This issue is settled in Section \ref{decmp}.

(2) To employ  Calderon-Zygmund theorem  for Hodge system $\div F=\sn G+e$
\begin{equation}\label{czin}
\|F\|_{L^\infty(S)}\les\|G\|_{L_x^\infty(S)}\ln(2+r^{\frac{3}{2}-\frac{2}{p}}\|\sn
G\|_{L^p(S)})+r^{1-\frac{2}{p}}\|e\|_{L^p(S)},
\end{equation}
where $S=S_{t,u}$. 
Relative to a regularized metric $\bg_{\le\la}$, using (\ref{czin})
to estimate $\|F(t,\cdot)\|_{L_x^\infty}$ only leads to a loss of $\ln \la$, which can be easily
balanced after integration in a frequency dependent time subinterval. However, relative to non-smoothed
metric, the application of (\ref{czin}) relies on the norm of $\sup_{u}\|\sn \ti \pi\|_{L^{2+}(S_{t,u})}$,
which requires the boundedness of the $H^{\frac{3}{2}+}$ norm of $\hn g$ and $k$ that can only be achieved
under the assumption of $H^{s}, s>\frac{5}{2}$  on data. We fix this problem
by squeezing a bit more differentiability out of
$\hn g,\, k$ in  $L_t^2 L_x^\infty$, which  can be achieved by Strichartz estimate via a bootstrap argument.
To implement this idea, we  establish a modified Calderon-Zygmund inequality relative to the rough metric, which can be seen in Section \ref{czz}.
\\

\noindent{\bf Acknowledgement.}  Part of this work was done during a four-month visit to L'Institut des Hautes \'{E}tudes Scientifiques, Paris. The author would like to thank
Professors Sergiu Klainerman and Lars Andersson
for interesting discussions,  and Qinian Jin for carefully reading the manuscript and giving many very useful comments.

\section{\bf  $H^2$ estimates }\label{h2es}

We first derive some simple consequences of (\ref{ba1}) that will be used throughout this paper.

Let $X$ be an arbitrary vector field on $\Sigma$. We use $|X|_g$ and $|X|_{\hat{g}}$ to denote
the lengths of $X$ measured by $g$ and $\hat{g}$ respectively. It then follows from
(\ref{eqn8}) that
\begin{align*}
\p_t (|X|^2_g )&= Y^m \hn_m g_{ij}X^i X^j -2nk_{ij}X^i X^j +(g_{im} \hn_j Y^m +g_{mj}\hn_i Y^m ) X^i X^j.
\end{align*}
Therefore
\begin{equation*}
\left|\p_t |X|^2_g\right|\le \left(2|\hn Y| +|Y|_g |\hn g| +2n|k|\right)|X|^2_{g}.
\end{equation*}
In view of (\ref{n1}) and the maximum principle, we can derive that $0<n\le C$, where $C$ is a constant
depending only on $t_0$; see \cite[Section 2]{Wang10}. Recall that $|Y|_g\le n$. We thus have
\begin{equation*}
\left|\p_t |X|^2_g\right|\le C\left(|\hn Y|+|\hn g|+|k|\right)|X|^2_{g}.
\end{equation*}
This together with the bootstrap assumption (\ref{ba1}) implies $ C^{-1}|X|_{g(t_0)}\le |X|_{g(t)}
\le C|X|_{g(t_0)}$. Since $g(t_0)$ and $\hg$ are always equivalent on compact $\Sigma$, we therefore have
\begin{equation}\label{11.5.1}
C^{-1} \hat{g}\le g\le C \hat{g}, \quad \mbox{on } [t_0, t_0+T]\times \Sigma
\end{equation}
for some universal constant $C>0$.\begin{footnote}{We will always use $C$ to denote a universal constant that
depends only on the constant in the bootstrap assumptions, information on $g(t_0)$ and
$\|(g, k)\|_{H^{s}\times H^{s-1}(\Sigma_{t_0})}$. For two quantities $\Phi$ and $\Psi$ we will use
$\Phi\les \Psi$ to mean that $\Phi\le C\Psi$ for some universal constant $C$.}\end{footnote}
This equivalence between $g$ and $\hg$ on each $\Sigma_t$ gives us the freedom to use
$g$ or $\hg$ to measure the length of any $\Sigma$-tangent tensor.

Using (\ref{11.5.1}) and (\ref{ba1}), we can follow the arguments in \cite[Sections 2 and 3]{Wang10}
to derive that
\begin{align}
 & C^{-1}<n<C,\quad Q(t)\les C, \quad \|H, E, \Ric\|_{L_x^2}\le C, \quad \|\pi, \bd_{n\bT}n\|_{H^1}\le C\label{q1}\\
&\|\nab^3 n\|_{L_x^2}+\|\nab^2 \bd_{n\bT} n\|_{L_x^2}\les \|k\|_{L_x^\infty}, \label{q2}
\end{align}
where $\pi$ is the deformation tensor of $\bT$ with components $k$ and $\nab \log n$, $E$ and $H$ are
the electric and magnetic parts of spacetime curvature defined by $E_{ij}=\bR_{0i0j}$
and $H_{ij}={}^\star\bR_{0i0j}$ respectively,  and $Q(t)$ is the Bel-Robinson energy defined by
$$
Q(t)=\int_{\Sigma}\left(|E|_g^2+|H|_g^2\right) d\mu_g.
$$
As a consequence of (\ref{q2}), we have
\begin{align}
\|\nab n, \bd_{nT} n\|_{L_x^\infty}\les 1+ \|k\|_{L_x^\infty}^{\frac{3}{2}-\frac{3}{p}}, \quad 3<p\le 6\label{q3}.
\end{align}


Let us fix the convention that $F\ast G$ denotes
contraction by $g$ and $\cdot$ denotes either usual multiplication
or contraction by $\hat g$.

\begin{lemma}\label{shift1}
Under the spatial harmonic gauge, the shift vector field $Y$ satisfies the equation
\begin{equation}\label{rshift}
\hdt Y=\pi\ast U+\pi\ast \pi+g\c \hn g\c \hn g \c Y+g^3 \c
\hat{R}\c Y
\end{equation}
where $\hat{\Delta}$ is defined in (\ref{11.13.1}), $U$ is defined in (\ref{11.5.2}),
and $\hat R$ is the Riemannian curvature with respect to $\hat g$.
\end{lemma}

\begin{proof}
Straightforward calculation shows for any vector field $Y$ and tensor $F$ that
\begin{equation*}
\nab_j Y^i=\hn_j Y^i+U_{jq}^{i} Y^q,\quad\, \nab_j F_m^i=\hn_j F_m^i+U_{jp}^i F_m^p-U_{jm}^p F_p^i.
\end{equation*}
In view of the spatial harmonic gauge condition $U^i:=g^{jl} U^i_{jl}=0$,  we obtain
\begin{align*}
g^{mj} \nab_m \nab_j Y^i
&=g^{mj}\hn_m\hn_j Y^i+g^{mj} \hn_m U_{jq}^i Y^q + 2 g^{mj} U_{jq}^i \hn_m Y^q +g^{mj} U_{mp}^i U_{jq}^p Y^q.
\end{align*}
Recall the identity
\begin{equation}\label{dell2}
R_{jl}={\hat R}_{jl}+\hn_i U_{jl}^i -\hn_j U_{il}^i+U_{jl}^p U_{pi}^i-U_{il}^p U_{pj}^i
\end{equation}
which can be verified directly, we can obtain
\begin{align*}
\Delta Y^i+R_p^i Y^p &=\hdt Y^i +\Omega_p^i  Y^p + 2 g^{mj} U_{jq}^i \hn_m Y^q
+g^{mj} U_{mp}^i U_{jq}^p Y^q \\
& \quad \, +U\c U \c g\c Y+\hat{R}\c g\c Y,
\end{align*}
where
$$
\Omega_p^i = g^{mj} \hn_m U_{jp}^i +(\hn_m U_{pk}^m -\hn_p U_{mk}^m) g^{ki}.
$$
By using the expression of $U$, the commutation formula and the gauge condition
$U^p=g^{ij} U_{ij}^p=0$ we have
\begin{align*}
\Omega_p^i &= g\c\hn g\c \hn g+\f12 g^{ki} g^{ml} \left(\hn_m \hn_p g_{kl}-\hn_p \hn_m g_{kl}\right)\\
&\quad\, +\f12 g^{ki} g^{ml} \left(\hn_m \hn_p g_{kl}+\hn_p\hn_l g_{mk}-\hn_p\hn_k g_{ml} \right)\\
&=g\c \hn g\c \hn g + g\c g\c g \c \hR.
\end{align*}
Thus
\begin{align*}
\Delta Y^i+R_p^i Y^p &=\hdt Y^i+(g\c \hn g\c \hn g+g^3\c \hat R)\c Y
+ 2g^{mj}U_{jq}^i \hn_m Y^q \\
&\quad \, +g^{mj}U_{mp}^i U_{jq}^p Y^q.
\end{align*}
Combining this with (\ref{shi1}) gives
\begin{align*}
\hdt Y^i +2 g^{mj} U_{jq}^i  & \hn_m Y^q+g^{mj} U_{mp}^i U_{jq}^p Y^q\\
&=-2n k^{mj} U_{mj}^i+2\nab^m Y^j U_{mj}^i +2 \nab^m n k_m^i-\nab^i n k_m^m \\
& +g\c \hn g\c \hn g\c Y+g^3\c \hat{R}\c Y.
\end{align*}
Since $\nab^m Y^l U_{ml}^i =g^{mj} \nab_j Y^l U_{ml}^i=g^{mj} \hn_j Y^l U_{ml}^i +g^{mj}
U_{jq}^p U_{mp}^i Y^q$, we obtain
\begin{align*}
\hdt Y^i &=g^{mj} U_{mp}^i U_{jq}^p Y^q-2 n k^{mp} U_{mp}^i + 2 \nab^m n k_m^i -\nab^i n k_m^m  \\
& \quad \, +g\c \hn g\c
\hn g\c Y+g^3\c \hat{R}\c Y
\end{align*}
which is the desired equation.
\end{proof}

\begin{lemma}\label{Cor8.12.1}
For any $\Sigma$-tangent tensor field $F$, on each $\Sigma_t$ there holds
\begin{equation*}
\|\hn^2 F\|_{L_x^2}\les \|\hdt F\|_{L_x^2}+\|\hn g\c \hn
F\|_{L_x^2}+\|\hn F\|_{L_x^2}+\|F\|_{L_x^2}.
\end{equation*}
\end{lemma}
\begin{proof} Let $du_g$ denote the volume form induced by $g$ on $\Sigma_t$.
Then, under the spacial harmonic gauge, there holds $\hn_i(g^{ij}d\mu_g)=0$ (see \cite[Page 3]{AnMon}).
Thus,  by integration by part, we have
\begin{align*}
\int_{\Sigma} |\hn^2 F|^2_g d\mu_g&=\int_{\Sigma} g^{ij}
g^{pq} \hn_i \hn_p F^l \hn_{j}\hn_{q}F_l d\mu_g\nn\\
&=-\int_\Sigma \left(\hn_p F^l \hdt \hn_{q} F_l g^{pq}+g^{ij} \hn_i
g^{pq}\hn_p F^l \hn_{j}\hn_{q} F_l \right)  d\mu_g
\end{align*}
Here and throughout the paper we will $\hg$ to raise and lower the indices in tensors.
It is easy to check the following commutator formula
\begin{equation}\label{commt2}
\hdt \hn_{q}F_l -\hn_{q} \hdt F_l= g\c \hn g\c \hn^2 F+ g\c {\hat
R}\c \hn F+g \c \hn{\hat R}\c F.
\end{equation}
Therefore we can derive that
\begin{align*}
\int_{\Sigma} |\hn^2 F|^2_g d\mu_g & =\int_{\Sigma}\left(
-g^{pq} \hn_p F^l \hn_{q} \hdt F_l  +g\c \hn g\c \hn F\c \hn^2 F \right.\\
& \qquad \qquad \left.+(g\c {\hat R}\c\hn F+g\c \hn {\hat R}\c F)\hn F\right) d\mu_g
\end{align*}
where the first term is $\int_\Sigma \hdt F^l\hdt F_l d\mu_g $ by
integration by part.
\end{proof}

\begin{lemma}\label{lem6}
On each $\Sigma_t$ there hold
\begin{align}
\|\hn Y\|_{L_x^2}&\les \|\hn g\|_{L_x^2}+1\label{lm7}, \\
\|\hn^2 Y\|_{L_x^2}&\les \|(\pi, \hn Y, \hn g)\c \hn g\|_{L_x^2}+\|\hn
g\|_{L_x^2}+1, \label{lm6}\\
\|\hn^3 Y\|_{L_x^2}&\les \|\hn Y, \hn g\|_{H^1}(\|\hn g\|_{L_x^\infty}^{\frac{4}{3}} \|\hn g\|_{L_x^2}^{\frac{2}{3}}+\|\hn g\|_{L_x^\infty})\nn\\&+(\|\hn^2g \|_{L_x^2}+1) \c \|\hn g, \pi\|_{L_x^\infty}+\|\hn g\|_{L_x^2} +1.\label{lmm7}
\end{align}
\end{lemma}

\begin{proof}
Consider (\ref{lm7}) first. By using (\ref{rshift}) and $\hn_j(g^{ij} d\mu_g)=0$, we have
\begin{align*}
\|\hn Y\|_{L_x^2}^2&\approx \int g^{ij} \hn_i Y^l \hn_{j} Y_l d\mu_g
=\int -\hdt Y^l \c Y_l d\mu_g \\
& \les \|\pi \c \pi\|_{L_x^1}+\|\hn g\c
(\hn g, \pi)\|_{L_x^1}+1.
\end{align*}
In view of (\ref{q1}), we thus obtain (\ref{lm7}).

Next, by using (\ref{rshift}) we have
\begin{equation*}
\|\hdt Y\|_{L_x^2}\les \|(\pi, \hn g)\c \hn
g\|_{L_x^2}+\|\pi\|_{L_x^4}^2+1.
\end{equation*}
It then follows from Lemma \ref{Cor8.12.1} that
\begin{equation*}
\|\hn^2 Y\|_{L_x^2}\les \|(\pi, \hn Y, \hn g)\c \hn g\|_{L_x^2}+\|\hn
Y\|_{L_x^2}+1.
\end{equation*}
We thus obtain (\ref{lm6}) in view of (\ref{lm7}).

Finally we derive (\ref{lmm7}). From (\ref{rshift}) and (\ref{commt2}) we first have
\begin{align}\label{onedy}
\hdt \hn Y& =\hn \hdt Y+[\hdt, \hn]Y\nn\\
&=\hn \pi \c \hn g+ \pi \c \hn^2 g + \pi \c \hn g \c \hn g + \pi \c \hn \pi +\pi\c\pi\c \hn g \nn\\
& + \hn g \c \hn^2 g\c Y + \hn g\c \hn g \c \hn g \c Y +\hn g\c \hn g \c \hn Y  + g\c \hn g \c \hn^2 Y \nn\\
& +  g\c \hR\c \hn Y +g\c \hn \hR \c Y +\hn (g^3 \c \hR \c Y).
\end{align}
It is easy to see that the last three terms involving $\hR$ can be bounded by $\les 1+\|\hn g\|_{L_x^2}$.
Note that
\begin{align*}
&\|\pi\c \pi\c \hn g\|_{L_x^2}\les \|\hn g\|_{L_x^\infty}\|\pi\|_{L_x^4}^2\les \|\hn g\|_{L_x^\infty}, \\
&\|(\hn Y, \hn g, \pi)\c \hn g \c \hn g\|_{L_x^2}\les (\|\hn g, \hn Y\|_{H^1}+\|\pi\|_{L_x^6})\|\hn g\|_{L_x^2}^{\frac{2}{3}}
\|\hn g\|_{L_x^\infty}^{\frac{4}{3}}.
\end{align*}
Thus, using Lemma \ref{Cor8.12.1}, we can obtain (\ref{lmm7}).
\end{proof}

\subsection{Energy estimate for $\hn g$.}\label{sec2.1}

In order to proceed further, besides (\ref{ba1}) we also need the following bootstrap assumption
\begin{equation}
\|\hn g\|_{L^2_{[0, T]} L_x^\infty}+\|k\|_{L^2_{[0,T]} L_x^\infty} \le B_0. \label{BA2} \tag{BA2}
\end{equation}
which is a stronger version for the corresponding part in (\ref{ba1}). The verification of (\ref{ba1}) and
(\ref{BA2}) will be carried out in Section 4.

We first introduce some conventions. For any 2-tensors $u$ and $v$ we define
$$
\langle u, v\rangle:=\hat g^{ik} \hat g^{jl} u_{ij} v_{kl} \quad \mbox{and} \quad
\langle \hn u, \hn v\rangle_g:= g^{ij} \langle \hn_i u, \hn_j v\rangle.
$$
We will use $|u|^2:=\langle u, u\rangle$ and $|\hn u|_g^2 :=\langle \hn u, \hn v\rangle_g$.

In the following we will derive some estimates on $\hn g$ and the derivatives on $Y$.
By using the formula
$$
\Lie_Y u_{ij}=\hn_Y u_{ij} +u_{im} \hn_j Y^m +u_{mj} \hn_i Y^m
$$
for any $2$-tensor $u$ and the formula under the spatial harmonic gauge,
\begin{equation}\label{ricid}
R_{ij}=-\f12 \hdt g_{ij}+\widehat{\mbox{R}}_{ij}+g\c \hn g\c\hn g
\end{equation}
 we can derive from  (\ref{eqn8}) and (\ref{eqn9})
that the $2$-tensors $u:=g$ and $v:=-2k$ satisfy the hyperbolic system (\ref{lie3})
with $F_u$ and $F_v$ given symbolically by
\begin{equation}\label{fuv11}
F_u =u \cdot \hn Y \quad \mbox{and} \quad F_v=2\nab^2 n+n\c k\ast k+k\ast \hn Y.
\end{equation}
From (\ref{lie3}), the commutation formula (\ref{commt2}),
we can derive
\begin{align}
&\p_t \hn u- \hn_Y \hn u=n \hn v+F_{\hn u}\label{lie.5}\\
&\p_t\hn v-  \hn_Y \hn v=n \hdt\hn u+F_{\hn v}\label{lie.6}
\end{align}
where
\begin{align*}
F_{\hn u}&=\hn Y\c \hn u + Y \cdot \hR \c u + \hn n \c v+\hn F_u, \\
F_{\hn v}&=\hn Y\c \hn v + Y \cdot \hR \c v + \hn n \hdt u + \hn F_v \\
&  +n(g\c \hn g\c \hn^2 u+g\c \hR\c \hn u+g\c\hn \hR\c u).
\end{align*}
It is straightforward to derive that
\begin{equation*}
 \left\{\begin{array}{lll}
\hn^2 F_u= \hn^2 Y\c\hn g+g\cdot \hn^3 Y+\hn Y\c\hn^2 g,\\
\hn F_v=g\cdot \hn^2 Y\c k+ 2 \nab^3 n +g\cdot \hn k\c \hn Y+\hn (n g \cdot k\c k)
\end{array}\right.
 \end{equation*}
 and
\begin{align}
|\hn F_{\hn u}, F_{\hn v}|&\le |\hn^2 Y\c (k, \hn g)|+|\nab^3 n, \hn^3 Y|
+|(\hn Y, \nab n , k, \hn g)\c \hn (\hn g, k)|\nn\\
& \quad \, +|\hn n \c k\c (\hn g, k)|+|\hn g\c k \c\hn Y|+|\hn g\c k\c k| +|\nab^2 n \c k| \nn\\
& \quad \, +|k| +|\hn Y| +|\hn g| +1.
\label{d1lu3}
\end{align}

In order to derive the estimates, we use the energy introduced in \cite[Section 2]{AnMon}
\begin{equation}\label{energy0}
\E^{(0)}(t)=\E^{(0)}(u,v)(t):=\frac{1}{2} \int_{\Sigma} \left(|u|^2+|\hn u|_g^2+|v|^2\right) d\mu_g
\end{equation}
with $u=g$ and $v=-2k$.

\begin{proposition}\label{thm1}
Under the bootstrap assumption (\ref{ba1}), there holds
\begin{equation*}
\sup_{[t_0, t_0+T]}\|\hn g\|_{L^2(\Sigma_t)}\le C.
\end{equation*}
\end{proposition}

\begin{proof}
Recall that for any vector fields  $Z$ tangent to $\Sigma_t$ and any scalar function $f$ there holds
$\int_{\Sigma_t} \Lie_Z (f d\mu_g) =\int_{\Sigma_t} \mbox{div}(f Z) d\mu_g=0$. Therefore
\begin{align*}
\p_t \E^{(0)}(t) &= \frac{1}{2} \int_{\Sigma_t} (\p_t -\Lie_Y) \left\{(|u|^2 +|\hn u|_g^2+|v|^2) d\mu_g\right\}\\
& =\int_{\Sigma_t} \left[\l u, \p_t u-\hn_Y u\r + \l v, \p_t v-\hn_Y v\r
   +g^{ij} \l \hn_i u, (\p_t-\hn_Y)\hn_j u\r\right] d\mu_g\\
&\quad \, +\frac{1}{2} \int_{\Sigma_t} (\p_t g^{ij} -\Lie_Y g^{ij}) \l \hn_i u, \hn_j u\r d\mu_g\\
&\quad \, +\frac{1}{2} \int_{\Sigma_t} (|u|^2+|\hn u|_g^2 +|v|^2) (\p_t -\Lie_Y)(d\mu_g)
\end{align*}
By using (\ref{eqn8}) we have $\p_t g^{ij}-\Lie_Y g^{ij}=2n k^{ij}$ and
$(\p_t -\Lie_Y) (d\mu_g) =-n \Tr k d\mu_g$. These two identities together with (\ref{lie3}),
(\ref{lie.5}) and (\ref{lie.6}) give
\begin{align*}
\p_t \E^{(0)}(t) &=\int_{\Sigma_t}  \left( n \l u, v\r+\l u,  F_u\r +\l v, F_v\r + n k^{ij}\hn_i u\hn_j u\right) d\mu_g\\
&\quad \, +\int_{\Sigma_t}  g\cdot \left(\hn F_u \cdot \hn u +Y\c \hat R\c u\c
\hn u +\hn Y \c \hn u \c \hn u \right)d\mu_g\\
&\quad -\f12\int_{\Sigma_t} n \Tr k(|u|^2+|\hn u|_g^2+|v|^2)\}d \mu_g.
\end{align*}
In view of the bounds on $n$, $|Y|$ and $g$, we can derive that
\begin{equation}\label{e0eng}
\p_t \E^{(0)}(t) \les \left(\|k, \hn Y\|_{L^\infty}+1\right)\E^{(0)}(t) + \|\hn F_u\|_{L^2}
\|\hn u\|_{L^2}+\|v\|_{L^2}\|F_v\|_{L^2}.
\end{equation}
By using Lemma \ref{lem6} we have
\begin{align*}
\|\hn F_u\|_{L^2} &\le \|\hn Y\|_{L^\infty} \|\hn g\|_{L^2}+\|\hn^2 Y\|_{L^2}
\les (\|\hn Y, \hn g,\pi\|_{L^\infty}+1) \|\hn g\|_{L^2}+1.
\end{align*}
and
\begin{align*}
\|F_v\|_{L^2} & \les \|\nabla^2 n\|_{L^2} +\|k\|_{L^4}^2 +\|k\|_{L^6} \|\hn Y\|_{L^3}
\les 1+\|\hn^2 Y\|_{L^2}+\|\hn Y\|_{L^2}\\
& \les (\|\pi, \hn Y, \hn g\|_{L^\infty}+1) \|\hn g\|_{L^2}+1.
\end{align*}
Therefore
$$
\p_t \E^{(0)}(t) \les \left(\|\pi, \hn Y, \hn g\|_{L^\infty(\Sigma_t)} +1\right) \E^{(0)}(t) +1.
$$
This together with the bootstrap assumption (\ref{ba1}) gives $\E^{(0)}(t)\le \E^{(0)}(t_0)+1$
for all $t\in [t_0, t_0+T]$. The proof is thus complete.
\end{proof}

We now consider the energy $\E^{(1)}(t)=\E^{(0)}(\hn u, \hn v)$. From (\ref{e0eng}) it follows easily that
\begin{align*}
\p_t \E^{(1)}(t)\les \left (\|k, \hn Y\|_{L_x^\infty}+1\right)\E^{(1)}(t)
+\left(\|\hn F_{\hn u}\|_{L_x^2}+\|F_{\hn v}\|_{L_x^2}\right) \sqrt{\E^{(1)}(t)}.
\end{align*}
In view of  (\ref{d1lu3}), we derive
\begin{align*}
\|\hn F_{\hn u}, F_{\hn v}\|_{L_x^2}&\les \|\hn^2 Y\|_{L_x^2}\|\hn g, k\|_{L_x^\infty}+\|\hn^3 Y\|_{ L_x^2}+\|\nab^3 n\|_{L_x^2}\\&+\|\hn Y, \hn n , k, \hn g\|_{L_x^\infty}\|\hn(\hn g, k)\|_{L_x^2}+\|\hn g\|_{L_x^6}\|\hn Y\|_{L_x^\infty}\|k\|_{L_x^3}\\&+\|\hn g\|_{L_x^\infty}\|k\|^2_{L_x^4}+\|\pi\|_{L_x^6}^3+\|\nab^2 n\|_{L^2_x} \|k\|_{L_x^\infty}\\
&+\|k\|_{L^2_x} +\|\hn Y\|_{L_x^2} +\|\hn g\|_{L_x^2}+1.
\end{align*}
In view of (\ref{q1}), (\ref{q2}),  Proposition \ref{thm1},  and (\ref{lm7}) and (\ref{lmm7}) in Lemma \ref{lem6},  we then have
\begin{align*}
\|\hn F_{\hn u}, F_{\hn v}\|_{L_x^2}
&\les \left(\|\hn g, k,\hn Y, \hn n  \|_{L_x^\infty}+\||\hn g\|_{L_x^\infty}^{\frac{4}{3}}\right)
\left(\|\hn^2 g\|_{L_x^2}+1\right)\\
&+\|\hn Y\|_{H^1}\left(\|\hn g, k\|_{L_x^\infty}+\|\hn g\|_{L_x^\infty}^{\frac{4}{3}} +1\right).
\end{align*}
Using this estimate, the inequality $\|\hn^2 g\|_{L_x^2}\le \sqrt{\E^{(1)}(t)}$, and the Young's inequality,
we can obtain
\begin{align*}
\p_t \E^{(1)}(t) &\les \left(1+\|k, \hn n, \hn g, \hn Y\|_{L_x^\infty} +\|k, \hn g\|_{L_x^\infty}^2\right) \E^{(1)}(t) \\
& \quad \, + \|k, \hn n, \hn g, \hn Y\|_{L_x^\infty} +\|\hn g\|_{L_x^2}^2 +\|\hn^2 Y\|_{L_x^2}^3+1
\end{align*}
In view of (\ref{ba1}) and (\ref{BA2}), it follows easily that
\begin{equation*}
\E^{(1)}(t)\les \E^{(1)}(t_0) +1 +\|\hn^2 Y\|_{L_t^3 L_x^2}^3.
\end{equation*}
This in particular implies that
\begin{equation}\label{hatn2g}
\|\hn^2 g\|_{L^2(\Sigma_t)}\les 1+ \|\hn^2 Y\|_{L_t^3 L_x^2}^{3/2}.
\end{equation}
On the other hand, it follows from (\ref{lm6}), (\ref{q1}), and Proposition \ref{thm1} that
\begin{align*}
\|\hn^2 Y\|_{L_x^2}&\les  \|\hn Y\c \hn g\|_{L_x^2}+\|\hn g\|_{L_x^4}^2+\|\pi\|_{L_x^4}^2 +1 \\
&\les\|\hn Y\|_{L_x^6}^{\f12}\|\hn Y\|_{L_x^2}^{\f12}\|\hn g\|_{L_x^6}+\|\hn g\|_{L_x^4}^2+1.
\end{align*}
Using $\|\hn Y\|_{L_x^6}\les \|\hn^2 Y\|_{L_x^2}+1$ and (\ref{lm7}) we can obtain
\begin{align*}
\|\hn^2 Y\|_{L_x^2} &\les  \|\hn g\|_{L_x^6}^2+ 1\les \|\hn^2 g\|_{L_x^2}^2 +1.
\end{align*}
This together with (\ref{hatn2g}) gives
\begin{equation}\label{12.23.1}
\|\hn^2 Y\|_{L_x^2}\les 1+ \|\hn^2 Y\|_{L_t^3 L_x^2}^3.
\end{equation}
Integrating with respect $t$ over $[t_0, t_0+T]$ yields
$$
\|\hn^2 Y\|_{L^3_{[t_0, t_0+T]} L_x^2}
\les T^{\frac{1}{3}} \left(1 +\|\hn^2 Y\|_{L_{[t_0, t_0+T]}^3 L_x^2}^3\right).
$$
Therefore we can choose a small but universal $T>0$ such that
$\|\hn^2 Y\|_{L^3_{[t_0, t_0+T]} L_x^2}\le C$ for some universal constant $C$.
Consequently, by using (\ref{hatn2g}) and (\ref{12.23.1}) we can obtain
(\ref{gh2}) and (\ref{cor3}) in the following result.

\begin{proposition}\label{cor2}
Under the bootstrap assumption (\ref{ba1}) and (\ref{BA2}), there hold
\begin{align}
\|g\|_{H^2(\Sigma_t)} +\|\hn k\|_{L^2(\Sigma_t)} &\le C,\label{gh2}\\
\|\hn^2 Y\|_{L^2(\Sigma_t)} +\|\hn Y\|_{L^2(\Sigma_t)} & \le C,\label{cor3}\\
\|e_0 (\hn g)\|_{L^2(\Sigma_t)}+\|\p_t \hn g\|_{L^2(\Sigma_t)} & \le C\label{eg}
\end{align}
for all $t\in [t_0, t_0+T]$ with $T>0$ being a universal number, where, for any $\Sigma_t$-tangent tensor
field $F$, we use the notation $e_0(F):=n^{-1} (\p_t F -\hn_Y F)$.
\end{proposition}

\begin{proof}
It remains only to prove (\ref{eg}). We use (\ref{lie.5}), (\ref{gh2}) and (\ref{cor3}) to deduce that
 \begin{align*}
\|e_0 (\hn g)\|_{L^2} &\les \|\hn k\|_{L^2}+ \|\hn Y\|_{L^6} \|\hn g\|_{L^3}
+\|\hn^2 Y\|_{L^2} \\
& +\|\hn n\|_{L^4} \|k\|_{L^4} + \|Y\|_{L^\infty} \|g\|_{L^2} \les 1.
\end{align*}
Finally, in view of (\ref{gh2}) we obtain $\|\p_t \hn g\|_{L^2} \les \|e_0(\hn g)\|_{L^2}
+\|\hn^2 g\|_{L^2} \les 1$.
\end{proof}

\begin{lemma}
Under the bootstrap assumptions (\ref{ba1}) and (\ref{BA2}), for $3<p\le 6$ there hold
\begin{align}
\|\hn^3 Y\|_{L^2} &\les  \|\hn g, k\|_{L^\infty}+1 , \label{shift3}\\
\|\hn Y\|_{L^\infty} &\les \|k, \hn g\|_{L^\infty}^{3/2-3/p}+1. \label{shift3.3}
\end{align}
\end{lemma}

\begin{proof}
In view of Proposition \ref{cor2}, (\ref{q1}), and Lemma \ref{Cor8.12.1}, we obtain from (\ref{onedy}) that
\begin{equation*}
\|\hn^3 Y\|_{L^2}\les \|\hn \pi, \hn g\|_{L_x^2}\|\hn g, \pi\|_{L_x^\infty}+ \|\hn^2 Y\c \hn g\|_{L^2} + \|\hn g, k\|_{L_x^\infty}+1.
\end{equation*}
We may write
$$
\|\hn^2 Y\c \hn g\|_{L^2}\les \|\hn^2 Y\|_{L^3} \|\hn g\|_{L^6}\les \|\hn^2 Y\|_{L^3} \|\hn^2 g\|_{L^2}.
$$
Applying the Sobolev type inequality  (\cite[Lemma 2.5]{Wang10}) to $\|\hn^2 Y\|_{L^3}$, and
using (\ref{q1}), (\ref{q2}), (\ref{gh2}) and (\ref{cor3}), we can obtain (\ref{shift3}).
Finally we can use the Sobolev embedding given in \cite[Lemma 2.6]{Wang10} to conclude (\ref{shift3.3}).
\end{proof}

\section{\bf $H^{2+\ep}$ estimates}\label{sec3}


In this section, under the bootstrap assumptions (\ref{ba1}) and (\ref{BA2}), we will
establish $H^{1+\ep}$ type energy estimates for $k, \hn g$ and $\bd \phi$ with $\phi$ being
solutions of homogeneous wave equation $\Box_\bg \phi=0$. We will also obtain the
$H^{\frac{3}{2}+\ep}$ and $H^{2+\ep}$ estimates for $\hn n, \hn Y, n e_0 (n)$ simultaneously.
As the main  building block of this section, established in Appendix III are a series of product
estimates in fractional Sobolev spaces and estimates for commutators between the Littlewood-Paley
projections $P_\mu$ and the rough coefficients.

For simplicity of exposition, we fix some conventions. We will use $\ti\pi$ to denote any term from the set
$\hn Y$, $\hn n$, $k$, $\hn g$ and $e_0(n)$, where $e_0(n)=n^{-1}(\p_t n-\hn_Y n)$ as defined before.
It follows from Proposition \ref{cor2} and (\ref{q1})  that $\|\ti\pi\|_{H^1}\le C$.
We also introduce the error terms
\begin{eqnarray}
 \er_1=\bg\c \ti \pi\c \hn\ti\pi, \qquad \er_2=\bg\c \ti\pi\c\ti\pi\c\ti\pi,\label{ertp}
\end{eqnarray}
where $\bg$ denotes any product of the components of $n$, $g$ and $Y$. We denote
by $\er(\hR)$ any term involving $\hR$ and its derivatives satisfying
\begin{equation}\label{er3}
\|\er(\hR)\|_{H^1(\Sigma_t)}\le C
\end{equation}
for all $t\in [t_0, t_0+T]$.

\begin{proposition}\label{eplem}
For $0<\ep<1/2$ there hold
\begin{align}
&\|\La^{1/2+\ep}\big( \hn^2 n, \hn (ne_0(n)), \hn^2 Y\big)\|_{L^2}\les \|\hn g, k\|_{H^{1+\ep}}+1, \label{elp.1}\\
&\|\La^{\ep} \big( \hn^3 n, \hn^2( ne_0 (n)), \hn^3 Y\big)\|_{L^2}\les
\|\hn g, k\|_{L^\infty} \|\hn g,k\|_{H^{1+\ep}}+1\label{ertp3}
\end{align}
and, for the error type terms defined in (\ref{ertp}), there hold
\begin{align}
&\|\La^{\ep}\er_1\|_{L^2}\les \|k, \hn g\|_{L^\infty} (\|\hn g, k\|_{H^{1+\ep}}+1) +1, \label{ertp1}\\
&\|\La^{\ep} \er_2\|_{L^2}\les  \|\hn g, k\|_{H^{1+\ep}}+1. \label{ertp2}
\end{align}
\end{proposition}

\begin{proof}
For any scalar function $f$ it is easy to derive the commutation formula
\begin{equation}\label{comm9}
[\Delta, \nab_{n\bT}] f=-2n k_i^l \nab_l \nab^i f-\nab^i n k_i^l
\nab_l f.
\end{equation}
To obtain the estimates of $\hn n$ and $n e_0(n)$, we first use (\ref{n1}) and (\ref{comm9})
to derive the identities
\begin{align}
\hdt \hn n &=\hn(n|k|^2_g)+g \hR\c\hn n, \label{onedn}\\
\hdt \big(n e_0(n)\big)&=n e_0(n |k|_g^2)-2n k_a^l \nab_l \nab^a n-\nab^a n
\nab_l n k_a^l.\label{onedtn}
\end{align}
In view of (\ref{onedn}), we have
\begin{equation}\label{hdtnn}
 \hdt \hn n= n \hn  k
\c k \c g+(\hn n,k,\hn g)^3\c (n, g)+\er(\hR).
 \end{equation}
It then follows from (\ref{ebeq.1}) that
\begin{align*}
&\|\La^{-1/2+\ep}\hdt \hn n \|_{L^2}\les \|\hn k \|_{H^\ep}\|n g \c k\|_{H^1}
+\|(\hn n,k,\hn g)^3\c (n, g)\|_{L^2}+\|\er(\hR)\|_{L^2}.
\end{align*}
By using (\ref{elep2}) and $\|\ti \pi\|_{H^1}\le C$, we can conclude that
\begin{equation}\label{1n2}
\|\La^{1/2+\ep}\hn^2 n\|_{L^2}\les \|\hn k\|_{H^\ep}+1.
\end{equation}

In view of (\ref{onedtn}) and (\ref{lie3}), we have
\begin{equation}\label{hdtn2}
\hdt \big(n e_0 (n) \big)=(n\hdt g+\nab^2n) ng\c k+\bg\c \ti \pi\c \ti \pi\c \ti \pi.
\end{equation}
Thus, with the help of (\ref{ebeq.1}) and (\ref{1n2}), it follows
\begin{align*}
\|\La^{-1/2+\ep}\hdt(n e_0 (n))\|_{L^2}
&\les (\|\hn^2 g\|_{H^\ep}+\|\hn^2 n \|_{H^\ep}) \|\bg \c k\|_{H^1} +\|\bg |\ti\pi|^3\|_{L^2} \\
& \les \|\hn^2 g,\hn k\|_{H^\ep}+1
\end{align*}
which implies, in view of (\ref{elep2}) and $\|\ti \pi\|_{H^1}\le C$, that
\begin{equation}\label{e0nes1}
 \|\La^{1/2+\ep}\hn\big( n e_0(n)\big)\|_{L^2}\les \|\hn^2 g, \hn k\|_{H^\ep}+1.
\end{equation}

Next we use (\ref{onedy}) and (\ref{ebeq.1}) to obtain
$$
\|\La^{-1/2+\ep}\hdt \hn Y\|_{L^2}\les \|\hn (\pi, \hn g, \hn Y)\|_{H^\ep}+\|\er(\hR)\|_{L^2}.
$$
This together with (\ref{1n2}), (\ref{elep2}), $\|\ti \pi\|_{H^1}\le C$  and the interpolation inequality gives
\begin{equation}\label{n2yes1}
\|\La^{1/2+\ep}\hn^2 Y\|_{L^2}\les \|\hn^2 Y\|_{H^\ep}+\|\hn^2 g, \hn
k\|_{H^\ep}+1\les\|\hn^2 g, \hn k\|_{H^\ep}+1.
\end{equation}
Combining the estimates (\ref{1n2}), (\ref{e0nes1}) and (\ref{n2yes1}), we therefore complete the
proof of (\ref{elp.1}).

As a byproduct of (\ref{elp.1}), we have
\begin{equation}\label{igel1}
\|\La^\ep \hn \ti\pi\|_{L^2}\les \|k, \hn g\|_{H^{1+\ep}}+1.
\end{equation}
It then follows from (\ref{epeq1}) and (\ref{igel1}) that
\begin{equation}\label{ertp.1}
\|\La^\ep (\ti \pi\c \hn \ti \pi)\|_{L^2}\les \|\ti \pi\|_{L^\infty}\| \ti\pi\|_{H^{1+\ep}}
\les( \|k, \hn g\|_{H^{1+\ep}}+1)\|\ti \pi\|_{L^\infty}.
\end{equation}
By Lemma \ref{epeq2} and (\ref{igel1}), we have
\begin{equation}\label{ertp.2}
\|\La^\ep(\ti\pi\c \ti\pi\c \ti \pi)\|_{L_x^2}
\les \|\ti \pi\|_{H^1}^2\|\ti \pi\|_{H^{1+\ep}}\les \|k, \hn g\|_{H^{1+\ep}}+1.
\end{equation}
To treat the factor $\bg$ in the definition (\ref{ertp}), in view of $\|\bg\|_{H^2}\le C$ in
Proposition \ref{cor2}, using Lemma \ref{prd3}, and (\ref{ertp.1}) and (\ref{ertp.2}), we thus obtain (\ref{ertp1}) and (\ref{ertp2}).

Finally, we consider (\ref{ertp3}) with the help of (\ref{elep1}). Let $F=n e_0(n), \hn n, \hn Y$. Then
it follows from (\ref{q1}) and (\ref{cor3}) that $\|F\|_{H^1}\les 1$. Moreover,  the elliptic equations
(\ref{hdtnn}), (\ref{hdtn2}) and  (\ref{onedy}) can be written symbolically as
$
\hdt F=\er_1+\er_2+\er(\hR).
$
In view of (\ref{ertp1}), (\ref{ertp2}), and the definition of $\er(\hR)$, we thus obtain (\ref{ertp3}).
\end{proof}

\subsection{First order hyperbolic systems}

\subsubsection{Energy estimates}

We consider a pair of tensors  $(u, v)$ satisfying the first order hyperbolic system
\begin{equation}\label{lu3}
 \left\{\begin{array}{lll}
\p_t u-\hn_Y u=n v+F_u\\
\p_t v-\hn_Y v=n \hdt u+F_v.
\end{array}\right.
 \end{equation}
Note that for  $(u,v)$ satisfying (\ref{lu3}), the pair  $(U_1, V_1)=(\hn u,\hn v)$ satisfies
an equation system of the form of (\ref{lu3}) with
\begin{equation}\label{fu1v1}
 \left\{
 \begin{array}{lll}
 F_{U_1}=\hn Y^m \hn_m u+\hn n\c  v+\hn F_u+Y\c \hat{R}\c u\\
 F_{V_1}=\hn Y^m \hn_m v+\hn n \c\hdt u+n \hat{R} \c \hn u+Y\c \hat{R}\c v+\hn F_v+n\hn (\hat{R}\c u)
 \end{array}\right.
 \end{equation}
where the last term in $F_{U_1}$ and $F_{V_1}$ can be dropped in the case that $(u,v)$ is a pair of scalar functions.

We also can check that the pair of functions $(U^\mu,V^\mu):=(P_\mu u, P_\mu v)$ satisfies (\ref{lu3})
with $F_{U^\mu}$ and $F_{V^\mu}$ given by
\begin{footnote}{We remark that the precise form of the first term in $F_{V^\mu}$ should
be $[P_\mu, ng\hn^2]u$ which consists of $[P_\mu, ng]\hn^2 u$ and $ng [P_\mu, \hat\Ga]\p u$.
The latter is of much lower order, which not only can be treated similar to the first term,
but also can be done in a much easier way since $\hat \Ga$ is smooth. Thus, we will omit this
term for ease of exposition.} \end{footnote}
\begin{equation}\label{puuv}
\left\{\begin{array}{lll}
F_{U^\mu}=[P_\mu, Y^m]\p_m u +[P_\mu, n]v+P_\mu F_u,\\
F_{V^\mu}=[P_\mu, ng]\hn_{ij}^2 u+P_\mu F_v+[P_\mu, Y^m]\p_m v.
\end{array}\right.
\end{equation}
Thus, it is easy to check that $(U_1^\mu,V_1^\mu):=(P_\mu \hn u, P_\mu \hn v)$ satisfies the system
(\ref{lu3}) with $F_{U_1^\mu}$ and $F_{V_1^\mu}$ given by
 \begin{equation}\label{puuv1}
\left\{\begin{array}{lll}
F_{U_1^\mu}=[P_\mu, Y^m]\p_m \hn u +[P_\mu, n]\hn v+P_\mu F_{U_1},\\
F_{V_1^\mu}=[P_\mu, ng]\hn_{ij}^2 \hn u+[P_\mu, Y^m]\p_m \hn v+P_\mu F_{V_1}.
\end{array}\right.
\end{equation}

\begin{lemma}
Let $0<\epsilon<1/2$. Then for $F_{U^\mu}$ and  $F_{V^\mu}$ defined by (\ref{puuv}) there hold the estimates
\begin{align}
\|\mu^{\f12+\ep}F_{U^\mu}\|_{l_\mu^2 L_x^2} & +\|\mu^{-\f12+\ep}\hn F_{U^\mu}\|_{l_\mu^2 L_x^2}
\les\|\hn n, \hn Y\|_{H^1}\|\hn u, v\|_{H^\ep}+\|\mu^{\f12+\ep}P_\mu F_u\|_{l_\mu^2 L_x^2}, \label{rem2.1}
\end{align}
\begin{align}
\|\mu^\ep & \hn F_{U^\mu}\|_{l_\mu^2 L_x^2}+\|\mu^{1+\ep} F_{U^\mu}\|_{l_\mu^2 L_x^2} \nn\\
&\les  \|\hn^2 Y, \hn^2 n \|_{H^{\f12+\ep}}\| \hn u, v\|_{L_x^2} +\|\hn n,\hn Y \|_{L_x^\infty}\|\hn u, v\|_{H^\ep}
+\|\mu^{\ep}\hn P_\mu F_u\|_{l_\mu^2 L_x^2}, \label{dfu}
\end{align}
and
\begin{align}
\| & \mu^\ep F_{V^\mu}\|_{l_\mu^2 L_x^2}\les\|\hn (n g), \hn Y\|_{L_x^\infty}\| \hn u, v\|_{H^\ep}
+\|\mu^{\ep}P_\mu F_ v\|_{l_\mu^2 L_x^2}. \label{fv}
\end{align}
\end{lemma}

\begin{proof}
 (\ref{dfu}) follows from (\ref{epeq3}),
(\ref{fv}) follows from (\ref{com3}), and  (\ref{rem2.1}) follows from (\ref{lem2eq}) and (\ref{lem4eq}).
\end{proof}

\begin{lemma}\label{dfu.1}
For $0<\ep<1/2$ there hold
\begin{align}
&\|\mu^{1+\ep} \hn F_{U^\mu}\|_{l_\mu^2 L_x^2}+\|\mu^\ep \hn F_{U_1^\mu}\|_{l_\mu^2 L_x^2}
     \les\|\hn^2 F_u\|_{H^\ep}+\I(\ep, u,v), \label{1dfu1}\\
&\|\mu^\ep F_{V_1^\mu}\|_{l_\mu^2 L_x^2}
    \les (\|\hn g\|_{L_x^\infty}+1)\|\hn^2 u\|_{H^\ep}+ \|\hn F_v\|_{H^\ep}+\I(\ep,u,v)+\|\hn u, v\|_{H^\ep}, \label{1dfu.2}
\end{align}
where
$$
\I(\ep, u,v)=\|\hn^2 Y, \hn^2 n \|_{H^{\f12+\ep}}\|\hn^2 u,\hn v\|_{L_x^2}
+\|\hn Y, \hn n\|_{L_x^\infty}\|\hn u, v\|_{H^{1+\ep}}.
$$
\end{lemma}

\begin{proof}
The first part of (\ref{1dfu1}) follows from (\ref{comm10}) and (\ref{puuv}). In order to prove
the second part of (\ref{1dfu1}), we may use the same argument for deriving (\ref{dfu}) to obtain
\begin{equation*}
\|\mu^\ep \hn F_{U_1^\mu}\|_{l_\mu^2 L_x^2}\les \I(\ep, u,v)+\|\mu^\ep \hn P_\mu F_{U_1}\|_{l_\mu^2 L_x^2}.
\end{equation*}
In view of (\ref{fu1v1}), we apply (\ref{prodd.1}) to obtain
\begin{align*}
\|\mu^\ep \hn P_\mu F_{U_1}\|_{l_\mu^2 L_x^2} &\les \I(\ep, u, v) + \|\mu^\ep \hn P_\mu \hn  F_u\|_{l_\mu^2 L_x^2}
+\|\La^\ep \hn(Y\cdot \hat{R}\cdot u)\|_{L^2}\\
&\les \I(\ep, u,v) + \|\mu^\ep \hn P_\mu \hn  F_u\|_{l_\mu^2 L_x^2}.
\end{align*}
Combining the above two estimates, we therefore obtain the second part of (\ref{1dfu1}).

Next we prove (\ref{1dfu.2}). We first apply Lemma \ref{lem1}  to derive that
\begin{align}\label{9.30.1}
\|\mu^\ep F_{V_1^\mu}\|_{l_\mu^2 L_x^2}
   &\les \|\hn(ng)\|_{L^\infty} \|\hn^2 u\|_{H^\ep} +\|\hn Y\|_{L^\infty} \|\hn v\|_{H^\ep}  + \|\mu^\ep P_\mu F_{V_1}\|_{l_\mu^2 L_x^2}.
\end{align}
By using Lemma \ref{prd3}  we have
\begin{align*}
\|\mu^\ep P_\mu F_{V_1}\|_{l_\mu^2 L_x^2}
   &\les \|\hn Y\|_{L_x^\infty}\|\hn v\|_{H^{\ep}}+\|\hn^2 Y\|_{H^{\f12+\ep}}\|\hn v\|_{L_x^2}
    +\|\hn n\|_{L_x^\infty}\|\hdt u\|_{H^{\ep}}\\
   &\quad\, +\|\hn^2 n\|_{H^{\f12+\ep}}\|\hdt u\|_{L_x^2} +\|n \hat{R}\c \hn u\|_{H^\ep}+\|Y \hat{R}\c v\|_{H^\ep}+\|\hn F_v\|_{H^\ep}.
\end{align*}
With the help of Lemma \ref{prd3} and $\|\ti \pi\|_{H^1}\le C$, we obtain
\begin{align*}
\|\hdt u\|_{H^\ep}&\les \|\mu^\ep g^{ij} P_\mu \hn^2_{ij} u\|_{l_\mu^2 L_x^2}
+ \|\mu^\ep [P_\mu, g^{ij}] \hn^2_{ij} u\|_{l_\mu^2 L_x^2}\\
&\les \|\hn^2 u\|_{H^\ep}+\|\hn g\|_{H^{\f12+\ep}}\|\hn^2 u\|_{L_x^2}\les \|\hn^2 u\|_{H^\ep}, \\
\|n \hat{R} \c \hn u\|_{H^\ep}&\les \|\mu^\ep[P_\mu, n\hat{R}]\hn u\|_{l_\mu^2 L_x^2}+\|\hn u\|_{H^\ep}\\
&\les \|\hn u\|_{H^{\ep}}+\|\hn (n \hat{R})\|_{H^{\f12+\ep}}\|\hn u\|_{L_x^2}\les \|\hn u\|_{H^\ep}
\end{align*}
Similarly, we have with the help of $\|\hn  Y\|_{H^1}\le C$ that
$
\|Y \c \hat{R} \c v\|_{H^\ep}\les \|v\|_{H^\ep}.
$
Therefore
\begin{align*}
\|\mu^\ep P_\mu F_{V_1}\|_{l_\mu^2 L_x^2}
   &\les \I(\ep, u,v) +\|\hn F_v\|_{H^\ep} + \|\hn u\|_{H^\ep}+\| v\|_{H^\ep}.
\end{align*}
Combining this estimate with (\ref{9.30.1}) we thus obtain (\ref{1dfu.2}).
\end{proof}

In the following we will derive the estimates on $\|\hn^2 g\|_{H^\ep}$ and $\|\hn k\|_{H^\ep}$. Recall
the energy $\E^{(0)}(u,v)$ defined in (\ref{energy0}). Let $P_\mu$ be the Littlewood-Paley projection with frequency size
$\mu$, we can introduce
\begin{equation*}
\ei_\mu(t)= \ei_\mu(u,v):=\E^{(0)}(P_\mu \hn u, P_\mu \hn v),
\end{equation*}
and  the energy
\begin{equation}\label{epengy.1}
\E^{(1+\ep)}(u,v)(t) := \sum_{\mu>1}\mu^{2\ep}\E_\mu^{(1)}(u,v)(t)+\sum_{i=0}^1\E^{(i)}(u,v)(t).
\end{equation}
In view of (\ref{e0eng}) we can derive that
\begin{align}\label{eimu}
\p_t\ei_\mu(t) &\le\left(\|k, \hn Y\|_{L^\infty}+1\right)\ei_\mu(t)+\|\hn F_{U_1^\mu}\|_{L^2} \|\hn U_1^\mu\|_{L^2}
 +\|V_1^\mu\|_{L^2}\|F_{V_1^\mu}\|_{L^2}.
\end{align}
Hence, by the Cauchy-Schwartz inequality, we obtain
\begin{align}
\p_t\left(\sum_{\mu}\mu^{2\ep}\ei_\mu(t)\right)
   &\le \left(\|k, \hn Y\|_{L^\infty}+1\right)\sum_{\mu}\mu^{2\ep}\ei_\mu(t)\nn\\
   &+\|\mu^\ep \hn F_{U_1^\mu} \|_{l_\mu^2 L^2}\|\mu^\ep \hn U_1^\mu\|_{l_\mu^2 L^2}+
\|\mu^\ep V_1^\mu\|_{l_\mu^2 L^2}\|\mu^\ep F_{V_1^\mu}\|_{l_\mu^2 L^2}\label{engep.1}.
\end{align}

We will apply (\ref{engep.1}) to the pair $(u,v)=(g, -2k)$ and (\ref{fuv11}) to derive energy estimates.
Lemma \ref{dfu.1} will be used to estimate the terms $\|\mu^\ep \hn F_{U_1^\mu} \|_{l_\mu^2 L^2}$
and $\|\mu^\ep F_{V_1^\mu}\|_{l_\mu^2 L^2}$ which involve terms related to $F_u$ and $F_v$.
The following result gives such estimates.

\begin{lemma}\label{L10.1.1}
For $0<\ep<1/2$ there hold
\begin{align}
\|\La^\ep \hn^2 F_u\|_{L_x^2}+\|\La^\ep \hn F_v\|_{L_x^2}
&\les (\|\hn g,\hn Y, k,\hn n\|_{L_x^\infty}+1)\|\hn g, k\|_{H^{1+\ep}}, \label{estfu.1}\\
\|\La^{\f12+\ep}\hn F_u\|_{L_x^2} &\les\|\hn g, k\|_{H^{1+\ep}}+1. \label{estfu.2}
\end{align}
\end{lemma}

\begin{proof}
Recall $F_u$ and $F_v$ from (\ref{fuv11}). By straightforward calculation, symbolically we have
\begin{align*}
\hn^2 F_u=g\c \hn^3 Y+ \er_1+\er_2, \qquad \hn F_v=\hn \nab^2 n+ \er_1+\er_2.
\end{align*}
where $\er_1$ and $\er_2$ denote the terms introduced in (\ref{ertp}). Then (\ref{estfu.1}) follows
from Proposition \ref{eplem}.
Applying  Lemma \ref{lem22} to $F=g$ and $G= \hn Y$, and  using (\ref{elp.1}) and (\ref{cor3}) we obtain
\begin{align*}
\|\La^{\ep+\f12} \hn(g\c \hn Y)\|_{L_x^2}&\les \|\hn g\|_{H^{1+\ep}}\|\hn Y\|_{H^1}
+\| g\|_{L^\infty } \|\hn^2 Y\|_{H^{\f12+\ep}} \les \|\hn g, k\|_{H^{1+\ep}}+1
\end{align*}
which gives (\ref{estfu.2}).
\end{proof}

\begin{proposition}\label{eg1}
For $0<\ep\le s-2$ there holds
\begin{align}
\|\hn^2 g (t)\|_{H^\ep} +\|\hn k(t)\|_{H^\ep} \le C \label{eng4}
\end{align}
and for any pair $(u,v)$ satisfying (\ref{lu3}) there holds
\begin{align}
\E^{(1+\ep)}(u,v)(t)&\les \E^{(1+\ep)}(t_0)
+\int_{t_0}^t \left(\| \hn u \|_{H^{1+\ep}}\| \hn F_u\|_{H^{1+\ep}}+
\|  v\|_{H^{1+\ep}}\|F_v\|_{H^{1+\ep}}\right). \label{geneng}
\end{align}
\end{proposition}

\begin{proof}
Now we consider the energy defined by (\ref{epengy.1}) for the pair $(u,v)=(g, -2k)$
by using (\ref{engep.1}). In view of (\ref{q2}), Propositions \ref{thm1} and Proposition \ref{cor2}, we have
\begin{equation}\label{e10}
 \E^{(1)}+\E^{(0)}\le C.
 \end{equation}
 Combining this fact with  (\ref{elp.1}), we have
$$
\I(\ep, u, v)\les \left(\|\hn Y, \hn n\|_{L^\infty} +1\right) \sqrt{\E^{(1+\ep)}(u, v)}.
$$
This together with Lemma \ref{dfu.1}, Lemma \ref{L10.1.1} and (\ref{engep.1}) implies
\begin{align*}
\p_t\left(\sum_{\mu}\mu^{2\ep}\ei_\mu(t)\right)
   &\le \left(\|k, \hn Y\|_{L^\infty}+1\right)\sum_{\mu}\mu^{2\ep}\ei_\mu(t)\nn\\
   &+\left(\|k, \hn g, \hn Y, \hn n\|_{L^\infty} +1\right) \E^{(1+\ep)}(u, v).
\end{align*}
which combined with (\ref{e10}) gives
\begin{align}
\p_t\left(\sum_{\mu}\mu^{2\ep}\ei_\mu(t)\right) \les \left(\|k, \hn g, \hn Y, \hn n\|_{L^\infty} +1\right) \left(\sum_{\mu}\mu^{2\ep}\ei_\mu(t)+1\right).
\end{align}
By the bootstrap assumption (\ref{ba1}), we obtain $\E^{(1+\ep)}(u,v)(t)\les \E^{(1+\ep)}(u,v)(t_0)+1$
for $t_0\le t\le t_0+T$ which implies (\ref{eng4}).

It remains to prove  (\ref{geneng}). From (\ref{eng4}), (\ref{elp.1}) and Sobolev embedding, it follows that
\begin{align}\label{ell.1}
\|\hn Y, &\hn n, ne_0 (n)\|_{L^\infty}
+\|\La^{1/2+\ep}(\hn^2 Y, \hn^2 n, \hn (ne_0(n)))\|_{L^2}\les \|\hn(\hn g, k)(t_0)\|_{H^\ep}+1.
\end{align}
Thus for any pair $(u,v)$ satisfying (\ref{lu3}) there holds $\I(\ep,u,v)\les \|\hn u, v\|_{H^{1+\ep}}$
on $\Sigma_t$. We will rely on Lemma \ref{dfu.1} to treat $\|\mu^\ep \hn F_{U_1^\mu} \|_{l_\mu^2 L^2}$ and
$\|\mu^\ep V_1^\mu\|_{l_\mu^2 L^2}$ in (\ref{engep.1}). We then obtain from (\ref{engep.1}) that
\begin{align}
\p_t\left(\sum_{\mu}\mu^{2\ep}\ei_\mu(t)\right) &\le \left(\|k, \hn Y,
   \hn g\|_{L^\infty}+1\right)\sum_{\mu}\mu^{2\ep}\ei_\mu(t)\nn\\
   &+(\|\hn^2 F_u\|_{H^\ep}+
\|\hn F_v\|_{H^\ep}+\|\hn u, v\|_{H^{1}})\sum_{\mu}\mu^{2\ep}\ei_\mu(t)\label{engep.12}.
\end{align}
 Now we consider lower order energy $\E^{(1)}(u,v)$,
 by using (\ref{fu1v1}) for the pair $(U_1, V_1)=(\hn u, \hn v)$ and
\begin{equation*}
\|\hn F_{U^1}, F_{V^1}\|_{L_x^2}\les \|\hn u, v\|_{H^1}+\|u\|_{L_x^2}+\|\hn^2 F_u\|_{L_x^2}+\|\hn F_v\|_{L_x^2}
\end{equation*}
which can be derived by using Sobolev embedding and (\ref{ell.1}).  By (\ref{e0eng}), we derive
\begin{align}
\p_t \E^{(1)}(u,v)(t) &\les \left(\|k, \hn Y\|_{L^\infty}+1\right)\E^{(1)}(u,v)(t)\nn\\
&+\|\hn^2 u, \hn v\|_{L_x^2}(\|\hn(\hn F_u, F_v)\|_{L^2}+\|\hn u, v\|_{H^1}+\|u\|_{L_x^2}).\label{ee1}
\end{align}
For $\E^{(0)}(u,v)(t)$, we employ  (\ref{e0eng}) again.
Combining (\ref{engep.12}), (\ref{ee1}), (\ref{e0eng}), Lemma \ref{dfu.1} and the
Gronwall inequality gives (\ref{geneng}).
\end{proof}

\subsubsection{Geometric wave operator}

For the pair $(u, v)$ satisfying (\ref{lu3}), we can show that $u$ satisfies the geometric wave equation
\begin{equation}\label{wave1}
n^2 \Box_\bg u=-(n F_v+n e_0 (F_u))+e_0(n) F_u-n^2 \pi_{0a} \nab^a u+n^2
\Tr k \, e_0 (u).
\end{equation}
Indeed,  relative to an orthonormal frame ${e_0:=\bT, e_j, j=1,2,3}$,
by  straightforward calculation we have
\begin{eqnarray*}
\bg^{ij}\bd^2_{ij}u=\Delta u+\Tr k \bd_\bT u, &&\, \bd_\bT \bd_\bT u
=e_0(e_0 (u))+\pi_{0j} \nab^j u.
\end{eqnarray*}
Since $\Box_\bg u=-\bd_\bT\bd_\bT u+\bg^{ij}\bd^2_{ij} u$, we obtain
\begin{equation}\label{eq2.1}
n^2 \Box_\bg u=-n^2 e_0(e_0 (u))+n^2 \hdt u -n^2 \pi_{0j} \nab^j u +n^2 \Tr k \, \bd_\bT u.
\end{equation}
In view of (\ref{lu3}) we have
\begin{align*}
n e_0(n e_0 (u)) =n e_0 (n)\c v+n^2 \hdt u+ n F_v+n e_0 (F_u).
\end{align*}
Combining this with (\ref{eq2.1}) and using (\ref{lu3}) we
obtain (\ref{wave1}) as desired.

Therefore, if $\psi$ is a solution of the geometric wave equation
\begin{equation}\label{wv2}
\Box_\bg \psi = W,
\end{equation}
we can check by (\ref{wave1}) that  $(u,v):=(\psi, e_0 (\psi))$ satisfies the hyperbolic system
\begin{equation}\label{wv3}
\begin{array}{lll}
e_0 (u) =v, &&\, ne_0 (v)=n \hdt u +F_v, \\
F_u=0,  &&F_v= -n W-n \pi_{0j}\nab^j u+n v \Tr k.
\end{array}
\end{equation}

\begin{lemma}
Let $\psi$ be a scalar function satisfying the geometric wave
equation $\Box_\bg \psi=0$. Then for $0<\ep\le s-2$ there holds
\begin{equation}\label{dfv}
\|\hn F_v\|_{H^{\ep}}\les \|\hn(\hn \psi, e_0(\psi))\|_{H^{\ep}}.
\end{equation}
\end{lemma}

\begin{proof}
Indeed,  in view of (\ref{prodd.1}) and $\Tr k=t$ we have
\begin{equation*}
\|\hn F_v\|_{H^\ep}\les \|\hn^2 n \|_{H^{\f12+\ep}}\|\hn^2 u\|_{L_x^2}
+\|\hn n \|_{L_x^\infty}\|\hn^2 u\|_{H^{\ep}}+\|\hn v\|_{H^\ep}.
\end{equation*}
Since $0<\ep\le s-2$, we have from (\ref{ell.1}) that
$\|\hn F_v\|_{H^\ep}\les \|\hn^2 u\|_{L_x^2}+\|\hn(\hn u, v)\|_{H^{\ep}}$ which gives the estimate.
\end{proof}


It is standard to derive for $\psi$  satisfying  $\Box_\bg \psi=0$ that
\begin{equation}\label{h1en}
\|\bd\psi(t)\|_{L^2}\les \|\bd\psi(t_1)\|_{L^2}, \quad t_0\le t_1\le t\le t_0+T.
\end{equation}
We now give the following energy estimate.

\begin{proposition}\label{geoeg}
Let $\psi$ be a scalar function satisfying the geometric wave
equation $\Box_\bg \psi=0$. Then for any $0<\ep<1/2$ and $t_0\le t_1\le t\le t_0+T$ there
hold the energy estimates
\begin{equation}\label{h1ena}
 \E^{(i)}(\psi, e_0(\psi))(t)\les \sum_{0\le j\le i}\E^{(j)}(\psi, e_0(\psi))(t_1), \quad i=0,1,
\end{equation}
and
\begin{equation*}
\E^{(1+\ep)}(\psi, e_0(\psi))(t)\les \E^{(1+\ep)}(\psi, e_0(\psi))(t_1)
\end{equation*}
\end{proposition}

\begin{proof}
Since $(u,v)=(\psi, e_0 (\psi))$ satisfies (\ref{wv3}) with $W=0$,
we can easily derive that
\begin{align*}
\|F_v\|_{L_x^2}&\les (\|\hn n \|_{L_x^\infty}+\|\Tr k\|_{L_x^\infty})\left(\E^{(0)}(u,v)\right)^{\f12}
\end{align*}
Recall that $F_u=0$, an application of (\ref{e0eng}) gives (\ref{h1ena}) with $i=0$.
The case $i=1$ can be proved by employing (\ref{e0eng}) with $(U^1, V^1):=(\hn \psi, \hn (e_0 (\psi)))$. Indeed,
in view of (\ref{fu1v1}) and (\ref{ell.1}), we have
\begin{align*}
\|F_{V^1}, \hn F_{U^1}\|_{L_x^2}&\les\|\hn F_v\|_{L_x^2}+(\E^{(1)}(u,v))^{\f12},
\end{align*}
and
\begin{align*}
\|\hn F_v\|_{L_x^2}&\les (\|\hn \nab n\|_{L_x^3}+\|\hn n \|_{L_x^\infty})\|\hn u\|_{H^1}+(\|\hn v\|_{L_x^2}+\|\hn n\|_{L_x^\infty} \|v\|_{L_x^2})\|\Tr k\|_{L_x^\infty}\\
&\les\left(\E^{(1)}(u, v)(t)\right)^{\f12}+\left(\E^{(0)}(u,v)(t)\right)^{\f12}.
\end{align*}
By substituting to (\ref{e0eng}),  we can complete the proof of (\ref{h1ena}). Using  (\ref{geneng}), (\ref{dfv}) and the Gronwall inequality, we can complete the proof of Proposition \ref{geoeg}.
\end{proof}

Consider either $(u,v)=(g, -2k)$ in view of  (\ref{lie3}), or  the solution of $\Box_\bg \psi=0$
in view of the reduction in  (\ref{wv3}) with $(u,v)=(\psi, e_0(\psi))$ and  $W=0$ therein.
For $0<\ep\le s-2$ we have obtained Proposition \ref{eg1} and Proposition \ref{geoeg}, which can be rewritten as
\begin{equation}\label{eng5}
\E^{({1+\ep})}(u,v)(t)\les  \E^{({1+\ep})}(u,v)(t_0)+1, \qquad  t_0\le t\le t_0+T.
\end{equation}

By the inequality (\ref{eimu}), we have
\begin{align*}
\sup_{t_0\le t\le t_0+T} \mu^{2\ep}\ei_\mu(t)&\le \mu^{2\ep}\ei_\mu(t_0)\nn\\
&+\int_{t_0}^{t_0+T} \left(\|\mu^\ep \hn F_{U_1^\mu}\|_{L^2} \|\mu^\ep \hn U_1^\mu\|_{L^2}
+\|\mu^\ep V_1^\mu\|_{L^2}\|\mu^\ep F_{V_1^\mu}\|_{L^2}\right) dt. 
\end{align*}
Using Lemma \ref{dfu.1} together with Young's inequality, we have
\begin{align*}
\left(\sum_{\mu}\sup_{t_0\le t\le t_0+T} \mu^{2\ep}\ei_\mu(t)\right)^{\f12}
& \les  \left(\E^{(1+\ep)}(t_0)\right)^{\f12}\\
& +\int_{t_0}^{t_0+T} \left(\|\mu^\ep P_\mu(\hn^2 F_u)\|_{l_\mu^2 L_x^2}+\|\mu^\ep \hn P_\mu F_v\|_{l_\mu^2 L_x^2}\right).
\end{align*}
We then conclude that

\begin{proposition}\label{remaindd}
For $0<\ep \le s-2$ there holds
\begin{align*}
\left(\sum_{\mu}\sup_{t_0\le t\le t_0+T} \mu^{2\ep}\ei_\mu(t)\right)^{\f12}
&\les \|\hn u,  v\|_{H^{1+\ep}(t_0)} \\
& +\int_{t_0}^{t_0+T} \left(\|\mu^\ep P_\mu(\hn^2 F_u)\|_{l_\mu^2 L_x^2}+\|\mu^\ep \hn P_\mu F_v\|_{l_\mu^2 L_x^2}\right).
\end{align*}
\end{proposition}

\subsection{Flux}\label{flx}

In view of (\ref{lie.5}) and (\ref{lie.6}), we can see that $(u,v):=(\hn g, -2\hn k)$ satisfies the
hyperbolic system (\ref{lu3}) with
\begin{equation}\label{8.20.0}
\left\{\begin{array}{lll}
F_u=g \hn^2 Y+ \hn Y\c \hn g + Y \c \hR \c g + \hn n \c k ,\\
F_v=\hn (2 \nabla^2 n + n * k* k +k * \hn Y) + \hn n \hdt g  + \hn Y \c \hn k + Y \c \hR \c k\\
\quad \quad\,\,\,\, + n (g \c \hn g\c \hn^2 g + g \c \hR \c \hn g + g \c \hn \hR \c g).
\end{array}\right.
\end{equation}
By straightforward calculation we have
$$
\left\{
\begin{array}{lll}
\hn  F_u=g\c \hn^3 Y+ \er_1+\er_2+\er(\hR),\\
 F_v=\hn^3 n+ \er_1+\er_2+\er(\hR).
\end{array}
\right.
$$

Next we give the first order hyperbolic system for the pair $(k, E)$.
Recall that (see \cite[(3.11a)]{AnF})
\begin{align}\label{le}
n^{-1}( \p_t -\Lie_Y )E_{ij}
&=\curl H_{ij}-n^{-1} (\nab n\wedge H)_{ij}-\frac{5}{2}(E\times k)_{ij} \nn\\
& \quad \,  -\frac{2}{3}(E\ast k) g_{ij}-\frac{1}{2}\Tr k E_{ij},
\end{align}
where, for any symmetric 2-tensor $F$,
\begin{equation}\label{cl1}
\curl F_{ab}=\frac{1}{2}(\ep_a^{\mbox{ }cd}\nab_d
F_{cb}+\ep_b^{\mbox{ }cd}\nab_d F_{ca}).
\end{equation}
with $\ep_a^{\mbox{ }cd}$ denoting the components of the volume form of $(\Sigma_t, g)$.
When $\div F=0$ and $\Tr F=t$, symbolically we can obtain the identity
\begin{equation}\label{dcurl}
\curl \curl F=-\hdt F+\Ric\ast F+\hn g\c \hn g\c g\c F
\end{equation}
In view of $\curl k=-H$, we can use (\ref{dcurl}) with $F=k$ to treat the term $\curl H$ in (\ref{le}).
Consequently we obtain
\begin{equation}\label{eqtE}
n^{-1}(\p_t-\Lie_Y)E_{ij}=\hdt k+\Ric\ast k+\hn g\c \hn g\c k\c g+\nab
n\ast H+k\ast k\ast k.
\end{equation}
By coupling (\ref{eqtE}) with (\ref{eqn9}), we can see that the pair $(u,v):=(k,E)$ satisfies the first order
hyperbolic system (\ref{lu3}) with
\begin{equation}\label{8.20.1}
\left\{\begin{array}{lll}
F_u=\nab^2_{ij}n+n k\ast k+k\c \hn Y, \\
F_v=n R\ast k+n \hn g\c\hn g \c g\c k +\nab n\ast H+n k\ast k\ast k+E\c \hn Y.
\end{array}\right.
\end{equation}
Using the Gauss equation $E=Ric+ k\star k$ and (\ref{ricid}) to treat $E$, we have
\begin{align}
\left\{\begin{array}{lll}
&\hn F_u= \hn^3 n+\er_1+\er_2+\er(\hR),\nn\\
 &F_v=\er_1+\er_2+\er(\hR). \nn
\end{array}\right.
\end{align}

In view of Proposition \ref{eplem}, we obtain

\begin{proposition}[Remainder estimates]\label{hepeg}
Let $(F_u, F_v)$ be defined by either (\ref{8.20.0}) or (\ref{8.20.1}). Then for
any $0<\ep<1/2$ there hold
\begin{equation}\label{eqnep1}
\left\{
\begin{array}{lll}
\|\hn  F_u\|_{H^\ep}+\|F_v\|_{H^\ep}\les (\|\hn g, k\|_{L^\infty}+1)\|\hn g, k\|_{H^{1+\ep}},\\
\|F_u\|_{H^{1/2+\ep}}\les\|\La^{\ep}\hn(k,\hn g)\|_{L^2}
\end{array}\right.
\end{equation}
and
\begin{equation}\label{tu0}
\|\hn F_u\|_{L^2}+\|F_v\|_{L^2}\les(\|\hn g, k\|_{L^\infty}+1)\|\hn g, k\|_{H^1}.
\end{equation}
\end{proposition}

We now fix a point $p$ in $\Sigma\times I$ and use $\Ga^+$ to denote the time axis passing through $p$ which
is defined to be the integral curve of the forward unit normal $\bT$ with $\Ga^+(t_p)=p$.
We use $\Ga_t$ to denote the intersection point of $\Ga$ with $\Sigma_t$. Let $u$ be the outgoing
solution of the eikonal equation $\bg^{\a\b}\p_\a u\p_\b u=0$ satisfying the initial condition
$u(\Ga_t)=t-t_p$ on the time axis. We will call this $u$ an optical function. We denote by $C_u$
the level sets of $u$ which are the outgoing null cones with vertex on the time axis.
Let $S_{t,u}=C_u\cap \Sigma_t$ and let $\{e_1, e_2\}$ be an orthonormal frame on $S_{t,u}$.
 Let $N$ be the exterior unit normal, along $\Sigma_t$,
to the surface $S_{t,u}$. We define
\begin{equation}\label{bb1}
\bb^{-1}:=\bT(u),\quad L:=\bT+N, \quad \Lb:=\bT-N=2\bT-L.
\end{equation}
We will call $L, \Lb$ the canonical null pair. Then $\{e_1, e_2, e_3:=\Lb, e_4:=L\}$ also forms a
null frame.

For $0<u\le t_0+T-t_p$ and $t_p\le t_1\le t_2\le t_0+T$, we introduce the null hypersurface
$\H:=\cup_{t_1\le t\le t_2}S_{t,u}$. We will use $D^+$ to denote the region enclosed by $\H$,
$\Sigma_{t_1}$ and $\Sigma_{t_2}$.  For any scalar function $\psi$ we introduce the flux
\begin{equation*}
\F[\psi]=\int_{\H}\left(|L\psi|^2+\ga^{AB}\sn_A \psi \sn_B \psi\right),
\end{equation*}
where $\ga$ is the induced metric on $S_{t,u}$ and $\sn$ is the corresponding covariant differentiation.
For any scalar functions $\phi$ and $\psi$ we introduce the energy-momentum tensor
\begin{equation}\label{momentum}
Q[\phi, \psi]_{\mu\nu}=\f12 (\bd_\mu \phi\bd_\nu \psi+\bd_\nu \phi \bd_\mu \psi)
-\frac{1}{2}\bg_{\mu \nu} (\bg^{\a\b} \bd_\a\phi \bd_\b\psi).
\end{equation}
Let $Q[\psi]:=Q[\psi, \psi]$ and define the energy
$$
Q(\psi)(t):=\int_{\Sigma_t} Q[\psi](\bT, \bT) d\mu_g.
$$
It is straightforward to check $\F[\psi]=2\int_{C_u} Q[\psi](\bT, L)$.

For any $\Sigma$-tangent tensor field $F_\mu$, we set
\begin{equation}\label{ddf1}
\nab_L F_i= e_i^\mu L^\nu \bd_\nu F_{\mu}, \qquad \sn_A F_i=e_i^\mu\nab_A F_\mu
\end{equation}
and introduce the norms
\begin{equation*}
|\nab_L F|_g^2:=g^{ij}\bd_L F_i \bd_L F_j, \qquad |\sn F|_g^2:=\ga^{AB}g^{ij}\sn_A F_i \sn_B F_j.
\end{equation*}
We will drop the subscript $g$ in the definition of norms whenever there occurs no confusion.

Following the same proof in  \cite[Section 5]{Wang10}, we can obtain the following result on tensorial
$k$-flux.

\begin{proposition}\label{dfk}
Under the bootstrap assumption (\ref{ba1}),
for the tensorial $k$-flux there holds on the null cone $C_u$ the estimate
\begin{equation*}
\int_{C_u} \left(|\sn k|_g^2 +|\nab_L
k|_g^2\right)\le C.
\end{equation*}
\end{proposition}

The following estimate is the main result of this subsection.

\begin{proposition}\label{fluxg}
Let the bootstrap assumptions (\ref{ba1}) and (\ref{BA2}) hold. 
Let $f$ be the scalar components of $\hn g$ and $k$. Then for $0<\ep\le s-2$ there holds
\begin{equation*}
\F^\f12[f]+\|\mu^\ep\F^\f12[P_\mu f]\|_{l_\mu^2}\le C.
\end{equation*}
\end{proposition}


In the following we will give the proof of Proposition \ref{fluxg}.
By the standard energy estimate we have
\begin{align}
\F[\psi]&\le |Q(\psi)(t_2)-Q(\psi)(t_1)|\nn\\
& +\int_{t_1}^{t_2} Q^{\f12}(\psi)(t')\|\Box_\bg \psi\|_{L_x^2}
+\int_{t_1}^{t_2} C(\|\pi\|_{L_x^\infty}+1) Q(\psi)(t') dt'.\label{fpsi}
\end{align}
Recall that $(\hn g, -2 \hn k)$ and $(k, E)$ satisfy the first order hyperbolic system (\ref{lu3}).
Thus, for $\psi=\hn g$ or $k$, the expression of $\Box_\bg \psi$ derived from (\ref{wave1}) contains
time derivatives of the shift vector field $Y$ since $F_u$ contains the term $\psi\c \hn Y$ and
other terms involving $Y$. The lack of control on $\bd_\bT Y$ makes it impractical to apply (\ref{fpsi})
directly to $\psi=\hn g, k$.

To get around the difficulty, we consider the following modified energy momentum tensor
\begin{equation*}
\ti P_\mu=-n^{-1} F_u \bd_\mu u+Q_{\mu\nu}\bT^\nu+\f12 (n^{-1}F_u)^2 n \bd_\mu (t).
\end{equation*}
When $(u,v)$ satisfies the system (\ref{lu3}), $u$ must satisfy (\ref{wave1}).
We claim that
\begin{align}\label{dpu}
\bd^\mu \ti P_\mu & =\left(-\pi_{0a} \nab^a u+\Tr k e_0 (u)-n^{-1}F_v\right)v
 -\bd^i(n^{-1} F_u)\bd_i u+Q_{\mu\nu}{}^{(\bT)}\pi^{\mu\nu}.
\end{align}
Indeed, since (\ref{lu3}) implies $-n^{-1}F_u+\bd_\bT u=v$, we have from the definition of $\ti P_\mu$
that
\begin{align*}
\bd^\mu \ti P_\mu &=-\bd^\mu(n^{-1}F_u) \bd_\mu u -n^{-1}F_u \Box_\bg u+\bd^\mu Q_{\mu\nu}\bT^\nu
 +Q_{\mu\nu}{}^{(\bT)}\pi^{\mu\nu}\\
&\quad \, +(n^{-1}F_u)\bd^\mu(n^{-1}F_u)n \bd_\mu t\\
&=\bd_0(n^{-1}F_u) \bd_0 u-\bd^i (n^{-1}F_u) \bd_i u-n^{-1} F_u \Box_\bg u+\Box_\bg u \bd_{\bT}u
+Q_{\mu\nu}{}^{(\bT)}\pi^{\mu\nu}\\
&\quad\, +(n^{-1}F_u) \bd^\mu(n^{-1} F_u) n \bd_\mu t\\
&=(-n^{-1}F_u+\bd_\bT u) (\Box_\bg u+\bd_0(n^{-1}F_u))-\bd^i(n^{-1}F_u)\bd_i u +Q_{\mu\nu}{}^{(\bT)}\pi^{\mu\nu}\\
&=\left(\Box_\bg u+\bd_0(n^{-1} F_u)\right) v-\bd^i(n^{-1}F_u)\bd_i u+Q_{\mu\nu}{}^{(\bT)}\pi^{\mu\nu}.
\end{align*}
In view of (\ref{wave1}), we obtain (\ref{dpu}).

By the divergence theorem we have for $\ti P_\mu$ that
\begin{equation}\label{div1}
\int_\H L^\mu \ti P_\mu=\int_{\Sigma_{t_2}\cap \D^+}\ti P_\mu \bT^\mu
-\int_{\Sigma_{t_1}\cap \D^+} \ti P_\mu \bT^\mu-\int_{\D^+} \bd^\mu \ti P_\mu.
\end{equation}
Note that
\begin{align*}
\int_\H L^\mu \ti P_\mu&=\int_{\H}\f12(|\bd_L u|^2+|\sn u|^2+(n^{-1} F_u)^2)-n^{-1} F_u \bd_L u\\
&=\int_\H \f12(\frac{1}{2}|\bd_L u|^2+|\sn u|^2-(n^{-1}F_u)^2)+(\f12\bd_L u- n^{-1}F_u)^2.
\end{align*}
Thus
\begin{equation*}
 \int_\H \f12(\frac{1}{2}|\bd_L u|^2+|\sn u|^2)\le \int_\H L^\mu \ti P_\mu+\f12(n^{-1}F_u)^2.
\end{equation*}
Also using (\ref{div1}), we obtain
\begin{equation}\label{fu1}
\F[u](\H)\les \int_\H \f12(n^{-1}F_u)^2+\left|\int_{\Sigma_{t_2}\cap \D^+}\ti P_\mu \bT^\mu-\int_{\Sigma_{t_1}\cap \D^+} \ti P_\mu \bT^\mu\right|+\left|\int_{\D^+} \bd^\mu \ti P_\mu\right|.
\end{equation}
Now  consider the terms on the right of (\ref{fu1}).
By trace inequality,
\begin{align}
\int_{\H}(n^{-1}F_u)^2&\les \int_{t_1}^{t_2}\|F_u\|_{H^1}\|F_u\|_{L_x^2}.\label{flfu}
\end{align}
By definition of $\ti P_\mu$ and $C^{-1}<n<C$,  for any $0< t'\le T$,
\begin{equation}\label{eng3}
\left|\int_{\Sigma_{t'}\cap \D^+} \bT^\mu \ti P_\mu\right| \les \|\bd u\|_{L_x^2}^2+\|F_u\|_{L_x^2}^2.
\end{equation}
For the third term, by (\ref{dpu}), there holds
\begin{align}
\left|\int_{\D^+}\bd^\mu \ti P_\mu\right|&\le \int_{t_0}^t\|{}^{(\bT)}\pi\|_{L_x^\infty}\|\bd u\|_{L_x^2}(\|v\|_{L_x^2}+\|\bd u\|_{L_x^2})\nn\\&+\int_{\D^+}\left| n^{-1} v F_v+\bd^i (n^{-1}F_u)\bd_i u\right|.\label{eng2}
\end{align}

\begin{proof}[Proof of Proposition \ref{fluxg}]

We first apply (\ref{eng2}) to the modified energy momentum tensor $\ti P_\mu$
corresponding to $(u,v)=(\hn g, -2\hn k)$ or $(k,E)$. In view of (\ref{tu0}), we obtain
\begin{align*}
\left|\int_{\D^+}\bd^\mu \ti P_\mu\right|&\les \|\hn g, k\|_{L_t^1 L_x^\infty}\|\hn g, k\|_{L_t^\infty H^1}^2.
\end{align*}
By Proposition \ref{cor2} and (\ref{tu0}) we have
\begin{equation*}
\left|\int_{\Sigma_{t'}\cap \D^+} \bT^\mu \ti P_\mu\right|\le  C \qquad \mbox{and} \qquad
\int_{\H}(n^{-1}F_u)^2\le C.
\end{equation*}
Therefore we can conclude that $\F[\hn g, k]\le C$.

Recall again that $(u,v) =(\hn g, -2\hn k)$ and $(k,E)$ satisfy (\ref{lu3}) with $(F_u, F_v)$ given by
(\ref{8.20.0}) and (\ref{8.20.1}) respectively. Then the pair $(U^\mu,V^\mu)=(P_\mu u, P_\mu v)$
satisfies (\ref{lu3}) with $(F_{U^\mu}, F_{V^\mu})$ given by in (\ref{puuv}).
In view of (\ref{rem2.1})--(\ref{fv}), (\ref{eng4}), and Proposition \ref{hepeg}, we have
\begin{align}
&\|\mu^{\f12+\ep}F_{U^\mu}\|_{l_\mu^2 L^2}+\|\mu^{-\f12+\ep}\hn F_{U^\mu}\|_{l_\mu^2 L_x^2}
  \les \|\hn(\hn g, k)\|_{H^\ep}\les 1, \label{rem2.2}\\
&\|\mu^\ep\hn F_{U^\mu}\|_{l_\mu^2 L^2}
\les \left(\|k, \hn g\|_{L_x^\infty}+1\right)\|\hn(\hn g, k)\|_{H^\ep}
 \les\|k, \hn g\|_{L_x^\infty}+1, \label{rem2.3}\\
&\|\mu^\ep F_{V^\mu} \|_{l_\mu^2 L^2}
 \les\left(\|k,\hn g\|_{L_x^\infty}+1\right)\|\hn(\hn g,k)\|_{H^\ep}
 \les \|k,\hn g\|_{L_x^\infty}+1.  \label{rem2.4}
\end{align}

Define
\begin{equation}\label{defb1mu}
  \B^{(1)}_\mu=\mu^{2\ep}\|F_{U^\mu}\|_{H^1}\|F_{U^\mu}\|_{L_x^2}.
  \end{equation}
 Similar to (\ref{flfu}), we have
$
\int_{\H}\mu^{2\ep}(n^{-1}F_{U^\mu})^2\les \int_{t_1}^{t_2}\B^{(1)}_\mu.
$
Using (\ref{rem2.2}), it yields
\begin{equation}\label{B1mu}
\sum_{\mu>1}\B^{(1)}_\mu\les \|\mu^{\ep-\f12}F_{U^\mu}\|_{l_\mu^2 H^1}\|\mu^{\ep+\f12}F_{U^\mu}\|_{l_\mu^2 L_x^2}\le C
\end{equation}
In view of (\ref{eng2}), it follows that
\begin{align*}
\sum_{\mu>1} & \left|\int_{\D^+} \mu^{2\ep}\bd^\a \ti P_\a\right|\\
&\les\int_{t_1}^{t_2} \left(\|{}^{(\bT)}\pi\|_{ L_x^\infty}\sum_{\mu>1}\|\mu^{\ep}(P_\mu \p u, P_\mu v)\|_{L_x^2}^2\nn
+\sum_{\mu>1}\|\mu^{\ep} F_{V^\mu}\|_{L_x^2}\|\mu^{\ep} P_\mu v\|_{L_x^2}\right.\\
&\left.+\sum_{\mu>1}\|\mu^{\ep} \hn F_{U^\mu}\|_{ L_x^2}\|\mu^{\ep}\bd_i P_\mu  u\|_{L_x^2}\right).
\end{align*}
Using (\ref{rem2.2})-(\ref{rem2.4}), we obtain
\begin{equation}\label{2ndt}
\sum_{\mu>1} \left|\int_{\D^+}\mu^{2\ep} \bd^\a \ti P_\a\right|
\les \|{}^{(\bT)}\pi, \hn g, \hn Y\|_{L_t^1 L_x^\infty}\|\hn(\hn g, k)\|_{H^\ep}^2\le C.
\end{equation}
In view of (\ref{eng3}),  (\ref{rem2.2}) and  (\ref{eng4}),
\begin{align}
\left|\int_{\Sigma_{t'}\cap \D^+} \mu^{2\ep}\bT^\a \ti P_\a\right|
&\les \sum_{\mu>1} \mu^{2\ep}(\|\bd P_\mu u\|_{L_x^2}^2+\| F_{U^{\mu}}\|_{L_x^2}^2)
\le C\label{3rdt}
\end{align}
We conclude in view of (\ref{B1mu}), (\ref{2ndt}) and (\ref{3rdt}) that
$$
\sum_{\mu>1}\mu^{2\ep}(\F[P_\mu \hn g]+\F[P_\mu k])\le C.
$$
The proof is thus complete.
\end{proof}

\section{\bf Strichartz estimate and Main estimates}\label{sec4}

In this section we will show that (BA1) and (\ref{BA2}) can be improved. For ease of exposition,
we shift the origin of time coordinate to $t_0$ and consider $[0,T]\times\Sigma$. Now we make
the following additional bootstrap assumption: there is a constant $B_0$ such that
\begin{align}
\|\mu^{\d}P_\mu \hn g\|_{l_\mu^2 L^2_{[0, T]} L_x^\infty}
+\|\mu^{\d} P_\mu k\|_{l_\mu^2 L^2_{[0,T]} L_x^\infty} &\le B_0,\label{BA3}\tag{BA3}
\end{align}
where $0<\d<s-2$ is a sufficiently small number. As an immediate consequence of  (\ref{BA2}), (\ref{BA3}), 
and (\ref{eqn20}), there holds
\begin{equation*}
\|\mu^{\d} P_\mu (g\hn g)\|_{l_\mu^2 L^2_{[0,T]} L_x^\infty} \les B_0.
\end{equation*}
This estimate will always be used together with (\ref{BA3}). 
Our goal is  to show that the estimates
in (\ref{ba1}), (\ref{BA2}) and (\ref{BA3}) can be improved by shrinking the time interval if necessary.
We will achieve this by establishing the following main estimates.


\begin{theorem}[Main Estimates]\label{mainformetric}
Let (\ref{BA2}) and (\ref{BA3}) hold for some sufficiently small number $0<\d<s-2$.
Then for any number $q>2$ that is sufficiently close to $2$ there holds
\begin{equation*}
\|\hn g, k\|_{L^2_{[0, T]} L_x^\infty}
+\|\mu^{\delta}P_\mu (\hn g, k)\|_{l_\mu^2 L^2_{[0, T]} L_x^\infty} \les T^{\f12-\frac{1}{q}}.
\end{equation*}
If $\phi$ is a function satisfying $\Box_\bg \phi=0$,  then there holds
\begin{equation*}
\|\p\phi\|_{L_I^2 L_x^\infty}^2+\|\mu^{\delta} P_\mu \p\phi\|_{l_\mu^2 L_I^2 L_x^\infty}^2
\le C T^{1-\frac{2}{q}}\|\hn \phi, e_0 \phi\|^2_{H^{1+\ep}(0)}.
\end{equation*}
\end{theorem}

\subsection{Decay estimate $\Rightarrow$ Strichartz estimates}

Let us rescale the coordinate $(t,x)\rightarrow (\frac{t}{\lambda},
\frac{x}{\lambda})$ for some positive constant $\la$. We first prove Strichartz
estimate by assuming the following decay estimate.

\begin{theorem}[Decay estimate]\label{decayth}
Let $0<\ep_0<s-2$ be a given number. There exists a large number $\La$ such that for
any $\la\ge \La$ and any solution $\psi$ of the equation
\begin{equation}\label{wave.4}
\Box_\bg \psi=0
\end{equation}
on the time interval $I_*=[0, t_*]$ with $t_*\le \la^{1-8\ep_0} T$ there is a function $d(t)$
satisfying
\begin{equation}\label{correccondi}
 \|d\|_{L^{\frac{q}{2}}}\les 1, \mbox{ for } q>2 \mbox{ sufficiently close to }2
 \end{equation}
such that for any $0\le t\le t_*$ there holds
\begin{equation}\label{decay}
 \|P e_0 \psi(t) \|_{L_x^\infty}\le \left(\frac{1}{{(1+t)}^{\frac{2}{q}}}+d(t)\right)
 (\|\psi[0]\|_{H^1}+\|\psi(0)\|_{L^2}),
\end{equation}
where $\psi[0]:=(\psi(0), \p_t \psi(0))$ and $\|\psi[0]\|_{H^1}:=\|\hn \psi(0)\|_{L^2}+\|\p_t \psi(0)\|_{L^2}$.
\end{theorem}

Using Theorem \ref{decayth}, we can prove the following result.

\begin{theorem}[Dyadic Strichartz estiamte]\label{ttstar}
There is a large universal constant $C_0$ such that if on the time interval $I_*:=[0, t_*]$
there holds
\begin{equation}\label{smles}
C_0 \|\pi, \hn g, \hn Y\|_{L^1_{I_*} L_x^\infty}\le 1,
\end{equation}
then for any $\phi$ satisfying the wave equation $\Box_\bg \phi=0$ and  $q>2$  sufficiently close to $2$, there holds
\begin{equation}\label{str3}
\|P\p \phi\|_{L^q_{I_*} L_x^\infty} \les \| \phi[0]\|_{H^1},
\end{equation}
where $P$ denote the Littlewood-Paley projection on the frequency
domain $\{1/2\le |\xi|\le 2\}$.
\end{theorem}

We will prove Theorem \ref{ttstar} by adapting a $\T\T^*$ argument from \cite{KRduke,KR1}.
Applying the $\T\T^*$ argument therein  directly to our setting
requires the  control over  $\p \bg$ including the undesired quantity $\bd_\bT Y$.
To get around this difficulty, we  give a careful refinement.

\begin{definition}
Let $\omega:=(\omega_0, \omega_1)\in H^1(\Sigma)\times L^2(\Sigma)$.
We denote by $\phi(t;s,\omega)$ the unique solution of the homogeneous geometric wave
equation $\Box_\bg \phi=0$ satisfying the initial condition $\phi(s;s,\omega)=\omega_0$
and $\bd_0 \phi(s;s,\omega)=\omega_1$. We set $\Phi(t;s,\omega):=(\phi(t;s,\omega),
\bd_0\phi(t;s,\omega))$. By uniqueness we have $\Phi\left(t;s,\Phi(s;t_0,\omega)\right)=\Phi(t;t_0,\omega)$.
\end{definition}

We first show that
\begin{equation}\label{pen}
\|P (e_0\phi)\|_{L^q_{I_*} L_x^\infty} \les \| \phi[0]\|_{H^1}.
\end{equation}
To this end, we let $\H:= H^1(\Sigma)\times L^2(\Sigma)$ endowed with the inner product
\begin{equation*}
\l \omega, v\r=\int_{\Sigma}\big(\omega_1\c v_1+\delta^{ij} \bd_i \omega_0\c \bd_j v_0\big)
\end{equation*}
relative to the orthornormal frame $\{e_0=\bT, e_i=1,2,3\}$. Let $I=[t',t_*]$ with $0\le t'\le t_*$  and let $X=L^q_{I} L_x^\infty$.
 Then the dual of $X$ is $X'=L^{q'}_{I} L_x^1$, where $1/q'+1/q=1$. Let $\T(t'): \H \to X$ be the
 linear operator defined by
\begin{equation}\label{defofT}
\T(t') \omega:=P \bd_0 \phi(t;t',\omega),
\end{equation}
where $\phi:=\phi(t;t',\omega)$ is the unique solution of $\Box_{\bg} \phi=0$ satisfying
$\phi(t')=\omega_0$ and $D_0\phi(t')=\omega_1$ with $\omega:=(\omega_0, \omega_1)$.

By using the Bernstein inequality for LP projections and the energy estimate it is easy to see that
$\T(t'):\H\to X$ is a bounded linear operator, i.e.
\begin{equation}\label{asstrch}
\|\T(t')\omega\|_X=\|P (e_0\phi)\|_{L^q_{I} L_x^\infty}\le C(\la)\|\bd\phi(t')\|_{L_x^2}
\end{equation}
for some constant $C(\la)$ possibly depending on $\la$.  Let $M(t'):=\|\T(t')\|_{\H\to X}$. Then $M(t')<\infty$,
and for the adjoint $\T(t')^*: X'\to \H$ we have
$$
\|\T(t')^*\|_{X'\to \H} =M(t'), \quad \|\T(t')\T(t')^*\|_{X'\to X}=M(t')^2.
$$
Note that $M(\cdot)$ is a continuous function on $I_*$, whose  maximum, denoted by $M$, is achieved at certain $t_0\in [0,t_*)$.
Our goal is to show that $M$ is independent of $\la$.
 Our strategy is to show that
\begin{equation}\label{gl1}
M^2\le C+\f12 M^2
\end{equation}
for some universal positive constant $C$ independent of $\la$. Let us set $I_0=[t_0, t_*]$ and consider $X=L^q_{I_0} L_x^\infty$ , $X'=L^{q'}_{I_0} L_x^1$, and the operators $\T(t_0)$ and $\T(t_0)^*$. For convenience, we drop the $t_0$ in the notation for operators.

We first calculate $\T^*: X'\to \H$. For any $f\in X'$ and $\omega\in \H$ we have
\begin{align*}
\l \T^* f, \omega\r_\H &=\l f, \T\omega\r_{X', X} =\int_{I_0\times \Sigma} f P \bd_0 \phi
=\int_{I_0\times \Sigma} (Pf) \bd_0 \phi(t,t_0, \omega).
\end{align*}
We introduce the function $\psi$ to be the solution of the initial value problem
\begin{equation}\label{psiastz}
\left\{\begin{array}{lll}
\Box_\bg \psi=-Pf, \quad \mbox{in } [t_0, t_*)\times \Sigma,\\
\psi(t_*)=\p_t \psi(t_*)=0.
\end{array}\right.
\end{equation}
Recall the energy momentum tensor $Q[\phi, \psi]$ introduced in (\ref{momentum}).
For any vector field $Z$ we set $P_\mu:=Q[\phi, \psi]_{\mu\nu} Z^\nu$.  In view of $\Box_\bg \phi=0$, it is easy to check that
\begin{equation*}
 \bd^\b P_\b=\frac{1}{2}\left((Z \phi) \Box_\bg \psi+Q[\phi, \psi]_{\a\b}{}^{(Z)}\pi^{\a\b} \right).
\end{equation*}
By the divergence theorem we have
\begin{align}\label{enestr1}
\int_{\Sigma_{t_*}} Q[\phi, \psi]_{\mu\nu} Z^\mu \bT^\nu
-\int_{\Sigma_{t_0}} Q[\phi, \psi]_{\mu\nu} Z^\mu \bT^\nu
&=\int_{I_0\times \Sigma} \bd^\b P_\b
\end{align}
which together with the initial conditions in (\ref{psiastz}) implies that
\begin{equation}\label{enbinlinear}
\int_{I_0\times \Sigma} -(Z\phi) \Box_\bg \psi=2\int_{\Sigma_{t_0}}\bT^\a P_\a
+\int_{I_0\times\Sigma} Q[\phi, \psi]_{\a\b}{}^{(Z)}\pi^{\a\b}.
\end{equation}
Now we take $Z=\bT$. Then it follows from (\ref{enbinlinear}) that
\begin{align}\label{expr1}
\int_{I_0\times \Sigma} -\bd_0\phi \Box_\bg \psi
&=\int_{\Sigma_{t_0}} \bd_0 \phi\bd_0 \psi+\delta^{ij} \bd_i \phi\bd_j \psi
+ \int_{I_0\times\Sigma} Q[\phi, \psi]_{\a\b}{}^{(\bT)}\pi^{\a\b}.
\end{align}
Therefore
\begin{equation}\label{10.4.1}
\l \T^*f, \omega\r_\H=\l \psi[t_0],\omega\r_\H+l(\omega),
\end{equation}
where $l(\cdot)$ is a linear functional on $\H$ defined by
\begin{equation*}
l(\omega):= \int_{I_0\times\Sigma} Q[\phi, \psi]_{\a\b}{}^{(\bT)}\pi^{\a\b}.
\end{equation*}
We claim that $l(\cdot)$ is a bounded linear functional on $\H$. To see this, let
$\omega\in \H$ with $\|\omega\|_\H\le 1$. Then by the energy estimate we have
$\|\bd \phi\|_{L_t^\infty L_x^2}\le \|\omega\|_\H\les 1$.  Thus
$$
|l(\omega)|\le \|\pi\|_{L_t^1 L_x^\infty} \|\bd \phi \|_{L_{I_0}^\infty L_x^2}
\|\bd \psi \|_{L_{I_0}^\infty L_x^2}
\les \|\pi\|_{L_t^1 L_x^\infty} \|\bd \psi \|_{L_{I_0}^\infty L_x^2}.
$$
Hence, by the Riesz representation theorem we have $l(\omega)=\l R(f), \omega\r_\H$ for some
$R(f)\in \H$ and there is a universal constant $C_1$ such that
$$
\|R(f)\|_\H\le C_1 \|\pi\|_{L_t^1 L_x^\infty} \|\bd \psi \|_{L^\infty_{I_0} L_x^2}.
$$
Moreover, we have from (\ref{10.4.1}) that $\T^*f=\psi[t_0]+R(f)$ and hence
\begin{equation}\label{10.4.2}
\T \T^* f=\T \psi[t_0]+\T R(f).
\end{equation}
We claim that there is a universal constant $C_2$ such that
\begin{equation}\label{dpsi}
\|\bd\psi\|_{L^\infty_{I_0} L_x^2}\le C_2 M \|f\|_{L^{q'}_{I_0} L_x^1}.
\end{equation}
Assuming this claim for a moment. Then it follows from the definition of $M$ that
\begin{equation*}
\|\T R(f)\|_{L^{q}_{I_0} L_x^\infty}\le C_1 C_2  M^2 \|\pi\|_{L_t^1 L_x^\infty}
\|f\|_{L^{q'}_{I_0} L_x^1}.
\end{equation*}
Thus, if (\ref{smles}) holds with $C_0\ge 2C_1 C_2$, then
\begin{equation}\label{10.4.3}
\|\T R(f)\|_{L^{q}_{I_0} L_x^\infty}\le \frac{1}{2} M^2 \|f\|_{L^{q'}_{I_0} L_x^1}.
\end{equation}

Next we will estimate $\|\T \psi[t_0]\|_{L_{I_0}^q L_x^\infty}$. We set $F:=(0,-n^2 Pf)$.
By the Duhamel principle we have
\begin{equation*}
\psi[t]=\int_{t_*}^t \Phi(t; s, F(s)) ds.
\end{equation*}
Then $\psi[t_0]=-\int_{t_0}^{t_*} \Phi(t_0;s, F(s)) ds$ and thus
\begin{align*}
\T \psi[t_0]&=P \left[e_0 \phi\left(t; t_0, -\int_{t_0}^{t_*} \Phi(t_0;s, F(s))\right)\right]
=-P\left[ e_0\left(\int_{t_0}^{t_*}\Phi(t,s, F(s)) ds\right)\right]\\
&=-\int_{t_0}^{t_*} P \left[e_0 \Phi(t, s, F(s))\right] ds.
\end{align*}
It follows from Theorem \ref{decayth} that
\begin{equation*}
\|P \left[e_0 \Phi(t,s, F(s))\right]\|_{L_x^\infty} \le C\left((1+|t-s|)^{-\frac{2}{q}}+d(|t-s|)\right)\|\hn(n^2 Pf)(s)\|_{L_x^2}.
\end{equation*}
Observe that on $\Sigma$ there holds
\begin{align*}
\|\hn(n^2 Pf)\|_{L_x^2}&\les \|\hn(n^2) P f\|_{L_x^2}+\|n^2\hn Pf\|_{L_x^2}
\les (\|\hn n\|_{L_x^\infty}+1)\|f\|_{L_x^1}\les \|f\|_{L_x^1}.
\end{align*}
Thus, in view of the Hardy-Littlewood-Sobolev inequality, (\ref{correccondi}) and  Hausdorff Young inequality we obtain
\begin{equation}\label{t1}
\|\T \psi[t_0]\|_{L^q_{I_0} L_x^\infty}
\les \|f\|_{L^{q'}_{I_0}L_x^1}+\left\|\int_{t_0}^{t_*} d(|t-s|)\|f(s)\|_{L_x^1} ds\right\|_{L^q_{I_0}}
\les \|f\|_{L^{q'}_{I_0} L_x^1}.
\end{equation}
Combining (\ref{10.4.2}), (\ref{10.4.3}) and  (\ref{t1}), we therefore obtain (\ref{gl1}).

It remains to prove (\ref{dpsi}). Let $\ti \phi$ be a solution of $\Box_\bg \ti \phi=0$ in $I_*$.
Then for any $t_0\in [0, t_*]$ there holds the energy estimate $\|\bd \ti\phi(t)\|_{L^2(\Sigma)}
\les \|\bd\ti \phi(t_0)\|_{L^2(\Sigma)}$ for $t\in [t_0, t_*]$. Let $t_0\le t'<t_*$.  Similar to the derivation of
(\ref{expr1}), we have on  $I=[t',t_*]$ that
\begin{align}\label{expr1t}
\int_{I\times \Sigma} -\bd_0\ti \phi \Box_\bg \psi&=\int_{\Sigma_{t'}}\bd_0 \phi\bd_0 \psi
+\delta^{ij} \bd_i \phi\bd_j \psi+ \int_{I\times\Sigma} Q[\phi, \psi]_{\a\b}{}^{(\bT)}\pi^{\a\b},
\end{align}
which together with $\Box_\bg \psi =-P f$ gives
\begin{align*}
\l \bd\psi, \bd\ti\phi\r(t')&\le \|P e_0 \ti \phi\|_{L^q_{I} L_x^\infty}\|f\|_{L^{q'}_{I} L_x^1}+\|{}^{(\bT)}\pi\|_{L^1_{I}L_x^\infty}\|\bd\psi\|_{L_t^\infty L_x^2}\|\bd \ti \phi\|_{L_I^\infty L_x^2}
\end{align*}
According to definition of $M$, we can obtain
$\|P e_0 \ti \phi\|_{L^q_{I} L_x^\infty}\le M \|\bd \ti \phi(t')\|_{L_x^2}$.  Thus
\begin{equation*}
\l \bd\psi, \bd\ti\phi\r(t') \les \left( M\|f\|_{L^{q'}_{I} L_x^1}
+ \|{}^{(\bT)}\pi\|_{L^1_{I}L_x^\infty}  \|\bd\psi\|_{L_t^\infty L_x^2}\right) \|\bd \ti \phi(t')\|_{L_x^2}.
\end{equation*}
Since $\bd \ti\phi(t')$ can be arbitrary, there is a universal constant $C_3$ such that
$$
\| \bd\psi(t')\|_{L_x^2} \le C_3  M\|f\|_{L^{q'}_I L_x^1}
+ C_3 \|{}^{(\bT)}\pi\|_{L^1_{I}L_x^\infty}  \|\bd\psi\|_{L_I^\infty L_x^2}.
$$
Recall that $t'\in [t_0, t_*)$ is arbitrary. Thus, if (\ref{smles}) holds with $C_0\ge 2 C_3$ then
$$
\| \bd\psi\|_{L_{[t_0, t_*)}^\infty L_x^2} \le C_3 M\|f\|_{L^{q'}_{[t_0, t_*)} L_x^1}
+ \frac{1}{2} \|\bd\psi\|_{L_{[t_0, t_*)}^\infty L_x^2}.
$$
This implies claim (\ref{dpsi}) with $C_2=2 C_3$.
 The proof of (\ref{pen}) is thus completed. We also have proved for any $t\in I_*$
 \begin{equation}\label{dsp2}
\|\bd\psi\|_{L^\infty_{[t, t_*)} L_x^2}\le C_2 M \|f\|_{L^{q'}_{[t_,t_*)} L_x^1}.
\end{equation}

Now we consider $\|P (\p_m \phi)\|_{L^q_{I_*} L_x^\infty}$. It suffices to estimate
\begin{equation*}
\I=\int_{I_*\times \Sigma} f P (\p_m \phi) =\int_{I_*\times \Sigma} \p_m \phi Pf
\end{equation*}
for any function $f$ satisfying $\|f\|_{L^{q'}_{I_*} L_x^1} \le 1$.
Let $\psi$ be the solution of (\ref{psiastz}), then
$$
\I=\int_{I_*\times \Sigma}-\p_m \phi\Box_\bg \psi.
$$
In view of (\ref{enbinlinear}), we have with $Z=\p_m$ that
\begin{align*}
\int_{I_*\times \Sigma} -Z\phi \Box_\bg \psi &=2\int_{\Sigma_0}\bT^\a P_\a
+ \int_{I_*\times\Sigma} Q[\phi, \psi]_{\a\b}{}^{(Z)}\pi^{\a\b}.
\end{align*}
By direct calculation we can see that ${}^{(Z)} \pi=\bg\c\ti\pi$.
Thus it follows from the energy estimate (\ref{dsp2}) and (\ref{smles}) that
\begin{align*}
\left|\int_{I_*\times\Sigma} Q[\phi, \psi]_{\a\b}{}^{(Z)}\pi^{\a\b}\right|
&\les \|\ti\pi\|_{L^1_{I_*} L_x^\infty}\|\bd \psi\|_{L^\infty_{I_*}L_x^2}\|\bd\phi\|_{L^\infty_{I_*} L_x^2}\\
&\les \|\bd\phi(0)\|_{L_x^2}\|f\|_{L^{q'}_{I_*}L_x^1}
\end{align*}
and
\begin{align*}
\left|\int_{\Sigma_0}\bT^\a P_\a\right|
\les\|\bd\phi(0)\|_{L^2}\|\bd \psi\|_{L^\infty_{I_*} L_x^2}
\les\|\bd\phi(0)\|_{L^2} \|f\|_{L_{I_*}^{q'}L_x^1}.
\end{align*}
Therefore $|\I|\les \|\bd\phi(0)\|_{L_x^2}\|f\|_{L^{q'}_{I_*} L_x^1}$.
Hence we can conclude that
\begin{equation}\label{strspd}
\|P \p_m \phi\|_{L^q_{I_*} L_x^\infty}\le C \|\bd \phi(0)\|_{L_x^2}.
\end{equation}

Finally we prove (\ref{str3}) for the case that $\p=\p_t$, i.e.
\begin{equation}\label{ppartt}
\|P \p_t \phi\|_{L^q_{I_*} L_x^\infty}\le C\|\bd \phi(0)\|_{L_x^2}.
\end{equation}
Note that $\p_t f=n e_0 (f)+Y^m \p_m f$. We can write
\begin{equation*}
P \p_t\phi=n P e_0 (\phi)+Y^m P \p_m \phi+[P,n]e_0 \phi+[P,Y^m]\p_m \phi.
\end{equation*}
By using (\ref{bdcp}), the Bernstein inequality and the finite band property for
the Littlewood-Paley projections, we obtain
\begin{align*}
\|[P, n]e_0\phi\|_{L^{q}_{I_*} L_x^\infty}
&\les \|\hn n\|_{L^{q}_{I_*} L_x^\infty}\|(e_0\phi)_{\le 1}\|_{L^\infty}
+\sum_{\ell>1}\|P(n_{\ell}\c (e_0 \phi)_{\ell})\|_{L^q_{I_*}L_x^\infty}\\
&\les \|\hn n\|_{L^q_{I_*} L_x^\infty}\|e_0 \phi\|_{L^\infty_{I_*}L_x^2}.
\end{align*}
By using (\ref{q3}) and (\ref{BA2}) under  the rescaling coordinates
$
\|\hn n \|_{L^q_{I_*} L_x^\infty}\les\la^{-1+1/q},
$
 also  using the energy estimate for $\phi$, we can obtain
\begin{equation*}
\|[P, n]e_0\phi\|_{L^q_{I_*} L_x^\infty}\les \|\bd \phi(0)\|_{L_x^2}.
\end{equation*}
Similarly, we have $\|\hn Y\|_{L^q_{I_*} L_x^\infty}\les \la^{-1+\frac{1}{q}}$ and
\begin{equation*}
\|[P, Y^m ]\p_m \phi\|_{L^q_{I_*} L_x^\infty}\les \|\bd \phi(0)\|_{L_x^2}.
\end{equation*}
 Combining the above two estimates with (\ref{pen}) and (\ref{strspd}) we therefore obtain (\ref{ppartt}).
 The proof is thus complete.

\subsection{Strichartz estimates $\Rightarrow$ Main estimates}

In this section we will use Theorem \ref{ttstar} to prove Theorem \ref{mainformetric}.
According to the properties of Littlewood-Paley projections, it is easy to derive
the desired estimates for the low frequency part. Therefore, to complete the proof of
Theorem \ref{mainformetric}, it suffices to establish the following result.

\begin{proposition}\label{mainpro}
There exists a large number $\La\ge 1$ such that for any $q>2$ sufficiently close to $2$
and any $\delta>0$ sufficiently close to $0$ there holds on $I=[0,T]$
\begin{equation*}
\sum_{\mu>\La} \left\{\|\mu^{\delta} P_\mu \p_m g\|_{L_{I}^2 L_x^\infty}^2
+\|\mu^{\delta}  P_\mu \p_t g\|_{L_{I}^2 L_x^\infty}^2\right\}\les T^{1-\frac{2}{q}}.
\end{equation*}
Moreover for any solution of $\Box_\bg \phi=0$, there holds
\begin{equation*}
\sum_{\mu>\La} \|\mu^{\delta} P_\mu \p\phi\|_{L_{I}^2 L_x^\infty}^2
\les T^{1-\frac{2}{q}}\|\hn \phi, e_0 \phi\|^2_{H^{1+\ep}(0)}.
\end{equation*}

\end{proposition}

In order to carry out the proof of Proposition \ref{mainpro}, we pick a sufficiently small
$\epsilon_0>0$ and for each $\mu>1$ we partition the  interval $[0,T]$ into disjoint union of
subintervals $I_k=[t_{k-1}, t_k)$ with the properties that
\begin{equation}\label{part}
|I_k|\les \mu^{-8\ep_0}T \quad \mbox{and}\quad
\|k, \hn g, \hn Y, \bd n\|_{L^2_{I_k} L_x^\infty}\le \mu^{-4\ep_0}.
\end{equation}
Such partition is always possible. Let $\kappa_\mu$ denote the total number of subintervals
in the partition. It is even possible to make $\kappa_\mu\les \mu^{8\ep_0}$.

We first consider any pair $(u, v)$ satisfying (\ref{lu3}). Then $(U^\mu, V^\mu):=
(P_\mu u, P_\mu v)$ also satisfies the system (\ref{lu3}) with $F_{U^\mu}$ and $F_{V^\mu}$
given by (\ref{puuv}), i.e.
\begin{align}
&F_{U^\mu}=[P_\mu, Y^m]\p_m u+[P_\mu, n]v+P_\mu F_u, \label{gmu}\\
&F_{V^\mu}=[P_\mu, ng]\hn^2 u+P_\mu F_v+[P_\mu, Y^m] \p_m v.   \label{fvmu}
\end{align}
Consequently, it follows from (\ref{wave1}) that
\begin{align*}
n^2\Box_\bg P_\mu u&=-n\bd_{\bT} F_{U^\mu}+n(-F_{V^\mu}-n\pi_{0a}\nab^a {U^\mu}+e_0(\ln n) F_{U^\mu}+n\Tr k e_0 U^\mu).
\end{align*}

Now we will use the Duhamel principle to represent $P_\mu u$. To simplify the notation, we use
$W(t,s)$ to denote the operator defined on $\H$ such that, for each $\omega:=(\omega_0, \omega_1)\in \H$,
$\phi:=W(t,s)(\omega)$ is the unique solution of the initial value problem
\begin{equation}\label{eqn23}
\Box_\bg \phi=0,\quad  \phi(t;s,x)=\omega_0,\quad  \p_t \phi(t;s,x)=\omega_1.
\end{equation}
Then, by the Duhamel principle, we have for $t\in I_k=[t_{k-1}, t_k]$ that
\begin{align}
P_\mu u(t)&=W(t,{t_{k-1}})\left(P_\mu u(t_{k-1}), \p_t P_\mu u(t_{k-1})- F_{U^\mu}(t_{k-1})\right)\nn\\
&+\int_{t_{k-1}}^t W(t,s)(0, - R_\mu(s)) +W(t,s)(F_{U^\mu}(s), 0)ds.\label{pmug}
\end{align}
Now we apply $P_\mu$ to the both sides and take the spacial derivative.
Writing $P_\mu^2=P_\mu$ by abuse of notation a little, we have
\begin{align}
P_\mu \p_m u(t)&=\int_{t_{k-1}}^t \left\{\p_m P_\mu W(t,s)(0, - R_\mu(s))
  +\p_m P_\mu W(t,s)(F_{U^\mu}(s), 0)\right\}ds\nn\\
&\quad \, +\p_m P_\mu W(t,{t_{k-1}})\left(P_\mu u(t_{k-1}),
  \p_t P_\mu u(t_{k-1})-F_{U^\mu}(t_{k-1})\right).\label{ppg}
\end{align}
where
$$
R_\mu=n(-F_{V^\mu}-n\pi_{0a}\nab^a {U^\mu}+e_0(\ln n) F_{U^\mu}+n\Tr k e_0 U^\mu)-Y^i \p_i F_{U^\mu} .
$$

By using (\ref{str3}) in Theorem \ref{ttstar} with suitable change of coordinates,
we have for any one-parameter family of data $\omega(s):=(\omega_0(s), \omega_1(s))\in \H$ with
$s\in I_k:=[t_{k-1}, t_k]$ that
$$
\mu^{-1+\frac{1}{q}}\|P_\mu \p W(t,s)(\omega(s))\|_{L_{[s,t_k]}^{q}L_x^\infty}
\les \mu^{\f12} \|\omega(s)\|_{\H}.
$$
In view of the Minkowski inequality we then obtain
\begin{align*}
\left\|\int_{t_{k-1}}^t P_\mu \p W(t,s)(\omega(s))ds\right\|_{L_{I_k}^2 L_x^\infty}
&\les \int_{t_{k-1}}^{t_k} \|P_\mu  W(t,s)(\omega(s))\|_{L^2_{[s,t_k]} L_x^\infty} ds\\
&\les |I_k|^{\frac{1}{2}-\frac{1}{q}}\mu^{\frac{3}{2}-\frac{1}{q}}\int_{I_k}\|\omega(s)\|_{\H} ds.
\end{align*}
Since $|I_k|\les T \mu^{-8\ep_0}$, it follows that
\begin{equation*}
\left\|\int_{t_{k-1}}^t P_\mu \p W(t,s)(\omega(s))ds\right\|_{L_{I_k}^2 L_x^\infty}
\les T^{\f12-\frac{1}{q}}\mu^{(\frac{1}{2}-\frac{1}{q})(1-8\ep_0)}
\int_{I_k} \mu \|\omega(s)\|_{\H} ds.
\end{equation*}
Applying the above inequality to (\ref{ppg}) gives, with
$\delta_0:=(\frac{1}{2}-\frac{1}{q})(1-8\epsilon_0)$, that
\begin{align}
\|P_\mu \p_m u\|_{L^2_{I_k} L_x^\infty}
&\les T^{\f12-\frac{1}{q}}\mu^{\delta_0} \left(\|\mu(0,R_\mu)\|_{L^1_{I_k} \H}
+\|\mu(F_{U^\mu},0)\|_{L^1_{I_k}\H}\right)\nn\\
&\quad \, + T^{\f12-\frac{1}{q}} \left(B_\mu(t_{k-1})+ C_\mu(t_{k-1}) \right), \label{sumu}
\end{align}
where
\begin{align*}
B_{\mu}(t):=\mu^{\delta_0}\|\mu\left(P_\mu u(t), \p_t P_\mu u(t)\right)\|_{\H}, \quad
C_{\mu}(t):=\mu^{\delta_0}\|\mu (0,-F_{U^\mu}(t))\|_{\H}.
\end{align*}
In the following we will give the estimates on $R_\mu$, $F_{U^\mu}$, $B_\mu(t_{k-1})$
and $C_\mu(t_{k-1})$ separately.  Positive indices $\ep_0, q, \delta$ are chosen such that $4\ep_0+\delta_0+\delta<s-2$, and $\delta_0+\delta<4\ep_0.$

\subsubsection{ Estimates for $ R_\mu$, $F_{U^\mu}$ }

\begin{lemma}\label{rmd.1}
For any $\delta_1>\delta>0$ satisfying  $b:=\delta_0+\delta_1<4\epsilon_0$,  there holds
\begin{align}
\left(\sum_{\mu>\La}\sum_{k=1}^{\kappa_\mu} \|\mu^{1+\delta_0+\delta}R_\mu\|_{L^1_{I_k}L_x^2}^2\right)^\f12
&\le T \|\hn u, v\|_{L_t^\infty H^{1+b}} +\|\mu^{1+b} P_\mu F_u\|_{L_I^1 l_\mu^2 H^1}\nn\\
&+\|\mu^{1+b} P_\mu F_v\|_{L_I^1 l_\mu^2 L_x^2}+\| \hn^2 u\|_{L_I^\infty L_x^2}.\label{rem1.1}\\
\sum_{\mu>\La}\sum_{k=1}^{\kappa_\mu}\|\mu^{1+\delta_0+\delta}( F_{U^\mu}(s), 0)\|_{L_{I_k}^1 \H}^2
&\les \left(\|\p u, v\|_{L_I^1 H^{1+b}}
+\|\mu^{1+b}P_\mu \p F_u\|_{L_I^1 l_\mu^2 L_x^2}\right)^2\label{gmuu}
\end{align}
\end{lemma}

\begin{proof} We first write
$
R_\mu=-n[P_\mu, ng]\hn^2 u-Y^i \p_i F_{U^\mu}+\ckk R_\mu,
$
where
$$
\ckk R_\mu:=n\left(-P_\mu F_v-[P_\mu, Y^m] \p_m v-n\pi_{0a}\nab^a {U^\mu}+e_0(\ln n) F_{U^\mu}+n\Tr k e_0 U^\mu\right).
$$
Let us set
$$
\I_1:=\sum_{\mu>\La}\sum_{k=1}^{\kappa_\mu}\|\mu^{1+\delta_0+\delta}[P_\mu, ng]\hn^2 u\|_{L_{I_k}^1 L_x^2}^2, \quad
\I_2:=\sum_{\mu>\La}\sum_{k=1}^{\kappa_\mu}\|\mu^{1+\delta_0+\delta}\ckk R_\mu\|_{L_{I_k}^1 L_x^2}^2.
$$
It suffices to show that
\begin{align}
\I_1^{\f12}&\les \| \hn^2 u\|_{L_I^\infty L_x^2}, \label{rem1.2} \\
\I_2^{\f12}&\les \|\hn u, v\|_{L_t^\infty H^{1+b}}T+\|\mu^{1+b} P_\mu F_u\|_{L_I^1 l_\mu^2 L_x^2}
+\|\mu^{1+b} P_\mu F_v\|_{L_I^1 l_\mu^2 L_x^2}. \label{gnr.1}\\
\sum_{\mu>\La}\sum_{k=1}^{\kappa_\mu}&\|\mu^{1+\delta_0+\delta}\hn F_{U^\mu}(s)\|_{L_{I_k}^1 L_x^2}^2
\les \left(\|\p u, v\|_{L_I^1 H^{1+b}}
+\|\mu^{1+b}P_\mu \p F_u\|_{L_I^1 l_\mu^2 L_x^2}\right)^2\label{gmu1.1}.
\end{align}
By (\ref{part}), we have  $\|\hn(ng)\|_{L_{I_k}^1 L_x^\infty} \les \mu^{-8\epsilon_0}$. We can apply
Corollary \ref{base2com} to obtain
\begin{align*}
\|\mu^{1+\delta+\delta_0} [P_\mu, ng]\hn^2 u\|_{L^1_{I_k} L_x^2}
&\les \mu^{\delta+\delta_0}\|\hn (ng)\|_{L^1_{I_k}L_x^\infty}\|\hn^2 u\|_{L_t^\infty L_x^2}\\
&\les \mu^{\delta+\delta_0-8\epsilon_0} \|\hn^2 u\|_{L_t^\infty L_x^2}.
\end{align*}
Recall also that $\kappa_\mu\les \mu^{8\epsilon_0}$. We can obtain
\begin{align*}
\sum_{k=1}^{\kappa_\mu}&\|\mu^{1+\delta+\delta_0}[P_\mu, ng]\hn^2 u\|_{L^1_{I_k}L_x^2}^2
\le C\mu^{2(\delta+\delta_0-4\ep_0)}\|\hn^2 u\|_{L_t^\infty L_x^2}^2.
\end{align*}
Since $0<\delta<\delta_1$ and $b:=\delta_0+\delta_1<4\epsilon_0$, we have
$\I_1\les \La^{2(b-4\epsilon_0)} \|\hn^2 u\|_{L_t^\infty L_x^2}\les \|\hn^2 u\|_{L_t^\infty L_x^2}^2$
which gives (\ref{rem1.2}).

Next we prove (\ref{gnr.1}). Since $0<\delta<\delta_1$, we observe that for any function $a_\mu$ there holds
\begin{align}
\sum_{\mu>\La} \sum_{k=1}^{\kappa_\mu} \|\mu^\delta a_\mu\|_{L_{I_k}^1 L_x^2}^2
&\le \sum_{\mu>\La} \|\mu^\delta a_\mu\|_{L_I^1 L_x^2}^2
\le \left(\int_I \sum_{\mu>\La}\|\mu^\delta a_\mu\|_{L_x^2}\right)^2 \nn\\
&\les \left(\int_I \|\mu^{\delta_1} a_\mu\|_{l_\mu^2 L_x^2} \right)^2.\label{amu1}
\end{align}

(\ref{gmu1.1}) and (\ref{gmuu}) can be derived  immediately by using (\ref{amu1}) and (\ref{1dfu1}).

In view of (\ref{amu1}), it suffices to estimate $\int_I \|\mu^{1+b}\ckk R_\mu\|_{l_\mu^2 L_x^2}$.
From (\ref{epeq3}) it follows that
\begin{align*}
\|\mu^{1+b}[P_\mu, Y^m]\p_m v\|_{l_\mu^2 L_x^2}&\les \| \hn Y\|_{L_x^\infty}\|\hn v\|_{ H^{b}}+\|\hn^2 Y\|_{H^{\f12+b}}\|\hn v\|_{ L_x^2}
\end{align*}
In view of (\ref{gmu}), (\ref{dfu}) and (\ref{ell.1}) we have
\begin{align*}
\|\mu^{1+b} e_0\ln n F_{U^\mu}\|_{l_\mu^2 L_x^2}
&\les \|e_0(\ln n)\|_{L^\infty} \left(\|\hn Y, \hn n\|_{L_x^\infty}\|\p u, v\|_{H^b} \right.\nn\\
&\quad \, \left. +\|\hn^2 Y, \hn^2n \|_{H^{\f12+b}}\|\p u, v\|_{L_x^2}+\|\mu^{1+b}P_\mu F_u\|_{l_\mu^2 L_x^2}\right)\nn\\
&\les \|\p u, v\|_{H^b}+\|\mu^{1+b}P_\mu F_u\|_{l_\mu^2 L_x^2}. 
\end{align*}
Recall that $\|\nab n , \Tr  k\|_{L_I^\infty L^\infty}\le C$, we can derive that
\begin{align*}
\|\mu^{1+b} (|\Tr k \, e_0 (U^\mu)|+|\nab_i n \bd^i U^\mu| )\|_{l_\mu^2 L_x^2}
&\les \|\nab n, \Tr k\|_{L_x^\infty}\|\La^b\hn (\hn u , v)\|_{L_x^2}\\
&\les \|\La^b\hn (\hn u , v)\|_{L_x^2}
\end{align*}
Combining the above three estimates we thus obtain (\ref{gnr.1}).
\end{proof}

\subsubsection{Estimate for  $B_{\mu}(t_{k-1})$ and  $C_\mu(t_{k-1})$.}

\begin{lemma}\label{rmd.111}
For any $\delta>0$ satisfying $\a:=4\ep_0+\delta_0+\delta<s-2$ there holds
\begin{align*}
 \sum_{\mu>\La}\sum_{k=1}^{\kappa_\mu} \mu^{2\delta} B_{\mu}(t_{k-1})^2
 &\les \sum_{\mu>\La} \mu^{2\a} \sup_{t\in I} \ei_\mu(t)\\
&+\sup_{t} \left(\|\hn u, v\|_{H^{\f12+\a}}^2+\|\mu^{\frac{3}{2}+\a} P_\mu  F_u\|_{l_\mu^2 L_x^2}^2\right).
\end{align*}
\end{lemma}

\begin{proof}
Since $\kappa_\mu\les \mu^{8\epsilon_0}$, we have from the expression of $B_\mu(t)$ that
\begin{align*}
\sum_{\mu>\La} &\sum_{k=1}^{\kappa_\mu}\mu^{2\delta}B_\mu(t_{k-1})^2\\
&\les \sum_{\mu>\La} \sup_{t\in I}(\mu^{4\ep_0+\delta} B_{\mu}(t))^2
\les \sum_{\mu>\La} \mu^{2(1+\a)} \sup_{t\in I} \|(P_\mu u(t), \p_t P_\mu u(t))\|_{\H}^2
\end{align*}
According to the definition of $\ei_\mu(t):=\ei_\mu(u(t), v(t))$ and the equation
for $\p_t P_\mu u$ we obtain
\begin{align*}
\sum_{\mu>\La}\sum_{k=1}^{\kappa_\mu}\mu^{2\delta}B_\mu(t_{k-1})^2
&\les \sum_{\mu>\La} \mu^{2\a} \sup_{t\in I} \ei_\mu(t)
+\sum_{\mu>\La} \sup_{t\in I} \|\mu^{1+\a} F_{U^\mu}(t)\|_{L_x^2}^2.
\end{align*}
With the help of (\ref{epnew}) we have
\begin{align*}
\|\mu^{1+\a}  F_{U^\mu}\|_{L_x^2}^2
&\le \mu^{-1}  \sum_{\la}\| \la^{\frac{3}{2}+\a}F_{U^\la}(t)\|_{L_x^2}^2\\
&\les \mu^{-1}(\|\hn u, v\|_{H^{\frac{1}{2}+\a}}^2+\|\la^{\frac{3}{2}+\a} P_\la F_u\|_{l_{\la}^2 L_x^2}^2).
\end{align*}
Plugging this into the above inequality and summing over $\mu>\La$ gives the desired estimate.
\end{proof}

\begin{lemma}\label{rmd.11}
For any $\delta_1>\delta>0$ satisfying $b:=\delta_0+\delta_1<4\epsilon_0$ there hold
\begin{align*}
&\sum_{\mu>\La}\sum_{k=1}^{\kappa_\mu} \left(\mu^{\delta}C_\mu(t_{k-1})\right)^2
\le \|\p_m u, v\|_{L^\infty_t H^{4\ep_0+b}}^2
+\sup_{t\in I} \sum_{\mu>1}\|\mu^{1+4\ep_0+b}P_\mu F_u\|_{L_x^2}^2.  
\end{align*}
\end{lemma}

\begin{proof}
Since
\begin{equation*}
\sum_{k=1}^{\kappa_\mu} \left(\mu^{\delta} C_\mu(t_{k-1})\right)^2
\le \sup_{t\in I}\|\mu^{4\ep_0+\delta_0+\delta+1} F_{U^\mu}(t)\|_{L_x^2}^2.
\end{equation*}
with $0<\delta<\delta_1$, we have
$$
\sum_{\mu>\La} \sum_{k=1}^{\kappa_\mu} \left(\mu^{\delta} C_\mu(t_{k-1})\right)^2
\le \sup_t\|\mu^{4\ep_0+\delta_0+\delta_1+1}F_{U^\mu}(t)\|_{l_\mu^2 L_x^2}^2.
$$
In view of (\ref{epnew}), we complete the proof of Lemma \ref{rmd.11}.
\end{proof}

In view of (\ref{sumu}), Lemma \ref{rmd.1}, Lemma \ref{rmd.11}, Lemma \ref{rmd.111}
and writing
$$
\sum_{\mu>\La} \|\mu^{\delta} P_\mu \p_m u\|_{L_{I}^2 L_x^\infty}^2
=\sum_{\mu>\La} \sum_{k=1}^{\kappa_\mu} \|\mu^{\delta} P_\mu \p_m u\|_{L_{I_k}^2 L_x^\infty}^2,
$$
we can obtain the following result.

\begin{proposition}\label{smau}
For any $q>2$ sufficiently close to $2$ and any $\delta>0$ sufficiently small such that
$\a:=4\epsilon_0+\delta_0+\delta<s-2$, where $\delta_0:=(\f12-\frac{1}{q})(1-8\epsilon_0)$.
Then for any pair $(u, v)$ satisfying (\ref{lu3}) there holds
\begin{align*}
&\sum_{\mu>\La}  \|\mu^{\delta} P_\mu \p_m u\|_{L_{I}^2 L_x^\infty}^2\\
&\les T^{1-\frac{2}{q}} \left(\|\mu^{2+\a} P_\mu F_u\|_{L_I^1 l_\mu^2 L_x^2}^2
+\|\mu^{1+\a} P_\mu F_v\|_{L_I^1 l_\mu^2 L_x^2}^2
+\|\mu^{\f12+\a} P_\mu \hn F_u\|_{L_I^\infty l_\mu^2 L_x^2}^2\right)\\
&+ T^{1-\frac{2}{q}} \left(\sup_{t\in I} \E^{({1+\a})}(u,v)(t)
+\sum_{\mu>\La} \mu^{2\a} \sup_{t\in I} \ei_\mu(t)  \right).
\end{align*}
\end{proposition}

Now we are ready to derive the estimates on the spacial derivative part in Proposition
\ref{mainpro}. Recall that $(u, v):=(g, -2k)$ satisfies (\ref{lie3}). Recall also that
Proposition \ref{remaindd} implies
\begin{align*}
\sum_{\mu>\La}\sup_{t} \mu^{2\a}\ei_\mu(t)
\les \E^{(1+\a)}(u, v)(0)+ \|\mu^{2+\a} P_\mu F_u\|_{L_I^1 l_\mu^2 L_x^2}
+\|\mu^{1+\a} P_\mu F_v\|_{L_I^1 l_\mu^2 L_x^2}.
\end{align*}
In view of $\E^{(1+\a)}(u, v)(t)\les \E^{(1+\a)}(u, v)(0)$, we thus obtain from
Proposition \ref{smau} that
\begin{align*}
\sum_{\mu>\La}  \|\mu^{\delta} P_\mu \p_m u\|_{L_{I}^2 L_x^\infty}^2
&\les T^{1-\frac{2}{q}} \left(\|\mu^{2+\a} P_\mu F_u\|_{L_I^1 l_\mu^2 L_x^2}^2
+\|\mu^{1+\a} P_\mu F_v\|_{L_I^1 l_\mu^2 L_x^2}^2 \right.\\
& \left. +\|\mu^{\f12+\a} P_\mu \hn F_u\|_{L_I^\infty l_\mu^2 L_x^2}^2
+ \E^{({1+\a})}(u,v)(0) \right).
\end{align*}
With the help of (\ref{estfu.1}), (\ref{estfu.2}), (\ref{eng4}) and the bootstrap assumption (\ref{ba1}), it follows that
$\|\mu^{\f12+\a} P_\mu \hn F_u\|_{L_I^\infty l_\mu^2 L_x^2}\les 1$ and
$$
\|\mu^{1+\a} P_\mu F_u\|_{L_I^1 l_\mu^2 L_x^2}+\|\mu^{1+\a} P_\mu F_v\|_{L_I^1 l_\mu^2 L_x^2}
\les \|\hn g, k,\hn Y,\hn n\|_{L_I^1 L_x^\infty}+ 1\les 1.
$$
Therefore
\begin{equation*}
\|\p_m g\|_{L_I^2 L_x^\infty}^2+\sum_{\mu>\La} \|\mu^{\delta} P_\mu \p_m g\|_{L_{I}^2 L_x^\infty}^2
\le C T^{1-\frac{2}{q}}.
\end{equation*}

For a solution $\phi$ of the equation $\Box_\bg \phi=0$, we recall that $(u,v)=(\phi, e_0\phi)$
satisfies (\ref{wv3}) with $W=0$.  Thus, in view of (\ref{dfv}) and $F_u=0$, we may use the same argument
as above to conclude that
\begin{equation*}
\|\p_m \phi\|_{L_I^2 L_x^\infty}^2
+\sum_{\mu>\La} \|\mu^{\delta} P_\mu \p_m\phi\|_{L_{I}^2 L_x^\infty}^2
  \le C T^{1-\frac{2}{q}}\|\hn \phi, e_0 \phi\|^2_{H^{1+\ep}(0)}.
\end{equation*}

\subsubsection{Estimate for $P_\mu \p_t u$}

From (\ref{pmug}) it follows that
 \begin{align}\label{eqn25.5}
P_\mu \p_t u(t)&=P_\mu F_{U^\mu}(t)+P_\mu \p_t W(t,{t_{k-1}})
\left(P_\mu u(t_{k-1}), \p_t P_\mu u(t_{k-1})-F_{U^\mu}(t_{k-1})\right) \nn \\
&+\int_{t_{k-1}}^t P_\mu \{\p_t W(t,s)(0, - R_\mu(s)) +\p_t W(t,s)(F_{U^\mu}(s), 0)\}ds.
\end{align}
We can use the same argument for dealing with $P_\mu \p_m u$ to estimate the terms
on the right hand side except the first term $P_\mu F_{U^\mu}(t)$.

\begin{lemma}\label{eqn25}
For sufficiently small $\d>0$ there holds
\begin{equation*}
\sum_{\mu>\La} \|\mu^\d P_\mu F_{U^\mu}\|_{L^2_{I} L_x^\infty}^2
\les\sum_{\mu>\La}\|\mu^\d P_\mu F_u\|_{L^2_I L_x^\infty}^2
+T\|\hn^2 u, \hn v \|_{L_t^\infty H^\d}^2
\end{equation*}
\end{lemma}

\begin{proof}
From (\ref{eqn21}) we can find $0<\eta<1/2$ such that
\begin{align*}
 \mu^\d\|&[P_\mu, Y^m]\p_m u\|_{L_x^\infty}+\mu^\d\|[P_\mu, n]v \|_{L_x^\infty}\\
 &\les \mu^{-\eta} \left(\|\hn Y, \hn n \|_{L_x^\infty}\|\hn^2 u, \hn v \|_{H^\d}
 +\|\hn u, v\|_{H^1}\|\hn^2 Y, \hn^2 n\|_{H^{\delta+\f12}} \right).
\end{align*}
This together with (\ref{gmu}) implies that
\begin{align*}
\sum_{\mu>\La} \|\mu^\d P_\mu F_{U^\mu}\|_{L^2_{I} L_x^\infty}^2
&\les \sum_{\mu>\La}\|\mu^\d P_\mu F_u\|_{L^2_I L_x^\infty}^2
 + T \|\hn Y, \hn n \|_{L_I^\infty L_x^\infty}^2 \|\hn^2 u, \hn v \|_{L_I^\infty H^\d}^2\\
& \quad \, +T \|\hn u, v\|_{L_I^\infty H^1}^2 \|\hn^2 Y, \hn^2 n\|_{L_I^\infty H^{\f12+\d}}^2.
\end{align*}
In view of Proposition \ref{eplem}, we therefore obtain the desired estimate.
\end{proof}

By using (\ref{eqn25.5}), Lemma \ref{eqn25} and Proposition \ref{geoeg} for the solution $\phi$ of $\Box_\bg \phi=0$,
in view of $F_u=0$ we derive that
\begin{equation*}
\|\p_t\phi\|_{L_I^2 L_x^\infty}^2+\sum_{\mu>\La} \|\mu^{\delta} P_\mu \p_t\phi\|_{L_{I}^2 L_x^\infty}^2
\le C T^{1-\frac{2}{q}}\|\hn \phi, e_0 \phi\|^2_{H^{1+\ep}(0)}.
\end{equation*}
Next we consider $(u,v)=(g, -2k)$ which satisfies (\ref{lie3}). Recall that
$F_u=\hn Y\c g$ in (\ref{lie3}), we have
\begin{align*}
\|\mu^\d P_\mu F_u\|_{L^2_I L_x^\infty}
&\les \|\mu^{\d}[P_\mu, g]\hn Y\|_{L_t^2 L_x^\infty}+\|\mu^\d P_\mu \hn Y\|_{L_t^2 L_x^\infty}.
\end{align*}
By using (\ref{eqn20}) we have with $0<\eta<1/2$ that
\begin{equation*}
\|\mu^{\d}[P_\mu, g]\hn Y\|_{L_x^\infty}
\les \mu^{-\eta} \|\hn Y\|_{L_x^\infty}\|\hn^2 g\|_{H^\d}.
\end{equation*}
Therefore
\begin{align*}
 \sum_{\mu>\La}\|\mu^\d P_\mu F_u\|_{L^2_I L_x^\infty}^2
 &\les T\left(\| \hn^2 Y\|_{L_t^\infty H^{\f12+\d}}^2
 +\|\hn^2 g, \hn k \|_{L_t^\infty H^\d}^2\right)
\end{align*}
From this,  (\ref{eqn25.5}), (\ref{eng4})  and (\ref{ell.1}),  we conclude that
$\sum_{\mu>\La} \|\mu^\delta P_\mu \p_t g\|_{L^2_{I} L_x^\infty}^2 \le C T^{1-\frac{2}{q}}$.

\subsection{\bf Boundedness Theorem  $\Rightarrow$  Decay estimates}\label{convv}

In this subsection we give the proof of Theorem \ref{decayth} under the
rescaled coordinates. A time interval $I=[0,T]$ becomes $I_*=[0, \la T]$ after rescaling.
Let $\tau_*$ denote a number such that $t_*\le \tau_*\le \la T$ and let $t_0$ be 
certain number satisfying  $1\approx t_0< \tau_*$.  We may take a sequence of balls $\{B_J\}$ of radius $1/2$
such that their union covers $\Sigma_{t_0}$ and any ball in this collection
intersect at most $10$ other balls.
Let $\{\chi_J\}$ be a partition of unity subordinate to the cover $\{B_J\}$.
We may assume that $|\nab\chi_{J}|_{L_x^\infty}\le C_1$ uniformly in $J$.
By using this partition of unity and a standard argument we can reduce the proof
of Theorem \ref{decayth} by establishing the following dispersive estimate result
with initial data supported on a ball of radius $1/2$.

\begin{proposition}\label{lcestimate}
There exists a large constant $\La$ such that for any $\la\ge \La$ and
any solution $\psi$ of
\begin{equation*}
\Box_\bg \psi=0
\end{equation*}
on the time interval $[0, \tau_*]$ with $\tau_*\le \la T$, with certain $t_0\in [1,C]$
and any initial data
$\psi[t_0]=(\psi(t_0), \p_t \psi(t_0))$ supported in the geodesic ball $B_{1/2}$ of radius
$\f12$, there is a function $d(t)$ satisfying
\begin{equation}\label{correc}
 \|d\|_{L^\frac{q}{2}[0,\tau_*]}\les 1, \mbox{ for  } q>2 \mbox{ sufficiently close to }2
 \end{equation}
 such that for all $t_0\le t\le \tau_*$,
 \begin{equation}\label{decaycp}
 \|P e_0 \psi(t)\|_{L_x^\infty}\le \left(\frac{1}{{(1+|t-t_0|)}^{\frac{2}{q}}}+d(t)\right)(\|\psi[t_0]\|_{H^1}+\|\psi(t_0)\|_{L^2}).
 \end{equation}
\end{proposition}

\begin{proof}[Proof of Theorem \ref{decayth}]
 To derive Theorem \ref{decayth}, we apply the above result to $\psi_I$,
\begin{equation*}
\Box_\bg \psi_I=0, \psi_I(t_0)=\chi_I \c \psi(t_0)
\end{equation*}
with $\psi$ the solution of (\ref{wave.4}) with initial data $\psi[0]:=(\psi(0), \p_t \psi(0))$.
It is easy to see that $\psi(t,x)=\sum_I \psi_I(t,x)$.
By summing over $I$, we derive  for $t_0\le t\le \tau_*$
\begin{equation}\label{12.04.2}
 \|P e_0 \psi(t)\|_{L_x^\infty}\le \left(\frac{1}{{(1+|t-t_0|)}^{\frac{2}{q}}}+d(t)\right)(\|\psi[t_0]\|_{H^1}+\|\psi(t_0)\|_{L^2}).
\end{equation}
Then (\ref{decay}) follows by applying (\ref{h1ena}) to the solution $\psi$ of (\ref{wave.4})
\begin{equation*}
\|\psi[t_0]\|_{H^1}+\|\psi(t_0)\|_{L^2}\les \|\psi[0]\|_{H^1}+\|\psi(0)\|_{L^2}.
\end{equation*}
For $0<t<t_0$, it follows immediately from Bernstein inequality and (\ref{h1en}) that
\begin{equation}\label{12.04.1}
\|P e_0 \psi(t)\|_{L_x^\infty}\les \|e_0 \psi(t)\|_{L_x^2}\les\|\psi[0]\|_{H^1}.
\end{equation}
Combining (\ref{12.04.2}) with (\ref{12.04.1}) implies  Theorem \ref{decayth}.
\end{proof}

We will prove Proposition \ref{lcestimate} by establishing boundedness theorem 
for conformal energy. For this purpose, we introduce the setup and notation.
We denote by $\Ga^+$  the portion in $[0,\la T]$ of the integral curve of
$\bT$ passing through the center of $B_{\f12}$. We define
the optical function $u$ to be the solution of eikonal equation
$\bg^{\a\b}\p_\a u\p_\b u=0$ with
 $u=t  \mbox { on }\Ga^+$. We denote the outgoing null cone initiating
from $\Ga^+$ by $C_u$ with $0\le u\le \la T$.  Let
$S_{t,u}=C_u\cap \Sigma_t$.
 Let us set $\D^+_0=\cup_{\{t\in [t_0,\tau_*],
0\le u\le t\}} S_{t,u}$ and $\D^+=\cup_{\{t\in [0,\tau_*],
0\le u\le t\}} S_{t,u}$.   We denote the exterior region on 
$\Sigma_t, \, t\ge t_0$  by $\Et_t=\{0\le u\le 3t/4\}$. By $C^{-1}<n<C$, we can always choose
$t_0\in [1,2C]$ such that
$B_{\f12}\subset(\D^+_0\cap\Sigma_{t_0})$.

Next we extend the time axis $\Ga^+: u=t$  backward by
following the integral curve of $\bT$ to $t=-\la T$.
Let us denote the extended portion of the integral curve of $\bT$ by
$\Ga^-$. Let $C_u$ be the outgoing null cone initiating from vertex
$p(t)\in \Ga^-$ with $u=t$. We also foliate the null hypersurfaces
by time foliation, $C_u=\cup_{u\le t\le \tau_*}S_{t,u}$.

Let $\underline{\varpi}$ and $\varpi$ be  smooth cut-off functions 
depending only on two variables $t,u$. For $t>0$, they are defined as follows
\begin{equation*}
\underline{\varpi}=\left\{\begin{array}{lll}
 1 \quad\mbox{ on }\,  0\le u\le t\\
0 \quad \mbox { off } -\frac{t}{4} \le
u\le t,
\end{array}\right.
\quad
 \varpi=\left\{\begin{array} {lll}1 \quad \quad\mbox { on } 0\le\frac{u}{t}\le \f12 \\
 0 \quad\quad\mbox{ if } \frac{u}{t}\ge \frac{3}{4} \mbox { or } u\le
 -\frac{t}{4}\end{array}\right..
\end{equation*}
We also suppose $\varpi$ and $\underline{\varpi}$ coincide in the region $\cup_{\{t\in [t_0, \tau_*],-\frac{t}{4}<u\le 0 \}}S_{t,u}.$

 Let us denote by $N$ the outward unit normal of $S_{t,u}\in
\Sigma_t$. Define $\theta_{AB}=\l \bd_A N, e_B\r$ and  $\chi_{AB}=\l \bd_{e_A}L, e_B\r$. We decompose $\chi$ as  $\chi_{AB}=\chih_{AB}+\f12\tr\chi\ga_{AB}.$
\begin{equation}\label{thetan}
\sn_N N=-(\sn\log \bb) e_A, \qquad  \sn_A N_B=\theta_{AB}e_B, \qquad  \chi_{AB}=\theta_{AB}-k_{AB}
\end{equation}
Under the rescaled  coordinates we  recall some useful results in Proposition \ref{dfk},
Proposition \ref{fluxg} and those established  in
\cite[Sections 4 and 8]{Wang10}. Under (\ref{ba1}) and (\ref{BA2}) we have:

\begin{enumerate}
\item[(i)]there exists $\delta_*>0$ depending only on $B_1$ and the norm
of initial data $\|(g, k)\|_{H^2\times H^1(\Sigma_0)}$ such that if $T\le \delta_*$
then  the outgoing null radius of injectivity satisfies $i_*(p)> T-t(p)$ for any
$p\in [-\la T,\la T]\times \Sigma$.

\item[(ii)] Let $\N^+(p)$ be an outgoing null cone  initiating from $p\in [-\la T,\la  T]\times
\Sigma$ and contained therein. Then on every $\N^+(p)$ there holds $ \F^{\f12}[\hn g]\les \la^{-\f12}$, 
the curvature flux $\R$  together with flux type norm  of components of $\pi$ on $\N^+(p)$ satisfies (for definition we refer to \cite{Wang10} and Section \ref{ricc})
\begin{equation}\label{fluxdef}
 \R+\N_1[\slashed{\pi}]\les \la^{-\f12}.
 \end{equation}

\item[(iii)] For $
 0\le T\le\delta_*,
 $ consider  $C_u\subset[0,\la T]\times \Sigma$, with $S_{t,u}=C_u\cap \Sigma_t$ 
 and $r(t,u)=\sqrt{\frac{|S_{t,u}|}{4\pi}}$. As a consequence of (\ref{fluxdef}) and $C^{-1}<n<C$
 the metric $\gamma_{t,u}$ on ${\mathbb
 S}^2$, obtained by restricting the metric $g$ on $\Sigma_t$ to
 $S_{t,u}$ and then pulling it back to ${\mathbb S}^2$ by the
 exponential map ${\mathcal G}(t,u,\cdot)$, verifies with small quantity $0<\ep< 1/2$ that
$$
|r^{-2}\gamma_{t,u}(X, X)-\gas (X, X)|< \epsilon \gas(X, X), \quad
\forall X\in T{\mathbb S}^2,
$$
where $\gas$ is the standard metric on ${\mathbb
S}^2$; there holds  $(t-u)\approx r(t,u)$. There hold 
\begin{align}
&\left|\frac{\bb}{n}-1\right| \le \frac{1}{2}, \quad\quad r\tr\chi\approx 1, \quad\quad v_{t,u}:=\sqrt{\frac{|\ga_{t,u}|}{|\gas|}}\approx r^2\label{recl1}\\
&\|\ti \pi, \chih, \sn\log \bb\|_{L^4(S_{t,u})}+\|r^{-\f12}(\ti \pi, \chih, \sn\log \bb)\|_{L^2(S_{t,u})} \les  \la^{-\f12}\nn.
\end{align}
\end{enumerate}

 We will constantly employ the following result, where all the constants suppressed in $\les$ are independent of frequency $\la$.



\begin{lemma}\label{trace2}
For any $\Sigma$-tangent tensor field $F$ there hold for $-\tau_*/2\le u< t$
\begin{equation*}
\int_{S_{t,u}}|F|^2\les \|
F\|_{H^1(\Sigma_t)}\|F\|_{L^2(\Sigma_t)},\quad\quad \|F\|_{L^4(S_{t,u})}+\|r^{-\f12} F\|_{L^2(S_{t,u})}\les \| F\|_{H^1(\Sigma_t)}.
\end{equation*}
\end{lemma}

\begin{proof}
This is \cite[Proposition 7.5]{Wang10}.
\end{proof}

More results that can be
established under (\ref{ba1}), (\ref{BA2}) and the bounded $H^2$ norm of data will be revisited in Section
\ref{ricc}.
Now we prove  a commutator estimate for $P$, the Littlewood Paley projection with frequency $1$.  This estimate is slightly more general than needed.
\begin{lemma}\label{comh1}
For  scalar function $f$ and $G$, with $3<p\le\infty$,
\begin{equation*}
\|[P, G]\p_m f\|_{L_x^\infty}+\|[P,G]\p_m f\|_{H^1}\les \|\hn G\|_{L_x^p}\|\hn f\|_{L_x^2}
\end{equation*}
\end{lemma}
\begin{proof}
We first establish $L^\infty$ estimates in this lemma.
Using (\ref{bdcp}), we have
\begin{equation}\label{err6}
[P, G]\p_m f= [P, G](\p_m f)_{\le 1}+\sum_{\ell>1}P\left( G_{\ell}\c P_\ell (\p_m f)\right).
\end{equation}
By Sobolev embedding,  with $b>3$ sufficiently close to $3$,
\begin{align}
\|[P,& G](\p_m f)_{\le 1}\|_{L_x^\infty}\les \|\hn [P, G](\p_m f)_{\le 1}\|_{L_x^b}+\|[P, G](\p_m f)_{\le 1}\|_{L_x^2}.\label{err5}
\end{align}
Using Corollary \ref{base2com},  we can obtain for the second term with $\frac{1}{p}+\frac{1}{p'}=\f12$ that
\begin{align}
\|[P, G](\p_m f)_{\le 1}\|_{L_x^2}&\les \|\hn G\|_{L_x^p}\|\hn f_{\le 1}\|_{L_x^{p'}}
\les\|\hn G\|_{L_x^p}\|\hn f\|_{L_x^2}.\label{err8}
\end{align}
Apply (\ref{bdcp2}) to $(G, \p_m f)$ and $\mu=1$  we can obtain
\begin{align}
\hn [P, G] (\p_m f)_{\le 1} &=\int \hn M_1(x-y) (x-y)^j \int_0^1 \p_j G(\tau y+(1-\tau) x) d\tau (\p_m f)_{\le 1}(y) dy\nn\\
&- \int M_1 (x-y) \hn G(x) (\p_m f)_{\le 1}(y) dy.\label{red1}
\end{align}
Then with $\frac{1}{b*}+\frac{1}{p}=\frac{1}{b}$,
\begin{align}
&\|\hn [P, G] (\p_m f)_{\le 1}\|_{L_x^b}\les \|\hn G\|_{L_x^p}\|\p f_{\le 1}\|_{L_x^{b*}}.\label{err7}
\end{align}
By (\ref{err6}), using (\ref{err7}) and  (\ref{err8}), Bernstein inequality,
\begin{align*}
\|[P, G]\p_m f\|_{L_x^\infty}&\les \|\hn G\|_{L_x^p}\|\p_m f\|_{L_x^2}+\sum_{\ell>1}\|P(G_\ell P_\ell (\p_m f))\|_{L_x^\infty}.
\end{align*}
Using Bernstein inequality and finite band property,
\begin{align*}
\sum_{\ell>1}\|P\left(G_\ell P_\ell (\p_m f)\right)\|_{L_x^\infty}&\les\sum_{\ell>1}\ell^{-1} \|P_\ell \p G\|_{L_x^p}\|P_\ell \p_m f\|_{L_x^{p'}}\\
&\les\|\hn G\|_{L_x^p}\|\p_m f\|_{L_x^2}\sum_{\ell}\ell^{-1+\frac{3}{p}}\les\| \hn G\|_{L_x^p} \|\hn f\|_{L_x^2}.
\end{align*}
Thus  the proof of the first part of the lemma is completed.  Now we consider $H^1$ estimates in this lemma.
Let $I=\hn [P, G ](\p_m f)_{\le 1}$. Similar to (\ref{err7}), it follows from (\ref{red1}) that
\begin{equation*}
\|I\|_{L_x^2}\les \|\hn G\|_{L_x^p} \|(\p_m f)_{\le 1}\|_{L_x^{p'}}.
\end{equation*}
Now consider $J=
\sum_{\ell>1}\hn P\left(G_\ell P_\ell (\p_m f)\right)$, by finite band property,
\begin{align*}
\|J\|_{L_x^2}&\les \sum_{\ell>1} \ell^{-1} \|\hn G\|_{L_x^p} \|P_{\ell} (\p_m f)\|_{L_x^{p'}}\\
&\les \sum_{\ell>1}\ell^{-1+\frac{3}{p}}\|\p_m f\|_{L_x^2}\|\hn G\|_{L_x^p}\les \|\hn G\|_{L_x^p} \|\hn f\|_{L_x^2}.
\end{align*}
Combining the estimates for $I$ and $J$, in view of (\ref{err6}), we can complete the proof.
\end{proof}

For ease of exposition, let us introduce the first version of
conformal energy and state its boundedness theorem whose proof
occupies the rest of the paper.

\begin{theorem}[Boundedness theorem]\label{BT}
Let $\psi$ be a solution of $\Box_\bg \psi=0$ whose initial data is
supported in $B_{\f12}\subset(\D_0^+\cap \Sigma_{t_0})$.  In the
region $\D_0^+$,
\begin{equation}\label{confen}
C[\psi](t):=\int_{\Sigma_t}\{t^2(|\sn
\psi|^2+|\sn_L \psi|^2)+u^2|\nab\psi|^2+(\frac{t^2}{(t-u)^2}+1)\psi^2\} d\mu_g
\end{equation}
there holds  for $t\in [t_0, \tau_*]$, $C[\psi](t)\les
\|\psi[t_0]\|_{H^1}^2+\|\psi(t_0)\|_{L^2(\Sigma)}^2$.
\end{theorem}


\begin{lemma}\label{err11}
Let $q>2$ and $0<\delta\le 1-\frac{2}{q}$ be two numbers.
Assuming (\ref{BA2}) and
\begin{equation}\label{aricc}
\|\varpi(\chih, \sn \log \bb)\|_{L^2[0,\tau_*] L_x^\infty}\les \la^{-\f12},\quad \mbox{ with } \tau_*\le \la T,
\end{equation}
 for any solution
$\psi$ of $\Box_\bg \psi=0$, there holds
\begin{equation*}
\|[P, \varpi N^m]\p_m \psi\|_{L_x^\infty}+\|[\varpi\sn, P]\psi\|_{H^1}\les \ti d(t)(\|\psi[t_0]\|_{H^1}+\|\psi(t_0)\|_{L_x^2}),
\end{equation*}
where
\begin{equation*}
(1+t)^{\delta}\ti d(t)\les (1+t)^{-\frac{2}{q}}+d(t)
\end{equation*}
with $d(t)$ being a function  satisfying (\ref{correc}).
\end{lemma}
The condition
(\ref{aricc}) is incorporated in  (\ref{ric3})  in Proposition \ref{rices} and is proved in Section \ref{sec7}.
\begin{proof}
We first claim
for $t>t_0$ there hold
\begin{align}\label{h1theta}
\|\hn (\varpi N)\|_{L_x^\infty}&\les \|\varpi (\chih, \sn \log \bb), k,\hn g\|_{L^\infty}+(1+t)^{-1}
\end{align}
Indeed,
for $t\ge t_0$, on the support of $\varpi$, i.e. $\cup_{-\frac{t}{4}\le u\le \frac{3t}{4}} S_{t,u}$, the radius $r$ of $S_{t,u}$ within the support of $\varpi$ satisfies $r\approx (1+t)$. (\ref{h1theta}) follows by using (\ref{thetan}), (\ref{recl1}) and
 the fact that $\|\hn \varpi\|_{L_x^\infty}\les (1+t)^{-1}$ .

Let $\Pi_{ij}=g_{ij}-N_i N_j$ denote the projection tensor on $\Sigma$.
Then for any scalar function $f$, we have $\sn_j f=\Pi^{i}_{j}\p_i f$ and
\begin{equation}\label{psnf}
[P, \varpi\sn_j]f=-(\varpi N^i N_j) P\p_i  f+ P((\varpi N_j N^i) \p_i f)=[P, \varpi N_j N^i]\p_i f.
\end{equation}
Applying Lemma \ref{comh1} to $(G, f)=(\varpi N^i, \psi), \, (\varpi N_j N^i, \psi)$, and
using Proposition \ref{geoeg} for $\psi$, we have
\begin{equation*}
\|[P, G]\p_i f\|_{L_x^\infty} + \|[P, G]\p_i f\|_{H^1}
\les \|\hn G\|_{L_x^\infty}(\|\bd \psi(t_0)\|_{L_x^2}+\|\psi(t_0)\|_{L_x^2}).
\end{equation*}
Now for $q>2$, we set
$\ti d(t)=\|\hn G\|_{L_x^\infty}$.
We have from (\ref{h1theta}) that
\begin{align*}
\ti d(t)&\les(1+t)^{-1}+\|\varpi(\chih,\sn\log \bb), \hn g, k\|_{L_x^\infty}
= (1+t)^{-1}+\ti d^{(2)}(t).
 \end{align*}
By using (\ref{aricc}) and H\"{o}lder inequality, we have
$$
 \|\ti d^{(2)}(t)\|_{L^{\frac{q}{2}}}\les \la^{\frac{2}{q}-1} T^{\frac{2}{q}-\frac{1}{2}}.
$$
Thus,   with $0< \delta\le 1-\frac{2}{q}$ and $d(t)=(1+t)^\delta \ti d^{(2)}(t)$, we can complete the proof.
\end{proof}

\begin{lemma}\label{basic1}
(i) Let  $S_t=\Sigma_t\cap\N^+(p)$, with $p\in [-\la T, \la T]\times\Sigma$.   For $S_t$ tangent tensor $F$, there holds
\begin{equation}\label{sob.12}
\|r^{1-2/q}F\|_{L^q(S_t)}\les\|r\sn F\|_{L^2(S_t)}^{1-2/q}\|F\|_{L^2(S_t)}^{2/q}+\|F\|_{L^2(S_t)}.
\quad 2\le q<\infty.
\end{equation}
(ii) For any $\delta\in (0,1)$, any $q\in (2,\infty)$ and any scalar function $f$
there hold
\begin{align*}
\sup_{S_{t,u}}|f|\les  r^{\frac{2\delta(q-2)}{2q+\delta(q-2)}}
&\left(\int_{S_{t,u}} \left(|\sn f|^2+r^{-2} |f|^2\right)\right)^{\f12-\frac{\delta q}{2q+\delta(q-2)}} \\
& \qquad\qquad \times \left(\int_{S_{t,u}} \left(|\sn f|^q+r^{-q}|f|^q\right) \right)^{\frac{2\delta}{2q+\delta(q-2)}}.
\end{align*}
\end{lemma}

\begin{proof}
This is \cite[Theorem 5.2]{KRduke}
\end{proof}

Now we are ready to complete the proof of Proposition \ref{lcestimate}.

\begin{proof}[Proof of Proposition \ref{lcestimate}]
We first claim that
\begin{equation}\label{eqn16}
\|\varpi P \psi\|_{L^\infty(\Sigma_t)}\les\left(\frac{1}{{(1+|t-t_0|)}^{\frac{2}{q}}}+d(t)\right)(\|\psi[t_0]\|_{H^1}+\|\psi(t_0)\|_{L^2}).
\end{equation}
Since $\varpi$ vanishes outside the region $-t/4\le u< 3t/4$,
this claim is trivial there. Thus we may restrict our consideration to the region
$-t/4\le u< 3t/4$. In view of $r\approx t-u$, we thus have $r\approx t$ for $t>0$.
Recall that $\varpi$ is constant on each $S_{t,u}$,  from Lemma \ref{basic1} (ii), we can obtain
\begin{align*}
\sup_{S_{t,u}}|\varpi P\psi|^2
& \les r^\delta \left( \int_{S_{t,u}} \left(|\varpi \sn P \psi|^2
+r^{-2}|\varpi P\psi|^2\right)\right)^{1-\delta} \\
& \qquad \qquad\qquad \times \left(\int_{S_{t,u}} \left(|\varpi \sn P\psi|^4
+r^{-4}|\varpi P\psi|^4\right) \right)^{\f12\delta}.
\end{align*}
Applying Lemma \ref{trace2} and the finite band property, we then obtain
\begin{align*}
\sup_{S_{t,u}}|\varpi P\psi|^2 & \les r^\delta \left(\int_{S_{t,u}} \left(|P(\varpi\sn
\psi)|^2+r^{-2}|\varpi P\psi|^2+|[P,\varpi\sn]\psi|^2\right) \right)^{1-\delta}\\
&\quad \, \times \left(\int_{S_{t,u}} \left(|P(\varpi \sn \psi)|^4+r^{-4}|\varpi P\psi|^4
+| [P,\varpi\sn]\psi|^4\right)\right)^{\f12\delta}\\
&\les r^\delta \left(t^{-2} C[\psi](t)+\|[P,\varpi\sn]\psi\|^2_{H^1}\right)
\end{align*}
By letting $0<\delta\le 2( 1-\frac{2}{q})$,  (\ref{eqn16}) then follows  from Theorem \ref{BT} and Lemma \ref{err11}.

In order to derive the estimate on $P (e_0 (\psi))$, we can write
$\|P \left(e_0 \psi\right)\|_{L_x^\infty}\le \|P (\varpi e_0 \psi)\|_{L_x^\infty}
+\|P\left((\underline{\varpi}-\varpi)e_0 \psi\right)\|_{L_x^\infty}$.
By  Bernstein inequality and  Theorem \ref{BT}, we have
\begin{align*}
\|P\left( (\underline{\varpi}-\varpi) e_0 \psi\right)\|_{L_x^\infty}
& \les\|(\underline{\varpi}-\varpi) e_0 \psi\|_{L_x^2}
\les (1+t)^{-1} C[\psi]^\f12(t)\\
& \les(1+t)^{-1} (\|\psi[t_0]\|_{H^1(\Sigma)}+\|\psi(t_0)\|_{L^2(\Sigma)})
\end{align*}
and
\begin{align*}
\|P(\varpi e_0\psi)\|_{L_x^\infty}&\les \|P(\varpi L\psi)\|_{L_x^\infty}
+\|P(\varpi N\psi)\|_{L_x^\infty}
\les \|\varpi L\psi\|_{L_x^2}+\|P(\varpi N\psi)\|_{L_x^\infty}.
\end{align*}
From Theorem \ref{BT} we have
$$\|\varpi L\psi\|_{L_x^2}\les (1+t)^{-1} C[\psi]^{\f12}(t)
\les (1+t)^{-1}(\|\psi[t_0]\|_{H^1(\Sigma)}+\|\psi(t_0)\|_{L^2_\Sigma}).$$
Moreover
\begin{align}
\|P(\varpi N\psi)\|_{L_x^\infty}&\le \|[P, \varpi N^l]\p_l \psi\|_{L_x^\infty}
+\|\varpi N^l P \p_l \psi\|_{L_x^\infty}. \label{eqn19}
\end{align}
The first term in (\ref{eqn19}) can be estimated by Lemma \ref{err11}.
By using (\ref{eqn16}), the second term  in (\ref{eqn19}) can be estimated as
\begin{equation*}
\|\varpi N^l P \p_l \psi(t)\|_{L_x^\infty}\le \|\varpi \ti P\psi(t)\|_{L_x^\infty}\les
\left((1+t)^{-\frac{2}{q}}+d(t)\right)(\|\psi[t_0]\|_{H^1}+\|\psi(t_0)\|_{L_\Sigma^2}),
\end{equation*}
where $\ti P$ denotes a Littlewood Paley projection with frequency $1$ associated to
a different symbol. Putting the above estimates together completes the proof.
\end{proof}

\section{\bf Boundedness theorem for Conformal energy}\label{sec5}

In this section we will present the proof of Theorem \ref{BT}. We will work under
the rescaled coordinates. Let $M_*=[0,\tau_*]\times \Sigma$, where $\tau_*\le \la T$,
where $\la\ge \La$ and $\La$ is a sufficiently large number.

Recall the definition of the optical function $u$. We will set $\ub:=2t-u$ and
introduce the Morawetz vector field $K:=\frac{1}{2} n (u^2 \Lb+\ub^2 L)$.
Associated to $K$ we introduce the modified energy density
\begin{equation*}
\bar Q(K, \bT)=\bar Q[\psi](K, \bT)=Q[\psi](K,\bT)+2t \psi
\bT\psi-\psi^2 \bT(t),
\end{equation*}
and the total conformal energy
\begin{equation*}
\bar Q[\psi](t)=\int_{\Sigma_t} \bar Q[\psi](K, \bT).
\end{equation*}


\begin{definition}
We define  $\C[F]$ for a scalar function $F$ with $\mbox{supp} F\subseteq\D_0^+$,   by
\begin{equation*}
\C[F](t)=\C[F]^\bi(t)+\C[F]^\be(t)
\end{equation*}
where $ t_0\le t\le \tau_* $, and
\begin{align*}
&\C[F]^\bi(t)=\int_{\Sigma_t} (\underline{\varpi}-\varpi)
\left(t^2|\bd F|^2+\left(1+\frac{t^2}{(t-u)^2}\right)|F|^2\right) d\mu_g,\\
 &\C[F]^\be(t)=\int_{\Sigma_t}\varpi\left(\ub^2|\bd_L F|^2+ u^2 |\bd_\Lb
 F|^2+\ub^2|\sn F|^2+|F|^2\right) d\mu_g.
 \end{align*}
\end{definition}

From the definition it is easy to see that $\bar Q[\psi](t)\les
\C[\psi](t)$ and $C[\psi](t)\les \C[\psi](t)$. We will prove the
following results.

\begin{theorem}[Comparison Theorem]\label{cprsthm}
$T>0$ can be chosen appropriately small but depending on universal constants,
such that for any function $\psi$ supported in $\D_0^+$ and any $1\le t\le \tau_*$ there holds
\begin{equation*}
\C[\psi](t)\approx \bar Q[\psi](t).
\end{equation*}
\end{theorem}

\begin{theorem}[Boundedness theorem]\label{bdness}
There exists a large universal number $\La$ and a small universal number $T>0$
such that for any $\la>\La$, $\tau_*\le \la T$ and any function $\psi$ satisfying
the geometric wave equation
\begin{equation}\label{wave.1}
\Box_\bg \psi= \frac{1}{\sqrt{|\bg|}} \p_\a
(\bg^{\a\b}\sqrt{|\bg|}\p_\b \psi)=0 \quad \mbox{in } [0,\tau_*]\times {\Bbb R}^3
\end{equation}
with initial data $\psi[t_0]$ supported on the ball $B_{1/2}(0)$  there holds
$$
\bar Q[\psi](t)\les \bar Q[\psi](t_0) \qquad \forall\, t_0\le t\le \tau_*.
$$
\end{theorem}

\subsubsection{Canonical null pair $L,\Lb$}

Recall that in (\ref{bb1}) we have introduced along the null hypersurface $C_u$
the canonical null frame $\{L, \Lb, e_1, e_2\}$, where $\{e_1, e_2\}$ is an orthonormal
frame on $S_{t,u}$. Let $e_3=\Lb=\bT-N$ and $e_4=L=\bT+N$. Then
from (\ref{bb1}) it follows that
$$
\bd_3 u=2\bb^{-1}, \quad \bd_3 \ub=2(n^{-1}-\bb^{-1}),\quad \bd_4
\ub=2n^{-1}, \quad \bd_4 u=0.
$$
Associated to this canonical null frame, we define on each null cone $S_{t,u}$ the Ricci
coefficients
\begin{eqnarray*}
\chi_{AB}=\l \bd_A e_4, e_B\r, && \chib_{AB}=\l \bd_A e_3, e_B\r\\
\zeta_A=\f12 \l \bd_3 e_4, e_A\r, &&\zb_A=\f12\l \bd_4 e_3, e_A\r\\
\xi_A=\f12 \l\bd_3 e_3, e_A\r. &&
\end{eqnarray*}
It is well-known (\cite{KR1}) that there hold the identities
\begin{eqnarray*}
\chib_{AB}=-\chi_{AB}-2 k_{AB}, && \zb_A=-k_{AN}+\sn_A \log  n, \\
\xi_A=k_{AN}-\zeta_A+\sn_A\log n, &&\zeta_A=\sn_A\log\bb+k_{AN}.\label{idd1}
\end{eqnarray*}
and the frame equations
\begin{eqnarray*}
\bd_4 e_4=-(k_{NN}+\pi_{0N}) e_4,&& \bd_A e_4 =\chi_{AB}e_B-k_{AN}e_4,\\
\bd_A e_3=\chib_{AB} e_B+k_{AN} e_3,&& \bd_4 e_3=2\zb_A
e_A+(k_{NN}+\pi_{0N})e_3,\\
\bd_3 e_4=2\zeta_A e_A+(k_{NN}-\pi_{0N})e_4, &&\bd_3 e_3=2\xi_A
e_A-(k_{NN}-\pi_{0N}) e_3,\\
\bd_4 e_A=\sn_4 e_A+\zb_A e_4,&& \bd_3 e_A=\sn_3e_A+\zeta_A
e_3+\xi_A e_4.
\end{eqnarray*}
Let
\begin{equation}\label{zomega}
z=\tr\chi-\frac{2}{n(t-u)}, \qquad \Omega=\frac{\bb^{-1}-n^{-1}}{t-u}.
\end{equation}

We  rely on the following result to prove the boundedness theorem.  (\ref{ric4})-(\ref{ric3})  will be proved in Sections \ref{ricc} and \ref{sec7}, and (\ref{comp2}) is  (\ref{recl1}). (\ref{pi.2}) consists of  (\ref{BA2}) and the $L_t^2 L_x^\infty$ estimates on $e_0 n , \hn n , \hn Y$ that can be derived immediately by using (\ref{BA2}), (\ref{q3}) and (\ref{shift3.3}), under the rescaled coordinates.

\begin{proposition}\label{rices}
Under the bootstrap assumptions (\ref{ba1}), (\ref{BA2}) and (\ref{BA3}), on $\D^+\subset[0,\tau_*]\times \Sigma$ there hold the estimates
\begin{align}
&(t-u)\tr\chi\approx 1\label{comp2}\\
&\| \ti\pi \|_{L_t^2 L_x^\infty}\les
\la^{-1/2}\label{pi.2}\\
&\|\Omega\|_{L_t^2 L_x^\infty}+\|\Omega\|_{L^4({S_{t,u}})}
+\|(t-u)^{-1/2}\Omega\|_{L^2(S_{t,u})}\les \la^{-1/2}\label{ric4} \\
&\|z,\zeta\|_{L^4(S_{t,u})}+\|r^{-1/2}(z,\zeta,k)\|_{L^2(S_{t,u})}
\les \la^{-1/2}\label{ricstu}\\
&\|z\|_{L_t^2 L_x^\infty}\les \la^{-1/2},\label{ric1}\\
&\|r^{3/2}\sn z, r^{3/2}\Lb z\|_{L_t^\infty L_u^\infty L_\omega^p}\les \la^{-1/2}\label{ricp}\\
&\|\chih, \zeta\|_{L_t^2 L_x^\infty}\les\la^{-1/2},\label{ric3}
\end{align}
where $p>2$ is such that $0<1-2/p<s-2$.
\end{proposition}

Let ${}^{(K)}\pi$ denote the deformation tensor of $K$ and let $\bkpi:={}^{(K)}\pi-4t\bg$.
Then we have
\begin{align*}
\bkpi_{34}&=-4un(\bb^{-1}-n^{-1})+n u^2(k_{NN}-\pi_{0N}- \bd_3\log n)\\
&\quad \,     +n\ub^2(k_{NN}+\pi_{0N}-\bd_4\log n),\\
\bkpi_{44} & =-2u^2n(\nab_L \log n+k_{NN}+\pi_{0N}),\\
\bkpi_{4A} & =u^2n(\zb_A-k_{AN}-\sn_A\log n),\\
\bkpi_{33} & =-8\ub n(n^{-1}-\bb^{-1})-2n\ub^2(k_{NN}-\pi_{0N}+\bd_3\log n),\\
\bkpi_{3A} & =n\ub^2(\zeta_A+k_{AN}-\sn_A\log n)+n u^2\xi_A,\\
\bkpi_{AB} & = -2nu^2 {\hat k}_{AB}-nu^2 \tr k \delta_{AB}+4tn(t-u)\chih_{AB}\\
&\quad\,     +2tn(t-u)\left(\tr\chi-\frac{2}{n(t-u)}\right)
\delta_{AB}.
\end{align*}
For simplicity of presentation, we will dropped the superscript $K$ in $\bkpi$.
In view of (\ref{idd1}), as an immediate consequence of Proposition \ref{rices}, we have

\begin{proposition}\label{deformofk}
Under the conditions in Proposition \ref{rices} we have on $[0, \tau_*]\times \Sigma$ that
\begin{align*}
& \|u^{-2} \,\, \bpi_{44}\|_{L_t^2 L_x^\infty}+\|(u \ub)^{-1} \,\, \bpi_{34}\|_{L_t^2 L_x^\infty}
+\| \ub^{-2} \,\, \bpi_{33}\|_{L_t^2 L_x^\infty}\\
&\quad  +\|u^{-2} \,\, \bpi_{4A}\|_{L_t^2 L_x^\infty}+\|\ub^{-2} \,\, \bpi_{3A}\|_{L_t^2 L_x^\infty}
+\|\ub^{-2} \,\, \bpi_{AB}\|_{L_t^2 L_x^\infty}\les \la^{-1/2}.
\end{align*}
\end{proposition}

Now we are ready to give the proof of Theorem \ref{bdness} and Theorem \ref{cprsthm}.

\subsection{Proof of Theorem \ref{bdness}}

By calculating $\p_t \bar Q[\psi]$ and integrating over the interval $[t_0,t]$, we have
\begin{equation*}
\bar Q[\psi](t)=\bar Q[\psi](t_0)-\f12 \J_1+\J_2
\end{equation*}
where
\begin{equation*}
\J_1=\int_{[t_0, t]\times\Sigma}Q^{\a\b}[\psi] \, \, \bkpi_{\a\b},
\qquad \J_2=\int_{[t_0, t]\times \Sigma} \psi^2 \Box_{\bg}t.
\end{equation*}
It is easy to see that
\begin{align*}
\J_1&=\int_{[t_0,t]\times \Sigma} \left(\frac{1}{4} \bpi_{33} (
L\psi)^2+\frac{1}{4}\bpi_{44} (\Lb \psi)^2+\frac{1}{2}\bpi_{34}|\sn
\psi|^2-\bpi_{4A}\Lb \psi\sn_A \psi\right.\\
&\quad\quad \left.-\bpi_{3A} L\psi\sn_A \psi+\bpi_{AB}\sn_A \psi\sn_B
\psi+\tr\bpi\left(\f12 \Lb \psi L\psi-|\sn\psi|^2\right)\right).
\end{align*}
Observe that
\begin{equation*}
\tr\bpi=\delta^{AB} \bpi_{AB}=4 tn(t-u)
(\tr\chi-\frac{2}{n(t-u)})-2u^2 n \tr k.
\end{equation*}
It is easy to derive from (\ref{pi.2}) that
\begin{equation*}
\int_{[t_0,t]\times \Sigma}|u^2 n \tr k L\psi \Lb \psi| \les
T^{\f12} \sup_t \C[\psi](t).
\end{equation*}
Thus, by letting
\begin{equation*}
\B=\int_{[t_0,t]\times \Sigma} 2 t' n (t'-u)
\left(\tr\chi-\frac{2}{n(t'-u)}\right) \Lb \psi L \psi,
\end{equation*}
we have from Proposition \ref{deformofk} that
\begin{equation}\label{j1}
|\J_1-\B|\les T^{\f12}\sup_{t}\C[\psi](t).
\end{equation}
Since $\Box_\bg t=-e_0(n^{-1})+n^{-1}\Tr k$, we can conclude from (\ref{pi.2}) that
\begin{equation*}
|\J_2|\les T^{\f12}\sup_{t}\C[\psi](t)
\end{equation*}

In the following we will show that
\begin{align}\label{b11}
|\B|\les T^{\f12}\sup_{t}\C[\psi](t)
\end{align}
We can write $\B=\B^i+\B^e$, where
\begin{align*}
\B^e&=\int_{[t_0, t]\times \Sigma} 2 t'n(t'-u)
\left(\tr\chi-\frac{2}{n(t'-u)}\right) L\psi\Lb \psi \varpi, \nn\\
\B^i&=\int_{[t_0,t]\times \Sigma} 2 t'n (t'-u)
\left(\tr\chi-\frac{2}{n(t'-u)}\right) L\psi\Lb
\psi(\underline{\varpi}-\varpi).
\end{align*}
In view of (\ref{ric1}), we have
\begin{equation}\label{bi}
\B^i\les T^{\f12}\sup_t \C[\psi](t).
\end{equation}
We still need to estimate $\B^e$.  With the help of \cite[p. 1162]{KR1} and
the integration by part we have
\begin{align}
\f12 \B^e&=-I_1+I_2+I_3-I_4,\nn
\end{align}
where
\begin{align}
I_1&=\int_{[t_0,t]\times \Sigma} \varpi n t'(t'-u) z (\Lb L \psi)\psi, \nn\\
I_2&=\int_{[t_0, t]\times \Sigma} \left(-\Lb(\varpi nt'(t'-u)
z)+(\tr\theta+N\log n -\Tr k-\div Y\right) \nn  \varpi nt'(t'-u)z) L\psi \psi, \nn\\
I_3&=\int_{\Sigma_t}\varpi n t'(t'-u) z L\psi \psi,\nn\\
I_4&=\int_{\Sigma_{t_0}} \varpi n t'(t'-u)z L\psi \psi.\nn
\end{align}

Recall that in the exterior region $\{0\le u\le 3t'/4\}$ we have $r(t',u)\approx t'$.
Thus with the help of the Sobolev inequality (\ref{sob.12}) on $S_{t',u}$,
for any function $\psi$ there holds
\begin{equation}\label{sobex}
\int_{0\le u\le 3t'/4}\|{t'}^{1-2/q}\psi\|_{L^q({S_{t',u}})}^2 du\les \C[\psi](t'), \quad 2\le q<\infty.
\end{equation}
Therefore, by using (\ref{ricstu}) and (\ref{sobex}), for the boundary term $I_3$
and $I_4$ we have the estimate
\begin{align*}
|I_3|+|I_4|\les \|t L\psi\|_{L_t^\infty
L^2_\Sigma}\|r^{\f12}\psi\|_{L_t^\infty L_u^2 L_x^{4}}\sup_{t,
0<u<\frac{3 t}{4}}\| r^{\f12}z\|_{L^4(S_{t,u})}\les
T^{\f12}\sup_{t}\C[\psi](t).
\end{align*}

Now we consider $I_2$. We write $I_2=I_2^{(1)} +I_2^{(2)}$, where
\begin{align*}
I_2^{(1)}=&\int_{[t_0, t]\times \Sigma} \Lb (\varpi n t'(t'-u) z)L\psi \psi,\\
I_2^{(2)}=&\int_{[t_0,t]\times \Sigma}(\tr\theta+N\log n -\Tr k-\div
Y) \varpi nt'(t'-u)z L\psi \psi.
\end{align*}
In view of (\ref{comp2}), (\ref{pi.2}) and (\ref{ric1}) in Proposition \ref{rices},
by H\"older inequality we have
\begin{align*}
|I_2^{(2)}|&\les (\tau_*\|\pi, \nab Y,z\|_{L_t^2
L_x^\infty}+\tau_*^{\f12})\|z\|_{L_t^2
L_x^\infty}\sup_{t}\C[\psi](t)\les T^{\f12}\sup_{t}\C[\psi](t).
\end{align*}
Observe that
\begin{equation}\label{lbvp}
|\Lb\varpi|\les r^{-1},\quad  \Lb t=n^{-1},\quad  \Lb u=2\bb^{-1}.
\end{equation}
Let $p>2$ be close to $2$ such that $0<1-2/p<s-2$ and let $1/p+1/q=1/2$.
Then it follows from (\ref{ricp}) that
{\small
\begin{align*}
|I_2^{(1)}|&\les \left(\|(|(t'-u)\Lb \varpi|+1)z\|_{L_t^1
L_x^\infty}+\tau_*\|\Lb n \|_{L_t^2 L_x^\infty}\|z\|_{L_t^2
L_x^\infty}\right) \|t' L\psi\|_{L_t^\infty
L_\Sigma^2}\|\psi\|_{L_t^\infty L_\Sigma^2}\\
&+\|t'^{\frac{2}{q}}\varpi\Lb z\|_{L_t^1 L^\infty_uL_x^{p}}
\sup_{t'} \left(\int_{0\le u\le \frac{3t'}{4}}
\|{t'}^{1-\frac{2}{q}}\psi\|_{ L^q({S_{t',u}})}^2 du \right)^{\f12}
\|t'L\psi\|_{L_t^\infty L_\Sigma^2}.
\end{align*}
}
In view of (\ref{sobex}), (\ref{pi.2}), (\ref{ric1}) and (\ref{ricp}), we obtain
\begin{equation*}
|I_2^{(1)}|\les T^{\f12}\sup_{t} \C[\psi](t)
\end{equation*}

Finally we will use (\ref{wave.1}) to estimate $I_1$.
We first rewrite (\ref{wave.1}) as
\begin{equation}\label{wave.2}
\Box_\bg \psi=-\Lb L\psi+\sD \psi+2\zeta_A \sn_A \psi-\f12 \tr\chi
\Lb\psi-(\f12 \tr\chib+\nu) L\psi.
\end{equation}
Then $I_1$ can be written as $I_1=I_{11}+I_{12}+I_{13}$, where
\begin{align*}
I_{11}&=\int_{[t_0,t] \times \Sigma} \varpi n t'(t'-u) z\sD\psi \psi\\
I_{12}&=-\f12 \int_{[t_0,t]\times \Sigma} \varpi n t'(t'-u) z \tr\chi \psi \Lb \psi \\\
I_{13}&=\int_{[t_0,t]\times \Sigma}\varpi n t'(t'-u) z
\left(2\zeta_A \sn_A \psi-\left(\frac{1}{2}\tr\chib+\nu\right)L\psi\right)\psi.
\end{align*}
Recall that for any vector field $X$ tangent to $S_{t,u}$ there holds
\begin{equation}\label{in.b.p}
\int_{\Sigma_t} F \div X =-\int_{\Sigma_t} \{\sn+(\zeta+\zb)\}F\c X.
\end{equation}
In view of (\ref{in.b.p}) we have
\begin{align*}
I_{11}&=-\int_{[t_0,t]\times \Sigma} \sn(\varpi n t'(t'-u) z \psi)
\sn \psi+(\zeta+\zb) \varpi nt'(t'-u) z \psi\c \sn \psi.
\end{align*}
Now we introduce the following types of terms:
\begin{align*}
\Er_1&=nt'(t'-u)\left(\pi, z, \frac{\bb^{-1}-n^{-1}}{(t-u)}\right) \c (L\psi,
\sn \psi)\c z,\\
\Er_2&=n z t' L\psi,\qquad\qquad \qquad\qquad \,\,\Er_3=n t'(t'-u)\sn z \c \sn \psi,\\
\Er_4&=nt'z ((t'-u) \sn \psi, \psi), \qquad \quad \Er_5=nt'(t'-u) \zeta \c  \sn \psi\c z.
\end{align*}
Then, symbolically we can write $I_{11}$ and $I_{13}$ as
\begin{align*}
|I_{11}|+|I_{13}|=\left|\int_{[t_0,t]\times \Sigma}\varpi(\Er_1+\Er_2+\Er_3+\Er_5)\psi\right|
+\left|\int_{[t_0,t]\times\Sigma}\varpi \Er_4\c \sn \psi\right|.
\end{align*}
By using (\ref{pi.2}), (\ref{ric1}), (\ref{ric4}) and H\"older inequality, we can derive
\begin{align*}
\left|\int_{[t_0,t]\times\Sigma} |\varpi \Er_1\c \psi|\right|
&\les \|z\|_{L_t^2 L_x^\infty}
\left\|z, \pi, \frac{\bb^{-1}-n^{-1}}{n(t'-u)}\right\|_{L_t^2 L_x^\infty}\tau_*\sup_t\C[\psi](t)\\
&\les T \sup_{t}\C[\psi](t)
\end{align*}
and
\begin{align*}
&\left|\int_{[t_0,t]\times \Sigma}|\varpi \Er_2 \psi|+|\varpi\Er_4 \sn\psi|\right|\\
&\qquad \qquad \les \|z\|_{L_t^1 L_x^\infty} \sup_{t} \left\{\|t (|L\psi|+|\sn
\psi|)\|_{L^2_{\Sigma}}(\|\psi\|_{L^2_\Sigma}+\|t \sn \psi\|_{L^2_\Sigma})\right\}\\
&\qquad \qquad \les T^{\f12}\sup_{t} \C[\psi](t).
\end{align*}
By using (\ref{ricp}) and (\ref{sobex}) with $0<1-2/p< s-2$ and $1/q+1/p=\f12$, we have
\begin{align*}
&\left|\int_{[t_0,t]\times \Sigma}| \varpi\Er_3 \c\psi|\right|\\
&\les \sup_t \|t^{\frac{2}{q}}\varpi\sn z\|_{L_t^1 L_u^\infty L_x^p}
\sup_{t'} \left\{\C[\psi](t')^{\f12}\left(\int_{0\le u\le
\frac{3t'}{4}}\|{t'}^{1-\frac{2}{q}}\psi\|_{ L^q({S_{t',u}})}^2
du\right)^{\f12}\right\}\\
&\les T^{\f12}\sup_{t} \C[\psi](t).
\end{align*}
With the help of (\ref{ric1}), (\ref{sobex}) with $q=4$ and (\ref{ricstu}), we obtain
\begin{align}
&\left|\int_{[t_0,t]\times \Sigma}|\varpi \Er_5 \psi|\right|\\
&\les \int_{t_0}^t \sup_{u}\|r^{\f12} \varpi z\c \zeta\|_{L_x^{4}}\c
\sup_{t'}\left\{\|t'\sn\psi\|_{ L_\Sigma^2} \left(\int_{0\le u\le
\frac{3t'}{4}}\|{r'}^{\f12}\psi\|_{ L^4({S_{t',u}})}^2
du\right)^{\f12}\right\}\nn\\
&\les \|r^{\f12}\zeta\|_{ L_t^\infty L_u^\infty L_x^4}\|z\|_{L_t^1
L_x^\infty}\sup_t \C[\psi](t)\les T\sup_t \C[\psi](t).\nn
\end{align}
Now we consider $I_{12}$ by integration by part.
\begin{align}
I_{12}&=\frac{1}{4}\int_{[t_0,t]\times \Sigma} -\Lb (\varpi n
t'(t'-u) z\tr\chi)\psi^2\label{tp6}\\
&+\frac{1}{4} \int_{[t_0, t]\times \Sigma} \left(\tr\theta+N\log n-\Tr k-\div Y\right) \varpi n t'(t'-u)
z\tr\chi \, \psi^2\label{tp4}\\
&+\frac{1}{4}\int_{\Sigma_t} \varpi n t'(t'-u) z \tr\chi
\psi^2-\frac{1}{4}\int_{\Sigma_0} \varpi n t'(t'-u) z\tr\chi
\psi^2.\label{tp5}
\end{align}
Using (\ref{comp2}), (\ref{pi.2}) and  (\ref{ric1}), the terms in (\ref{tp4}) is bounded by
\begin{align*}
\int_{[t_0,t]\times \Sigma}& \left|\varpi(z+\pi+\nab
Y+\frac{1}{n(t-u)})z\right|t'n \psi^2\\&\les \|z\|_{L_t^2
L_x^\infty}\||\psi|^2\|_{L_t^\infty L_\Sigma^1}(\tau_*^{1/2}+\tau_*\|\pi,
\nab Y, z\|_{L_t^2 L_x^\infty})\\&\les T^\f12\sup_{t}\C[\psi](t).
\end{align*}
For the terms in (\ref{tp5}), using (\ref{comp2}), (\ref{ricstu})
and (\ref{sobex}) with $q=4$,  they can be bounded by
\begin{align*}
\sup_{t,u}&\| z\|_{L^2(S_{t,u})}\int_{0<u\le
\frac{3t}{4}}\|t'|\psi|^2\|_{L^2(S_{t,u})}\les T^{\f12}\sup_t
\C[\psi](t).
\end{align*}
It remains only to consider the terms in (\ref{tp6}). We have
\begin{align}
&\left|\int_{[t_0,t]\times \Sigma} \Lb (\varpi n t'(t'-u) z\tr\chi)
\psi^2\right|\nn\\
&\qquad\qquad \les\large|\int_{[t_0,t]\times\Sigma}\{\Lb \varpi n
t'(t'-u)\tr\chi+\varpi n \tr\chi\Lb(t'(t'-u))\nn\\
&\qquad\qquad +\varpi\tr\chi\Lb n
t'(t'-u)+\varpi \Lb(\frac{1}{n(t'-u)})  n t' (t'-u)\}z\psi^2\large|\label{tp9}\\
&\qquad \qquad +\int_{[t_0,t]\times\Sigma}|\varpi\Lb z n t'(t'-u)| (|\tr\chi
|+|z|)\psi^2.\label{tp8}
\end{align}
Since $r\approx (t-u)$ and  $t(|\tr\chi|+|z|)\approx 1$  on $\Et_t$,
with  $\frac{1}{p}+\frac{1}{q}=1$ and $p$ is slightly greater than
$2$, also using H\"older inequality, we obtain
\begin{align*}
(\ref{tp8})&\les \|\varpi r^{1-\frac{2}{p}}\Lb z\|_{L_t^1 L_u^\infty
L^p_x}\sup_{t'}(\int_{0\le u\le \frac{3t'}{4}}\|r^{\frac{2}{p}}|\psi|^2\|_{ L^q(S_{t',u})}du)\\
&\les  \|\varpi r^{1-\frac{2}{p}}\Lb z\|_{L_t^1 L_u^\infty
L^p_x}\sup_{t'}(\int_{0\le u\le
\frac{3t'}{4}}\|r^{1-\frac{2}{q_1}}\psi\|_{ L^{q_1}(S_{t',u})}
\|r^{1-\frac{2}{q_2}}\psi\|_{L^{q_2}(S_{t',u})}d u)
\end{align*}
where $\frac{1}{q_1}+\frac{1}{q_2}=\frac{1}{q}$ .  Using
(\ref{sobex}) and (\ref{ricp}) we obtain
\begin{equation*}
(\ref{tp8})\les T^{\f12}\sup_{t'} \C[\psi](t').
\end{equation*}

 To estimate (\ref{tp9}), in view of  (\ref{lbvp}) and
 \begin{equation}\label{lbr}
\Lb(\frac{1}{n(t-u)})=\frac{n\bb^{-1}}{n^2(t-u)^2}-\frac{\Lb \log n
}{n(t-u)}-\frac{n(n^{-1}-\bb^{-1})}{n^2(t-u)^2},
\end{equation}
 using (\ref{ric1}), (\ref{pi.2}) and (\ref{ric4}), we have
\begin{align*}
|(\ref{tp9})|&\les\|z\psi^2\|_{L_t^1 L_\Sigma^1}+\|\Lb n,
\frac{\bb^{-1}-n^{-1}}{n(t-u)} \|_{L_t^2
L_\Sigma^\infty}\|z\|_{L_t^2 L_\Sigma^\infty}\tau_*\sup_{t}\int_{0\le
u\le \frac{3 t}{4}}\|\psi^2\|_{L^1(S_{t,u})}\\& \les T^\f12\sup_t
\C[\psi](t).
\end{align*}
The proof is therefore complete.

\subsection{Proof of Comparison theorem}

We  will  adapt the argument in \cite{KR1} to prove Theorem \ref{cprsthm}.
For simplicity, we use $\Theta$ to denote any term from the collection
$$
\left\{\tr\chi-\frac{2}{n(t-u)}, \,\,\, \Tr k, \,\,\,
\frac{\bb^{-1}-n^{-1}}{t-u},\,\,\,{\hat k}_{NN}  \right\}.
$$
According to (\ref{ricstu}) and (\ref{ric4}) in Proposition \ref{rices} we have
\begin{equation}\label{riccm}
\|r^{-\f12}\Theta\|_{ L^2(S_{t,u})}\les \la^{-\f12}.
\end{equation}
By following the argument in \cite[Section 6]{KR1} we can derive
\begin{align*}
\bar Q[\psi](t)&\gtrsim \int_\Sigma  \left(\ub^2 (L\psi)^2+u^2
(\Lb\psi)^2+(u^2+\ub^2)|\sn\psi|^2+\left(1+\frac{t^2}{(t-u)^2}\right)\psi^2\right)\\
&\quad \, -\int_\Sigma \left(1+\frac{t^2}{(t-u)^2}\right)\psi^2(t-u)\Theta\\
&\gtrsim \C[\psi](t)-\int_\Sigma \left(1+\frac{t^2}{(t-u)^2}\right)\psi^2(t-u)\Theta.
\end{align*}
By using (\ref{riccm}) and the inequality, which can be derived in view of (\ref{sob.12}), $$\|t\psi\|_{L_u^2 L_\omega^4}\les \|t\sn
\psi\|_{L_x^2}+\|\frac{t}{t-u} \psi\|_{L_x^2}\les \C[\psi](t),$$ we can obtain
\begin{align*}
\int_\Sigma \frac{t^2}{t-u} \Theta\psi^2
&\les \|r^\f12 \Theta\|_{L_u^\infty L_\omega^2}\|r^{\f12} t^2 \psi^2\|_{L_u^1 L_\omega^2}
\les \la^{-\f12} \tau_*^{\f12}\|t^2 \psi^2\|_{L_u^1 L_\omega^2}
\les T^{\f12} \C[\psi](t).
\end{align*}
Similarly we have
\begin{align*}
\int_{\Sigma} \psi^2(t-u) \Theta &\les \la^{-\f12} \tau_*^{\f12} \|r^2 \psi^2\|_{L_u^1 L_\omega^2}
\les T^{\f12} \C[\psi](t).
\end{align*}
Therefore, there is a universal constant $C_0>0$ such that
\begin{equation*}
\C[\psi](t)\le C_0\bar Q[\psi](t)+ C_0 T^{\f12} \C[\psi](t).
\end{equation*}
This implies the desired conclusion by taking $T$ to be small universal constant.

\section{\bf Estimates for Ricci coefficients}\label{ricc}

In Sections \ref{ricc} and \ref{sec7}, we will complete the proof  of Proposition \ref{rices}. We
will consider only the case $\la=1$ since the general case can be obtained by a simple rescaling.
Thus, we will work on the spacetime $[0, T]\times \Sigma$.


Consider an outgoing null cone $C_u$
in $[0,T]\times \Sigma$ initiating from a point $p$ on $\Ga^+$ with $u=t(p)$.
Let $\D^+$ be the domain enclosed by $C_u$ and $\Sigma_T$ and let
$S_{t,u}:=C_u\cap \Sigma_t$. For simplicity, we will supress $u$ and write $S_t:=S_{t,u}$.
Instead of the canonical null pair $\{L, \Lb\}$, we will work under a new pair of null vector
fields $\{L', \Lb'\}$ which is defined as follows.
The null vectors ${l'}_{\omega}$ in $T_p\bM$ parametrized with $\omega\in {\Bbb S}^2$
are normalized by $\bg(l'_\omega, \bT_p)=-1$. We denote by $\Gamma_\omega(s)$ the
outgoing null geodesic from $p$ with $\Gamma_\omega(0)=p$ and
$\frac{d}{ds}\Gamma_{\omega}(0)=l'_\omega$ and define the null
vector field $L'$ by
$$
L'(\Ga_\omega(s))=\frac{d}{ds}\Ga_\omega(s).
$$
Then $\bd_{L'} L'=0$. The affine parameter $s$ of null geodesic is
chosen such that $s(p)=0$ and $L'(s)=1$. We will assume that the
exponential map $\G_t: \omega\rightarrow \Ga_\omega(s(t))$ is a
global diffeomorphism  from ${\Bbb S}^2$ to $S_t$ for any $t\in(u,
T]$. Along $C_u$ we introduce the null lapse function
$a^{-1}=-\bg( L', \bT)$. Then $a>0$ on $C_u$ with $a(p)=1$.
Moreover, along any null geodesic $\Ga_\omega$ there holds
\begin{equation}\label{st}
\frac{dt}{ds}=n^{-1}a^{-1},\quad t(p)=0.
\end{equation}
We can define a conjugate null vector $\underline{L'}$ on $C_u$ with
$\bg(L', \underline{L'})=-2$ and such that $\underline{L'}$ is
orthogonal to the leafs $S_t$. Let  $N$ be the outward unit normal
of $S_t$ in $\Sigma_t$. Then in terms of $L'$, $\Lb'$ and $a$ we
have
\begin{equation}\label{TN}
L'=a^{-1}(\bT+N), \quad \quad \Lb'=a(\bT-N).
\end{equation}

Relative to the canonical null frame ${L, \Lb}$, we recall the null components of the
Riemannian curvature tensor ${\bR}$ as follows
\begin{eqnarray}
\alpha_{AB}=\bR(L, e_A, L, e_B),&&\qquad  \beta_A=\frac{1}{2} \bR(e_A, L,
\underline{L}, L),\nonumber\\
\rho=\frac{1}{4} \bR(\underline{L},L, \underline{L},L),&&\qquad
\sigma=\frac{1}{4}{{}^\star \bR}(\underline{L}, L, \underline{L}, L),\label{f14}\\
\underline{\beta}_A=\frac{1}{2} \bR(e_A, \underline{L},
\underline{L},L),&& \qquad \underline{\a}_{A B}=\bR(\underline{L}, e_A,
\underline{L}, e_B).\nonumber
\end{eqnarray}
and  $\a', \b', \rho', \sig', \udb', \ab'$ can be similarly defined with ${L, \Lb}$ replaced by ${L',\Lb'}$.
We define
$$
\R(C_u)^2=\int_{C_u} \left(|\a|^2+|\b|^2+|\udb|^2+|\rho|^2+|\sigma|^2\right).
$$
It follows from Proposition \ref{dfk} and the  standard flux type estimate by
using the Bel-Robinson tensor and Einstein vacuum equation,
\begin{equation}\label{pi.3}
\||\nab_L k , \sn k|_g\|_{L^2(C_u)}+\R(C_u)\le C.
\end{equation}

Let us denote by  $\hk$ the traceless part of $k$, $\hk=k-\frac{1}{3}\Tr k \c g$.
We decompose $\hk$ on each $S_t$ by introducing components
\begin{equation}\label{compo}
\e_{AB}=\hk_{AB},\quad \quad \ep_A=\hk_{AN},\quad\quad
{\mathfrak{\delta}}= \hk_{NN}
\end{equation}
where $(e_A)_{A=1,2}$ is an  orthonormal frame on $S_t$. Let $\eh_{AB}$ be the traceless part of $\e$. Since
$\delta^{AB} \e_{AB}=-\delta$, we have
$
\eh_{AB}=\e_{AB}+\frac{1}{2}\delta_{AB}\delta.
$
The Ricci coefficients $\chi', \chib', \zeta', \zb', \xi', \mu'$ and null components of curvature
associated with the null frame $\{e_1, e_2, L', \Lb'\}$ are related to their
counterparts associated with $\{e_1, e_2, L, \Lb\}$ as follows:
\begin{align*}
&\bb=a n_p,\quad \sn\log \bb=\sn \log a, \quad \chi'=a^{-1}
\chi,\quad \chib'=a\chib,\quad \zeta'=\zeta,\\
&\zb'=\zb,\quad \xi'=a^2 \xi, \quad -2\mu'= \mu:=\bd_3 \tr\chi-\f12
(\tr\chi)^2-(k_{NN}+\nab_N \log
n)\tr\chi,\\
&\a'=a^{-2} \a,\quad\b'=a^{-1}\b, \quad \rho'=\rho, \quad \sigma'=\sigma,\quad
\udb'=a\udb
\end{align*}
In the following we will consider the Ricci coefficients relative to the
null frame $\{e_1, e_2, L', \Lb'\}$, and we will drop the prime  for convenience.
We will also fix the following conventions:
\begin{enumerate}
\item[$\bullet$]$\sl{\pi}$ denotes the collection of $\eh,\,\ep,\,\delta,\, \nab_N\log n, \, \sn\log
n, \,\Tr k,$
\item[$\bullet$]$\ckk{\pi}$ denotes the collection of $\sl{\pi}$, $\hn g$,
\item[$\bullet$]  $\iota:=\tr\chi-\frac{2}{r}$,\,$V:=\tr\chi-\frac{2}{s}$, \, $\kappa:=\tr\chi-(an)^{-1}\ovl{an\tr\chi},$
\item[$\bullet$] $A$ denotes the collection of $ \chih,\,\zeta,\, \zb,\, \nu=-\frac{da}{ds}=\pi_{0N}+\delta+\frac{\Tr k}{3},$
\item[$\bullet$]$\Ab$ denotes the collection of $A$ and  $\chibh,\, \sn\log a, \,\sl{\pi},$
\item[$\bullet$]The pair of quantities $(M, \D_0 M)$ denotes either $ (\sn \tr\chi,\sn \chih),$ or $(\mu,\sn \zeta),$
\item[$\bullet$] $R_0$ denotes the collection of $\a,\,\b,\rho, \sigma, \udb$,
\item[$\bullet$]$\bar R$ denotes the collection of $ R_0, \tr\chi \Ab,
A\c\Ab,$
\item[$\bullet$] $S:=S_t$,\, $\gac:=r^{-2}\ga$,\, $\ga^{(0)}:=r^2{\gas}$.
\end{enumerate}
We define the norm of  $\N_1[\cdot]$ on null cone $C_u$ for $S$ tangent tensor fields $F$ by
\begin{equation}
\N_1[F]=\|r^{-1} F\|_{L^2(C_u)}+\|\sn F\|_{L^2(C_u)}+\|\sn_L F\|_{L^2(C_u)}.
\end{equation}
We first recall the following results that hold under the bounded $H^2$ data condition.
The proofs can be found in can be found in \cite{Wang10,Wang1} and  rely on  the bounded
flux (\ref{pi.3}) and (\ref{q1})-(\ref{q3}).

\begin{theorem}\label{recalll2}
There exists a universal constant $C_1\ge 1$ such that on $C_u$ there holds
\begin{equation}\label{calc0}
\C_0:=\R(C_u)+\N_1[\slashed{\pi}]+\|r^{-\f12}\bd n \|_{L^2(S_t)}+\|\bd n\|_{L^4(S_t)}
+\|\bd n\|_{L_t^2 L_\omega^\infty}\le C_1.
\end{equation}
Moreover, let $\R_0:=\R(C_u)+\N_1[\slashed{\pi}]$, there exists a universal constant
$\delta_0>0$ such that if $T-u\le \delta_0$ then on $C_u$ there hold
\begin{align}
\left\|\tr \chi -\frac{2}{s}\right\|_{L^\infty}&\lesssim
\R_0^2,\label{mt0}\\
|a-1|&\le \frac{1}{2},\label{mt7}
\end{align}
\begin{align}
&\left\|\int_0^1 (|\chih|^2+|\zeta|^2+|\nu|^2+|\zb|^2) na dt
\right\|_{L_\omega^\infty}\les \R_0^2,\label{mt2}\\
&\|\sn \tr\chi\|_{L^2(C_u)}+\|\mu\|_{L^2(C_u)} +\left\|\sup_{t\le
1}|r^{3/2}\sn \tr \chi|\right\|_{L^2_\omega}+ \left\|\sup_{t\le
1}r^{\frac{3}{2}}|\mu| \right\|_{L^2_\omega}\les \R_0\label{mt3}\\
 &{\mathcal N}_1(\chih)+{\mathcal
N}_1(\zeta)+{\mathcal N}_1\left(\tr
\chi-\frac{2}{r}\right)+\N_1\left(\tr\chi-(an)^{-1}\ovl{an\tr\chi}\right)\les \R_0,\label{mt4}\\
&\left\|\tr\chi-\frac{2}{r}\right\|_{L_t^2L_\omega^\infty}
+\left\|\tr\chi-(an)^{-1}\ovl{an\tr\chi}\right\|_{L_t^2 L_\omega^\infty}\les \R_0.\label{mt5}
\end{align}
\end{theorem}

\begin{proposition}\label{cmps}
 On each $S_t$ we introduce the ratio of area elements
\begin{equation}\label{vt1}
v_t(\omega):=\frac{\sqrt{|\ga_t|}}{\sqrt{|\gas|}}, \qquad
\omega\in {\Bbb S}^2.
\end{equation}
Then there hold
\begin{equation}\label{cmps1}
C^{-1}\le v_t/s^2\le C,\qquad C^{-1}<r/s<C,
\end{equation}
where $C$ is a positive universal constant.
\end{proposition}

Finally we recall (\cite{Wang10,Wang1}) that for any $\Sigma$-tangent tensor field $F$
there holds the trace inequality
\begin{equation}\label{trc1}
\|r^{-\frac{1}{2}}F\|_{L^2(S_{t,u})}+\|F\|_{L^4(S_{t,u})}\les
\|F\|_{H^1(\Sigma)},
\end{equation}
 and  for any $-\frac{t}{4}\le u\le t$ and any $\Sigma$-tangent tensor field $F$,
there holds
\begin{equation}\label{fq}
\|F\|_{L^2(S_{t,u})}^2 \les(\|k\|_{L^3(\itt({S_{t,u}}))}+1)
\|F\|_{H^1(\Sigma)}\|F\|_{L^2(\itt(S_{t,u}))}.
\end{equation}
where $Int(S_{t,u})$ denotes the domain in $\Sigma_t$ enclosed by $S_{t,u}$.

We recall also  the transport Sobolev-type inequality for any $S_t$-tangent tensor field $F$,
\begin{equation}\label{trc2}
\|r^{\f12}F\|_{L_\omega^p L_t^\infty}\les \|r^{-\f12} F\|_{L_x^2 L_t^\infty(C_u)}+\|F\|_{L_x^4
L_t^\infty(C_u)}\les \N_1[F]
\end{equation}
with $2\le p\le 4$. Similar to \cite[Proposition 8.1]{Wang10}  and also using (\ref{trc1}),
we have for any $\Sigma$-tangent tensor field $F$ that
\begin{align}
\|s^{-1} F\|_{L^2(C_u)}&\les \|s^{-1/2}
F\|_{L^2(S_T)}+\|\sn_L F\|_{L^2(C_u)}+\|F\|_{L^2(C_u)}\nn\\
&\les\|F\|_{H^1(\Sigma_T)}+\|\sn_L F\|_{L^2(C_u)},\label{trr1}
\end{align}
As a immediate consequence of (\ref{trc2}), (\ref{calc0}) and (\ref{mt4}), on $C_u$ there holds
for $F=\chih, \zeta, \ckk \pi$ that
\begin{equation}\label{cor4}
\|r^{\f12}F\|_{L_\omega^p L_t^\infty}\les\N_1[F]\le  C, \mbox{ with }  2\le p\le 4.
\end{equation}

\begin{lemma}
For  $q>2$, there hold the following Sobolev inequalities on $C_u$:
\begin{align}
&\|F\|_{L_t^2 L_\omega^q}\les (\|\sn
F\|_{L^2(C_u)}^{1-\frac{2}{q}}+\|r^{-1}F\|_{L^2(C_u)}^{1-\frac{2}{q}})\|r^{-1}
F\|_{L^2(C_u)}^{\frac{2}{q}}\label{sob.2}\\
&\|F\|_{L_\omega^\infty}\les \|r\sn F\|_{L_\omega^{q}}
+\|F\|_{L_\omega^{q}}\label{sob.3}
\end{align}
\end{lemma}

\begin{lemma}\label{trans2}
Let $0\le 1-\frac{2}{p}<s-2$. For any $\Sigma$ tangent tensor $F$ there holds
\begin{align*}
\|r^{1-\frac{2}{p}}|\nab_L F, \sn F|_g\|_{L_t^2 L_x^p(C_u)}
&\les\sum_{i}\|r^{1-\frac{2}{p}}(\sn_L (F_{i}), \sn (F_{i}))\|_{L_t^2 L_x^p(C_u)}+ \|F\|_{L_t^\infty H^1}.
\end{align*}
\end{lemma}

Indeed,  we can check for $\Sigma$ tangent tensor $F_{ij}$ the following symbolic identities,
\begin{equation*}
 \nab_L F_{i}=\nab_L (F_{i})+\er(F),\qquad \sn F_{i}=\sn(F_{i})+\er(F)
\end{equation*}
where $|\er(F)|\le |\bg\c (\hn g, \hn Y, k)\c F|$, which gives
\begin{align*}
\|r^{1-\frac{2}{p}}|\nab_L F, \sn F|_g\|_{L_t^2 L_x^p(C_u)}&\les \sum_{i}\|r^{1-\frac{2}{p}}(\sn_L (F_{i}), \sn (F_{i}))\|_{L_t^2 L_x^p(C_u)}\\&+(T-u)^{\f12}\|\ti \pi\|_{L_t^2 L_x^\infty}\|r^{\f12-\frac{2}{p}}F\|_{L_t^\infty L^p(C_u)}
\end{align*}
where to derive the last inequality, we employed (\ref{trc1}).

\begin{lemma}
Let $f$ be the components of the $\Sigma$-tangent tensor fields  $\nab^2 n$, $\nab (ne_0 (n))$, and $\hn^2 Y$.
Then for $2\le q<4$ there holds
\begin{equation}\label{sobinff}
\|r^{1-\frac{2}{q}}f\|_{L_t^2 L_x^q(C_u)}\le C.
\end{equation}
For the $\Sigma$ tangent tensor $\pi_0=-\nab\log n $, there holds
\begin{footnote}{We can obtain stronger estimates than (\ref{sobinf}), but this estimate is sufficient for our purpose.}
\end{footnote}
\begin{equation}\label{sobinf}
\|r^{1-\frac{2}{q}}(\nab_L \pi_0, \nab \pi_0)\|_{L_t^2 L_x^q(C_u)}\le C .
\end{equation}
\end{lemma}

\begin{proof}
We decompose a scalar function $f$ by  $f=\sum_{\mu>1} P_\mu f+P_{\le 1} f$.
By the Sobolev inequality  (\ref{sob.12}) on $S_{t,u}$ we have
\begin{equation*}
r^{1-\frac{2}{q}}\|P_\mu f\|_{L_x^q}\les \|r\sn P_\mu f\|_{L_x^2}^{1-\frac{2}{q}}\|P_\mu f\|_{L_x^2}^{\frac{2}{q}}+\|P_\mu f\|_{L_x^2}.
 \end{equation*}
Using the finite band property of LP projections and (\ref{fq}), we derive that
\begin{align*}
 \|r^{1-\frac{2}{q}}P_\mu f\|_{L_t^2 L_x^q}&\les (\mu^{1-\frac{2}{q}}+1)\|P_\mu f\|_{L^2(C_u)}\\
 &\les(\mu^{1-\frac{2}{q}}+1)\|P_\mu f\|_{L_t^2 H^1([0,T]\times \Sigma)}^\f12\|P_\mu f\|_{L_t^2 L^2([0,T]\times \Sigma)}^\f12\\
 &\les \mu^{-\f12}(\mu^{1-\frac{2}{q}}+1)\|f\|_{L_t^2 H^1([0,T]\times \Sigma)}.
\end{align*}
The lower order part $P_{\le 1}f$ can be treated similarly. Therefore   summing over $\mu>1$, we obtain
\begin{equation*}
 \|r^{1-\frac{2}{q}}f\|_{L_t^2 L_x^q(C_u)}\les \|f\|_{L_t^2 H^1([0,T]\times \Sigma)}.
 \end{equation*}
The estimate (\ref{sobinff}) then follows by using (\ref{q2}) and (\ref{shift3}).

We next prove (\ref{sobinf}). It is straightforward to derive that
$\nab\pi_0=n^{-1}\nab^2 n+\pi_0\c\pi_0$.  Then using (\ref{sobinff}) and (\ref{trc1})
\begin{align*}
\|r^{1-\frac{2}{q}} \nab\pi_0\|_{L_t^2 L_x^q(C_u)}&\les \|r^{1-\frac{2}{q}} \nab^2 n\|_{L_t^2 L_x^q(C_u)}+\|\pi_0\|_{L_t^2 L_x^\infty}\|r^{\f12-\frac{2}{q}}\pi_0\|_{L_t^\infty L_x^q}\les 1.
\end{align*}
Note that we have the following symbolic identity
\begin{equation*}
\nab_{n\bT} \nab_i \log n =\nab_i \nab_{n\bT} \log n+n k \c \nab\log n.
\end{equation*}
In view of $L=\bT+N$, and using (\ref{sobinff}), we can derive
\begin{align*}
\|r^{1-\frac{2}{q}}\nab_L \pi_0\|_{L_t^2 L_x^q(C_u)}
&\les \|r^{1-\frac{2}{q}} (\nab (ne_0(\log n)), \nab\pi_0 )\|_{L_t^2 L_x^q(C_u)}\\
&+\|\pi\|_{L_t^2 L_x^\infty}\| r^{\frac{1}{2}-\frac{2}{q}}\pi\|_{L_t^\infty L_x^q(C_u)}
\les 1.
\end{align*}
Hence we complete the proof.
\end{proof}

\subsection{$L^p$ flux type estimates}

In this subsection, we will use Proposition \ref{fluxg} to derive $L^p$ flux type
estimates on $\hn g, k$ and null components of curvature
$R_0=(\a, \b, \rho, \sigma, \udb)$.
Let $0<1-\frac{2}{p}<s-2$.  For any $S_t$-tangent tensor $F$ along null hypersurface $C_u$ we define the
Sobolev norm
\begin{equation*}
\np[F]=\|r^{-\frac{2}{p}}F\|_{L_t^2 L_x^p}+\|r^{1-\frac{2}{p}} \sn
F\|_{L_t^2 L_x^p}+\|r^{1-\frac{2}{p}} \nab_L F\|_{L_t^2 L_x^p}.
\end{equation*}
We denote by
$$
\np[\ckk\pi]=\np[\slp]+\np[\hn g],
$$
where $\hn g$ in the last term denotes the scalar components of $\hn g$ relative to
arbitrary time independent coordinate frame in $\Sigma$.
When $p=2$, $\N_{1,2}[F]=\N_1[F]$. We have from Proposition \ref{fluxg} together with
$\N_1[\slp]\le C$ in Theorem \ref{recalll2} that
\begin{equation}\label{pi.31}
\N_1[\ckk \pi](C_u)\le C.
\end{equation}

We now state the main result of this subsection.

\begin{proposition}\label{fluxpb}
Let $f$ be the components of $k$ and $\hn g$. Then on $C_u$ there hold
\begin{align}
&\|r^{1-\frac{2}{p}}(\sn_L f, \sn f)\|_{L_t^2 L_x^p(C_u)}\le C, \label{pflux}\\
&\np[\ckk \pi]\le C, \label{npflux}\\
&\|r^{1-\frac{2}{p}} (\a, \b, \rho, \sigma, \b, \udb)\|_{L_t^2
L_x^{p}(C_u)}\le C,\label{cflux}
\end{align}
where $0< 1-\frac{2}{p}< s-2.$
\end{proposition}

As the first step, in what follows we build a connection between the dyadic flux in
Proposition \ref{fluxg} and the ones that will be used to derive (\ref{pflux}).

\begin{proposition}
Let $\slP_\mu$ be the classical Littlewood-Paley
decomposition on 2-dimensional slice $S_{t,u}$ on $C_u$. Let $\D_t=an\frac{d}{ds}$
along each null geodesic $\Ga_\omega$ on $C_u$. Then for any $0<\ep<1$ and any scalar function $f$ there hold
\begin{align}
\sum_{\mu>1}\|\mu^{\ep}\slP_\mu \sn f\|_{L^2(C_u)}^2
&\les \sum_{\ell>1}\ell^{2\ep}\left(\F[P_\ell f](C_u)+\|P_{\ell} f\|_{H^1(\Sigma_{T})}^2\right)\nn\\
&\quad \, +\F[f]+\|f\|^2_{H^1(\Sigma_{T})},\label{com1}\\
\sum_{\mu>1}\|\mu^{\ep}\slP_\mu \D_t f\|_{L^2(C_u)}^2
&\les \sum_{\ell>1}\ell^{2\ep}\left(\F[P_\ell f](C_u)+\|P_{\ell}
f\|_{H^1(\Sigma_{T})}^2\right)\nn\\
& \quad \, +\F[f]+\|f\|^2_{H^1(\Sigma_{T})}.\label{com2}
\end{align}
\end{proposition}
\begin{proof}
We first prove (\ref{com1}) and (\ref{com2}) by assuming the following inequality, with $\ell>1$ being dyadic
\begin{equation}\label{d4lst}
\|r\sn  \sn P_\ell
f\|_{L^2(C_u)}+\|r\sn \D_t P_\ell f\|_{L^2(C_u)}\les \ell \left(\F^\f12[\ti
P_{\ell}f]+\|\ti P_\ell f\|_{H^1(\Sigma_{T})}\right).
\end{equation}
We decompose $f=\sum_{l>1} P_\ell f+P_{\le 1} f$. First we have
\begin{align*}
\mu^{\ep}\left\|\slP_\mu\sn\sum_{\ell>\mu} P_\ell f \right\|_{L^2(C_u)}
&\les \sum_{\ell>\mu} \left(\frac{\mu}{\ell}\right)^{\ep}\ell^{\ep}\|\sn
P_\ell f\|_{L^2(C_u)}.
\end{align*}
Summing over $\mu>1$, we obtain
\begin{equation}\label{lowh}
\sum_{\mu>1}\mu^{2\ep} \left\|\slP_\mu\sn\sum_{\ell>\mu} P_\ell f \right\|_{L^2(C_u)}^2
\les\sum_{\ell} \ell^{2\ep}\|\sn P_\ell f\|_{L^2(C_u)}^2.
\end{equation}
Now we consider  the part $\sum_{1<\ell<\mu}P_\ell f$ in the
decomposition of $f$. The term $P_{\le 1} f$ can be treated
similarly. By the finite band property of $\slP_\mu$, we have
\begin{align*}
\|\slP_\mu\sn P_\ell f\|_{L^2(C_u)}&\les \mu^{-1}\|r\sn  \sn P_\ell
f\|_{L^2(C_u)}.
\end{align*}
In view of (\ref{d4lst}), it follows that
\begin{equation*}
\|\slP_\mu\sn P_\ell f\|_{L^2(C_u)}\les \mu^{-1}\ell\left(\F^\f12[\ti
P_{\ell}f]+\|\ti P_{\ell} f\|_{H^1(\Sigma_{t_M})}\right)
\end{equation*}
Hence  we have
\begin{align*}
\mu^{\ep} \left\|\slP_\mu\sn\sum_{\ell\le \mu} P_\ell f \right\|_{L^2(C_u)}
&\les\sum_{\ell\le\mu} \left(\frac{\ell}{\mu}\right)^{1-\ep}\ell^{\ep}\left(\F^\f12[\ti
P_\ell f]+\|\ti P_{\ell} f\|_{H^1(\Sigma_{t_M})}\right).
\end{align*}
Taking  $l_\mu^2$  and also combining with (\ref{lowh})  gives (\ref{com1}).

To see (\ref{com2}),  we first derive by the basic property of $\sl{P}_\mu$ that
\begin{equation*}
\mu^{\ep} \left\|\slP_\mu\D_t\sum_{\ell>\mu} P_\ell f \right\|_{L^2(C_u)}
\les\sum_{\ell>1} \ell^{\ep} \left(\frac{\mu}{\ell}\right)^\ep\|\D_t P_\ell
f\|_{L^2(C_u)}.
\end{equation*}
For the case $\ell\le \mu$, we obtain by the finite band property of $\slP_\mu$ and (\ref{d4lst}) that
\begin{align*}
\mu^{\ep} \left\|\slP_\mu\D_t\sum_{1<\ell\le \mu} P_\ell f \right\|_{L^2(C_u)}
& \les\sum_{1<\ell\le \mu}\mu^{\ep-1}\|r\sn\D_t P_\ell f\|_{L^2(C_u)}\nn\\
&\les \sum_{1<\ell\le \mu} \left(\frac{\ell}{\mu}\right)^{1-\ep}\ell^{\ep}\left(\F^\f12[\ti P_{\ell}
f](C_u)+\|\ti P_\ell f\|_{H^1(\Sigma_{t_M})}\right).\label{hlow2}
\end{align*}
Taking $l_\mu^2$ for $\mu>1$, we can obtain (\ref{com2}).

We next prove (\ref{d4lst}). Recall that  $\Pi_{m n}=g_{mn}-N_m N_n$ and
$\sn_m=\Pi_{m}^{n'}\hn_{n'}$. Note that for functions
$f$ on $C_u$ we have
\begin{align*}
\sn\sn P_\ell f&=\sn(\Pi_m^n\hn_n P_\ell f)=\sn(\delta_m^n-N_m\c
N^n)\hn P_\ell f+\Pi_m^n \sn\hn P_\ell f\\
&=-\sn N_m\c N^n \hn_n P_l f-N_m \sn N^n \hn_n P_l f+\Pi_m^n \sn\hn P_\ell f.
\end{align*}
Therefore
\begin{equation*}
|\sn\sn P_\ell f|_{\ga}\le \ell r^{-1}\left(|(\tr\theta+\hat \theta)\c \ti P_{\ell} f|_{\ga}
+|\Pi_m^n \sn \ti P_{\ell}f|_\ga\right),
\end{equation*}
where $\ti P$ denote a classical Littlewood-Paley projection on ${\Bbb R}^3$ induced by a different symbol.
By the commutation formula $[\sn,\D_t] f=an\chi\c \sn f$ for scalar functions $f$, we obtain
\begin{equation*}
\sn \D_t P_\ell f=[\sn, \D_t]P_\ell f+\D_t \sn P_{\ell} f=an \chi
\sn P_\ell f+\D_t \sn P_{\ell}f.
\end{equation*}
For the last term on the right, we have
\begin{align*}
\D_t \sn P_\ell f&=\D_t (\Pi_m^n \hn_n P_{\ell}
f)=\D_t(\Pi_m^n)\hn_n
P_{\ell} f+\Pi_m^n \D_t\hn_n P_\ell f\\
&=\D_t(\delta_m^n-N_m N^n)\hn_n P_\ell f+\ell\Pi_m^n \D_t \ti P_\ell f.
\end{align*}
Thus, in view of $\bd_{L'} N^\rho=a^{-1}L'(a)\bT^\rho-a^{-1}\zb_A e_A^\rho$ and
$\D_t=an\bd_{L'}$, we have
\begin{equation*}
|\sn\D_t P_\ell f|_{\ga}\les \ell r^{-1} \left(|\zb\ti P_\ell
f|_{\ga}+|\D_t\ti P_\ell f|_{\ga}+|(\tr\chi+\hat \chi)\c \ti P_{\ell} f|_{\ga}\right).
\end{equation*}
With the help of the symbolic identity $\tr\theta=\tr\chi+k$ and $\hat\theta=\chih+k$, we obtain
\begin{align}
\|r\sn^2& P_\ell f\|_{L^2(C_u)}+\|r\sn\D_t P_\ell f\|_{L^2(C_u)} \nn\\
&\les \ell \left(\|\sn\ti P_{\ell}f\|_{L^2(C_u)}+\|\D_t \ti P_\ell f\|_{L^2(C_u)}+\|(\tr\chi+\chih+\ckk\pi)\ti
P_\ell f\|_{L^2(C_u)}\right)\nn\\
&\les \ell \left(\F^\f12[\ti P_\ell f]+\|r^{-1}\ti P_\ell f\|_{L^2(C_u)}
+\|r^{-\f12} \ti P_\ell f\|_{L_t^2 L_x^4(C_u)}\right)\label{11.20.1}
\end{align}
where we employed the H\"{o}lder inequality, (\ref{cor4}) and $r \tr\chi \approx 1$.
By the Sobolev embedding, we deduce that
\begin{align*}
\|r^{-\f12}\ti P_{\ell} f\|_{L_t^2 L_x^4(C_u)}&\les \|\sn \ti P_\ell f\|_{L^2(C_u)}^{\f12}\|r^{-1} \ti P_\ell f\|_{L^2(C_u)}^{\f12}+\|r^{-1}\ti P_{\ell}f\|_{L^2(C_u)}.
\end{align*}
Applying (\ref{trr1}) to $F=\ti P_{\ell} f$ gives
\begin{equation}\label{lst1}
\|r^{-1}\ti P_{\ell} f\|_{L^2(C_u)}\les \|\sn_L \ti
P_{\ell}f\|_{L^2(C_u)}+\|\ti P_{\ell} f\|_{H^1(\Sigma_{t_M})},
\end{equation}
Therefore
\begin{equation}\label{lst3}
 \|r^{-\f12}\ti P_{\ell} f\|_{L_t^2 L_x^4(C_u)}\les \F^\f12[\ti
P_{\ell}f](C_u)+\|\ti P_{\ell}f\|_{H^1(\Sigma_{t_M})}.
\end{equation}
In view of (\ref{11.20.1}), (\ref{lst1}) and (\ref{lst3}), we thus obtain  (\ref{d4lst}).
\end{proof}

Note that the energy estimate on the Bel-Robinson tensor does not implies the
$L^p, p>2$, type curvature flux in (\ref{cflux}).  As the second step,
we now build the connection between (\ref{npflux}) and (\ref{cflux}).

\begin{proposition}\label{dcmp2}
Let $\er$ denote the type of error terms taking the form  $(\ckk\pi+F)\c \ckk\pi+\tr\chi\ckk\pi$
with $F=\chih, \sn \log a, \hat \theta$. Relative to canonical null frame $\{e_1. e_2, L, \Lb\}$,
there hold the following results for null components of curvature:
\begin{enumerate}
\item[(a)] There exist  scalar functions $\cpi$,   $S$-tangent $2$-tensor fields
$\cpi_{AB}$  such that the components $\b, \udb$
admit the decompositions
\begin{equation*}
\b,\, \udb=\sn_A \cpi+\sn^B \cpi_{AB}+\slashed{\curl}\cpi+\er.
\end{equation*}
\item[(b)] There exist two $S$-tangent vector fields $\cpi^{(1)}_A$ and $\cpi^{(2)}_A$ such that
\begin{equation*}
\rho=\div \cpi^{(1)}+\er,\quad
\sigma=\curl \cpi^{(2)}+\er.
\end{equation*}
\item[(c)] There holds the decomposition $ \a=\sn_L \eh+\sn \slashed{\pi}+\er$.
\end{enumerate}
where $\cpi, \cpi^{(1)}$, and $\cpi^{(2)}$ denote the terms formed by the sums of
$N^\mu(g\c \hn g, k)_{\mu\cdots}$ and $\Pi^{\mu'}_{\mu}(g\c \hn g, k)_{\mu'\cdots}$.
\end{proposition}

As a key feature required by the proof of (\ref{ric3}) (see Section \ref{sec8}),
the higher order terms in the decompositions for $\b$, $\rho$ and $\sigma$ can only
contain  $\sn \ti \pi$. Since time derivative of shift has to be excluded in our
decompositions,  Proposition \ref{dcmp2} does not follow from the approach given
in \cite[Section 4]{KR2}.  We will derive in Appendix I the decompositions in a more
invariant fashion.

\begin{lemma}\label{error}
Let $\er$ be the error term in Proposition \ref{dcmp2}.
Let $0\le 1-2/p<s-2$.  There holds
\begin{equation*}
\|r \c\er\|_{L_t^2 L_\omega^{p}(C_u)}\les
(T-u)^{\f12}+1.
\end{equation*}
\end{lemma}

\begin{proof}
It follows from (\ref{pi.31}), $r\tr\chi\approx 1$  and (\ref{sob.2}) that
$\|r\tr\chi \ckk\pi\|_{L_t^2 L_\omega^p(C_u)}\les \N_1[\ckk\pi]\le C$.
In view of (\ref{trc1}) and (\ref{trc2}),  we can obtain
\begin{equation*}
\left\{\begin{array} {lll}\|r\ckk\pi \c \ckk\pi\|_{L_\omega^{p}}\les
\|\ckk\pi\|_{ L_x^\infty}\|r\ckk\pi\|_{L_\omega^{p}}\les
(t-u)^{\f12} \|\ckk\pi\|_{H^1(\Sigma)}\|\ckk\pi\|_{L_x^\infty}\\
 \|rF\c\ckk\pi\|_{ L_\omega^{p}}\les (t-u)^{\f12}
\N_1[F]\|\ckk\pi\|_{ L_x^\infty}\end{array}\right.
\end{equation*}
where $F=\chih, \sn \log a, \hat\theta$. By integrating in $t$ and using (\ref{BA2}),
$\|\ti \pi\|_{H^1(\Sigma)}\le C$ and (\ref{cor4}) we can  complete the proof.
\end{proof}

\begin{proof}[Proof of Proposition \ref{fluxpb}]
First, applying (\ref{sob.12}) to $F=\slP_\mu \sn f$ gives
\begin{equation*}
\|r^{1-\frac{2}{p}} \sn f\|_{L_t^2 L_x^p(C_u)}\les
\sum_{\mu>1}(\mu^{1-\frac{2}{p}}+1)\| \slP_\mu \sn
f\|_{L^2(C_u)}+\|\sn f\|_{L^2(C_u)}
\end{equation*}
With the help of (\ref{com1}), we then obtain
\begin{equation*}
\|r^{1-\frac{2}{p}} \sn f\|_{L_t^2 L_x^p(C_u)}\les
\F^\f12[f](C_u)+\sum_{\ell>1}\ell^{1-\frac{2}{p}} \left(\F^\f12[P_\ell
f](C_u)+\|P_\ell f\|_{H^1(\Sigma_{T})}\right).
\end{equation*}
On the other hand, applying (\ref{sob.12}) to $F=\slP_\mu \D_t f$ and using (\ref{com2})
we can obtain
\begin{equation*}
\|r^{1-\frac{2}{p}} \D_t f\|_{L_t^2 L_x^p(C_u)}\les
\F^\f12[f](C_u)+\sum_{\ell>1}\ell^{1-\frac{2}{p}}\left (\F^\f12[P_\ell
f](C_u)+\|P_\ell f\|_{H^1(\Sigma_{T})}\right).
\end{equation*}
Thus, applying the above two inequalities with $f=k_{ij}, \hn g_{ij}$, and using
\begin{equation*}
\|\mu^{1-\frac{2}{p}}P_\mu f\|_{l_\mu^1H^1_\Sigma}\le
\|f\|_{H^{1+\ep}} \mbox{ with }1-\frac{2}{p}<\ep<s-2
\end{equation*}
and Proposition \ref{fluxg}, we have
\begin{align*}
 \|\mu^{1-\frac{2}{p}}r(\sn,
\nab_L)P_\mu f\|_{l_\mu^1 L_t^2 L_\omega^p(C_u)}\le C.
\end{align*}
This completes the proof of (\ref{pflux}).

We next prove (\ref{npflux}). Let $\sn_*$ denote either $\sn_A$ or $\sn_{L'}$ for $S$-tangent
tensor, and $\nab_*$ be either $\nab_L$ or $\sn_{A}$ for $\Sigma$ tangent tensor as
in (\ref{ddf1}). We recall from \cite[(8.14), (8.15)]{Wang10} that with $\pi$ being $\Sigma$ tangent tensor $\pi_0$ or $k$,
\begin{equation}\label{trans1}
\sn_*\slp=\nab_*\pi+\tr\theta\c\slp+(\hat \theta, \zb)\c \slp+(an)^{-1}=\nab_*\pi+\er.
\end{equation}
By using (\ref{pflux}) and (\ref{sobinf}),we have $\||r^{1-\frac{2}{p}}\nab_*\pi|_g\|_{L^p(C_u)}\le C,$
 which together with (\ref{trans1}) implies that,
\begin{align*}
\|r^{1-\frac{2}{p}}\sn_*\slp\|_{L_t^2L^p(C_u)}&\les \||r^{1-\frac{2}{p}}\nab_*\pi|_g\|_{L^p(C_u)}+1+(T-u)^{\f12}\le C,
\end{align*}
where to derive the last inequality we employed  Lemma \ref{error}.
Hence we obtain (\ref{npflux}).

In view of the decompositions in Proposition \ref{dcmp2}, we have $R_0=\sn_* \slp+g \c\sn(\hn g)+\er$.
By using (\ref{npflux}) and Lemma \ref{error}, we can obtain (\ref{cflux}).
\end{proof}

\subsection{$L^{p}$ estimates for Ricci coefficients}

We fix a null cone $C_u$ contained in $\D^+$ with $t=u$ at its vertex on $\Ga^+$.
By rescaling the space-time coordinate as $(t,x) \to (\frac{t-u}{d},\frac{x}{d})$ with
$d=(T-u)$, we may restrict our consideration to the time interval $[0,1]$. We denote by $\H$,
the null cone $C_u$ after change of coordinates. Then $\H=\cup_{0<t\le 1}S_t$ with $S_t$
being the intersection  of $\H$ and $\Sigma_t$. We also denote
$\H_t:=\cup_{t'\in [0,t]} S_{t'}$ for $0<t\le 1$.

We will rely on  the quantities
\begin{equation}\label{bkm1}
\Rp:=\nps[\ckk \pi]+\|rR_0\|_{L_t^2 L_\omega^{p}(\H)}
\end{equation}
and
\begin{equation}\label{cond1}
\RR:= \R(\H)+\N_1[\sl{\pi}].
\end{equation}

We will prove an $L^{p}$ version of weakly spherical property on $S_t$ in
Proposition \ref{swsph} and  the following estimates

\begin{proposition}\label{smr}
Let $M$ denote either $\mu$ or $\sn \tr\chi$.  If $\Rp+\RR$ is sufficiently small,  there holds on $\H$
\begin{equation}\label{smry1}
\nps(\Ab)+\|\Ab\|_{L_t^2
L_\omega^\infty}+\|r^{\f12}\Ab\|_{L_\omega^{2p}
L_t^\infty}+\|r^{\frac{3}{2}}M\|_{L_{\omega}^{p}
L_t^\infty}+\|rM\|_{L_t^2 L_\omega^{p}}\les \Rp+\RR.
\end{equation}
\end{proposition}

We remark the the smallness on $\RR+\Rp$ can be obtained by requiring $T-u$ sufficiently
small but depends only on universal constants.

Let us record the structure equations for the Ricci coefficients relative to $\{e_1, e_2, \Lb', L'\}$:
\begin{align}
&\frac{d\tr \chi}{ds}+\frac{1}{2}(\tr\chi)^2=-|\chih|^2 \label{s1}\\
&\frac{d}{ds}\chih+\tr\chi\chih=-\a\label{s2}\\
&\frac{d}{ds}\zeta=-\chi\c\zeta+\chi\c\zb-\b\label{s3}\\
&\frac{d}{ds}\sn \tr\chi+\frac{3}{2}\tr\chi \sn\tr\chi=-\chih\c \sn
\tr\chi-2\chih\c
\sn\chih-(\zeta+\zb)(|\chih|^2+\frac{1}{2}(\tr\chi)^2) \label{s4}\\
&\div \chih=\frac{1}{2}\sn \tr \chi +\frac{1}{2}\tr\chi\c\zeta-\chih\cdot \zeta-\beta \label{stc1}\\
&\div\zeta=-\mu-\rho+\frac{1}{2}\chih\c \chibh-|\zeta|^2+\frac{1}{2}a\delta \tr\chi-\frac{1}{3}a\Tr k \tr\chi\label{stc2}\\
&\curl\zeta=\sigma-\frac{1}{2}\chih\wedge\chibh.\label{stc3}
\end{align}

Let us record the following transport equation for $\mu'$ where $'$
will be dropped for convenience,
\begin{align}
L\mu+\tr\chi\mu &=2\chih\c \sn \zeta+(\zeta-\zb) \c (\sn
\tr\chi+\tr\chi\zeta)-\frac{1}{2}\tr\chi(\chih\c
\chibh-2\rho+2\zb\c \zeta)\nn\\&\quad\quad+2\zeta\c \chih\c \zeta
+\left(\frac{1}{4}a^2 \tr\chi+a
(-\delta+\frac{2}{3}\Tr k)\right)|\chih|^2-\frac{1}{2}a \nu
(\tr\chi)^2\label{mumu}
\end{align}


Recall the following result on initial data (see \cite{Wang09}).

\begin{lemma}\label{inii}
There holds
\begin{align*}
&V,\sn a,  r\sn \tr\chi, r^2\mu\rightarrow 0 \mbox{ as }
t\rightarrow 0,\quad \lim_{t\rightarrow 0}\|\chih,
\zeta,\zb,\nu\|_{L^\infty(S_t)}<\infty.
\end{align*}
\end{lemma}

Let us recall the transport lemma (see \cite{KR2,Wang10}).

\begin{lemma}\label{tsp2}
For an $S$ tangent tensor field $F$ verifying
\begin{equation*}
\sn_L F+\frac{m}{2} \tr\chi F= G\c F+ H
\end{equation*}
with $m\ge 1$ certain integer, if $\lim_{t\rightarrow 0}
r(t)^{m} F=0$ and $\|G\|_{L_\omega^\infty L_t^2}\le C$, then
there holds
\begin{equation}\label{tran.2}
|F|\les v_t^{-\frac{m}{2}}\int_0^t v_\tt^{\frac{m}{2}} |H|na d\tt.
\end{equation}
\end{lemma}

We also need the following useful Sobolev inequality.

\begin{lemma}
For any $q\ge 2$ and any $S_t$ tangent tensor field $F$, there holds the estimate
\begin{equation}\label{intp1}
\|r^{\f12-\frac{1}{q}} F\|_{L_x^{2q} L_t^\infty(\H)}^2\les
\left(\|r\sn_L F\|_{L_\omega^{q} L_t^2} + \| F\|_{L_\omega^{q}
L_t^2}\right)\|F\|_{L_\omega^\infty L_t^2}.
\end{equation}
\end{lemma}

\begin{proof} The estimate with $q=2$ is proved in \cite[Lemma 8.2]{Wang10}; the same
argument can be used to the result for any $q>2$.
\end{proof}

We now prove the following result on null cone $\H$,

\begin{proposition}\label{prl1}
For any $p\ge 2$ there hold
\begin{align}
&\|r\sn_L(\chih,\zeta)\|_{L_t^2 L_\omega^{p}}\les
\Rp+\RR\label{prl2}\\
&\|\Ab\|_{L_t^2 L_\omega^{p}}+\|r\nab_L \Ab\|_{L_t^2
L_\omega^{p}}\les \Rp+\RR\label{simp3}
\end{align}
\end{proposition}

\begin{proof}
By the Sobolev embedding (\ref{sob.2}) we have for $p>2$ that
$$\|\Ab\|_{L_t^2 L_\omega^{p}}\les \N_1[\Ab]\les \RR.$$
 The second inequality in (\ref{simp3}) will follow from (\ref{prl2})  and
(\ref{bkm1}). We need only to prove (\ref{prl2}).
 In view of (\ref{s2}) and (\ref{bkm1}), we have
\begin{equation*}
\|r\nab_L \chih\|_{L_t^2 L_\omega^p}\les \|\chih\|_{L_t^2
L_\omega^p}+\|r\a\|_{L_t^2 L_\omega^p}\les \Rp.
\end{equation*}
Let $p''$ be such that $\frac{1}{p}=\frac{1}{4}+\frac{1}{p{''}}$. By using
(\ref{sob.2}) we have
\begin{equation}\label{blin1}
\|rA\c \Ab\|_{L_t^2 L_\omega^p}\les\|r^{\f12} A\|_{L_t^\infty
L_\omega^4}\|r^{\frac{1}{2}}\Ab\|_{L_t^2 L_\omega^{p{''}}}\les
\N_1(A)\N_1(r^{\f12} \Ab)\les\RR^2.
\end{equation}
In view of (\ref{s3}), (\ref{blin1}), (\ref{bkm1}) and the first part of (\ref{simp3}),
we obtain
\begin{align*}
\|r\nab_L \zeta\|_{L_t^2 L_\omega^p}&\les \|\zeta\|_{L_t^2
L_\omega^p}+\|r\chih\c \zb\|_{L_t^2 L_\omega^p}+\|\zb\|_{L_t^2
L_\omega^p}+\|r\b\|_{L_t^2 L_\omega^p}\les \Rp+\RR^2.
\end{align*}
The proof is thus complete.
\end{proof}

\subsubsection{Weakly spherical condition}

Let $\ga$ be the induced metric on $S_t$, and define the rescaled metric $\gac$ on
$S_t$ by $\gac=r^{-2}\ga$.  Note that
\begin{equation}\label{8.1.1}
\lim_{t\rightarrow 0}\stackrel{\circ}
\gamma_{ij}={\gas}_{ij},\qquad  \lim_{t\rightarrow
0}\p_k\stackrel{\circ}\ga_{ij}=\p_k{\gas}_{ij}
\end{equation}
where $i, j, k=1,2$.  Recall that we
have proved in \cite{Wang1} that
\begin{align}
&\|\stackrel{\circ}\ga_{ij}(t)-{\gas}_{ij}\|_{L^\infty}\les \RR\label{ws1}\tag{WS1}\\
&\|\p_k \stackrel{\circ}\gamma_{ij}(t)-
\p_k{\gas}_{ij}\|_{L^{2}_\omega L_t^\infty}\les\RR.\label{ws2}
\end{align}
Now we make the following bootstrap assumption
which will be improved at the end of this section.
\begin{assumption}\label{l8.1}
For the transport local coordinates $(t,\omega)$, the following
properties hold true for all surfaces ${S}_t$ of the time foliation
on the null cone $\H$: the metric $\stackrel{\circ}\gamma_{ij}(t)$
on each ${S}_t$ verifies weakly spherical conditions i.e.
\begin{align*}
\|\p_k \stackrel{\circ}\gamma_{ij}(t)-
\p_k{\gas}_{ij}\|_{L^{p}_\omega L_t^\infty}\le \Delta_0 \mbox{ with } 0<1-\frac{2}{p}<s-2,
\end{align*}
where we assume $0<\Delta_0< \f12.$
\end{assumption}
Let $\sn^{(0)}$ be the covariant differentiation relative to $\gaz_{ij}$.
 We will improve (WS2) with the right side replaced by
$\RR+\Rp$. By choosing $\Delta_0<\Rp+\RR$ sufficiently small, we can
close the bootstrap argument.

\begin{lemma}\label{frm1}
Under the Assumption \ref{l8.1}, for $2<q\le p$,
\begin{equation*}
\|\sn F\|_{L_x^{q}}+\|r^{-1}F\|_{L_x^{q}}\approx
\|\sn^{(0)}F\|_{L_x^{q}}+\|r^{-1}F\|_{L_x^{q}}.
\end{equation*}
\end{lemma}

\begin{proof}
It suffices to consider  $F$ to be 1-form. Relative to any coordinates on $S$,
$\sn_i F_j =\sn^{(0)}_i F_j +(\Ga_{ij}^{\dum\dum l} -{\Gz_{ij}}^ l )\c F_l .$
 Note that using (\ref{ws1}) and Assumption \ref{l8.1} we have $\|r^{1-\frac{2}{q}}|\Ga-\Gz|_\ga\|_{L_x^q}\les \Delta_0$. Also in view of  (\ref{sob.3}) and H\"older inequality, we obtain
\begin{align}
\||(\Ga-\Gz)_{ij}^l F_l|_\ga\|_{L_x^{q}}&\les \|r^{1-\frac{2}{q}}|\Ga-\Gz|_\ga\|_{L_x^q} \|r^{\frac{2}{q}-1}
F\|_{L_\omega^\infty}\nn\\&\les \Delta_0 (\|\sn F\|_{L_x^q}+\|r^{-1}
F\|_{L^q_x}).\label{sobinf1}
\end{align}
 Note that we can choose $0<\Delta_0<\f12$ sufficiently
small, then for $2<q\le p$,
\begin{equation*}
\|\sn F\|_{L_x^q}+\|r^{-1} F\|_{L_x^q}\les
\|\sn^{(0)}F\|_{L_x^q}+\|r^{-1} F\|_{L_x^q}.
\end{equation*}
The other direction can be proved in the same way.
\end{proof}

We recall from \cite{Wang1}  the following result which uses the smallness of $\RR$ given in the
remark immediately after Proposition \ref{smr}.

\begin{proposition}
 On null cone $\H$ under the condition (\ref{cond1}) and
 $C^{-1}<n<C$ that for $S$ tangent tensor $F$, there holds
\begin{equation}\label{elpl2}
\|\sn F\|_{L^2(S)}+\|r^{-1}F\|_{L^2(S)}\les \|\D F\|_{L^2(S)}.
\end{equation}
\end{proposition}

By Assumption \ref{l8.1}, let us prove  the following $L^p$
estimate for Hodge system.

\begin{lemma}\label{hdgm1}
Let $\D$ denote either $\D_1$ or $\D_2$ and  $F$ be $S$ tangent tensor in the domain  of $\D$,  there holds for $2<q\le p,$
\begin{equation*}
\|\sn  F\|_{L^q(S)}+\|r^{-1}F\|_{L^q(S)}\les \|\D F\|_{L^q(S)}
\end{equation*}
\end{lemma}

\begin{proof}
We only give the proof for  $\D=\D_2$. The proof for the case that $\D=\D_1$ follows similarly.  
Relative to orthonormal frame associated to $\gaz$,
for $S$-tangent, symmetric, traceless 2-tensor $F$,  
let us define $\ti F_{A B}=\gaz_{A C} \ga^{C D}F_{D B}$ and $2\bar F_{A B}=\ti F_{AB}+\ti F_{BA}$.
Then we have
\begin{equation}\label{barf}
  \dvz {\bar F}=\div F+(\Ga-\Gz)F\c \ga,\quad \quad\Tr^{(0)}(\bar F)=0.
\end{equation}
Indeed, under the orthonormal frame associated to $\gaz$,
\begin{align*}
\sn_C(\ga^{CD}F_{DA})=\sz_C(\ga^{CD}F_{DA})+(\Ga-\Gz)\ga\c F,\\
\sn_C (F^{CD} \ga_{DA})= \sz_C(F^{CD}\ga_{DA})+(\Ga-\Gz)F\c \ga.
 \end{align*}
By Calderon Zygmund theory  on $2$-sphere with  $\gaz$ and also using (\ref{barf}),
 we  obtain
\begin{align*}
\int(|\sz \bar F|^q&+| r^{-1} \bar F|^q)\les \int |\D^{(0)} \bar F|^q \\
&\le \int(|\D F|^q+|(\Ga-\Gz)F\c \ga|^q )\les \int|\D F|^q
\end{align*}
where to derive the last inequality, we employed (\ref{sobinf1}),
Lemma \ref{frm1} and the smallness of $\Delta_0$. By (\ref{ws1}) and Assumption \ref{l8.1}
we can show that 
$$
\|\sn^{(0)}F\|_{L_x^q}+\|r^{-1} F\|_{L_x^q}\les \|\sn^{(0)} \bar F\|_{L_x^q}+\|r^{-1} \bar F\|_{L_x^q}.
$$
Therefore, Lemma \ref{hdgm1} follows by using Lemma \ref{frm1}.
\end{proof}

Since
$$
\D_2\chih=\sn\tr\chi+\bar R,\,\,\quad  \D_1 \zeta=(\mu,0)+\bar R,
$$
using Lemma \ref{hdgm1}, we can derive that for $2\le q\le p$
\begin{align*}
\|\sn \chih\|_{L^q(S)}+\|r^{-1}\chih\|_{L^q(S)} &\les
\|\sn\tr\chi\|_{L^q(S)}+\|\bar R\|_{L^q(S)},\nn\\
\|\sn \zeta\|_{L^q(S)}+\|r^{-1}\zeta\|_{L^q(S)} &\les \|\mu\|_{L^q(S)}+\|\bar
R\|_{L^q(S)}.
\end{align*}
Since using (\ref{blin1}), (\ref{simp3}) and (\ref{bkm1})
 $$\|r\bar R\|_{L_t^2 L_\omega^{p}}\les\|\Ab\|_{L_t^2
L_\omega^{p}}+\|rR_0\|_{L_t^2 L_\omega^{p}}+\|rA\c \Ab\|_{L_t^2
L_\omega^{p}}\les\Rp+\RR^2,$$ then on $\H_t$
\begin{equation}\label{czm}
\|r\D_0 M\|_{L_t^2 L^{p}_\omega}\les \|r M\|_{L_t^2
L^{p}_\omega}+\Rp+\RR.
\end{equation}

Recall that $M=\sn \tr\chi \mbox{ or } \mu$.  (\ref{s4}) and
(\ref{mumu}) can be symbolically
recast as
\begin{equation}\label{LM}
\sn_L M+\frac{m}{2}\tr\chi M=\chih\c M+H_1+H_2+H_3
\end{equation}
where \begin{footnote} {In view of $|a-1|\le 1/2$ in (\ref{mt7}), the
factors $a^m, \,m\in {\Bbb N}$  can be ignored in $H_i$ when we
employ (\ref{LM}) to prove Proposition \ref{MN1}.}
\end{footnote}
$$
H_1=\chih\c \D_0 M+F\c A, \quad H_2=\Ab\c A\c A,\quad
\mbox{and}\quad H_3= r^{-1}\bar R, \iota \c \bar R,
$$
with
\begin{equation*} 
(m,F)=\left \{\begin{array}{lll}(2,  \sn \tr\chi), & \mbox{if } M=\mu,\\
(3, 0),  & \mbox{if } M=\sn\tr\chi.
\end{array}\right.
\end{equation*}

We now establish the following estimates, which,
in view of (\ref{sob.3}) and (\ref{intp1}), imply Proposition \ref{smr}.

\begin{proposition}\label{MN1}
Let $M$ denote either $\sn \tr\chi$ or $\mu$, there hold
\begin{equation}\label{M1}
\|r^{\frac{3}{2}} M\|_{L_\omega^{p} L_t^\infty}+ \|r M\|_{L_t^2
L_\omega^{p}}\les \Rp+\RR,
\end{equation}
\begin{equation}\label{n1cz}
\|r\sn \zeta, r\sn \chih\|_{L_t^2L_\omega^{p}}\les \Rp+\RR.
\end{equation}
\end{proposition}

\begin{proof}
 (\ref{n1cz}) can be obtained immediately by combining the second
estimate in (\ref{M1}) with (\ref{czm}), we  consider the first norm
in (\ref{M1}). By  (\ref{LM}), Lemma \ref{tsp2} and \ref{inii},
integrating along the null geodesic $\Ga_\omega$ initiating from the
vertex, we obtain
\begin{equation*} 
|M|\le \sum_{i=1}^3 \left|v_t^{-\frac{m}{2}}\int_0^t
v_{t'}^{\frac{m}{2}}|H_i|na dt'\right|=I_1(t)+I_2(t)+I_3(t).
\end{equation*}
 We have
\begin{align}
v_t^{\frac{3}{4}} I_3(t) &\le v_t^{\frac{3}{4}-\frac{m}{2}}
\int_0^{s(t)} v_\tt^{\frac{m}{2}-\f12} |\bar R| ds(\tt) \les
s^{\frac{3}{2}-m} \int_0^{s(t)} s^{m-2} |\bar R r| d s(\tt)\label{i1.2}\\
&\les\left(\int_0^{s(t)} |\bar R r|^2 d s(t')
\right)^{\f12}\nn
\end{align}
where we used  $s^2\approx v_t$, $r\approx s$ in Proposition
\ref{cmps} to derive the second inequality. Thus
\begin{equation}\label{i32}
\left\|\sup_{t\in(0,1]}|v_t^{\frac{3}{4}}
I_3(t)|\right\|_{L_\omega^{p}} \les \|r\bar R\|_{L_\omega^{p}
L_t^2}\les \Rp +\RR.
\end{equation}
We can proceed in a similar fashion for the other two terms,
\begin{align}
v_t^{\frac{3}{4}} I_2(t) &\le
v_t^{\frac{3}{4}-\frac{m}{2}}\int_0^{s(t)} v_\tt^{\frac{m}{2}} |A|^2
|\Ab| d s(\tt) \les s^{\frac{3}{2}-m} \int_0^{s(t)} s^{m}|A|^2 |\Ab|
d s(\tt)\label{i2.1}\\
&\les \|A\|_{L_t^2}^2\|r^{3/2}\Ab\|_{L_t^\infty}\nn,
\end{align}
then
\begin{equation}\label{i2t}
\left\|\sup_{0<t\le 1}|v_t^{\frac{3}{4}}
I_2(t)|\right\|_{L_\omega^{p}} \les \|A\|_{L_\omega^\infty
L_t^2}^2\|r^{3/2} \Ab\|_{L_\omega^{p} L_t^\infty}\les \RR^3.
\end{equation}
Similarly,
\begin{align}
&v_t^{\frac{3}{4}} I_1(t) \les r(t)^{\frac{3}{2}-m}  \int_0^t s^m
|(\D_0 M+F)\c A |dt'.\label{vti1}
\end{align}
Taking $L_\omega^{p}$, using (\ref{czm}) for $\D_0 M$,
\begin{align}
\left\|\sup_{0<t\le 1}|v_t^{\frac{3}{4}} I_1(t)|
\right\|_{L_\omega^{p}}&\les (\|v_t^{\frac{3}{4}}\D_0
M\|_{L_\omega^{p} L_t^2}+\|v_t^{\frac{3}{4}}F\|_{L_\omega^{p}
L_t^2})\|A\|_{L_\omega^\infty L_t^2}\nn\\
&\les (\|v_t^{\frac{3}{4}} M\|_{L_t^2
L_\omega^{p}}+\|v_t^{\frac{3}{4}}F\|_{L_t^2
L_\omega^{p}}+\Rp)\|A\|_{L_\omega^\infty L_t^2} \label{i1t1}.
\end{align}
Note that $\|A\|_{L_\omega^\infty L_t^2}\les \RR$ can be
sufficiently small, then
 we combine (\ref{i32}), (\ref{i2t}) and (\ref{i1t1}) to obtain
  $$\|r^{3/2} M\|_{L_\omega^{p} L_t^\infty(\H)}\les \RR+\Rp.$$

Using Minkowski inequality, we can obtain in view of (\ref{i1.2})
\begin{align*}
r(t)^{-\f12}\|v_t^{\frac{3}{4}}I_3(t)\|_{L_\omega^{p}}&\les
r(t)^{-1}\int_0^t r(t') \|\bar R\|_{L_\omega^p} dt'
\end{align*}
then
\begin{equation*}
\|r I_3(t)\|_{L_t^2 L_\omega^{p}}\les \|r\bar R\|_{L_t^2[0,1]
L_\omega^{p}}\les \Rp.
\end{equation*}
We can also obtain in view of (\ref{i2.1}) that
\begin{equation*}
v_t^{\f12} I_2(t)\les\|A\|_{L_t^2}^2 \|r \Ab\|_{L_t^\infty}
\end{equation*}
by taking $L_t^2 L_\omega^{p}$ norm and using (\ref{trc2}), we have
\begin{align*}
\|v_t^{\f12}I_2\|_{L_t^2 L_\omega^{p}}&\les\|A\|_{L_\omega^\infty
L_t^2}^2 \|r \Ab\|_{L_\omega^{p} L_t^\infty}\les \RR^3.
\end{align*}
For $I_1(t)$, in view of (\ref{vti1})
\begin{align*}
&v_t^{\frac{1}{2}} I_1(t) \les r(t)^{1-m}  \int_0^t s^m
|(\D_0 M+F)\c A |dt'.
\end{align*}
Now taking $L_t^2([0,1])L_\omega^{p}$ by H\"{o}lder inequality,
\begin{align*}
\|v_t^{\f12} I_1\|_{L_t^2 L_\omega^{p}}&\les\|A\|_{L_\omega^\infty
L_t^2}\|r(|\D_0 M|+|F|)\|_{
L_\omega^{p} L_t^2}\\&\les\RR(\|rM\|_{L_t^2
L_\omega^{p}}+\|rF\|_{L_t^2 L_\omega^{p}}+\Rp)
\end{align*}
where for the last step we employed Minkowski inequality and (\ref{czm}). Since $\RR$ can be sufficiently small, the second inequality in
(\ref{M1}) can be proved.
\end{proof}

 Now we prove Assumption \ref{l8.1}.
\begin{proposition}\label{swsph}
For the transport local coordinates $(t,\omega)$, the following
properties hold true for all surfaces ${S}_t$ of the time foliation
on the null cone $\H$: the metric $\stackrel{\circ}\gamma_{ij}(t)$
on each ${S}_t$ verifies weakly spherical conditions i.e.
\begin{equation}\label{8.0.3}
\|\p_k \stackrel{\circ}\gamma_{ij}(t)-
\p_k{\gas}_{ij}\|_{L^{p}_\omega L_t^\infty}\lesssim \RR+\Rp,
\end{equation}
where $ i,j,k=1,2$.
\end{proposition}

\begin{proof}
Since relative to the transport coordinate on $\H$,
$\frac{d}{ds}\ga_{ij}=2\chi_{ij}$,
\begin{equation}\label{dtga}
\frac{d}{dt}(\gac_{ij})=an(\kp \c \ga_{ij}+2\chih_{ij})r^{-2}.
\end{equation}
Differentiate the above equation again, we obtain,
\begin{align*}
\frac{d}{dt}\p_k \gac_{ij}&=an\left(\p_k \log(an)\kp \gac_{ij}+\p_k
\tr\chi \gac_{ij}\right.\\
&\quad \, \left. +\kp \p_k \gac_{ij}+2 \p_k \chih_{ij} r^{-2}+2\p_k
\log(an) \chih_{ij}r^{-2}\right)
\end{align*}
where $i, j, k=1,2$,  with the initial condition given by
(\ref{8.1.1}). Integrating the following transport equation along a
null geodesic initiating from vertex,
 by $\|\kp\|_{L_t^2 L_\omega^\infty}\les \RR$ in
(\ref{mt5}) and a similar argument to Lemma \ref{tsp2}, we can
obtain
\begin{align*}
\left|\p_k\stackrel{\circ}\gamma_{ij}(t)-\p_k{\gas}_{ij}\right|
&\les\Big|\int_0^{s(t)}\left(\p_k \tr\chi\cdot
{\stackrel{\circ}\ga}_{ij}+r^{-2}(\sn_k\chih_{ij}-\Gamma\cdot
\chih)\right.\\
&\quad \, \left. +\p_k \log(an)\gac_{ij}\kp+2\p_k \log(an)
\chih_{ij}r^{-2}+\kp \p_k {\gas}_{ij}\right)  d s(\tt)\Big|,
\end{align*}
where $\Gamma$ represents Christoffel symbols, and $\Gamma \cdot
\chih$ stands for the terms $\sum_{l=1}^2\Ga_{ki}^l\chih_{lj}$ with
$l=1,2$. Then with the help of  $\|\kappa\|_{L_t^2
L_\omega^\infty(\H)}\les \RR$ in (\ref{mt5}) and Lemma \ref{tsp2},
\begin{align}
& \left\|\sup_{0<t\le 1}\left|\p_k
\stackrel{\circ}\gamma_{ij}(t)-\p_k{\gas}_{ij}\right|\right\|_{L_\omega^{p}} \nn\\
&\qquad \qquad \lesssim\|\Gamma\|_{L_\omega^{p}
L_t^2}\||\chih|_\ga\|_{L_\omega^\infty L_t^2}+\|r|\sn \tr \chi|_\ga
\|_{L_\omega^{p} L_t^1}+\|r|\sn \chih|_\ga\|_{L_\omega^{p}
L_t^1}\nn\\
&\qquad \qquad +\|r|\zeta+\zb|_\ga\c\kp\|_{L_\omega^{p} L_t^1}+
\left\|r|(\zeta+\zb)\c\chih|_\ga \right\|_{L_\omega^{p} L_t^1}
+\|\kp\|_{L_\omega^{p}L_t^1}\label{li2}.
\end{align}
We estimate the terms in the line of (\ref{li2}) by (\ref{mt2}) and
(\ref{blin1})
\begin{equation*}
(\ref{li2})\les\|\kp\|_{ L_\omega^{p} L_t^2}(\|
A\|_{L_\omega^\infty L_t^2}+1)+\|r A\c
A\|_{L_\omega^{p}L_t^1}\les\RR^2+\RR
\end{equation*}
 Using
Proposition \ref{MN1},
\begin{equation*}
\left\|\sup_{0<t\le 1}\left|\p_k
\stackrel{\circ}\gamma_{ij}(t)-\p_k{\gas}_{ij}\right|\right\|_{L_\omega^{p}}\les\Rp+\|\chih\|_{L_\omega^\infty
L_t^2} \|\Gamma\|_{L_\omega^{p} L_t^2}.
\end{equation*}

Summing over all $i,j,k=1,2$, we have for the component of Christoffel Symbol
$$
\|\Ga\|_{L_\omega^{p}}\les\sum_{i,j,k=1,2} \|\p_k
(\gac_{ij}-{\gas}_{ij})\|_{L_\omega^{p}}+C,
$$
where $C$ is the constant such that  the Christoffel symbol of
$\gas$ satisfies $|\p \gas|\le C$, (\ref{8.0.3}) then
follows by   $\|\chih\|_{L_\omega^\infty L_t^2}\le C \RR$ which can
be sufficiently small.
\end{proof}

\section{\bf {Proof of Proposition \ref{rices}}}\label{sec7}

In this section, we give the proofs of (\ref{pi.2})--(\ref{ric3}) in Proposition \ref{rices}.
Recall the definition of the region $\D^+$ from Section \ref{convv}.  We consider $\D^+$ under
the rescaling  coordinates according to $(t, x)\rightarrow (\frac{t}{\la}, \frac{ x}{\la})$.
The quantities $\C_0$ and $\R_0$ introduced in Theorem \ref{recalll2} in the frequency dependent
coordinates  then satisfy $\C_0\les \la^{-\f12}$, $\R_0\les \la^{-\f12}$ and  the quantity
$\R_{0,p}\les \la^{-\f12}$.  Estimates in Theorem \ref{recalll2} hold with $\R_0$ replaced
by $\la^{-\f12}$ and Proposition \ref{smr}  holds with $\R_0+\R_{0,p}$ replaced by $\la^{-\f12}$.
The estimates for $\zeta, k$ in (\ref{ricstu}) were derived by combining (\ref{mt4}), (\ref{calc0}) and (\ref{trc2}).

\subsection{Estimates for $z$ and $\Omega$}

We first derive the estimates on $z$ and $\Omega$ given in Proposition \ref{rices},
where $z$ and $\Omega$ are defined by (\ref{zomega}), i.e. $z=a\tr\chi'-\frac{2}{n(t-u)}$
and $\Omega=\frac{\bb^{-1}-n^{-1}}{t-u}$.

\begin{proposition}
Let $\U=\frac{1}{s}-\frac{1}{an(t-u)}$, and let $z$ and $\Omega$ be defined by (\ref{zomega}).
then there holds on $S:=S_{t,u}$ that
\begin{equation}\label{u2}
\|r^{-\f12}(\U,z, \Omega)\|_{L^2(S)}+\|\U,z, \Omega\|_{L^4(S)}\les \la^{-\f12}.
\end{equation}
\end{proposition}

\begin{proof}
Observe that $z=aV+2a\U$ and $\|V\|_{L^4(S)}+\|r^{-\f12} V\|_{L^2(S)}\les \tau_*^{\f12}\RR^2\les\la^{-\f12}$.
Thus the estimates for $z$ follow immediately from those estimates on $\U$.

In order to derive the estimates on $\U$, we write $\U=z_1-z_2$, where
\begin{equation}\label{z12}
z_1=\frac{1}{s}-\frac{1}{an_p (t-u)}, \quad
z_2=\frac{1}{a(t-u)} \left(\frac{1}{n}-\frac{1}{n_p}\right)
\end{equation}
with $n_p$ being the value of the lapse function $n$ at the vertex $p$ of the null cone $C_u$.
Integrating along any null geodesic initiating from vertex, we obtain,
\begin{equation}\label{z2e}
|z_2|\le \frac{1}{a (t-u)} \left|\int_u^t \frac{d}{ds} (\frac{1}{n})na dt'\right|.
\end{equation}
Therefore
\begin{equation}\label{tn1}
\|z_2\|_{L_\omega^\infty L_t^2}\les \|\nab_L n\|_{L_\omega^\infty L_t^2}\les\C_0.
 \end{equation}
Recall that $|a-1|\le 1/2$, $C^{-1}<n<C$ and $t-u\approx s$. We can derive from (\ref{z2e}) that
\begin{equation}\label{z2e2}
\|r^{-\f12}z_2\|_{L^2(S)}\les\|r^{-\f12}\nab_L n\|_{L^2(S)}\les \C_0\les \la^{-\frac{1}{2}},
\quad \|z_2\|_{L^4(S)}\les \|\nab_L n\|_{L^4(S)}\les \C_0 \les \la^{-\frac{1}{2}}.
\end{equation}

We next consider $z_1$. On the outgoing null cone $C_u$ with vertex $p$, we have the transport
equation
\begin{equation}\label{dz1}
\frac{d}{ds}z_1+\frac{2}{s}z_1=z_1^2+\frac{1}{an_p (t-u)}z_2+\frac{{\frac{da}{ds}}}{a^2n_p (t-u)}
\end{equation}
Integrating along any null geodesic initiating from the vertex, we obtain,
\begin{equation}\label{tspz1}
|z_1|\les \left|\frac{1}{s^2}\int_u^t s^2 \left( z_1^2+
\frac{1}{an_p (t-u)}z_2+\frac{{\frac{da}{ds}}}{a^2n_p (t-u)}\right)\right|.
\end{equation}
In view of (\ref{tn1}), it follows that
\begin{align}
\|z_1\|_{L_\omega^\infty L_t^2}&\les \tau_*^{\f12}\|z_1\|_{L_\omega^\infty
L_t^2}^2 +\|z_2\|_{L_\omega^\infty L_t^2}+\|\nab_L
a\|_{L_\omega^\infty L_t^2}\nn\\&\le C_2 \left(\tau_*^\f12 \|z_1\|_{L_\omega^\infty L_t^2}^2+\C_0\right).\label{z1est}
\end{align}
Recall that $\C_0\le C_1 \la^{-1/2}$ for some universal constant $C_1>0$. We claim that
$\|z_1\|_{L_\omega^\infty L_t^2} \le 2 C_2 \C_0$  on any time interval $[0, T]$ with
$T<(4 C_1 C_2^2)^{-2}$. Observe that $\|z_1\|_{L_\omega^\infty L_t^2} \rightarrow 0$ as $T\rightarrow 0$,
the estimate holds on sufficiently small interval $[0, T]$. Thus, it follows from (\ref{z1est}) and
$\tau_*\le \la T$ that on the same time interval $[0, T]$ with $T<(4C_1 C_2^2)^{-2}$ there holds
$\|z_1\|_{L_\omega^\infty L_t^2}< 2 C_2 \C_0$. Consequently the estimate holds on larger interval and thus
we must have the claim on any interval $[0, T]$ with $T<(4C_1 C_2^2)^{-2}$. Therefore
\begin{equation}\label{tn2}
\|z_1\|_{L_\omega^\infty L_t^2}\les \C_0\les \la^{-1/2}.
\end{equation}
With the help of (\ref{dz1}), (\ref{tn2}) and Lemma \ref{tsp2} we thus obtain
\begin{equation*}
\|r^{-\f12}z_1\|_{L^2(S)}\les
\|r^{-\f12}z_2\|_{L^2(S)}+\|r^{-\f12}\nab_L a\|_{L^2(S)}
\end{equation*}
which together with (\ref{z2e2}) gives
\begin{equation*}
\|r^{-\f12}z_1\|_{L^2(S)}\les\C_0+\N_1[\nu]\les\C_0\les \la^{-1/2}.
\end{equation*}
In the same way, we can also derive that
$$
\|z_1\|_{L^4(S)}\les \|z_2\|_{L_t^\infty L^4_x(C_u)}+\|\nu\|_{L_t^\infty
L_x^4(C_u)} \les\N_1[\nu]+\C_0\les\C_0\les \la^{-1/2}.
$$
Combining this with (\ref{z2e2}), we obtain the estimate for $\U$ in (\ref{u2}).

Finally we consider $\Omega$. Observe that $\Omega=\frac{a^{-1}-1}{n_p(t-u)}+\frac{n_p^{-1}-n^{-1}}{t-u}$,
we have
\begin{equation*}
|\Omega|\le \frac{1}{n_p(t-u)} \left|\int_u^t L(a^{-1}) ds \right|
+\frac{1}{(t-u)} \left|\int_u^t L(n^{-1}) ds \right|.
\end{equation*}
Similar to (\ref{z2e2}) we can derive that
\begin{align*}
&\|r^{-\f12}\Omega\|_{L^2(S)}\les \|r^{-\f12}(|\nab_L n|+|\nab_L a|)\|_{L^2(S)}\les\C_0 \les \la^{-1/2}, \\
&\|\Omega\|_{L^4(S)}\les \|\nab_L n\|_{L^4(S)}+\|\nab_L a\|_{L^4(S)}\les \C_0 \les \la^{-1/2}.
\end{align*}
Thus we complete the proof of (\ref{u2}).
\end{proof}

\begin{proposition}
Relative to the canonical null frame, there holds on every $S_{t,u}$
that
\begin{equation}\label{ric6}
\|r^{\frac{3}{2}}\Lb z\|_{L^{p}_\omega}+\|r^{\frac{3}{2}}\sn z\|_{L^{p}_\omega}\les \la^{-\f12}.
\end{equation}
\end{proposition}

\begin{proof}
We consider $\Lb z$ first. By the definition of $\mu$ and  using (\ref{lbr}), we derive that
\begin{align*}
\Lb z&-\left(\mu+\tr\chi(k_{NN}+\nab_N \log n )\right)\\
& =\Lb
z-\Lb\tr\chi+\frac{1}{2}(\tr\chi)^2=\f12 (\tr\chi)^2-2 \Lb
\left(\frac{1}{n(t-u)}\right)\\
&=\f12 \left(z+\frac{2}{n(t-u)}\right)^2-2\left(\frac{n\bb^{-1}}{n^2(t-u)^2}-\frac{\Lb
\log n }{n(t-u)}-\frac{n(n^{-1}-\bb^{-1})}{n^2(t-u)^2}\right)\\
&=\f12 \left(z^2+\frac{4z}{n(t-u)}\right)+4\frac{1-n\bb^{-1}}{n^2(t-u)^2}+2\frac{\Lb
\log n }{n(t-u)}.
\end{align*}
In view of  $r\tr\chi\approx 1$ and (\ref{u2}), we have $\|r^{\frac{3}{2}}z^2\|_{L_\omega^{p}}\les \RR \les \la^{-1/2}$.
Therefore, by using (\ref{smry1}) and applying
\begin{equation}\label{intpl}
  \|r^{\f12} F\|_{L_\omega^{p}}\les
  \|r^{-\f12}F\|_{L^2(S)}+\|F\|_{L^4(S)},
\end{equation}
to $F=\Omega,\, \Lb \log n,\, z$, we can deduce in view of (\ref{u2}), (\ref{calc0}) and
$\|r^{\frac{3}{2}}\mu\|_{L_\omega^p}\les \la^{-\f12}$ in Proposition \ref{smr} that
\begin{align*}
\|r^{\frac{3}{2}} \Lb z\|_{L^{p}_\omega}&\les
\|r^{\frac{3}{2}}\mu\|_{L^{p}_\omega}+\|r^{\frac{3}{2}}z^2\|_{L_\omega^{p}}+\|r^{\f12}(z,
\Lb \log n, \Omega)\|_{L_\omega^{p}}\les\la^{-\f12}.
\end{align*}

In order to obtain the estimate for $\sn z$, we apply (\ref{intpl}) to $\sn a$ and $\sn n$
and use (\ref{smry1}) for $\sn \tr\chi'$. It then follows that
\begin{align*}
\|r^{\frac{3}{2}} \sn z\|_{L_\omega^{p}}&\les \|r^{\frac{3}{2}}
\tr\chi' \sn a\|_{L_\omega^{p}}+\|r^{\frac{1}{2}}\sn (n^{-1})
\|_{L_\omega^{p}}+\|r^{\frac{3}{2}}a\sn
\tr\chi'\|_{L_\omega^{p}}\les \RR+\Rp.
\end{align*}
This completes the proof.
\end{proof}

We define Hardy-Littlewood maximal function for scalar function $f(t)$  by
\begin{equation*}
\M(f)(t)=\sup_{t'} \frac{1}{|t-t'|}\int_{t'}^t |f(\tau)| d\tau.
\end{equation*}
It is well-known that for any $1<q<\infty$ there holds
\begin{equation}\label{hlm}
\|\M(f)\|_{L_t^q} \les \|f\|_{L_t^q}.
\end{equation}

\begin{proposition}
There holds on $\D^+$ that $\|z\|_{L_t^2 L_x^\infty}\les \la^{-\f12}$.
\end{proposition}

\begin{proof}
Recall the estimate (\ref{mt0}) in Theorem \ref{recalll2}. We then have
\begin{align*}
\|z\|_{L_t^2 L_x^\infty}&\les
\|aV\|_{L_t^\infty L_x^\infty}\tau_*^{\f12}
+\left\|\frac{1}{s}-\frac{1}{an(t-u)}\right\|_{L_t^2 L_x^\infty}
\les \la^{-\f12}T^{\f12}+\|\U\|_{L_t^2 L_x^\infty} .
\end{align*}
Assuming   the following estimate on $\D^+$
\begin{equation}\label{sandt}
\|\U\|_{L_t^2 L_x^\infty}
\les \|\nab_L n,\nab_L a\|_{L_t^2 L_x^\infty},
\end{equation}
noting that  (\ref{pi.2}) implies $\|\nu, \nab_L n \|_{L_t^2 L_x^\infty} \les \la^{-1/2}$,
therefore we can obtain  (\ref{ric1}).

To see (\ref{sandt}), due to $\U=z_1-z_2$, let us consider $z_2$ first.
In view of  (\ref{z2e}), $|a-1|\le 1/2$, $C^{-1}<n<C$ and
$t-u\approx s$, we can obtain
\begin{equation}\label{z2}
\sup_{\omega\in {\Bbb S}^2, u}|z_2|(t)\les \M(\|\nab_L n \|_{L^\infty_x}).
\end{equation}
In view of (\ref{tspz1}), we also have
\begin{align*}
|z_1|&\les \frac{1}{s(t)^2}\int_{u}^t s(t')^2 \left|\frac{1}{an_p(t-u)}z_2
+\frac{1}{a^2n_p(t-u)} {\frac{da}{ds}}\right| dt'.
\end{align*}
Therefore
\begin{equation}\label{z1}
\sup_{\omega\in {\Bbb S}^2, u}|z_1|(t)\les\M(\|z_2\|_{L_x^\infty})+\M(\|\nab_L a\|_{L_x^\infty}).
\end{equation}
Combining (\ref{z2}) and (\ref{z1}), taking $L_t^2$ norm, and using (\ref{hlm}) implies (\ref{sandt}).
\end{proof}

\subsection{\bf $L_t^2 L_x^\infty$ estimate for $\chih$ and $\zeta$}\label{sec8}

We will use the equations (\ref{stc1})-(\ref{stc3}) to derive the estimates for
$\chih$ and $\zeta$. We first recall from \cite{KR2} the following  version of Calderon-Zygmund-type result
for Hodge systems.

\begin{proposition}\label{cz}
Let $F$ be a covariant traceless symmetric $2$-tensor satisfying the Hodge system
\begin{equation}\label{fe1}
\div F=\sn G+e \qquad \mbox{on } S_{t,u}
\end{equation}
for some scalar function $G$ and $1$-form $e$.
Then for $2<p<\infty$ and $\frac{1}{q}=\frac{1}{2}+\frac{1}{p}$ there holds
\begin{equation}\label{lpp2}
\|F\|_{L^p(S_{t,u})}\les\|G\|_{L^p(S_{t,u})}+\|e\|_{L^q(S_{t,u})}
\end{equation}
and for $p>2$ there holds
\begin{equation}\label{cz0}
\|F\|_{L^\infty(S_{t,u})}\les\|G\|_{L^\infty(S_{t,u})}\ln {\small \left(2+r^{\frac{3}{2}-\frac{2}{p}}\la^{\f12}\|\sn
G\|_{L^p(S_{t,u})}\right)}+r^{1-\frac{2}{p}}\|e\|_{L^p(S_{t,u})}.
\end{equation}
Similarly, for the Hodge system
\begin{equation}\label{divcurl}
\left\{
\begin{array}{lll}
\div F= \sn\c G_1+e_1,\\
\curl F=\sn\c G_2+e_2,
\end{array}
\right.
\end{equation}
where $G=(G_1, G_2)$ are $1$-forms and $e=(e_1, e_2)$ are scalar functions, we have
(\ref{lpp2}) and for any $p>2$ there holds
\begin{align}
\|F\|_{L^\infty(S_{t,u})}\les\|G\|_{L^\infty(S_{t,u})}\ln {\small \left(2+r^{\frac{3}{2}-\frac{2}{p}}\la^{\f12}\|\sn
G, r^{-1}G\|_{L^p(S_{t,u})}\right)}+r^{1-\frac{2}{p}}\|e\|_{L^p(S_{t,u})}\label{cz1}.
\end{align}
\end{proposition}

Relative to rough metric, the proof of Proposition \ref{cz} can follow from the standard Calderon-Zygmund theory
with the help of weakly spherical condition (\ref{8.0.3}) and (\ref{ws1}).
Consider the application of (\ref{cz0}) to  $\div \chih=\sn \cpi+\sn \tr\chi+\cdots$
which follows from the Hodge system (\ref{stc1}) and Proposition \ref{dcmp2} (a),
where $\cpi$ is the major part of $G$.
In principle we only have the bounds of  $\|r\sn \cpi\|_{L_t^2 L_\omega^p}$ on null cones,
which needs one-half more derivative to give the estimate of
$\sup_{u}r^{\frac{3}{2}}\|\sn G\|_{L^p_\omega}$ for $p>2$ in (\ref{cz0}). This
poses an issue of lack of  half-derivative. Note that,  by Strichartz estimate,
we may expect to  control   $\|\mu^{0+}P_\mu(\hn g, k) \|_{l_\mu^1 L_t^2  L_x^\infty}$, which is slightly stronger  than $\|\cpi\|_{L_t^2 L_x^\infty}$. On the other hand,  since  $\cpi$ takes the form of  $N^\mu(g\hn g, k)_{\mu\cdots}$ or $\Pi^{\mu'}_{\mu}(g\hn g, k)_{\mu'\cdots}$ relative to spacial harmonic coordinates, to control $\|\mu^{0+}P_\mu\cpi\|_{l_\mu^1 L_t^2  L_x^\infty}$ causes a serious technical baggage due to the factor $N^\mu$ involved in the paradifferential calculation. To solve all the potential issues, our strategy is  to establish (\ref{cz2}) so as to take advantage of  the extra
differentiability of $k, \hn g$.

\begin{proposition}\label{cz.2}
Let $F$ and $G$ be $S$-tangent tensor fields of suitable type satisfying (\ref{fe1}) or
(\ref{divcurl}) with certain term $e$. Suppose $G$ is  a projection of $\Sigma$ tangent
tensor $\ti G$ to tangent space of $S$ by $ \Pi_{\mu}^{\mu'}{\ti G}_{\mu'\cdots}$ or
takes the form of $N^\mu{\ti G}_{\mu\cdots}$. Then for $p>2$, $\delta>0$ sufficiently close to $0$,
and $1\le c<\infty$ there holds
\begin{align}
\|F\|_{L^\infty(S_{t,u})}\les \|\mu^{\delta}P_\mu \ti G\|_{l_\mu^c L^\infty(S_{t,u})} +\|\ti G\|_{L^\infty(S_{t,u})}+r^{1-\frac{2}{p}}\|e\|_{L^p(S_{t,u})}.\label{cz2}
\end{align}
\end{proposition}

The proof of (\ref{cz2}) is  presented  in the Appendix II.

We slightly extend the definition of $\er$ in  Proposition \ref{dcmp2} to
\begin{equation}\label{tie}
 \er: =\Ab\c\Ab+\tr\chi\c \ckk\pi.
\end{equation}
where $\Ab$ was defined in Section \ref{ricc}.  By using (\ref{smry1}) we have
$
\|r^{1-\frac{2}{p}} \Ab\c \Ab\|_{L^{p}(S_{t,u})}\les(\Rp+\RR)^2.
$
Also using the fact $r\tr\chi\approx 1$,
 we can obtain the following estimate on $\er$.
\begin{lemma}
For any $p>2$ there holds
\begin{equation}\label{err3}
\|r^{1-\frac{2}{p}} \er\|_{L^{p}(S_{t,u})}\les
\|\ckk\pi\|_{L_\omega^\infty}+(\Rp+\RR)^2
\end{equation}
\end{lemma}

\subsubsection{Estimate of $\chih$}

In view of the equation (\ref{stc1}) and  Proposition \ref{dcmp2} (a), we have
\begin{equation*}
\div \chih=\frac{1}{2}\sn \tr \chi +\sn\cpi+\er.
\end{equation*}
We define $\schih$, a symmetric traceless 2-tensor, to be the solution of
\begin{equation}\label{hdge1}
\div\schih=\frac{1}{2}\sn\left(\tr\chi-\frac{2}{n(t-u)}\right).
\end{equation}
It is straightforward to see that
\begin{equation}\label{hdge2}
\div (\chih-\schih)=\sn \cpi+\er.
\end{equation}
Applying (\ref{cz0}) in  Proposition \ref{cz} to (\ref{hdge1}) and  also using (\ref{ric6}), we obtain
\begin{equation}\label{schi}
\|\schih\|_{L^\infty(S_{t,u})}\les \|z\|_{L^\infty(S_{t,u})}\ln(2+\|r^{\frac{3}{2}-\frac{2}{p}}\la^{\f12}\sn z\|_{L^{p}(S_{t,u})})\les \|z\|_{L^\infty(S_{t,u})}
\end{equation}
Note that with $ 2\le q\le 4$,
\begin{equation}\label{pi.i1}
\|r^{\f12-\frac{2}{q}}\ckk \pi\|_{L^q(S_{t,u})}\les\N_1[\ckk \pi] \la^{-\f12}.
\end{equation}
Applying  (\ref{cz2}) in  Proposition \ref{cz.2} to (\ref{hdge2}) and using (\ref{pi.i1}) give
\begin{equation}\label{dchi}
\|\chih-\schih\|_{L^\infty(S_{t,u})}\les \|\mu^{0+}P_\mu \cpi\|_{l_\mu^2
L^\infty (S_{t,u})}+\|\cpi\|_{L^\infty(S_{t,u})}+\|r^{1-\frac{2}{p}}\er\|_{L^{p}(S_{t,u})}
\end{equation}
where $p>2$ sufficiently close to $2$.

Combining (\ref{schi}) and (\ref{dchi}), and using  (\ref{err3}) and (\ref{ricp}), we can obtain \begin{footnote}{We denote by $0+$ a constant $b>0$ arbitrarily close to $0$}\end{footnote}
\begin{align*}
\|\chih\|_{L_t^2 L^\infty_x} &\les \|\mu^{0+}P_\mu \cpi\|_{L_t^2
l_\mu^2 L_x^\infty}+\|\ckk\pi\|_{L_t^2 L_x^\infty}+(\Rp+\RR)^2 \tau_*^{\f12}.
\end{align*}
By using (\ref{BA3}), we conclude $\|\chih\|_{L_t^2 L_x^\infty(\D^+)}\les \la^{-1/2}$, which is the first part of (\ref{ric3}).

\subsubsection{Estimate for $\zeta$}

In view of (\ref{stc2}), (\ref{stc3}) and Proposition \ref{dcmp2}, we can write
\begin{equation*}
\D_1\zeta=\sn \cpi+\er+(\mu,0)
\end{equation*}
where $\cpi=(\cpi^{(1)},\cpi^{(2)})$.

Let $\smu_A$ be an $S$-tangent co-vector which is defined as a solution of the Hodge
system
\begin{equation}\label{hge1}
\div \smu=\mu-\bar\mu\quad \curl \smu=0.
\end{equation}
It follows from (\ref{elpl2}) and (\ref{mt3}) that
\begin{equation}\label{l2smu}
\|\sn
\smu\|_{L^2(C_u)}+\|r^{-1}\smu\|_{L^2(C_u)}\les\|\mu\|_{L^2(C_u)}\les
\RR.
\end{equation}
Observe that by (\ref{stc2}), Proposition \ref{dcmp2} (b) and the divergence theorem, we can treat
$\bar\mu$ as $\overline{n A\c A+n \tr\chi \pi}$. By ignoring the average sign, in view of (\ref{tie}),
$\bar\mu$ can also be incorporated to $\er$. With the help of (\ref{hge1}), we have
\begin{equation}\label{d1zeta}
\D_1(\zeta-\smu)=\sn\cpi+\er.
\end{equation}
In view of (\ref{cz2}) in Proposition \ref{cz.2} and (\ref{pi.i1}), we obtain
\begin{equation}\label{zein}
\|\zeta\|_{L^\infty(S_{t,u})}\les \|\smu\|_{L^\infty(S_{t,u})}+
\|\mu^{0+}P_\mu \cpi\|_{l_\mu^2 L_x^\infty}+\|\cpi\|_{L^\infty(S_{t,u})}+\|r^{1-\frac{2}{p}}\er\|_{L^{p}(S_{t,u})}.
\end{equation}

\begin{lemma}\label{smu3}
In the region $\D^+$, there holds
\begin{align*}
\left(\int_{0}^{\tau_*} \sup_{u}\|\smu\|_{L^\infty(S_{t,u})}^2dt'\right)^{1/2}
&\les \|\mu^{0+}P_\mu \cpi\|_{L_t^2 l_\mu^2 L_x^\infty}+\|\ckk\pi\|_{L_t^2 L_x^\infty}
 +T^{\f12}\la^{-\f12}.
\end{align*}
\end{lemma}

Assuming this result, we can obtain from (\ref{zein}) and (\ref{err3}) that
\begin{equation}\label{zees1}
\left(\int_{0}^{\tau_*} \sup_{u}\|\zeta\|_{L^\infty(S_{t,u})}^2\right)^{\f12}\les \la^{-\f12}
\end{equation}
which, in view of (\ref{BA3}), gives the second part in (\ref{ric3}).

In order to prove  Lemma \ref{smu3}, we first derive an equation for
\begin{equation}\label{4smu}
\bd_4 \smu:=\D_t\smu+\f12n\tr\chi
\smu-n \chih\c \smu,
\end{equation}
 where $\D_t F:=n \bd_L F$ for any tensor field $F$. We need the commutation
formula (\cite[Chapter 13]{CK}, \cite{Wang1})
\begin{equation}\label{commt4}
[\D_t, \sn _B] F_A=-n\chi_{BC}\sn_CF_A+n(\chi_{AB}\zb_C-\chi_{BC}
\zb_A+\ep_{AC}{}^\star \b_B)F_C.
\end{equation}

\begin{lemma}\label{lem4}
$\bd_4 \smu$ satisfies the following Hodge system.
\begin{equation}\label{hge2}
\left\{
\begin{array}{lll}
\div\bd_4 \smu=\frac{1}{t-u}\sn
\cpi+e_1'+n(\tr\chi)^2\pi \\
\curl\bd_4 \smu=n(\chi \c \pi+R_0) \smu+n\Ab\c (A+\tr\chi) \smu+n \sn
A \c \smu,
\end{array}\right.
\end{equation}
where
$
e'_1=n\left(A\c \sn A+(A+\smu)\c( A\c \Ab+\Ab\c \tr\chi)+(A+z)\c R_0\right).
$
\end{lemma}
\begin{proof}
Let $G:=L\mu+\tr\chi\mu$. We can derive
\begin{equation}
 \D_t (\mu-\bar\mu)+n\tr\chi(\mu-\bar\mu)=n G-\overline{nG}-n\kappa\bar\mu
\end{equation}
In view of (\ref{commt4}) we then obtain
\begin{align*} \div \D_t \smu&=\D_t \div \smu+
\delta^{AB}n\chi_{BC} \sn_C
\smu_A-\delta^{AB}n(\chi_{AB}\zb_C-\chi_{BC} \zb_A+\ep_{AC}{}^\star
\b_B)\smu_C\\
&=\D_t\div \smu+\f12 n\tr\chi\div\smu+n\chih\c \sn\smu-\f12n\tr\chi \zb\c
\smu+n\chih\c \zb\c \smu-n \b\c \smu\\
&=n G-\overline{nG}-n\kappa\bar\mu-\frac{1}{2}n \tr\chi\div \smu+n\chih\c \sn\smu-\f12n\tr\chi \zb\c \smu\\
& \quad \, +n\chih\c \zb\c \smu-n \b\c \smu.
\end{align*}
Consequently
\begin{align*}
\div\bd_4 \smu & =n G-\overline{nG}-n\kappa\bar\mu+\f12 \sn (n \tr\chi) \smu-\div(n \chih)\smu\\
&\quad \, -\f12n\tr\chi \zb\c \smu+n\chih\c \zb\c \smu-n \b\c \smu.
\end{align*}
As explained before (\ref{d1zeta}), $
\kappa\bar\mu$  can be treated as $\kappa\c(an A\c A+an \tr\chi \pi)$, where the average sign has been ignored.
We also will not distinguish $\overline{nG}$ with $nG$. Then,
in view of (\ref{stc1}) and (\ref{mumu}), symbolically there holds
\begin{align*}
\div\bd_4 \smu &=n\left(A\c \sn A+(A+\smu)\c( A\c \Ab+\Ab\c \tr\chi)  +(\tr\chi)^2\pi+\tr\chi\c \rho\right).
\end{align*}
In view of (\ref{commt4}) and (\ref{hge1}), we also have
\begin{align*}
&\curl\bd_4\smu=n(\chi\c {}^\star \zb+{}^\star \b)\c \smu+\sn( n
\chih)\smu+\sn(n \tr\chi) \smu.
\end{align*}
Thus, with the help of Proposition \ref{dcmp2} (b), we obtain (\ref{hge2}).
\end{proof}

\begin{proof}[Proof of Lemma \ref{smu3}]
In view of  (\ref{sob.2}) and (\ref{l2smu}), there holds
$\| \smu\|_{L_t^2 L_\omega^{p}}\les \RR$.
By Lemma \ref{hdgm1}, (\ref{hge1}) and (\ref{smry1}), we can obtain
\begin{equation}\label{lqmu}
\|r\sn \smu\|_{L_t^2 L^{p}_\omega}\les \|r\mu\|_{L_t^2
L_\omega^{p}}+\|\smu\|_{L_t^2 L_\omega^{p}}\les \RR+\Rp.
\end{equation}
Let
$$
\E:=\|r\sn(\Ab,\smu)\|_{L_t^2 L_\omega^{p}}+\|rR_0\|_{L_t^2
L_\omega^{p}}+\|r\tr\chi (\Ab, \smu)\|_{L_t^2 L_\omega^{p}}+\|r\Ab\c
\Ab\|_{L_t^2 L_\omega^{p}},
$$
we have  $\E \les \RR+\Rp$.
By (\ref{sob.3}), we also have
\begin{equation}\label{lqmu2}
\left(\int_{u}^{\tau_*}\|\smu\|_{ L^\infty(S_{t,u})}^2\right)^{\f12}\les\RR+\Rp.
\end{equation}
By (\ref{4smu}) and Lemma \ref{tsp2},
 \begin{equation*}
\|\smu\|_{L^\infty(S_{t,u})}\les \left\|\frac{1}{t-u}\int_{u}^t s\bd_4\smu\right\|_{L^\infty(S_{t,u})}.
 \end{equation*}
Using (\ref{cz2}) in Proposition \ref{cz.2} , (\ref{pi.i1}), (\ref{hge2}) and $C^{-1}<s\tr\chi<C$ on $C_u$, we can obtain
\begin{align}
\left\|\frac{1}{t-u}\int_{u}^t s\bd_4\smu\right\|_{L^\infty(S_{t,u})}
&\les \int_{u}^tr^{1-\frac{2}{p}} \left(\|e_1'\|_{L^{p}(S_{t,u})}+\|\curl
\bd_4\smu\|_{L^{p}(S_{t,u})}\right)\nn\\
&+\frac{1}{t-u}\int_{u}^t \|\mu^{0+}P_\mu\cpi\|_{l_\mu^2 L_\omega^\infty}+\frac{1}{t-u}
\int_{u}^t \|\pi\|_{L_\omega^{p}}\label{lsmu2}
\end{align}

 Thus, by using (\ref{lqmu2}), we can derive that
\begin{align*}
\int_{u}^t \left\|r^{1-\frac{2}{p}}(|e_1'|+|\curl \bd_4\smu|)\right\|_{L^{p}(S_{t,u})}
& \les\||A|+|z|+|\smu|\|_{L_t^2 L_\omega^{\infty}(C_u)}\cdot \E\\
& \les (\Rp+\RR)\la^{-\f12}.
\end{align*}
In view of (\ref{pi.2}) and (\ref{hlm}), we have
\begin{align*}
\left(\int_{0}^{\tau_*} \sup_u \|\smu\|_{L^\infty(s_{t,u})}^2 dt'\right)^{\f12}
& \les\|\mu^{0+}P_\mu \cpi\|_{L_t^2 l_\mu^2 L_x^\infty}+\|\ckk\pi\|_{L_t^2 L_x^\infty}\\
&\quad \, +(\RR+\Rp) (\la^{-1}\tau_*)^{\f12}
\end{align*}
This completes the proof.
\end{proof}

\section{\bf Appendix I: Proof of Proposition \ref{dcmp2}}\label{decmp}

\subsection{Decomposition for $\b, \udb$.}

We need the following decomposition result on $R_{AN}$. 

\begin{lemma}\label{L12.29.1}
On $S_{t,u}\subset\Sigma_t,$  relative to orthonormal frame $\{N, e_1, e_2\}$,where
$N$ denotes the unit normal of $S_{t,u}$ in $\Sigma_t$,  there holds
\begin{footnote}
{We ignored the frame coefficients $e_A^j, N^k$ in the last three terms, since they are harmless for the purpose of application. }
\end{footnote}
\begin{equation}\label{ran}
R_{AN}=\slashed{\curl}\cpi_1+\sn_A \cpi_2+g\c \theta\c\hn g+g\c\hn g\c\hn g+\hat R\c g
\end{equation}
where the 2-tensor $\cpi_1$ and the scalar $\cpi_2$ are terms formed by the sums of
$N^\mu(g\c \hn g, k)_{\mu\cdots}$ and $\Pi^{\mu'}_{\mu}(g\c \hn g, k)_{\mu'\cdots}$.
\end{lemma}

By the Gauss equation $E_{AN}=R_{AN}+k\c k$, the identity $E_{AN}=\f12(\b+\udb)$ (see 
\cite[(7.3.3e)]{CK}),  and 
\begin{equation*}
\div \eh=\f12(\b-\udb)-\f12 \sn \delta+\hat \theta\c \ep-\f12\tr\theta\c \ep
\end{equation*}
(see \cite[(11.1.2d)]{CK}), we have symbolically that
\begin{equation*}
\b=\div \eh+\sn\delta+R_{AN}+\theta \c \ep+k\c k.
\end{equation*}
Hence Proposition \ref{dcmp2} (a) can be proved by using (\ref{ran}).

\begin{proof}[Proof of Lemma \ref{L12.29.1}]
We write $R_{AN}=N^k e_A^j R_{jk}$. Recall that
\begin{equation*}
R_{jk}=\hR_{jk}+\hn_m U_{jk}^{\dum m}-\hn_j U_{m k}^{\dum
m}-U_{jk}^{\dum n} U_{nm}^{\dum m}-U_{mk}^{\dum n}U_{nj}^{\dum m}.
\end{equation*}
Then we have 
\begin{align*}
R_{AN}
&=N^k e_A^j \hn_m U_{jk}^{\dum m}-\left(e_A^j \hn_j(N^k U_{mk}^{\dum m})-e_A^j \hn_j N^k U_{mk}^{\dum m}\right)+N^k e_A^j {\hat R}_{jk}-U\c U.
\end{align*}
For simplicity of exposition, we set $U_{kj l}=U_{kj}^{\dum p}\c g_{pl}$ and $A_{klm}:=\hn_m g_{lk}-\hn_l g_{mk}$. 
Then we have 
\begin{align*}
\hn_m U_{jk}^{\dum m}&=\f12 \hn_m \{g^{ml}(\hn_k g_{jl}+\hn_j g_{kl}-\hn_l g_{jk})\}\\
&=\hn_m g^{ml} U_{kjl}+\f12 g^{ml}(\hn_m \hn_k
g_{jl}+\hn_m\hn_j g_{kl}-\hn_m\hn_l g_{jk})\\
 &=\hn_m g^{ml} U_{kj l}+\frac{1}{2}g^{ml}(\hn_j \hn_m g_{kl}+\hat R)+\f12 g^{ml}\hn_m A_{jlk}.
\end{align*}
We now define 
\begin{align*}
\Omega_{ij}&=\f12(\ep_{i}^{\dum lm}\hn_l g_{mj}+\ep_j^{\dum lm}\hn_l
g_{mi}), \qquad \varrho_k=g^{mn}\hn_k g_{mn}-g^{mn}\hn_m g_{kn}.
\end{align*}
Since the spacial harmonic condition implies that $g^{ij}\hn_i g_{jl}=\f12 g^{ij} \hn_l g_{ij}$,
 we can derive that $\varrho_k=g^{mn}\hn_m g_{nk}$.  We claim that
\begin{equation}\label{adcomp}
A_{iab}= \left(\Omega_{ij}+\frac{1}{2}\varrho_n \ep_{ij}^{\dum n} \right)
\ep_{ab}^{\dum j}.
\end{equation}
Using (\ref{adcomp}), we can deduce that
\begin{align*}
g^{ml}\hn_m A_{jlk}
&=\ep_{lk}^{\dum s}g^{ml}\hn_m \left(\Omega_{js}+\f12 \varrho_n\ep_{js}^{\dum n } \right)
+g^{ml}\hn_m \ep_{lk}^{\dum s} \left(\Omega_{js}+\f12 \varrho_n \ep_{js}^{\dum n }\right).
\end{align*}
Hence, with $\dot{\pi}_{jC}=\f12(\Omega_{jC}+\f12 \varrho_n \ep_{jC}^{\dum n})$, we have
\begin{align*}
e_A^j & N^k \hn_m U_{jk}^{\dum m}\\
&=e_A^j N^k \ep^m{}_k{}^s \hn_m (\Omega_{js}+\f12 \varrho_n \ep_{js}^{\dum n})+\f12 e_A^j N^k g^{ml} \hn_j \hn_m g_{kl}+g\hn g\c \hn g +g\c \hat R\\
&=-e_A^j \ep^{BC} \sn_B \dot{\pi}_{jC}+\f12 e_A^j N^k g^{ml}  \hn_j\hn_m g_{kl}+\theta\c g\c \hn g+g \hn g\c \hn g+g\c\hat R\\
&=-(\slashed{\curl} \dot{\pi})_{A}+\f12 e_A^j \hn_j(N^k g^{ml} \hn_m  g_{kl})+\theta\c g\c \hn g+g \hn g\c \hn g+g\c \hat R.
\end{align*}
We then conclude symbolically that
\begin{equation*}
R_{AN}= \sn_A \left(\f12 g^{ml} \hn_m g_{kl}N^k+N^k U_{mk}^{\dum m}\right)
+\slashed{\curl}\dot{\pi}+g\c\theta\c\hn g +g\hn g\c\hn g +g\c \hat R
\end{equation*}
as desired.

It remains to verify (\ref{adcomp}). By the definition of $\Omega_{ij}$ we have
\begin{align*}
2\Omega_{ij}\ep_{ab}^{\dum \dum \dum j}&=\ep_{i}^{\dum \dum lm}
\ep_{ab}^{\dum \dum j} \hn_l g_{mj}+\ep_j^{\dum\dum
lm}\ep_{ab}^{\dum\dum\dum j} \hn_l g_{mi}.
\end{align*}
From the properties of the volume form components ${\ep_i}^{jl}$, we have
\begin{equation*}
\ep_j^{\dum\dum lm}\ep_{ab}^{\dum\dum\dum j} \hn_l g_{mi}=\hn_a g_{bi}-\hn_b g_{ai},
\quad 
\ep_i^{\dum \dum lm}\ep_{ab}^{\dum\dum j}\hn_l g_{mj}
=\hn_a g_{bi}-\hn_b g_{ai}+g_{ia}\varrho_b-g_{ib}\varrho_a.
\end{equation*}
It is easy to check that $g_{ia}\varrho_b-g_{ib}\varrho_a=- \varrho_n
\ep_{ij}^{\dum\dum n} \ep_{ab}^{\dum \dum j}$. Hence we have
\begin{equation*}
\Omega_{ij}\ep_{ab}^{\dum \dum \dum j}=\hn_a g_{bi}-\hn_b
g_{ai}-\f12 \varrho_n \ep_{ij}^{\dum\dum n} \ep_{ab}^{\dum \dum j}.
\end{equation*}
This completes the proof of (\ref{adcomp}).
\end{proof}

\subsection{Decompositions of $\a$ and $\sigma$}

The decompositions of $\a$ and $\sigma$ follow immediately from the structure equations
\begin{align*}
\frac{1}{2}\a&=\bd_4\eh+\frac{1}{2}\tr\chi \eh+n^{-1}
{\hat\sn^2}n-\frac{1}{2}\sn\hot\ep+\f12\delta \eh-a^{-1}\sn a \hot
\ep-\frac{3}{2} \hat\theta \delta\\
& \quad \, -\ep\hot \ep+(\zeta-\sn \log
n )\hot \ep
\end{align*}
and $\sigma=\slashed{\curl} \ep-\hat\theta\wedge\eh$ 
which can be found in \cite[Chapter 11]{CK}.

\subsection{Decomposition for $\rho$}

Let $(e_A)_{A=1,2}$ be an orthonormal frame on $S$, and let ${\cpi}_{AB}^C$ be 
the projection of $U_{ij}^l$ onto $S$, i.e. 
$$
{\cpi}_{AB}^C=e_A^i e_B^j (e_C)_k U_{i'j'}^{\dum\dum k'}\Pi_{i}^{i'}\Pi_j^{j'} \Pi_{k'}^{k}.
$$
where $U_{ij}^l$ is the tensor defined by (\ref{11.5.2}) and $\Pi_i^j=\delta_i^j -N_i N^j$ is the 
projection operator. Recall that 
\begin{equation*}
R^i{}_{jkl}=\hn_k U_{lj}^{\dum\dum i}-\hn_l U_{kj}^{\dum\dum i}
+U_{km}^{\dum\dum i}U_{lj}^{\dum\dum m}-U_{lm}^{\dum\dum i} U_{kj}^{\dum\dum m}+\hat{R}^i{}_{jkl}
\end{equation*}
By using $\hn_k U_{lj}^{\dum\dum i}=\nab_k U_{lj}^{\dum\dum i}+U\c U$, we have
\begin{equation}\label{dcmp1}
R_{ABCD}=e_{Ai}e_B^j e_C^k e_D^l R^i_{jkl}=\sn_C \cpi_{DBA}-\sn_D \cpi_{CBA}+\ti E_{ABCD}.
\end{equation}
where $\ti E$ consists of the projections $U\c U$  to $S$ and the terms 
$\theta\c (N^l U_{l\cdots}, \Pi^l_{m}U_{l\cdots})$
and for $|\ti E|:=\max_{A,B,C,D} |\ti E_{ABCD}|$ we have 
$|\ti E|\les |\hn g\c \hn g|+|\theta\c \hn g|$.
Recall from \cite[P.167]{CK} that $\ga^{AC} \ga^{BD} \bR_{ADCB}=-2\rho$.
By the Gauss equation, we then have $-2\rho=\ga^{AC}\ga^{BD}(R_{ABCD}+k\c k)$. 
This together with (\ref{dcmp1}) gives the decomposition for $\rho$ in Proposition \ref{dcmp2}.

\section{\bf Appendix II: Proof of  Proposition \ref{cz.2} }\label{czz}

In this section we will complete the proof of Proposition \ref{cz.2} by employing the geometric
Littlewood-Paley theory on the surfaces $S_t$ developed in \cite{KRsurf}.  In Section 6 we introduced
the metric $\gaz:=r^2 \gamma_{\mathbb S^2}$ on $S_{t,u}$, where $r$ denotes the the radius of
$S_{t,u}$ with respect to $\gamma$, i.e $r:=\sqrt{\frac{1}{4\pi} \int_{S_{t,u}} d\mu_\ga}$.
Let $m$ be a function in the Schwartz class defined on $[0,\infty)$ having finite number of vanishing moments
and set $m_\mu(\tau):=\mu^2 m(\mu^2 \tau)$ for any dyadic numbers $\mu>0$. The geometric Littlewood-Paley
projection $\Pz\mu$ associated with $\gaz$ is defined by
\begin{equation*}
\Pz{\mu} H=\int_0^\infty m_\mu(\tau) U(\tau)H d\tau
\end{equation*}
for any $S$-tangent tensor $H$, where $U(\tau') H$ denotes the solution of the heat flow
\begin{equation*}
\frac{d}{d\tau}U(\tau)H=r^2\sD^{(0)} U(\tau)H, \quad\quad U(0)H=H.
\end{equation*}
One can refer to \cite{KRsurf} for various properties of $U(\tau)$ and $\Pz\mu$. In particular,
it has been shown that one can always find an $m$ such that the associated $\Pz\mu$ satisfies
$\sum_{\mu>1} {P_\mu^{(0)}}^2 +{P_{\le 1}^{(0)}}^2=Id$.

\begin{lemma}
For $2\le p<\infty$, and any $S$-tangent tensor $H$ there hold on $S_t$ that
\begin{align}
&\|[\Pz{\mu},{\sz}^2]H\|_{L^p}\les r^{-2}\|H\|_{L^p}, \label{cmm.1}\\
&[\Pz{\mu}, \sz]H\|_{L^p}\les r^{-2+\frac{2}{p}}\mu^{-\frac{2}{p}}\|H\|_{L^2}. \label{cmm.2}
\end{align}
\end{lemma}

\begin{proof}
Symbolically we have $[\sD^{(0)}, {\sz}^2]H= r^{-2} \gaz\c {\sz}^2 H$.
By the definition of $\Pz\mu H$ and the Duhumal principle for heat flow, we have
\begin{align*}
[\Pz\mu, {\sz}^2]H&=r^2\int_0^\infty m_\mu(\tau) \int_0^\tau U(\tau-\tau')[\sD^{(0)}, {\sz}^2]U(\tau') H d\tau'\\
&= \int_0^\infty m_\mu(\tau) \int_0^\tau U(\tau-\tau') {\sz}^2 U(\tau') H d\tau'.
\end{align*}
Taking the $L^p$ norm gives
\begin{align*}
\|[\Pz\mu, {\sz}^2]H\|_{L^p}
&\les r^{-2}\|H\|_{L^p}\int_0^\infty m_\mu(\tau)\int_0^\tau (\tau-\tau')^{-\f12} {\tau'}^{-\f12}d\tau'
&\les r^{-2}\|H\|_{L^p}.
\end{align*}

We next prove (\ref{cmm.2}). In view of
\begin{align*}
[\Pz{\mu}, \sz]H&=r^2\int_0^\infty m_\mu(\tau)\int_0^\tau U(\tau-\tau')[\sD^{(0)}, \sz]U(\tau') H \\
&=\int_0^\infty m_\mu(\tau)\int_0^\tau U(\tau-\tau')\sz U(\tau')H.
\end{align*}
We then derive by the Sobolev inequality and the $L^2$ estimate of heat operator $U(\tau)$  that
\begin{align*}
\|[\Pz{\mu}, \sz]H\|_{L^p}
&\les r^{-2+\frac{2}{p}}\|H\|_{L^2}\int_0^\infty m_\mu(\tau)\int_0^\tau {\tau'}^{-\f12} (\tau-\tau')^{-\f12+\frac{1}{p}} d\tau'\\
&\les r^{-2+\frac{2}{p}}\mu^{-\frac{2}{p}}\|H\|_{L^2}
\end{align*}
as desired.
\end{proof}

Now we divide the proof of (\ref{cz2}) into two steps. The first step is to prove

\begin{lemma}\label{cz1.2}
For $S$-tangent tensor fields $F$ and $G$ satisfying the Hodge system in
Proposition \ref{cz}, there holds with $\delta>0$ and $p>2$ that
\begin{equation*}
\|F\|_{L^\infty(S_{t,u})}\les \|\mu^{\delta}\Pz{\mu} G\|_{l_\mu^1 L^\infty(S_{t,u})}
+\|G\|_{L^\infty(S_{t,u})}+r^{1-\frac{2}{p}}\|e\|_{L^p(S_{t,u})}
\end{equation*}
\end{lemma}

\begin{proof}
Let $\bar F$ be defined by (\ref{barf}). Using  the following  standard identity for
Hodge operators (see \cite[Page 38]{CK})
\begin{equation*}
-\f12 \sD^{(0)}\bar F={}^\ast\D_2^{(0)}\D_2^{(0)}\bar F-\frac{1}{r^2}\bar F,
\end{equation*}
we can obtain
\begin{equation*}
-\f12 \sD^{(0)}\bar F={}^\ast\D_2^{(0)}(\D_2 F+(\Ga-\Gz)F\c\ga)-\frac{1}{r^2}\bar F.
\end{equation*}
We can always find the Littlewood Paley projection $\ti P^{(0)}_\mu$ associated to
a different symbol $\ti m$ such that
\begin{equation}\label{invlp}
\mu^2 P^{(0)}_\mu =r^2 \sD^{(0)} \ti P^{(0)}_\mu
\end{equation}
we can derive symbolically that
\begin{equation}\label{pobarf}
 P_{\mu}^{(0)}\bar F=\mu^{-2}r^2\left(\tPz{\mu}{}^\ast\D_2^{(0)}(\sz G+(\Ga-\Gz)(F\c \ga+G)+e)\right)+\mu^{-2}\tPz{\mu}\bar F.
\end{equation}
In what follows,  we will not distinguish $\ti P_{\mu}^{(0)}$ with $\Pz{\mu}$.
By the Bernstein inequality and the finite band property, for $p>2$ satisfying $0<1-2/p<s-2$, we
can derive on $S$ that
\begin{align*}
I_\mu&=:\mu^{-2}r^2\|\Pz{\mu}\sz\left((\Ga-\Gz)(F\c \ga+G)+e\right)\|_{L^\infty}+\mu^{-2}\|\Pz{\mu} \bar F\|_{L^\infty}\\
&\les\mu^{-2+\frac{2}{p}+1}r^{1-\frac{2}{p}}\| (\Ga-\Gz)(F\c \ga+G)+e\|_{L^p}+\mu^{-2} \|\Pz{\mu}\bar F\|_{L^\infty}\\&\les \mu^{\frac{2}{p}-1}\left(r^{1-\frac{2}{p}}(\|e\|_{L^p}+\|F\|_{L_x^\infty}\|\Ga-\Gz\|_{L^p})+\|G\|_{L^\infty}\right)
\end{align*}
where, with
 the help of the Bernstein inequality and (\ref{lpp2}), the lowest  order term
 $\mu^{-2}\|\Pz{\mu} \bar F\|_{L_x^\infty}$ has been treated as
\begin{align*}
 \|\Pz{\mu} \bar F\|_{L_x^\infty}
 &\les \mu^{\frac{2}{p''}}r^{-\frac{2}{p''}}\| F\|_{L_x^{p''}}
 \les \mu^{\frac{2}{p''}}(\|G\|_{L_\omega^{p''}}+\|r e\|_{L_\omega^q}),
\end{align*}
with $p''>2$ being a large number and $1/q=1/2+1/{p''}$. Therefore, we obtain that
\begin{equation*}
\sum_{\mu>1}I_\mu\les \|r^{1-\frac{2}{p}}|\Ga-\Gz|_\ga\|_{L^p}\|F\|_{L_x^\infty}
+\|G\|_{L_x^\infty}+r^{1-\frac{2}{p}}\|e\|_{L^p}.
\end{equation*}
It remains to consider the first term in (\ref{pobarf}). Using the Bernstein inequality with $q>2$, we have
$
\mu^{-2}r^2\| \Pz{\mu}{}^\ast\D_2^{(0)}\sz G\|_{L^\infty}\les a_\mu+b_\mu
$
where
$$
a_\mu=(\mu^{-1}r)^{2-\frac{2}{q}}\|{}^\ast \D_2^{(0)}\sz \Pz{\mu}G\|_{L^q} \quad \mbox{ and } \quad
b_\mu=(\mu^{-1}r)^{2-\frac{2}{q}}\|[\Pz{\mu}, {\sz}^2]G\|_{L^q}.
$$
By the $L^q$ theory for elliptic equations, the finite band property, and (\ref{cmm.1}), we have
\begin{eqnarray}
a_\mu\les \mu^{\frac{2}{q}}r^{-\frac{2}{q}}\|\Pz{\mu}G\|_{L^q}, &&
b_\mu \les \mu^{-2+\frac{2}{q}}r^{-\frac{2}{q}}\|G\|_{L^q},
\end{eqnarray}
which together with the estimate on $I_\mu$ imply
\begin{align*}
\sum_{\mu>1}\|\Pz{\mu} \bar F\|_{L^\infty}
&\les \sum_{\mu>1}\mu^{\frac{2}{q}}\|\Pz{\mu}G\|_{L_\omega^q}
    +\|r^{1-\frac{2}{p}}|\Ga-\Gz|_{\gaz}\|_{L^p}\|F\|_{L_x^\infty}\\
&+\|G\|_{L_x^\infty}+r^{1-\frac{2}{p}}\|e\|_{L^p}.
\end{align*}
Using Proposition \ref{swsph} and (\ref{ws1}), we therefore complete the proof.
\end{proof}

The second step is to prove Lemma \ref{lemtrans}, for which we need the following result.

\begin{lemma}
For any $p>2$ and any $S$-tangent tensor $F$, there holds on $S:=S_{t,u}$ that
\begin{equation}\label{invbern}
\|\Pz{\mu}F\|_{L^\infty}\les \mu^{-1+\frac{2}{p}}r^{1-\frac{2}{p}}\|\Pz{\mu} \sz F\|_{L^p}+\mu^{-2}\|\sz F\|_{L^2}.
\end{equation}
\end{lemma}

\begin{proof}
With the help of (\ref{invlp}) and the Bernstein inequality, we have
\begin{align*}
\|\Pz{\mu}F\|_{L^\infty}&=\mu^{-2}r^2 \|\sD\tPz{\mu}F\|_{L^\infty}
\les \mu^{-2+\frac{2}{p}}r^{2-\frac{2}{p}}\|\tPz{\mu}\sD^{(0)} F\|_{L^p}\\
&\les \mu^{-2+\frac{2}{p}}r^{2-\frac{2}{p}}
\left(\|[\tPz{\mu}, \sz]\sz F\|_{L^p}+\|\sz\tPz{\mu}\sz F\|_{L^p}\right)\\
&:=\A^{(1)}_\mu+\A^{(2)}_\mu.
\end{align*}
From (\ref{cmm.2}) it follows that
\begin{equation*}
r^{2-\frac{2}{p}}\|[\tPz{\mu}, \sz]\sz F\|_{L^p}\les \mu^{-\frac{2}{p}}\|\sz F\|_{L^2},
\end{equation*}
which  implies
$
\A^{(1)}_\mu\les \mu^{-2} \|\sz F\|_{L^2}.
$
In view of the finite band property, we also have
\begin{equation*}
\A^{(2)}_\mu \les \mu^{-1+\frac{2}{p}}r^{1-\frac{2}{p}}\|\tPz{\mu}\sz F\|_{L^p}.
\end{equation*}
Thus, we complete the proof of (\ref{invbern}).
\end{proof}

\begin{lemma}\label{lemtrans}
(i) For $0<\ep<1$, $\mu>1$ and any scalar function $f$, there holds on $S:=S_{t,u}$ that
\begin{equation}\label{compr}
\mu^\ep\|\Pz\mu f\|_{L^\infty}\les \left(\sum_{\ell \ge \mu} \left(\frac{\mu}{\ell}\right)^\ep
+ \sum_{\ell<\mu} \left(\frac{\ell}{\mu}\right)^{1-\ep}\right)\|\ell^\ep P_\ell f\|_{L^\infty}+\mu^{-1+\ep}\|f\|_{L^\infty}.
\end{equation}

(ii) Let $\delta, \delta'>0$ be such that $\delta+\delta'<1$ and let $q>2$ be such that
$2/q<1-\delta$. Then for any scalar functions $f$ and $H$ and any $\mu>1$ there holds
\begin{align}
\mu^{\delta}\|\Pz{\mu} (f\c H)\|_{L^\infty}
&\les \|H\|_{L^\infty} \left(\sum_{\ell\ge \mu} \left(\frac{\mu}{\ell}\right)^\delta
+ \sum_{\ell<\mu} \left(\frac{\ell}{\mu}\right)^{1-\delta'-\delta}\ell^{\delta'}\right)
\|\ell^\delta P_\ell f\|_{L_x^\infty}\nn\\
&+\mu^{\frac{2}{q}-1+\delta}\|f\|_{L^\infty} \|r\sz H\|_{L^q_\omega}.\label{sumlem}
\end{align}
\end{lemma}
\begin{proof}
We first prove (i). We can write
$$
\Pz\mu f=\Pz\mu \sum_{\ell\ge\mu} P_\ell f+\Pz\mu \sum_{1\le \ell< \mu} P_\ell f
+\Pz\mu P_{\le 1} f=I^{(1)}_\mu+I^{(2)}_\mu+I^{(3)}_\mu.
$$
By using the  properties of the geometric Littlewood-Paley projections, we have on $S_{t,u}$ that
\begin{align*}
\mu^\ep\|I^{(1)}_\mu\|_{L^\infty}
   &\les \sum_{\ell \ge \mu} \left(\frac{\mu}{\ell}\right)^{\ep}\|\ell^\ep P_\ell f\|_{L^\infty},\\
\mu^\ep\|I^{(2)}_\mu\|_{L^\infty}
   &\les \sum_{1\le \ell<\mu}\mu^{-2+\ep}r^2\|\Delta^{(0)} \ti P_\ell f\|_{L^\infty}
   \les \sum_{1\le \ell<\mu} \left(\frac{\ell}{\mu}\right)^{2-\ep} \|\ell^{\ep}\ti P_\ell f\|_{L^\infty},\\
\|I^{(3)}_\mu\|_{L^\infty}&\les \mu^{-2} r^2\|\Delta^{(0)} P_{\le 1} f\|_{L^\infty}\les \mu^{-2} \|f\|_{L^\infty}.
\end{align*}
Combining the above estimates gives (\ref{compr}).

Now we prove (ii).  In view of $\sum_{\ell} P_\ell f= f$ and set $f_\ell:=P_\ell f$,
we can write
\begin{align*}
\Pz{\mu}(f\c H)&=J^{(1)}_\mu+J^{(2)}_\mu+J^{(3)}_\mu\\
&=\sum_{1<\ell<\mu}\Pz{\mu}(f_{\ell}\c H)+\sum_{\ell\ge \mu}\Pz{\mu}(f_{\ell}\c H)+\Pz{\mu}(f_{\le 1}\c H).
\end{align*}
It is straightforward to derive that
\begin{align*}
\mu^{\delta}\| J^{(2)}_\mu\|_{L^\infty}
\les \|H\|_{L^\infty} \sum_{\ell\ge \mu} \left(\frac{\mu}{\ell} \right)^{\delta} \ell^{\delta}\|f_{\ell}\|_{L^\infty}.
\end{align*}
We now estimate $J_\mu^{(1)}$ by using (\ref{invbern}). Let $f_{1<\cdot<\mu}:=\sum_{1<\ell<\mu}f_{\ell}$.
Then for sufficiently large $p$ we have
\begin{align*}
\mu^\delta \|J^{(1)}_\mu\|_{L^\infty}
&\les \mu^{\delta-1+\frac{2}{p}} r^{1-\frac{2}{p}}\|\Pz{\mu} \sz(f_{1<\cdot<\mu}\c H)\|_{L^p}
+\mu^{-2+\delta}\|\sz(f_{1<\cdot<\mu}\c H)\|_{L^2}\\
&\les \mu^{\delta-1+\frac{2}{p}}r^{1-\frac{2}{p}}\|\Pz{\mu}(\sz f_{1<\cdot<\mu} \c H
  +f_{1<\cdot<\mu} \c \sz H)\|_{L^p}\\
&\quad\quad \quad  +\mu^{-2+\delta}\|\sz(f_{1<\cdot<\mu} \c H)\|_{L^2}\\
&\les \mu^{\delta-1+\frac{2}{p}} \|H\|_{L^\infty} \sum_{1<\ell<\mu}\ell\|f_{\ell}\|_{L_\omega^p}
+\mu^{\delta-2}\|f\|_{L^\infty}\|\sz H\|_{L^2}\\
&\quad\quad\quad+\mu^{\delta-1+\frac{2}{p}} r^{1-\frac{2}{p}}\|\Pz{\mu} (f_{1<\cdot<\mu}\c \sz H)\|_{L^p}.
\end{align*}
For the last term we proceed with $q>2$ sufficiently close to $2$ and the Bernstein inequality to obtain
\begin{align*}
\mu^{\delta-1+\frac{2}{p}}&r^{1-\frac{2}{p}}\|\Pz{\mu}(f_{1<\cdot<\mu}\c \sz H)\|_{L^p}\les \mu^{\delta-1+\frac{2}{q}}r^{1-\frac{2}{q}}\|f\|_{L^\infty}\|\sz H\|_{L^q}.
\end{align*}
Thus
\begin{equation*}
\mu^\delta \|J^{(1)}_\mu\|_{L^\infty}
\les \|H\|_{L^\infty} \sum_{1<\ell<\mu} \left(\frac{\ell}{\mu}\right)^{1-\delta-\frac{2}{p}}
\|\ell^{\delta+\frac{2}{p}}f_{\ell}\|_{L_\omega^p}+ \mu^{\delta-1+\frac{2}{q}}
\|f\|_{L^\infty}\|r\sz H\|_{L^q_\omega}
\end{equation*}
The low frequency part
$J^{(3)}_\mu$ can be treated similarly. Hence we obtain (\ref{sumlem}).
\end{proof}

Now we prove (\ref{cz2}). By abuse of notation, we can write $G$ as the sum of two types of terms
$ \ti G_\mu N^\mu$ and $\ti G_\mu$. We may apply (\ref{compr}) to the components of $\ti G_\mu$
to treat the second type of terms. To treat the first type of terms, we may apply (\ref{sumlem}) to
$f=\ti G_{\mu}$ and $H=N^{\mu}$. Note that  $|\sz N|\le |\tr\theta, \hat\theta, g\c \hn g|$
for which we have
\begin{equation*}
  r^{1-\frac{2}{p}}\|\sz N\|_{L^p(S_{t,u})}\les 1,\, \,
  \mbox{ with } 2\le p\le 4.
\end{equation*}
This together with Lemma \ref{lemtrans} and Lemma \ref{cz1.2} implies (\ref{cz2}).

\section{\bf Appendix III: Commutator estimates}\label{apiii}

In this section we derive various commutator estimates involving the LP projections $P_\la$ in fractional
Sobolev spaces that are extensively used in this paper, where $P_\la$ is defined by  (\ref{LP2012}).
One can refer to \cite{Stein2,KRsurf} for various properties of LP projections. In view of the LP decomposition,
the norm in the Sobolev space $H^\ep$ with $0\le \ep<1$ is defined by
$$
\|F\|_{H^\ep} :=\|F\|_{L^2} + \left(\sum_{\la> 1} \la^{2\ep} \|P_\la F\|_{L^2}^2 \right)^{1/2}
$$
for any scalar function $F$. For any nonnegative integer $m$ and $0\le \ep <1$, we define 
$\|F\|_{H^{m+\ep}} :=\|F\|_{H^m} +\|\hn^m F\|_{H^\ep}$. 
For simplicity of exposition, we will write  $F_\la:=P_\la F$,
$F_{\le\la} : =\sum_{\mu\le \la} P_\mu F$, and  
$\|\La^r F\|_{L^2}: = \left(\sum_{\la> 1} \la^{2r} \|P_\la F\|_{L^2}^2 \right)^{1/2}$
. For any sequence $(a_\la)$ we will use $\|a_\la\|_{L_\la^2}^2$
to denote $\sum_{\la\ge 1} |a_\la|^2$. 

\subsection{Product estimates}

We first derive some useful product estimates. Let $P_\mu$ denote the LP projection.
According to the Littlewood-Paley (LP) decomposition, one has the trichotomy law which schematically says that
for any scalar functions $F$ and $G$ there holds
\begin{equation}\label{trico}
P_\mu(F\c G)=P_\mu(F_{\le \mu} \c G_\mu)+P_\mu(F_\mu\c G_{\le \mu}) +\sum_{\la>\mu}P_\mu(F_\la\c G_\la).
\end{equation}
We will use this decomposition repeatedly.

\begin{lemma}\label{lem3}
For any $0<\ep <1$ and any scalar functions $F$ and $G$ there hold
\begin{align}
&\|\La^\ep(F\c G)\|_{L^2}  \les \| F\|_{H^{1/2+\ep}}\|G\|_{H^1}+\|G\|_{H^{1/2+\ep}}\|F\|_{H^1},\label{simlp2}\\
&\|\mu^{-1/2+\ep}P_\mu (F\c G)\|_{l_\mu^2 L^2}\les \|G\|_{H^\ep}\|F\|_{H^1}. \label{ebeq.1}
\end{align}
\end{lemma}

\begin{proof}
We first prove (\ref{simlp2}). By using the Bernstein inequality and the finite
band property of the LP projections, we have
\begin{align*}
\mu^{\ep}& \|P_\mu(F_{\le \mu} \c G_\mu)\|_{L^2}\\
&\les \mu^\ep \sum_{\la\le \mu} \|F_\la\|_{L^\infty}\|G_\mu\|_{L^2}
\les \sum_{\la\le \mu} \left(\frac{\la}{\mu}\right)^{1-\ep} \|\la^{1/2+\ep} F_\la\|_{L^2}\|\hn G_\mu\|_{L^2}.
\end{align*}
Therefore
\begin{equation}\label{prd1}
\|\mu^{\ep} P_\mu(F_{\le \mu} \c G_\mu)\|_{l_\mu^2 L_x^2}
\les \| F\|_{H^{1/2+\ep}}\|\hn G\|_{L^2}.
\end{equation}
Similarly we have
\begin{equation}\label{prd1.5}
\|\mu^{\ep} P_\mu(F_\mu \c G_{\le \mu})\|_{l_\mu^2 L^2}
\les \| G\|_{H^{1/2+\ep}}\|\hn F\|_{L^2}.
\end{equation}
Moreover, we have
\begin{align*}
\mu^\ep\|P_\mu( F_\la\c G_\la)\|_{L^2}
&\les \mu^{1/2+\ep} \| F_\la\c G_\la\|_{L^{\frac{3}{2}}} \les \mu^{1/2+\ep}\|F_\la\|_{L^2}\|G_\la\|_{L^6}\\
&\les \left(\frac{\mu}{\la}\right)^{1/2+\ep} \|\la^{1/2+\ep} F_\la\|_{L^2}\|\hn G_\la\|_{L^2}.
\end{align*}
This implies that
\begin{equation}\label{prd2}
\left\|\mu^\ep\sum_{\la>\mu} P_\mu(F_\la\c G_\la)\right\|_{l_\mu^2 L^2}
\les \|\La^{1/2+\ep} F\|_{L^2}\|\hn G\|_{L^2}.
\end{equation}
Combining the above estimates and using the trichotomy law we obtain (\ref{simlp2}).

Next we prove (\ref{ebeq.1}) by using again the trichotomy. We have
\begin{align*}
&\|\mu^{-1/2+\ep} P_\mu (F_{\le \mu} \cdot G_\mu)\|_{l_\mu^2 L^2}
\les \left\|\mu^{-1/2+\ep} \|G_\mu\|_{L^2} \|F_{\le \mu} \|_{L^\infty} \right\|_{l_\mu^2}\\
&\qquad \qquad \les \left\|\sum_{\la\le \mu} \left(\frac{\la}{\mu}\right)^{1/2} \|\mu^\ep G_\mu\|_{L^2} \|\hn F_\la\|_{L^2}\right\|_{l_\mu^2}
\les \|F\|_{H^1} \|G\|_{H^\ep},
\end{align*}
\begin{align*}
&\|\mu^{-1/2+\ep} P_\mu (F_\mu\cdot G_{\le \mu})\|_{l_\mu^2 L^2}
\les \left\|\mu^{-1/2+\ep} \|F_\mu\|_{L^2} \|G_{\le \mu}\|_{L^\infty} \right\|_{l_\mu^2}\\
&\qquad \qquad \les \left\|\sum_{\la\le \mu} \left(\frac{\la}{\mu}\right)^{3/2-\ep} \|\hn F\|_{L^2} \|\la^\ep G_\la\|_{L^2} \right\|_{l_\mu^2}
\les \|\hn F\|_{L^2} \|G\|_{H^\ep}
\end{align*}
and
\begin{align*}
&\left\|\sum_{\la>\mu} \mu^{-1/2+\ep} P_\mu(F_\la\cdot G_\la)\right\|_{l_\mu^2 L^2} \\
& \quad \qquad \les \left\|\mu^{1/2+\ep} \sum_{\la>\mu} \|F_\la\cdot G_\la\|_{L^{6/5}}\right\|_{l_\mu^2}
\les \left\|\mu^{1/2+\ep} \sum_{\la\ge \mu}  \|F_\la\|_{L^3} \|G_\la\|_{L^2} \right\|_{l_\mu^2}\\
&\quad \qquad \les \left\|\sum_{\la\ge \mu} \left(\frac{\mu}{\la}\right)^{1/2+\ep} \|\hn F_\la\|_{L^2} \|\la^\ep G_\la\|_{L^2} \right\|_{l_\mu^2} \les \|F\|_{H^1} \|G\|_{H^\ep}.
\end{align*}
Combining the above estimates yields (\ref{ebeq.1}).
\end{proof}

\begin{lemma}\label{epeq2}
For any $\ep>0$ and any scalar functions $G_1$, $G_2$ and $G_3$ there holds
\begin{equation*}
\left\|\La^\ep (G_1 G_2 G_3)\right\|_{L^2}\les\sum_{j=1}^3 \left(\| G_j\|_{H^{1+\ep}}
\prod_{l\neq j}\|G_l\|_{H^1}\right).
\end{equation*}
\end{lemma}

\begin{proof}
By using the properties of LP projections and the estimate (\ref{ebeq.1}) in Lemma \ref{lem3}
we have
\begin{align*}
&\|\mu^\ep P_\mu \left((G_1)_{\le \mu} (G_2 G_3)_\mu\right) \|_{l_\mu^2 L^2}
\les \left\|\sum_{\la\le \mu} \mu^\ep \|(G_1)_\la\|_{L^\infty} \|(G_2 G_3)_\mu\|_{L^2} \right\|_{l_\mu^2} \\
& \les \left\|\sum_{\la\le \mu} \mu^\ep \la^{1/2} \|\hn (G_1)_\la\|_{L^2} \|(G_2 G_3)_\mu\|_{L^2} \right\|_{l_\mu^2}
\les \|\hn G_1\|_{L^2}  \left\|\mu^{-1/2+\ep} P_\mu \hn (G_2 G_3)\right\|_{l_\mu^2 L^2} \\
&\les \|G_1\|_{H^1} \left(\| G_2\|_{H^{1+\ep}}\|G_3\|_{H^1}+\|G_3\|_{H^{1+\ep}}\|G_2\|_{H^1} \right),
\end{align*}
\begin{align*}
\left\|\mu^\ep P_\mu\left((G_1)_\mu (G_2 G_3)_{\le \mu}\right) \right\|_{l_\mu^2 L^2}
& \les \left\|\mu^\ep \|(G_1)_\mu\|_{L^3} \|(G_2 G_3)_{\le \mu} \|_{L^6}\right\|_{l_\mu^2}\\
& \les \left\|\mu^{\ep-1/2}\|\hn (G_1)_\mu\|_{L^3}\right\|_{l_\mu^2} \|G_2 G_3\|_{L^3}\\
&\les \|\mu^\ep \hn (G_1)_\mu\|_{l_\mu^2 L^2} \|G_2\|_{L^6} \|G_3\|_{L^6}\\
&\les \| G_1\|_{H^{1+\ep}} \|G_2\|_{H^1} \|G_3\|_{H^1}
\end{align*}
and
\begin{align*}
\left\|\sum_{\la>\mu} \mu^\ep P_\mu \left((G_1)_\la (G_2 G_3)_\la\right) \right\|_{l_\mu^2 L^2}
& \les \left\|\sum_{\la>\mu} \mu^\ep \|(G_1)_\la\|_{L^6} \|G_2 G_3\|_{L^3} \right\|_{l_\mu^2}\\
& \les \left\|\sum_{\la>\mu} \left(\frac{\mu}{\la}\right)^\ep \|\la^\ep \hn (G_1)_\la\|_{L^2} \right\|_{l_\mu^2}
\|G_2\|_{H^1} \|G_3\|_{H^1}\\
&\les \|G_1\|_{H^{1+\ep}} \|G_2\|_{H^1}\|G_3\|_{H^1}.
\end{align*}
In view of the trichotomy and the above estimates, we thus complete the proof.
\end{proof}

\begin{lemma}\label{epp1}
For any $0<\ep<1$ there hold
\begin{align}
&\|\La^\ep(F\c \hn G)\|_{L^2}\les \|F\|_{L^\infty}\|G\|_{H^{1+\ep}}
+\|G\|_{L^\infty}\| F\|_{H^{1+\ep}},\label{epeq1}
\end{align}
\end{lemma}

\begin{proof}
We will use the trichotomy law to derive the estimates. To derive (\ref{epeq1}), we
use the properties of the LP projections to obtain
\begin{align*}
\|\mu^\ep P_\mu (F_{\le\mu} \cdot \hn G_\mu)\|_{l_\mu^2 L^2}
&\les \|F\|_{L^\infty} \|\mu^{\ep}(\hn G)_\mu\|_{l_\mu^2 L^2}
\les \|F\|_{L^\infty} \|\La^\ep \hn G\|_{L^2},
\end{align*}
\begin{align*}
\|\mu^\ep P_\mu (F_\mu \cdot \hn G_{\le \mu})\|_{l_\mu^2 L^2}
&\les \left\|\mu^\ep \|F_\mu\|_{L^2} \sum_{\la <\mu} \la \|G_\la\|_{L^\infty}\right\|_{l_\mu^2}
\les  \|\mu^{1+\ep} F_\mu\|_{l_\mu^2 L^2}  \|G\|_{L^\infty} \\
&\les \|\La^\ep \hn F\|_{L^2} \|G\|_{L^\infty}
\end{align*}
and
\begin{align*}
\left\|\mu^{\ep} \sum_{\la>\mu} P_\mu(F_\la \cdot \hn G_\la)\right\|_{l_\mu^2 L^2}
&\les \|F\|_{L^\infty} \left\| \sum_{\la>\mu} \left(\frac{\mu}{\la}\right)^\ep \|\la^\ep \hn G_\la\|_{L^2}\right\|_{l_\mu^2}\\
&\les \|F\|_{L^\infty} \|\La^\ep \hn G\|_{L^2}.
\end{align*}
Combining the above three estimates we therefore obtain (\ref{epeq1}).
\end{proof}


\subsection{Commutator estimates}

In this subsection we will derive various estimates related to the commutators
$[P_\mu, F] G$. We first consider the general setting. Let $m(\xi)$ define a multiplier
\begin{equation}\label{multiplier1}
\mathcal{P} f(x)=\int e^{i x\xi} m(\xi) \hat f(\xi) d\xi.
\end{equation}
By introducing the function $P(x)$ defined by
\begin{equation}\label{multiplier2}
M(x)=\int e^{ix\cdot \xi} m(\xi) d\xi=\hat{m}(-x),
\end{equation}
then for any scalar functions $F$ and $G$ we can write
\begin{align}\label{expression}
[\P, F]G(x)&=\int M(x-y) (F(y)-F(x))G(y) dy \nonumber\\
&=\int M(x-y)(x-y)^j \int_0^1 \p_j F(\tau y+(1-\tau)x)d\tau G(y) dy \nonumber\\
&=\int M(h) h^j \int_0^1 \p_j F(x-\tau h) d\tau G(x-h) dh.
\end{align}
By taking the $L^q$-norm with $1\le q\le \infty$ and using the Minkowski inequality we obtain
\begin{equation}\label{commutator}
\|[\P, F]G\|_{L^q} \le \int |M(h)| |h| \int_0^1 \|\p F(\cdot-\tau h) G(\cdot-h)\|_{L^q} d\tau d h
\end{equation}
An application of the H\"{o}lder inequality gives the following result whose special case with $p=\infty$ and
$q=r=2$ is \cite[Lemma 8.2]{KRduke}.

\begin{lemma}\label{basecom}
Let $\P$ be the multiplier operator defined by (\ref{multiplier1}) and let $M$ be the function given by
(\ref{multiplier2}). Then, for any $1\le p, q, r\le \infty$ satisfying $1/p+1/r=1/q$ and any scalar
functions $F$ and $G$, there holds
\begin{equation*}
\|[\mathcal{P}, F]G\|_{L^q}\le \|\p F\|_{L^p}\|G\|_{L^r}\int |x||M(x)| dx.
\end{equation*}
\end{lemma}

Recall that the LP projection $P_\mu$ is a multiplier operator with $m(\xi)=\Psi(\mu^{-1}\xi)$, where $\Psi$ is
a mollifier with support on $\{1/2<|\xi|<2\}$. Observe that $M(x)=\mu^3 \widehat{\Psi}(-\mu x)$.
We have $\int |x| |M(x)| dx \les \mu^{-1}$. Therefore, from Lemma \ref{basecom} we obtain the following commutator estimate.

\begin{corollary}\label{base2com}
For any $1\le p, q, r\le \infty$ satisfying $1/p+1/r=1/q$ and any scalar functions $F$ and $G$ there holds
\begin{equation*}
\|[P_\mu, F]G\|_{L^q}\les \mu^{-1}\|\hn F\|_{L^p} \|G\|_{L^r}.
\end{equation*}
\end{corollary}

In the following we will give further estimates related to the commutator $[P_\mu, F]G$
for any scalar functions $F$ and $G$. We can write
$$
[P_\mu, F] G=[P_\mu, F] G_{\le 2\mu} +[P_\mu, F] G_{>2\mu}.
$$
By the orthogonality of the LP projections, we have
\begin{align*}
[P_\mu, F]G_{>2\mu} &= P_\mu (F\cdot G_{>2\mu})-F P_\mu G_{>2\mu}=\sum_{\mu_1>2\mu} P_\mu (F\cdot G_{\mu_1})\\
&=\sum_{\mu_1>2\mu} \sum_{\frac{\mu_1}{2}\le \mu_2 \le 2 \mu_1} P_\mu (F_{\mu_2}\cdot G_{\mu_1}).
\end{align*}
Thus, schematically we can write
\begin{align}\label{bdcp}
[P_\mu, F]G=[P_\mu, F] G_{\le \mu} + \sum_{\la>\mu} P_\mu (F_\la G_\la)
\end{align}
which is not quite accurate but harmless to derive estimates.

\begin{lemma}\label{lem1}
For $0<\ep<1$, there holds
\begin{align}
\mu^{\ep}\|[P_\mu, F]\hn G\|_{L^2}
&\les \|\hn F\|_{L^\infty} \sum_{\la\le \mu} \left(\frac{\la}{\mu}\right)^{1-\ep} \|\la^\ep G_\la\|_{L^2} \nonumber\\
&+ \|\hn F\|_{L^\infty} \sum_{\la>\mu} \left(\frac{\mu}{\la}\right)^\ep \|\la^\ep G_\la\|_{L^2}. \label{com3}
\end{align}
\end{lemma}

\begin{proof}
By using Corollary \ref{base2com} we have
\begin{align*}
 \|[P_\mu, F] \hn G_{\le \mu} \|_{L^2}
& \les  \mu^{-1} \|\hn F\|_{L^\infty} \sum_{\la\le \mu} \|\hn G_\la\|_{L^2}
\les \mu^{-1} \|\hn F\|_{L^\infty} \sum_{\la \le \mu} \la \|G_\la\|_{L^2}.
\end{align*}
On the other hand, by using the properties of the LP projections we have
\begin{align*}
\sum_{\la>\mu} \|P_\mu (F_\la \cdot \hn G_\la)\|_{L^2}
& \les \sum_{\lambda>\mu} \|F_\la\|_{L^\infty} \|\hn G_\la\|_{L^2}
\les \|\hn F\|_{L^\infty} \sum_{\la>\mu} \|G_\la\|_{L^2}.
\end{align*}
Combining these two estimates and using the decomposition (\ref{bdcp}) we therefore complete the proof.
\end{proof}

\begin{lemma}\label{lem2}
For $0<\ep<1/2$ and $\mu\ge 1$ there hold
\begin{align}
\mu^{-\f12+\ep}\|\hn [P_\mu, F]G\|_{L^2}
&\les \|\hn F\|_{L^6} \sum_{\la\le \mu} \left(\frac{\la}{\mu}\right)^{1/2-\ep} \|\la^\ep G_\la\|_{L^2} \nonumber\\
&+ \|\hn F\|_{L^6} \sum_{\la>\mu} \left(\frac{\mu}{\la}\right)^{1/2+\ep} \|\la^\ep G_\la\|_{L^2} \label{lem2eq}
\end{align}
and
\begin{align}
\mu^{\f12+\ep} \|[P_\mu, F]G\|_{L^2}
& \les \|\hn F\|_{L^6} \sum_{\la\le \mu} \left(\frac{\la}{\mu}\right)^{1/2-\ep} \|\la^\ep G_\la\|_{L^2} \nonumber\\
& +\|\hn F\|_{L^6} \sum_{\la\ge \mu} \left(\frac{\mu}{\la}\right)^{1+\ep} \|\la^\ep G_\la\|_{L^2}.\label{lem4eq}
\end{align}
\end{lemma}

\begin{proof}
From the expression (\ref{expression}) we can obtain
\begin{align}
\hn [P_\mu, F] G_{\le \mu} &=\int \hn M_\mu(x-y) (x-y)^j \int_0^1 \p_j F(\tau y+(1-\tau) x) d\tau G_{\le \mu}(y) dy\nn\\
&+ \int M_\mu (x-y) \hn F(x) G_{\le \mu}(y) dy,\label{bdcp2}
\end{align}
where $M_\mu$ is given by (\ref{multiplier2}) with $m(\xi)=\Psi(\mu^{-1} \xi)$. One can check
$\int |M_\mu (x)| dx\les 1$ and $\int |x||\hn M_\mu(x)| dx \les 1$. Thus, it follows from the Minkowski inequality that
\begin{equation*}
\|\hn [P_\mu, F] G_{\le \mu} \|_{L^2} \les \|\hn F\|_{L^6} \|G_{\le \mu} \|_{L^3}
\les \|\hn F\|_{L^6} \sum_{\la\le \mu} \la^{1/2} \|G_\la\|_{L^2}.
\end{equation*}
On the other hand, by the finite band property and the Bernstein inequality of LP projections, we have
\begin{align*}
\sum_{\la> \mu} \|\hn P_\mu (F_\la\cdot G_\la)\|_{L^2}
& \les \mu \sum_{\la>\mu} \|F_\la\|_{L^6} \|G_\la\|_{L^3} \les \mu \|\hn F\|_{L^6} \sum_{\la>\mu} \la^{-1/2} \|G_\la\|_{L^2}.
\end{align*}
With the help of the decomposition (\ref{bdcp}) we thus complete the proof of (\ref{lem2eq}).

Next we prove (\ref{lem4eq}). By using Corollary \ref{base2com} we have
\begin{align*}
\|\mu^{\f12+\ep} [P_\mu, F]G_{\le \mu}\|_{L^2}
& \les \mu^{-\f12+\ep} \|\hn F\|_{L^6} \|G_{\le \mu}\|_{L^3}
 \les \mu^{-\f12+\ep} \|\hn F\|_{L^6} \sum_{\la \le \mu} \la^{\f12} \|G_\la\|_{L^2},
\end{align*}
while by using the properties of the LP projections we have
\begin{align*}
\mu^{\f12+\ep} \sum_{\la>\mu} \|P_\mu(F_\la\cdot G_\la)\|_{L^2}
& \les \mu^{1+\ep} \sum_{\la>\mu} \|F_\la \cdot G_\la\|_{L^\frac{3}{2}}
 \les \mu^{1+\ep} \|\hn F\|_{L^6} \sum_{\la>\mu} \la^{-1} \|G_\la\|_{L^2}.
\end{align*}
The estimate (\ref{lem4eq}) thus follows from an application of (\ref{bdcp}).
\end{proof}

\begin{lemma}\label{Lem3.7}
For $0<\ep <3/2$ there holds
\begin{equation}\label{commt8}
\|\mu^{-1/2+\ep} [P_\mu, F]\hn G\|_{l_\mu^2 L_x^2}\les \|\hn F\|_{H^{1}}\|G\|_{H^\ep}.
\end{equation}
\end{lemma}

\begin{proof}
We will use the decomposition (\ref{bdcp}). We first consider the term
$$
a_\mu:=\sum_{\la>\mu} \mu^{-1/2+\ep} P_\mu(F_\la \cdot \hn G_\la).
$$
By the Bernstein inequality and the finite band property of LP projections, we have
\begin{align*}
\mu^{-1/2+\ep} \|P_\mu (F_\la \cdot \hn G_\la)\|_{L^2}
&\les \mu^\ep \|F_\la \cdot \hn G_\la\|_{L^{3/2}} \les \mu^\ep \|F_\la\|_{L^6} \|\hn G_\la\|_{L^2} \\
&\les \mu^\ep \|\hn F_\la\|_{L^6} \|G_\la\|_{L^2}
\les \left(\frac{\mu}{\la}\right)^\ep \|\la^\ep G_\la\|_{L^2} \|\hn F\|_{H^1}.
\end{align*}
Therefore
\begin{equation}\label{2ndt1}
\|a_\mu \|_{l_\mu^2L^2}\les \|\hn F\|_{H^1} \|\La^\ep G\|_{L^2}.
\end{equation}

Next we consider the term $\mu^{-1/2+\ep} [P_\mu, F] \hn G_{\le \mu}$. By using (\ref{commutator}) and
setting $F_{h, \tau} (x)=F(x-\tau h)$, we obtain
\begin{equation*}
\|[P_\mu, F] \hn G_{\le \mu} \|_{L^2} \les \mu^{-1} \sup_{h, \tau}
\|\hn  F_{h, \tau} \cdot \hn G_{\le \mu} \|_{L^2}.
\end{equation*}
By using the orthogonality of the LP projections, we can write
$\hn F_{h, \tau} \cdot \hn G_{\le \mu} =b_\mu+ c_\mu$, where
\begin{align*}
b_\mu &=\sum_{\la > \mu} P_\la ((\hn F_{h, \tau})_{\la} \cdot \hn G_{\le \mu})\quad \mbox{and}\quad
c_\mu = \sum_{\la\le \mu} P_\la (\hn F_{h,\tau} \cdot \hn G_{\le \mu}).
\end{align*}
We first have from the H\"{o}lder inequality, the finite band property and the Bernstein inequality that
\begin{align*}
\|b_\mu\|_{L^2} &\les \sum_{\la>\mu} \|(\hn F_{h, \tau})_\la\|_{L^2} \|\hn G_{\le \mu}\|_{L^\infty}
 \les \sum_{\la>\mu, \la'\le \mu} \la^{-1} \la'^{5/2} \|\hn F_{h,\tau}\|_{H^1} \|G_{\la'}\|_{L^2} \\
&\les \|\hn F\|_{H^1} \sum_{\la'\le \mu} \mu^{-1} \la'^{5/2} \|G_{\la'}\|_{L^2}
\end{align*}
Therefore
\begin{equation}\label{8.12.1}
\|\mu^{-3/2+\ep} b_\mu\|_{l_\mu^2 L^2}\les \|\hn F\|_{H^1}\|G\|_{H^\ep}.
\end{equation}

Next we consider $c_\mu$. By using the orthogonality of the LP projections we can write
\begin{align*}
c_\mu & =\sum_{\la\le\mu} \sum_{\la'\le \la} P_\la\left( (\hn F_{h,\tau})_\la\cdot \hn G_{\la'}
+(\hn F_{h,\tau})_{\la'}\c\hn G_{\la}\right) \\
& \quad \, + \sum_{\la\le \mu} \sum_{\la<\la'\le \mu} P_\la \left((\hn F_{h,\tau})_{\la'}\cdot \hn G_{\la'}\right).
\end{align*}
By using the Bernstein inequality and the finite band property, we obtain
\begin{align*}
\|c_\mu\|_{L^2} &\les \sum_{\la\le\mu} \sum_{\la'\le\la}
\left( \la^{-1} \|(\hn F_{h, \tau})_\la\|_{H^1} \|\hn G_{{\la}'}\|_{L^\infty}
+{\la'}^{\f12}\|\hn (\hn F_{h,\tau})_{\la'}\|_{L_x^2}\|\hn G_{\la} \|_{L_x^2} \right)\\
&\qquad \qquad \qquad + \sum_{\la\le \mu} \sum_{\la<\la'\le \mu} \la \|(\hn F_{h,\tau})_{\la'}\|_{L^2} \|\hn G_{\la'}\|_{L^3} \\
&\les \|\hn F\|_{H^1} \sum_{\la\le \mu} \sum_{\la'\le\la\le \mu} (\la^{-1} \la'^{5/2}+{\la'}^{\frac{1}{2}}\la) \|G_{\la'}\|_{L^2}\\
&\qquad \qquad \qquad + \|\hn F\|_{H^1} \sum_{\la\le \mu} \sum_{\la<\la'\le \mu} \la \la'^{1/2} \|G_{\la'}\|_{L^2} \\
&\les \|\hn F\|_{H^1}\c \left(\sum_{\la\le \mu}\sum_{\la'\le \mu}\la'^{1/2} \la \|G_{\la'}\|_{L^2}\right).
\end{align*}
Therefore
\begin{align}\label{8.12.2}
 \|\mu^{-3/2+\ep} c_\mu\|_{l_\mu^2 L^2} \les \|\hn F\|_{H^1} \| G\|_{H^\ep}.
\end{align}
Combining (\ref{8.12.1}) and (\ref{8.12.2}) yields 
$$
\|\mu^{-1/2+\ep} [P_\mu, F] \hn G_{\le \mu}\|_{l_\mu^2 L^2}
\les \|\hn F\|_{H^1} \|G\|_{H^\ep}
$$ 
which together with (\ref{2ndt1}) gives the desired estimate.
\end{proof}

\begin{lemma}
For $0<\ep<1$ there holds
\begin{align}
&\|\mu^{1+\ep}[P_\mu, F]G\|_{l_\mu^2 L^2}+\|\mu^\ep\hn[P_\mu, F]G\|_{l_\mu^2 L^2}
\les \|\hn^2 F\|_{H^{\f12+\ep}}\|G\|_{L^2}+\|\hn F\|_{L^\infty} \|G\|_{H^\ep}\label{epeq3}\\
&\|\mu^{1+\ep} [P_\mu, F] G\|_{l_\mu^2 L^2}\les (\|\hn^2 F\|_{H^{\f12}}+\|\hn F\|_{L^\infty})\|G\|_{H^\ep}\label{epnew}
\end{align}
\end{lemma}

\begin{proof}
We first write $\mu^{1+\ep} [P_\mu, F] G=a_\mu+b_\mu+c_\mu$, where
$$
a_\mu=\mu^{1+\ep} [P_\mu, F] G_{>\mu},
\quad b_\mu =\mu^{1+\ep} \sum_{\la\le \mu} [P_\mu, F_{\le \la}] G_\la,
\quad c_\mu=\mu^{1+\ep} \sum_{\la\le \mu} [P_\mu, F_{>\la}] G_\la.
$$
By Corollary \ref{base2com}, we obtain
\begin{equation*}
\|a_\mu\|_{l_\mu^2 L^2}\les \|\hn F\|_{L^\infty}
\left\|\sum_{\la>\mu} \left(\frac{\mu}{\la}\right)^\ep \|\la^\ep G_\la\|_{L^2}\right\|_{l_\mu^2}
\les \|\hn F\|_{L^\infty} \|\La^\ep G\|_{L^2}.
\end{equation*}

In order to estimate $b_\mu$ and $c_\mu$, we introduce the function $M_\mu(x)=\int e^{ix\cdot \xi} m_\mu(\xi) d\xi$
with $m_\mu(\xi)=\Psi(\mu^{-1}\xi)$. It is easy to see that $\int |x|^q |M_\mu(x)| dx\les \mu^{-q}$ for any $q>-3$.
It follows from (\ref{expression}) that
\begin{align*}
[P_\mu, F_{\le \la}]G_\la(x)=A_{\mu, \la}(x) +B_{\mu, \la}(x),
\end{align*}
where
\begin{align*}
A_{\mu, \la}(x) &=\p_j F_{\le \la} (x) \int M_\mu(x-y)(x-y)^j G_\la (y) dy,\\
B_{\mu, \la}(x) &= \int M_\mu (x-y)(x-y)^j \int_0^1 \left[\p_j F_{\le\la} (x-\tau(x-y))-\p_j F_{\le\la}(x) \right] d\tau G_\la (y) dy.
\end{align*}
For the term $A_{\mu, \la}(x)$, it is nonzero only if $\mu$ and $\la$ are at the same magnitude since both
$M_\mu$ and $G_\la$ are frequency localized at the level $\mu$ and $\la$ respectively. Thus
$$
\sum_{\la\le \mu} \|A_{\mu, \la}\|_{L^2} \les \mu^{-1} \|\hn F_{\le \mu}\|_{L^\infty} \|G_\mu\|_{L^2}
\les \mu^{-1} \|G_\mu\|_{L^2} \|\hn F\|_{L^\infty}.
$$
For the term $B_{\mu,\la}$ we write
\begin{align*}
B_{\mu,\la}(x)&=\int M_\mu(x-y)(x-y)^j (x-y)^l \\
&\times \int_0^1 \int_0^1 -\tau \p_l\p_j F_{\le\la}\left(x-\tau\tau'(x-y)\right) d\tau d\tau' G_\la(y) dy.
\end{align*}
Thus, by the Minkowski inequality we obtain
\begin{align*}
\|B_{\mu, \la}\|_{L^2} &\les \mu^{-2} \|\hn^2 F_{\le \la}\|_{L^\infty}\|G_\la\|_{L^2}
\les \mu^{-2} \sum_{\la'\le\la} \la' \|\hn F_{\la'}\|_{L^\infty} \|G_\la\|_{L^2} \\
&\les \mu^{-2} \la \|G_\la\|_{L^2} \|\hn F\|_{L^\infty}.
\end{align*}
We therefore obtain
$$
\|b_\mu\|_{L^2}\les \mu^\ep \| G_\mu\|_{L^2} \|\hn F\|_{L^\infty}
+\sum_{\la\le \mu} \left(\frac{\la}{\mu}\right)^{1-\ep} \|\la^\ep G_\la\|_{L^2} \|\hn F\|_{L^\infty}
$$
which implies that
$$
\|b_\mu\|_{l_\mu^2 L^2} \les \|\hn F\|_{L^\infty} \|G\|_{H^\ep}.
$$

In order to estimate the term $c_\mu$, we write
\begin{align*}
[&P_\mu, F_{>\la}]G_\la(x)\\
&=\int M_\mu(x-y)(x-y)^j \int_0^1 \{\p_j F_{>\la}(x-\tau(x-y))-\p_j F_{>\la}(x) \}d \tau G_{\la}(y) dy\\
&+\p_j F_{>\la}(x) \int M_\mu(x-y)(x-y)^j G_{\la} (y) dy.
\end{align*}
By using the similar argument as above, we can obtain
\begin{align*}
\sum_{\la\le\mu} & \|[P_\mu, F_{>\la}] G_\la\|_{L^2} \\
&\les \mu^{-2} \sum_{\la\le\mu} \|\hn^2 F_{>\la}\|_{L^2} \|G_\la\|_{L^\infty} +\mu^{-1} \|G_\mu\|_{L^2} \|\hn F\|_{L^\infty}\\
&\les \sum_{\la\le\mu} \sum_{\la'>\la} \mu^{-2} \la^{3/2} \|\hn^2 F_{\la'}\|_{L^2} \|G_\la\|_{L^2}
+\mu^{-1} \|G_\mu\|_{L^2} \|\hn F\|_{L^\infty}.
\end{align*}
Therefore
\begin{align*}
\|c_\mu\|_{L^2} & \les \sum_{\la\le \mu} \sum_{\la'>\la} \left(\frac{\la}{\mu}\right)^{1-\ep} \left(\frac{\la}{\la'}\right)^{1/2+\ep}
\|\la'^{1/2+\ep} \hn^2 F_{\la'}\|_{L^2} \|G_\la\|_{L^2} \\
& \quad \, +\mu^\ep \|G_\mu\|_{L^2} \|\hn F\|_{L^\infty}
\end{align*}
which implies that $\|c_\mu\|_{l_\mu^2 L^2} \les \|\hn F\|_{L^\infty} \|G\|_{H^\ep} +\|\hn^2 F\|_{H^{1/2+\ep}} \|G\|_{L^2}$.
This together with the estimates on $a_\mu$ and $b_\mu$ gives the first inequality of (\ref{epeq3}).
Note that $c_\mu$ can also be estimated as
\begin{align}
\|c_\mu\|_{L^2}&\les\sum_{\la'>\la} \left(\frac{\la}{\mu}\right)^{1-\ep} \left(\frac{\la}{\la'}\right)^{1/2}
\|\la'^{1/2} \hn^2 F_{\la'}\|_{L^2} \|\la^\ep G_\la\|_{L^2}\nn \\
& \quad \, +\mu^\ep \|G_\mu\|_{L^2} \|\hn F\|_{L^\infty}\nn
\end{align}
which implies $\|c_\mu\|_{l_\mu^2 L^2}\les (\|\hn F\|_{L^\infty}+\|\hn^2 F\|_{H^{\f12}})\|\La^\ep G\|_{L^2}$.
This together with the estimates for $a_\mu$ and $b_\mu$ gives (\ref{epnew}).

Next we prove the second part of (\ref{epeq3}). By using (\ref{bdcp}) we can write
$\mu^\ep \hn[P_\mu, F] G =I_\mu+J_\mu$, where
\begin{eqnarray*}
I_\mu:=\mu^\ep \sum_{\la>\mu} \hn P_\mu (F_\la \c G_\la) \quad \mbox{and} \quad J_\mu:=\mu^\ep \hn [P_\mu, F] G_{\le \mu}.
\end{eqnarray*}
We first have from the finite band property of the LP projections that
$$
\|I_\mu\|_{L^2} \les \mu^{1+\ep} \sum_{\la>\mu} \|F_\la\cdot G_\la\|_{L^2}
\les \|\hn F\|_{L^\infty} \sum_{\la>\mu} \left(\frac{\mu}{\la}\right)^{1+\ep} \|\la^\ep G_\la\|_{L^2}.
$$
This gives
\begin{equation*}
\|I_\mu \|_{l_\mu^2 L^2} \les \|\hn F\|_{L^\infty}\|\La^\ep G\|_{L^2}.
\end{equation*}

In order to estimate $J_\mu$, we write
\begin{align*}
J_\mu(x) &= \mu^\ep \int \hn_x \left(M_\mu(x-y) (F(x)-F(y))\right) G_{\le \mu}(y) dy\\
&=\mu^\ep \hn F(x) \int M_\mu(x-y) G_{\le \mu}(y) dy\\
&\quad\, +\mu^\ep \int \hn M_\mu(x-y) (F(x)-F(y)) G_{\le \mu}(y) dy.
\end{align*}
By writing $F(x)-F(y)=(x-y)^j \int_0^1 \p_j F(x-\tau(x-y)) d\tau$, we can decompose $J_\mu$ as
$J_\mu=J_\mu^{(1)}+J_\mu^{(2)} +J_\mu^{(3)}$, where
\begin{align*}
J_\mu^{(1)} &=\mu^\ep \hn F(x) \int M_\mu(x-y) G_{\le \mu}(y) dy,\\
J_\mu^{(2)} &=\mu^\ep \p_j F(x) \int \hn M_\mu(x-y) (x-y)^j G_{\le \mu}(y) dy,\\
J_\mu^{(3)} &=\mu^\ep \int \hn M_\mu(x-y) (x-y)^j \int_0^1 \left[\p_j F(x-\tau(x-y))-\p_j F(x)\right] d\tau G_{\le \mu}(y) dy.
\end{align*}
By using $\int |x||\hn M_\mu(x)| dx \les 1$ and the frequency localization of $M_\mu$ and $G_\la$ we can obtain
$$
\|J_\mu^{(1)}\|_{L^2}+\|J_\mu^{(2)}\|_{L^2} \les \mu^\ep \|G_\mu\|_{L^2} \|\hn F\|_{L^\infty},
$$
which gives
$$
\|J_\mu^{(1)}\|_{l_\mu^2 L^2} +\|J_\mu^{(2)}\|_{l_\mu^2 L^2}\les \|\hn F\|_{L^\infty} \|\La^\ep G\|_{L^2}.
$$
For the term $J_\mu^{(3)}$, we can write
\begin{align*}
J_\mu^{(3)}(x) &= \mu^\ep \sum_{\la\le\mu} \int \hn M_\mu(x-y) (x-y)^j \\
&\qquad \times \int_0^1 \left[\p_j F_{\le \la}(x-\tau (x-y))-\p_j F_{\le \la}(x)\right]d\tau G_\la (y) dy\\
&+ \mu^\ep \sum_{\la\le \mu} \int \hn M_\mu (x-y) (x-y)^j \\
&\qquad \times \int_0^1 \left[\p_j F_{>\la}(x-\tau (x-y))-\p_j F_{> \la}(x)\right]d\tau G_\la (y) dy.
\end{align*}
Now we can use the similar arguments in the proof of the first part of this lemma to obtain
\begin{align*}
\|J_\mu^{(3)}\|_{L^2} &\les \mu^{-1+\ep} \sum_{\la\le \mu} \|\hn^2 F_{\le\la}\|_{L^\infty} \|G_\la\|_{L^2}
+\mu^{-1+\ep} \sum_{\la\le\mu} \|\hn^2 F_{>\la}\|_{L^2} \|G_\la\|_{L^\infty}\\
&\les \|\hn F\|_{L^\infty} \sum_{\la\le\mu} \left(\frac{\la}{\mu}\right)^{1-\ep} \|\la^\ep G_\la\|_{L^2}  \\
&\quad \, +\sum_{\la\le\mu} \sum_{\la\le \la'} \left(\frac{\la}{\mu}\right)^{1-\ep} \left(\frac{\la}{\la'}\right)^{1/2+\ep}
\|\la'^{1/2+\ep} \hn^2 F_{\la'}\|_{L^2} \|G_\la\|_{L^2}.
\end{align*}
By taking the $l_\mu^2$-norm we obtain
$$\|J_\mu^{(3)}\|_{l_\mu^2 L^2}\les \|\hn^2 F\|_{H^{1/2+\ep}} \|G\|_{L^2}
+\|\hn F\|_{L^\infty}\|G\|_{H^\ep}.$$ This together with the estimates on
$J_\mu^{(1)}$ and $J_\mu^{(2)}$ gives
$$
\|J_\mu\|_{l_\mu^2 L^2} \les \|\hn F\|_{L^\infty} \|G\|_{H^\ep} +\|\hn^2 F\|_{H^{1/2+\ep}} \|G\|_{L^2}.
$$
Combining this with the estimate on $I_\mu$ completes the proof of the second part of (\ref{epeq3}).
\end{proof}

\begin{proposition}
For $0<\ep <1/2$ there holds
\begin{equation}\label{comm10}
\|\mu^{1+\ep}\hn [P_\mu, F] G\|_{l_\mu^2 L^2}
\les \|\hn F\|_{L^\infty} \|G\|_{H^{1+\ep}} +\|\hn^2 F\|_{H^{\f12+\ep}}\|G\|_{H^1}.
\end{equation}
\end{proposition}

\begin{proof}
As can be seen from the proof of the second part of (\ref{epeq3}), it suffices to estimate
the term
\begin{align*}
a_\mu:=\mu J^{(3)}_\mu
&= \mu^{1+\ep} \int \hn M_\mu (x-y) (x-y)^j(x-y)^l \\
&\qquad \times \int_0^1 \int_0^1 -\tau \p^2_{jl} F(x-\tau \tau'(x-y)) d\tau d\tau' G_{\le \mu}(y) dy
\end{align*}
which can be written as $a_\mu=a_\mu^{(1)}+a_\mu^{(2)} +a_\mu^{(3)}$, where
\begin{align*}
a_\mu^{(1)} &=\mu^{1+\ep}\int\hn M_\mu (x-y)(x-y)^j (x-y)^l \\
&\qquad \times \int_0^1\int_0^1-\tau \p^2_{jl} F_{\ge \mu}\left(x-\tau\tau'(x-y)\right) d\tau d\tau' G_{\le \mu}(y)dy\\
a_\mu^{(2)} &=\mu^{1+\ep} \p_{jl}^2 F_{<\mu}(x) \int\hn M_\mu(x-y)(x-y)^j (x-y)^l G_{\le \mu}(y)dy\\
a_\mu^{(3)} &= \mu^{1+\ep} \int\hn M_\mu(x-y)(x-y)^j (x-y)^l \\
&\qquad \times \int_0^1\int_0^1-\tau \left[\p^2_{jl} F_{<\mu}(x-\tau\tau'(x-y))-\p_{jl}^2 F_{<\mu}(x)\right] d\tau d\tau'
G_{\le \mu}(y)dy.
\end{align*}
By using the properties of the LP projections, it is easy to derive that
\begin{align*}
\|a_\mu^{(1)}\|_{L^2}
& \les \mu^\ep \sum_{\la\ge \mu} \|\hn^2 F_\la\|_{L^2}\|G_{\le \mu}\|_{L^\infty}
\les \sum_{\la>\mu} \left(\frac{\mu}{\la}\right)^{1/2+\ep} \|\la^{1/2+\ep}\hn^2 F_\la\|_{L^2} \|\hn G\|_{L^2},\\
\|a_\mu^{(2)}\|_{L^2}
&\les \mu^\ep \|\hn^2 F_{<\mu}\|_{L^\infty} \|G_\mu\|_{L^2}\les \mu^{1+\ep} \|G_\mu\|_{L^2} \|\hn F\|_{L^\infty},\\
\|a_\mu^{(3)}\|_{L^2}
& \les \mu^{-1+\ep} \|\hn^3 F_{<\mu}\|_{L^2} \|G_{\le \mu}\|_{L^\infty}
\les\sum_{\la< \mu} \left(\frac{\la}{\mu}\right)^{1/2-\ep} \|\la^{1/2+\ep} \hn^2 F_\la\|_{L^2} \|\hn G\|_{L^2}.
\end{align*}
Therefore, by taking the $l_\mu^2$-norm, we obtain
\begin{equation*}
\|a_\mu^{(1)}\|_{l_\mu^2 L^2} +\|a_\mu^{(3)}\|_{l_\mu^2 L^2} \les \|\hn^2 F\|_{H^{1/2+\ep}}\|G\|_{H^1}, \quad
\|a_\mu^{(2)}\|_{l_\mu^2 L^2}\les \|\hn F\|_{L^\infty} \|G\|_{H^{1+\ep}}.
\end{equation*}
The proof is thus complete.
\end{proof}

\begin{lemma}\label{prd3}
For $0<\ep<1$ there holds
\begin{equation*}
\|\mu^\ep [P_\mu, F]G\|_{l_\mu^2 L^2}\les \|F\|_{L^\infty}\| G\|_{H^\ep}+\|\hn F\|_{H^{1/2+\ep}} \|G\|_{L^2}.
\end{equation*}
\end{lemma}

\begin{proof}
First we have
\begin{align*}
\mu^\ep \sum_{\la>\mu} \|P_\mu(F_\la\cdot G_\la)\|_{L^2}
&\les \mu^\ep \sum_{\la>\mu} \|F_\la\|_{L^\infty} \|G_\la\|_{L^2}
\les \|F\|_{L^\infty} \sum_{\la>\mu} \left(\frac{\mu}{\la}\right)^\ep \|\la^\ep G_\la\|_{L^2}
\end{align*}
which implies that
\begin{equation*}
\left\|\mu^\ep \sum_{\la>\mu} P_\mu(F_\la\cdot G_\la)\right\|_{l_\mu^2 L^2} \les
\|F\|_{L^\infty}\|G\|_{H^\ep}.
\end{equation*}
Next we consider the term
$$
[P_\mu, F] G_{\le \mu}=\int M_\mu(x-y) [F(x)-F(y)] G_{\le \mu}(y) dy
$$
which can be decomposed as follows
\begin{align*}
[P_\mu, F] G_{\le \mu}(x) &= \sum_{\la\le \mu} \int M_\mu(x-y)
\left(F_{\le \la}(x)-F_{\le \la}(y)\right) G_\la(y) dy\\
&+ \sum_{\la\le \mu} \int M_\mu(x-y) \left(F_{>\la}(x)-F_{>\la}(y)\right) G_\la(y) dy.
\end{align*}
We can estimate as before to obtain
\begin{align*}
\mu^\ep  \|& [P_\mu, F] G_{\le \mu} \|_{L^2} \\
& \les \mu^{-1+\ep} \sum_{\la\le \mu} \|\hn F_{\le \la}\|_{L^\infty} \|G_\la\|_{L^2}
+\mu^{-1+\ep} \sum_{\la\le \mu} \|\hn F_{>\la}\|_{L^2} \|G_\la\|_{L^\infty} \\
&\les \sum_{\la\le \mu} \left(\frac{\la}{\mu}\right)^{1-\ep} \|\la^\ep G_\la\|_{L^2} \|F\|_{L^\infty}\\
&+\sum_{\la\le \mu} \sum_{\la'>\la} \left(\frac{\la}{\mu}\right)^{1-\ep} \left(\frac{\la}{\la'}\right)^{1/2+\ep}
\|\la'^{1/2+\ep}\hn F_{\la'}\|_{L^2} \|G_\la\|_{L^2}.
\end{align*}
Taking the $l_\mu^2$-norm gives
\begin{equation*}
\|\mu^\ep [P_\mu, F] G_{\le \mu}\|_{l_\mu^2 L^2}
\les \|F\|_{L^\infty} \|G\|_{H^\ep} + \|\hn F\|_{H^{1/2+\ep}} \|G\|_{L^2}.
\end{equation*}
The proof is therefore complete.
\end{proof}

\begin{lemma}\label{prodad1}
For any $\ep>0$ and any scalar functions $F$ and $G$, there holds
\begin{equation}\label{prodad.1}
\|\mu^\ep P_\mu(\hn F\c G)\|_{l_\mu^2 L^2}
\les \|\hn F\|_{H^{\f12+\ep}}\|G\|_{H^1}+\|F\|_{L^\infty}\|G\|_{H^{1+\ep}}.
\end{equation}
\end{lemma}

\begin{proof}
By the trichotomy law, we can write
\begin{align*}
P_\mu(\hn F\c G)&=a_\mu+b_\mu+c_\mu\\
&=P_\mu\left((\hn F)_\mu\c G_{\le \mu}\right)+\sum_{\la>\mu}P_\mu\left(\hn F_\la \c G_\la\right)
+P_\mu\left(\hn F_{\le \mu} \c G_\mu\right).
\end{align*}
For the terms $b_\mu$ and $c_\mu$, it is easy to derive that
\begin{align*}
\|\mu^\ep b_\mu\|_{l_\mu^2 L^2} \les \|F\|_{L^\infty}\|G\|_{H^{1+\ep}}, \qquad
\|\mu^\ep c_\mu\|_{l_\mu^2 L^2} \les \|F\|_{L^\infty}\|G\|_{H^{1+\ep}}.
\end{align*}
For the term $a_\mu$, we can write
\begin{equation}\label{decompa}
a_\mu=[P_\mu, G_{\le\mu}] \hn F_\mu+G_{\le \mu} P_\mu (\hn F)_\mu.
\end{equation}
By the Bernstein inequality for LP projections, it is easy to obtain
\begin{equation*}
\|\mu^\ep G_{\le \mu} P_\mu (\hn F)_\mu\|_{L^2}
\les \mu^\ep \|G_{\le \mu}\|_{L^6}\|P_\mu (\hn F)_\mu\|_{L^3}\les \|G\|_{L^6}\|\mu^{\f12+\ep}(\hn F)_\mu\|_{L^2},
\end{equation*}
while by using Corollary \ref{base2com} we have
\begin{align*}
\|\mu^\ep [P_\mu, G_{\le\mu}](\hn F)_\mu\|_{L^2}
&\les \mu^{\ep-1}\sum_{\la \le \mu} \|\hn G_\la\|_{L^\infty} \|(\hn F)_\mu\|_{L^2}\\
&\les \sum_{\la \le \mu} \left(\frac{\la}{\mu}\right)^{\frac{3}{2}}\|\hn G_\la\|_{L^2}
\|\mu^{\f12+\ep}(\hn F)_\mu\|_{L^2}.
\end{align*}
Therefore
\begin{equation*}
\|\mu^\ep a_\mu\|_{l_\mu^2 L_x^2}\les\|\hn F\|_{H^{\f12+\ep}}\|G\|_{H^1}.
\end{equation*}
The combination of the estimates for $a_\mu, b_\mu$ and $c_\mu$ gives (\ref{prodad.1}).
\end{proof}

By using Lemma \ref{prd3} and (\ref{prodad.1}), we can derive the following product estimate.

\begin{lemma}\label{lem22}
For $0<\ep<1$ there holds
\begin{align}
\|\La^\ep\hn(F\c G)\|_{L^2} \les \|\hn F\|_{H^{\f12+\ep}}\| G\|_{H^1}
+\|F\|_{L^\infty}\|G\|_{H^{1+\ep}}\label{prodd.1}
\end{align}
\end{lemma}

\begin{proof}
Observe that
$$
\|\La^\ep \hn (F\cdot G)\|_{L^2}\les \|\mu^\ep P_\mu (\hn F \c G)\|_{L_\mu^2 L^2}
+ \|\mu^\ep P_\mu (F\c \hn G)\|_{l_\mu^2 L^2}.
$$
From Lemma \ref{prd3} it follows that
\begin{align*}
\|\mu^\ep P_\mu(F\c \hn G)\|_{l_\mu^2 L^2}
&\le \|\mu^\ep [P_\mu, F] \hn G\|_{l_\mu^2 L^2}+\|F\|_{L_x^\infty}\|\mu^\ep P_\mu \hn G\|_{l_\mu^2 L^2}\\
&\les \|F\|_{L^\infty}\|\hn G\|_{H^\ep}+\|\hn F\|_{H^{\f12+\ep}}\|\hn G\|_{L^2}.
\end{align*}
This together with (\ref{prodad.1}) then completes the proof.
\end{proof}

\begin{lemma}\label{err12}
Let $0<\ep <1/2$ and $p>3$ be sufficiently close to $3$. Then for any $\mu\ge 1$ and any
scalar functions $F$ and $G$, there hold
\begin{align}
\|[P_\mu, F]G\|_{L^\infty}
&\les \mu^{-\f12-\ep}\|G\|_{L^\infty} \sum_{\la>\mu}
\left(\frac{\mu}{\la}\right)^{\frac{3}{p}-\frac{1}{2}+\ep} \|\la^\ep \hn^2 F_\la\|_{L^2} \nn \\
&+ \mu^{-\frac{1}{2}-\ep} \|G\|_{L^\infty} \sum_{\la\le \mu}
\left(\frac{\la}{\mu}\right)^{\frac{1}{2}-\ep} \|\la^\ep \hn^2 F_\la\|_{L^2} \label{eqn20}
\end{align}
and
\begin{align}
\|[P_\mu, F]G\|_{L^\infty}
&\les \mu^{-\frac{1}{2}-\ep} \|\hn F\|_{L^\infty} \sum_{\la>\mu}
   \left(\frac{\mu}{\la}\right)^{2+\ep} \|\la^\ep \hn G_\la\|_{L^2} \nn \\
&+ \mu^{-\frac{1}{2}-\ep} \|\hn F\|_{L^\infty} \sum_{\la\le \mu}
   \left(\frac{\la}{\mu}\right)^{\frac{1}{2}-\ep} \|\la^\ep \hn G_\la\|_{L^2} \nn \\
&+ \mu^{-\frac{3}{p}-\f12-\ep}\|G\|_{H^1} \sum_{\la>\mu}
   \left(\frac{\mu}{\la}\right)^{\frac{3}{p}+\ep} \|\la^{\frac{1}{2}+\ep} \hn^2 F_\la\|_{L^2}.\label{eqn21}
\end{align}
where $C>0$ is a constant depending only on $\ep$ and $p$.
\end{lemma}

\begin{proof}
In view of  (\ref{bdcp}), we can write $[P_\mu F]G=a_\mu+b_\mu +c_\mu$, where
$$
a_\mu=\sum_{\la>\mu} P_\mu(F_\la G_\la), \quad b_\mu=\sum_{\la>\mu}[P_\mu, F_\la]G_{\le \mu},\quad
c_\mu=\sum_{\la\le \mu}[P_\mu, F_\la] G_{\le \mu}.
$$
It is easy to derive that
\begin{equation*}
\|a_\mu\|_{L^\infty} \les \mu^{-\f12-\ep}\sum_{\la>\mu} \left(\frac{\mu}{\la}\right)^{2+\ep}
\|\la^{\ep}\hn^2 F_\la\|_{L_x^2}\|G_\la\|_{L^\infty}.
\end{equation*}
By using Corollary \ref{base2com} and the Bernstein inequality, we also have
\begin{align*}
\|c_\mu\|_{L^\infty} &\les \|G\|_{L^\infty}\sum_{\la\le \mu}\mu^{-1}\|\hn F_\la\|_{L^\infty}
\les \mu^{-\f12-\ep}\|G\|_{L^\infty} \sum_{\la\le \mu} \left(\frac{\la}{\mu}\right)^{\f12-\ep}
\|\la^{\ep}\hn^2 F_\la\|_{L^2}.
\end{align*}
Since $p>3$, we have the Sobolev embedding
$$
\|b_\mu\|_{L^\infty}\les \|\hn b_\mu\|_{L^p}+\|b_\mu\|_{L^2}.
$$
From (\ref{bdcp2}) and the properties of LP projections it follows that
\begin{align*}
\|\hn b_\mu\|_{L^p} &\les\|G\|_{L^\infty} \sum_{\la>\mu} \|\hn F_\la\|_{L^p}
\les \|G\|_{L^\infty}\sum_{\la>\mu} \la^{-\frac{3}{p}+\f12} \|\hn^2 F_\la\|_{L^2}\\
&\les \mu^{-\frac{3}{p}-\ep-\f12}\|G\|_{L^\infty}
\sum_{\la>\mu} \left(\frac{\mu}{\la}\right)^{\frac{3}{p}-\f12+\ep} \|\la^\ep\hn^2 F_\la\|_{L^2}.
\end{align*}
The term $\|b_\mu\|$ can be estimated similarly but much easier. We therefore complete the proof of (\ref{eqn20}).

Similarly, in order to obtain (\ref{eqn21}), we can use the properties of LP projections and
 Corollary \ref{base2com} to derive that
\begin{align*}
\|a_\mu\|_{L^\infty} &\les \mu^{-\f12-\ep} \|\hn F\|_{L^\infty}
\sum_{\la>\mu} \left(\frac{\mu}{\la}\right)^{2+\ep} \|\la^\ep \hn G_\la\|_{L^2},
\end{align*}
\begin{align*}
\|c_\mu\|_{L^\infty}&\les \mu^{-1} \sum_{\la\le \mu} \|\hn F_{\le \mu}\|_{L^\infty} \|G_\la\|_{L^\infty}
\les \mu^{-\f12-\ep}\|\hn F\|_{L^\infty} \sum_{\la\le \mu} \left(\frac{\la}{\mu}\right)^{\f12-\ep}
\|\la^{\ep} \hn G_\la\|_{L^2}.
\end{align*}
\begin{align*}
\|\hn b_\mu\|_{L^p}&\les\sum_{\la\le \mu}\sum_{\ell>\mu}\|G_\la\|_{L^\infty} \|\hn F_{\ell}\|_{L^p}
\les \sum_{\ell\le \mu}\sum_{\la>\mu}\ell^{\f12}\|\hn G_{\ell}\|_{L^2} \la^{-\frac{3}{p}-\ep}
\|\la^{\f12+\ep}\hn^2 F_\la\|_{L^2}\\
&\les \mu^{-\frac{3}{p}-\ep+\f12} \|G\|_{H^1} \sum_{\ell \le \mu}\sum_{\la>\mu} \left(\frac{\ell}{\mu}\right)^{\f12}
\left(\frac{\mu}{\la}\right)^{\frac{3}{p}+\ep}  \|\la^{\ep+\f12}\hn^2 F_\la\|_{L^2}.
\end{align*}
Hence the proof of (\ref{eqn21}) is completed.
\end{proof}

\subsection{$H^\ep$ elliptic estimates}

\begin{lemma}\label{ellp1}
For $0<\ep<1/2$ and any $\Sigma$-tangent tensor field $F$ there hold
\begin{align}
\|\hn^2 F\|_{\dot H^\ep} &\les \|\hdt F\|_{\dot H^\ep}+\|F\|_{H^1},\label{elep1}\\
\|\hn F\|_{\dot H^{1/2+\ep}} &\les \|\La^{\ep-1/2}\hdt F\|_{L^2}+\|F\|_{H^1}.\label{elep2}
\end{align}
\end{lemma}

\begin{proof}
Consider (\ref{elep1}) first. We will use $\er_\mu$ to denote any
error term satisfying
$$
\|\er_\mu\|_{ L^2} \les \mu^{-1}\| F\|_{H^1}.
$$
Using Corollary \ref{base2com}, we can obtain
\begin{equation}\label{phess1}
P_\mu \hn^2 F=\hn^2 P_\mu F+\er_\mu.
\end{equation}
Recall from Lemma \ref{Cor8.12.1} that
\begin{align}\label{8.12.3}
\|\hn^2 P_\mu F\|_{L^2} &\les \|\hdt P_\mu F\|_{L^2} +\|\hn P_\mu F\|_{L^2}
+\|P_\mu F\|_{L^2}
\end{align}
Recall also that $\hdt=g^{ij} \hn_i\hn_j$, we have
\begin{equation}\label{8.12.4}
\hdt P_\mu F=P_\mu \hdt F+[P_\mu, g^{ij}](\p_i\p_j
F-\hat\Ga_{ij}^k \p_k F)+ \er_\mu.
\end{equation}
By Lemma \ref{Lem3.7} and \ref{prd3}, we have
\begin{align}
&\|\mu^\ep [P_\mu, g^{ij}]\p_i \p_j F\|_{l_\mu^2
L^2} \les\|\p g\|_{H^1} \|\La^{\ep+1/2} \p F\|_{L^2},\label{er.1}\\
&\|\mu^\ep[P_\mu, g^{ij}](\hat \Ga_{ij}^k \p_k F)\|_{l_\mu^2 L^2} \les \|\p g\|_{H^{\f12+\ep}}\|\p F\|_{L^2}\label{er.2}.
\end{align}
By the Sobolev inequality and (\ref{8.12.4}), we obtain $\hdt P_\mu F=P_\mu \hdt F +\er_\mu$.
This together with (\ref{phess1}) and (\ref{8.12.3}) gives (\ref{elep1}).

To prove (\ref{elep2}),  we first use the equivalence between $g$ and $\hat{g}$ and
the integration by parts to obtain
\begin{equation}\label{8.12.5}
\|P_\mu \hn F\|_{L^2}^2 \approx  \int_{\Sigma} g^{ij} P_\mu \p_i F P_\mu \p_j F d\mu_g
=\int_{\Sigma}  P_\mu F\hdt P_\mu F d \mu_g.
\end{equation}
 In view of (\ref{commt8}) and (\ref{er.2}), we can obtain
\begin{align*}
\sum_{\mu} \mu^{1+2\ep} & \left|\int_{\Sigma}  P_\mu F [P_\mu, g^{ij}]\p_i\p_j F d \mu_g\right|\\
& \les \sum_\mu \mu^{3/2+\ep} \|P_\mu F\|_{L^2} \|\mu^{-1/2+\ep} [P_\mu, g^{ij}]\p_i \p_j F\|_{L^2}\\
&\les \|\hn F\|_{H^{1/2+\ep}} \|\hn g\|_{H^1} \|\La^\ep \p F\|_{L^2},
\end{align*}
 and in view of (\ref{lem4eq})
\begin{align*}
\sum_\mu \mu^{1+2\ep}  \left|\int P_\mu F\c [P_\mu, g\hat{\Ga}]  \p F\right|
&\les \|\mu^{\f12+\ep}P_\mu F\|_{l_\mu^2 L^2}\|\hn(g\c \hat{\Ga})\|_{L^6}\|\p F\|_{H^\ep}\\
&\les \|F\|_{H^{\f12+\ep}}\|\p F\|_{H^\ep}\| g\|_{H^2}.
\end{align*}
In view of (\ref{8.12.4}) and (\ref{8.12.5}), we have with $p=\ep/(1/2+\ep)$ that
\begin{align*}
\sum_\mu \mu^{1+2\ep} \|P_\mu \hn F\|_{L^2}^2
&\les  (\|\ F\|_{{\dot{H}}^{\f12+\ep}}+\|\p F\|_{H^{1/2+\ep}}) \|\p F\|_{H^{1/2+\ep}}^{p}\|\p F\|_{L^2}^{1-p}\| g\|_{H^2}\\
&+\sum_\mu \|\mu^{3/2+\ep} P_\mu F\|_{L^2} \|\mu^{-1/2+\ep} P_\mu \hdt F\|_{L^2}.
\end{align*}
By  the fact
$\|g\|_{H^2} \les 1$ and the Young's inequality,  we obtain (\ref{elep2}).

To prove (\ref{elep2}) for the vector field case, we note that
$$
P_\mu \hn_i F^m =P_\mu \hn_i(F^m)+ [P_\mu, \hat \Ga] F+\hat\Ga\cdot P_\mu F.
$$
Using  Corollary \ref{base2com}, then there holds
$
P_\mu \hn_i F^m =P_\mu \hn_i(F^m)+ \er_\mu ,
$
hence we can obtain
$$
\|\hn F^m\|_{H^{1/2+\ep}} \les \|\hn (F^m)\|_{H^{1/2+\ep}} +\|F\|_{H^1}.
$$
Now we can use (\ref{elep2}) for the scalar function case to derive
\begin{equation}\label{8.14.2}
\|\hn F^m\|_{H^{1/2+\ep}}\les \|\La^{\ep-1/2} \hdt (F^m)\|_{L^2} +\|F\|_{H^1}.
\end{equation}
In view of (\ref{8.12.4}), we can obtain
\begin{equation*}
\|\La^{\ep-1/2} (\hdt F^m)\|_{L^2}\les \|\La^{\ep-1/2} (\hdt F)^m\|_{L^2}+\|\hn F\|_{L^2}.
\end{equation*}
combining this with (\ref{8.14.2}) completes the proof.
\end{proof}

\end{document}